\documentclass[10pt]{amsart}
\usepackage{amssymb}
\usepackage{amsmath}
\usepackage{amsthm}
\usepackage{mathrsfs}
\usepackage{MnSymbol}
\usepackage{accents}
\usepackage[T1]{fontenc}
\usepackage{verbatim}
\usepackage{setspace}
\usepackage{color}
\usepackage[pdftex]{graphicx}
\usepackage{wasysym}
\usepackage{upgreek}
\usepackage{fancyhdr}
\usepackage{hyperref}
\usepackage{times}

\numberwithin{equation}{subsection}

\newtheorem{theorem}{Theorem}
\newtheorem{proposition}{Proposition}[section]
\newtheorem{lemma}[proposition]{Lemma}
\newtheorem{corollary}[proposition]{Corollary}

\theoremstyle{definition}
\newtheorem{definition}{Definition}[section]
\newtheorem{remark}{Remark}[section]

\theoremstyle{remark}

\renewcommand{\theproposition}{\arabic{section}-\arabic{proposition}}

\newcommand{\eqdef}{\overset{\mbox{\tiny{def}}}{=}}

\newcommand{\dParameter}{\ell}

\newcommand{\Ldual}{{^{\star \hspace{-.03in}} \mathscr{L}}}
\newcommand{\Far}{\mathcal{F}}
\newcommand{\Gar}{\mathcal{G}}
\newcommand{\Max}{\mathcal{M}}
\newcommand{\Fardual}{{^{\star \hspace{-.03in}}\mathcal{F}}}
\newcommand{\FarMinkdual}{{^{\ostar \hspace{-.03in}}\mathcal{F}}}

\newcommand{\Maxdual}{{^{\star \hspace{-.05in}}\mathcal{M}}}
\newcommand{\MaxMinkdual}{{^{\ostar \hspace{-.05in}}\mathcal{M}}}
\newcommand{\Stress}{\dot{Q}}

\newcommand{\ualpha}{\underline{\alpha}}
\newcommand{\uL}{\underline{L}}

\newcommand{\ualphadot}{{^{\odot} \underline{\alpha}}}
\newcommand{\alphadot}{{^{\odot} \alpha}}
\newcommand{\rhodot}{{^{\odot} \rho}}
\newcommand{\sigmadot}{{^{\odot} \sigma}}

\newcommand{\unabla}{\underline{\nabla}}
\newcommand{\conenabla}{\overline{\nabla}}

\newcommand{\angn}{\not \nabla}

\newcommand{\Minkvolume}{\upsilon}
\newcommand{\uvolume}{\underline{\upsilon}}
\newcommand{\angupsilon}{{\not \upsilon}}

\newcommand{\angm}{\not m}
\newcommand{\um}{\underline{m}}

\newcommand{\coneproject}{\overline{\pi}}

\newcommand{\nablamod}{\hat{\nabla}}
\newcommand{\Lie}{\mathcal{L}}
\newcommand{\Liemod}{\hat{\mathcal{L}}}

\newcommand{\Electricfield}{E}
\newcommand{\Displacement}{D}
\newcommand{\Magneticinduction}{B}
\newcommand{\Magneticfield}{H}

\newcommand{\Farinvariant}{\mbox{\lightning}}

\newcommand{\decayparameter}{\upkappa}

\newcommand{\myarray}[2][]{\Big(
		\begin{array}{lr}
    	 #1 \\
    	 #2 
     \end{array} \Big)}

\begin{document}
\pagestyle{fancy}
\title{The Global Stability of the Minkowski Spacetime Solution to the Einstein-Nonlinear Electromagnetic System in Wave Coordinates}
\address{Princeton University, Department of Mathematics,
Fine Hall, Washington Road, Princeton, NJ 08544-1000\\ \indent USA}
\email{jspeck@math.princeton.edu}
\author{ Jared Speck$^*$}
\thanks{$^*$Princeton University, Department of Mathematics, Fine Hall, Washington Road, Princeton, NJ 08544-1000, USA}
\thanks{The author was supported in part by the Commission of the European Communities, ERC Grant Agreement No 208007. He was also supported in part by an NSF All-Institutes Postdoctoral Fellowship administered by the Mathematical Sciences Research Institute through its core grant DMS-0441170.}

\begin{abstract}
In this article, we study the coupling of the Einstein field equations of general relativity to a family of models of 
nonlinear electromagnetic fields. The family comprises all covariant electromagnetic models that satisfy the following criteria: they are derivable from a sufficiently regular Lagrangian, they reduce to the linear Maxwell model in the weak-field limit, and their corresponding energy-momentum tensors satisfy the dominant energy condition. Our main result is a proof of the global nonlinear stability of the $1 + 3-$dimensional Minkowski spacetime solution to the coupled system for any member of the family, which includes the linear Maxwell model. This stability result is a consequence of a small-data global existence result for a reduced system of equations that is equivalent to the original system in our wave coordinate gauge. Our analysis of the spacetime metric components is based on a framework recently developed by Lindblad and Rodnianski, which allows us to derive
suitable estimates for tensorial systems of quasilinear wave equations with nonlinearities that satisfy the weak null condition. Our analysis of the electromagnetic fields, which satisfy quasilinear first-order equations, is based on an extension of a geometric energy-method framework developed by Christodoulou, together with a collection of pointwise decay estimates for the Faraday tensor developed in the article. We work directly with the electromagnetic fields, and thus avoid the use of electromagnetic potentials.

\end{abstract}

\date{Version of \today}
\maketitle

\setcounter{tocdepth}{2}

\pagenumbering{roman} 
\tableofcontents 
\newpage 

\pagenumbering{arabic}

\section{Introduction} \label{S:Introduction} 
The Einstein field equations of general relativity connect the \emph{Einstein tensor} $R_{\mu \nu} - \frac{1}{2}g_{\mu \nu} R,$
which contains information about the curvature of spacetime\footnote{By spacetime, we mean a four-dimensional time-oriented Lorentzian manifold $\mathfrak{M}$ together with a Lorentzian metric $g_{\mu \nu}$ of signature $(-,+,+,+).$} $(\mathfrak{M},g_{\mu \nu}),$ to the energy-momentum-stress-density tensor (energy-momentum tensor for short) $T_{\mu \nu},$ which contains information about the matter present in $\mathfrak{M}.$ Here, $g_{\mu \nu}$ is the \emph{spacetime metric}, $R_{\mu \nu}$ is the \emph{Ricci curvature tensor}, and $R = (g^{-1})^{\kappa \lambda} R_{\kappa \lambda}$ is the \emph{scalar curvature}. In this article, we show the stability of the $1 + 3-$dimensional vacuum Minkowski spacetime solution of the Einstein-nonlinear electromagnetic system

\begin{subequations}
\begin{align} 
	R_{\mu \nu} - \frac{1}{2}g_{\mu \nu} R & = T_{\mu \nu}, && (\mu, \nu = 0,1,2,3), 
		\label{E:IntroEinstein} \\
	(d \Far)_{\lambda \mu \nu} & = 0, && (\lambda, \mu, \nu = 0,1,2,3), \label{E:IntrodFaris0} \\
	(d \Max)_{\lambda \mu \nu} & = 0, && (\lambda, \mu, \nu = 0,1,2,3), \label{E:IntrodMis0}
\end{align}
\end{subequations}
where $T_{\mu \nu}$ (see \eqref{E:electromagnetictensorloweroinTermsofLagrangian}) is one of the energy-momentum tensors corresponding to a family of nonlinear models of electromagnetism, $d$ denotes the exterior derivative operator, the two-form $\Far_{\mu \nu}$ denotes the \emph{Faraday tensor}, the two-form $\Max$ denotes the \emph{Maxwell tensor}, and $\Max_{\mu \nu}$ is connected to $(g_{\mu \nu},\Far_{\mu \nu})$ through a constitutive relation. We make the following three assumptions concerning the electromagnetic matter model: $(i)$ its Lagrangian $\Ldual$ is a scalar-valued function of the two electromagnetic invariants\footnote{Throughout the article, we use Einstein's summation convention in that repeated indices are summed over.} $\Farinvariant_{(1)} \eqdef \frac{1}{2} (g^{-1})^{\kappa \mu} (g^{-1})^{\lambda \nu} \Far_{\kappa \lambda} \Far_{\mu \nu},$ $\Farinvariant_{(2)} \eqdef \frac{1}{4} (g^{-1})^{\kappa \mu} (g^{-1})^{\lambda \nu} \Far_{\kappa \lambda} \Fardual_{\mu \nu},$ where $\star$ denotes the Hodge duality operator corresponding to $g_{\mu \nu};$ $(ii)$ the energy-momentum tensor $T_{\mu \nu}$ corresponding to $\Ldual$ satisfies the \emph{dominant energy condition} (sufficient conditions on $\Ldual$ are given in \eqref{E:DECL1} - \eqref{E:DECTrace} below);
$(iii)$ $\Ldual$ is a sufficiently differentiable function of $(\Farinvariant_{(1)},$ $\Farinvariant_{(2)}),$ and its Taylor expansion around $(0,0)$ agrees with that of the linear Maxwell-Maxwell equations to first order; i.e., $\Ldual(\Farinvariant_{(1)}, \Farinvariant_{(2)}) = - \frac{1}{2} \Farinvariant_{(1)} + O^{\dParameter + 2}\big(|(\Farinvariant_{(1)}, \Farinvariant_{(2)})|^2 \big),$ where $\dParameter \geq 8$ is an integer; see Section \ref{SS:Oando} regarding the notation $O^{\dParameter + 2}(\cdots)$. We briefly summarize our main results here. They are rigorously stated and proved in Section \ref{S:GlobalExistence}.

\begin{changemargin}{.25in}{.25in} 
\textbf{Main Results.} \
The vacuum Minkowski spacetime background solution to the system \eqref{E:IntroEinstein} - \eqref{E:IntrodMis0} is globally stable. In particular, small perturbations of the trivial initial data corresponding to the background solution have maximal globally hyperbolic developments that are geodesically complete. Furthermore, the perturbed solution converges to the vacuum Minkowski spacetime solution. These conclusions are consequences of a small-data global existence result for the \emph{reduced} system \eqref{E:Reducedh1Summary} - \eqref{E:ReduceddMis0Summary}, which is equivalent to the study of \eqref{E:IntroEinstein} - \eqref{E:IntrodMis0} in a wave coordinate system (i.e., a coordinate system 
$\lbrace x^{\mu} \rbrace_{\mu = 0,1,2,3}$ on $\mathbb{R}^{1+3}$ satisfying 
$(g^{-1})^{\kappa \lambda} \mathscr{D}_{\kappa} \mathscr{D}_{\lambda} x^{\mu} = 0, (\mu = 0,1,2,3),$ where 
$\mathscr{D}$ is the Levi-Civita connection corresponding to $g_{\mu \nu}$).
\end{changemargin}

We recall the following standard facts (see e.g. \cite{dC2008}, \cite{rW1984}) concerning the initial data for the system \eqref{E:IntroEinstein} - \eqref{E:IntrodMis0}, which we refer to as ``abstract'' initial data. The abstract initial data consist of a $3-$dimensional manifold $\Sigma_0,$ together with the following fields on $\Sigma_0:$ a Riemannian metric $\mathring{\underline{g}}_{jk},$ a symmetric type $\binom{0}{2}$ tensorfield $\mathring{K}_{jk},$ and a pair of electromagnetic one-forms $\mathring{\mathfrak{D}}_j, \mathring{\mathfrak{B}}_j,$ $(j,k =1,2,3).$ Furthermore, they must satisfy the \emph{Gauss}, \emph{Codzazzi}, and \emph{electromagnetic} constraint equations, which are respectively given by

\begin{subequations}
\begin{align}
	\underline{\mathring{R}} - \mathring{K}_{ab} \mathring{K}^{ab} + \big[(\mathring{\underline{g}}^{-1})^{ab} \mathring{K}_{ab}\big]^2 & = 
		2T(\hat{N},\hat{N})|_{\Sigma_0}, && \label{E:GaussIntro} \\
	(\mathring{\underline{g}}^{-1})^{ab} \underline{\mathring{\mathscr{D}}}_a \mathring{K}_{bj} - (\mathring{\underline{g}}^{-1})^{ab}  \underline{\mathring{\mathscr{D}}}_j \mathring{K}_{ab} & = 
		T(\hat{N},\frac{\partial}{\partial x^j})|_{\Sigma_0}, && (j=1,2,3), \label{E:CodazziIntro} \\
	(\mathring{\underline{g}}^{-1})^{ab} \underline{\mathring{\mathscr{D}}}_a \mathring{\mathfrak{\Displacement}}_b & = 0,&& 
		\label{E:DivergenceD0Intro} \\
	(\mathring{\underline{g}}^{-1})^{ab} \underline{\mathring{\mathscr{D}}}_a \mathring{\mathfrak{\Magneticinduction}}_b & = 0.&& 
		\label{E:DivergenceB0Intro}
\end{align}
\end{subequations}
In the above expressions, the indices are lowered and raised with $\mathring{\underline{g}}_{jk}$ and $(\mathring{\underline{g}}^{-1})^{jk},$ $\underline{\mathring{R}}$ denotes the scalar curvature of 
$\mathring{\underline{g}}_{jk},$ $\underline{\mathring{\mathscr{D}}}$ denotes the Levi-Civita connection corresponding to $\mathring{\underline{g}}_{jk},$ and $\hat{N}^{\mu}$ is the future-directed unit $g-$normal to $\Sigma_0$ (viewed as an embedded submanifold of $(\mathfrak{M},g_{\mu \nu})$). The one-forms $\mathring{\mathfrak{D}}_j$ and $\mathring{\mathfrak{B}}_j$
together form a geometric decomposition of $\Far_{\mu \nu}|_{\Sigma_0},$ and the right-hand sides of \eqref{E:GaussIntro} - \eqref{E:CodazziIntro} can be computed (in principle) in terms of $\mathring{\underline{g}}_{jk},$ $\mathring{\mathfrak{\Displacement}}_j,$ and $\mathring{\mathfrak{\Magneticinduction}}_j$ alone; see Section \ref{SS:EBDH} for more details concerning the relationship of $\mathring{\mathfrak{D}}_j$ and $\mathring{\mathfrak{B}}_j$ to $\Far_{\mu \nu}|_{\Sigma_0}.$ The dominant energy condition manifests itself along $\Sigma_0$ as the inequalities $T(\hat{N},\hat{N}) \geq 0$ and $T(\hat{N},\hat{N})^2 - (\mathring{\underline{g}}^{-1})^{ab} T(\hat{N},\frac{\partial}{\partial x^a})T(\hat{N},\frac{\partial}{\partial x^b}) \geq 0.$ 

In this article, we consider the case $\Sigma_0 = \mathbb{R}^3.$ We will construct spacetimes of the form $\mathfrak{M} = I \times \mathbb{R}^3,$ where $I$ will be a time interval, and $\Sigma_0$ will be a spacelike Cauchy hypersurface in $(\mathfrak{M},g_{\mu \nu}).$ The constraints \eqref{E:GaussIntro} - \eqref{E:CodazziIntro} are necessary to ensure that \eqref{E:IntroEinstein} can be satisfied along $\Sigma_0,$ while the constraints \eqref{E:DivergenceD0Intro} - \eqref{E:DivergenceB0Intro} are necessary to ensure that the electromagnetic equations \eqref{E:IntrodFaris0} - \eqref{E:IntrodMis0} can be satisfied along $\Sigma_0.$ Our stability criteria include both decay assumptions at $\infty$ and smallness assumptions for the abstract initial data. We provide here a description of our decay assumptions at $\infty,$ which are based on the assumptions of \cite{hLiR2010}; our smallness assumptions will be addressed in detail in Section \ref{S:SmallDataAssumptions}.


\textbf{Assumptions on the abstract initial data:}
We assume that there exists a global coordinate chart $x = (x^1,x^2,x^3)$ on $\Sigma_0 = \mathbb{R}^3,$ a real number $\decayparameter > 0,$ and an integer $\dParameter \geq 8$ such that (with $r \eqdef |x| \eqdef \big[(x^1)^2 + (x^2)^2 + (x^3)^2 \big]^{1/2}$ and $j,k = 1,2,3$)

\begin{subequations}
\begin{align}
	\mathring{\underline{g}}_{jk} & = \delta_{jk} + \mathring{\underline{h}}_{jk}^{(0)} + \mathring{\underline{h}}_{jk}^{(1)},
	&& \label{E:metricdataexpansion} \\
	\mathring{\underline{h}}_{jk}^{(0)} & = \chi(r) \frac{2M}{r} \delta_{jk}, && \chi(r) \ \mbox{is defined in} \ 
		\eqref{E:chidef}, \label{E:h0AbstractDataAsymptotics} \\
	\mathring{\underline{h}}_{jk}^{(1)} & =  o^{\dParameter+1}(r^{-1 - \decayparameter}), && \mbox{as} \ r \to \infty, \label{E:h1AbstractDataAsymptotics} \\
	\mathring{K}_{jk} & = o^{\dParameter}(r^{-2 - \decayparameter}), && \mbox{as} \ r \to \infty, \label{E:KAbstractDataAsymptotics} \\
	\mathring{\mathfrak{\Displacement}}_j & = o^{\dParameter}(r^{-2 - \decayparameter}), && \mbox{as} \ r \to \infty, \\
	\mathring{\mathfrak{\Magneticinduction}}_j & = o^{\dParameter}(r^{-2 - \decayparameter}),&& \mbox{as} \ r \to \infty,
	\label{E:BdecayAssumption}
\end{align}
\end{subequations}
where the meaning of $o^{\dParameter}(\cdots)$ is described in Section \ref{SS:Oando}.

The parameter $M$ in \eqref{E:metricdataexpansion}, which is known as the \emph{ADM mass}, is constrained by the following requirements: according to the \emph{positive mass theorem} of Schoen-Yau \cite{rSstY1979}, \cite{rSstY1981}, and Witten \cite{eW1981}, under the assumption that $T_{\mu \nu}$ satisfies the dominant energy condition, the only solutions $\mathring{\underline{g}}_{jk}$ to the constraint equations \eqref{E:GaussIntro} - \eqref{E:DivergenceB0Intro} that have an expansion of the form \eqref{E:metricdataexpansion} with the asymptotic behavior \eqref{E:h0AbstractDataAsymptotics} - \eqref{E:KAbstractDataAsymptotics} either have $M > 0,$ or have $M = 0$ and $\mathring{\underline{g}}_{jk} = \delta_{jk}.$ The groundbreaking work \cite{dCsK1993} of Christodoulou and Klainerman (which is discussed further in Section \ref{SSS:MathematicalComparisons}) demonstrated the stability of the Minkowski spacetime solution to the Einstein-vacuum equations in the case that the initial data are \emph{strongly asymptotically flat}, which corresponds to the parameter range $\decayparameter \geq 1/2$ in the above expansions. Our work here, which relies on the framework developed by Lindblad and Rodnianski in \cite{hLiR2010} (see Section \ref{SSS:MathematicalComparisons}), allows for the parameter range $\kappa > 0.$

In this article, we do not consider the issue of solving the constraint equations. 
To the best of our knowledge, under the restrictions on $\Ldual$ described at the beginning of Section \ref{S:Introduction},
there are presently no rigorous results concerning the construction of initial data on the manifold $\mathbb{R}^3$ that satisfy the constraints. However, we remark that for the Einstein-vacuum equations $T_{\mu \nu} \equiv 0,$ initial data that satisfy the constraints and that coincide with the standard Schwarzschild data

\begin{subequations}
\begin{align}
	\mathring{\underline{g}}_{jk} & = \big(1 + \frac{2M}{r} \big) \delta_{jk},\\
	\mathring{K}_{jk} & = 0
\end{align}
\end{subequations}
outside of the unit ball centered at the origin were shown to exist in \cite{pCeD2002erratum} - \cite{pCeD2002} and \cite{jC2000}. The stability of the Minkowski spacetime solution to the Einstein-vacuum equations for such data follows from the methods of the aforementioned works \cite{dCsK1993}, \cite{hLiR2010} (and its precursor \cite{hLiR2005}), and also from the \emph{conformal method} approach of Friedrich \cite{hF1986a}.

\begin{remark}
	The only role of the dominant energy condition in this article is to ensure the physical condition $M \geq 0;$ 
	we assume this physical condition throughout the article. However, although the 
	smallness of $|M|$ is needed to prove our global stability result, the sign of $M$ does not enter into the stability 
	analysis. In particular, if there existed small initial data with small negative ADM mass, we would still be able to prove 
	that the corresponding solution to the equations exists globally. Similarly, if we made the replacement
	$T_{\mu \nu} \rightarrow - T_{\mu \nu}$ in the reduced equations \eqref{E:Reducedh1Summary} - \eqref{E:ReduceddMis0Summary},
	we could still prove a small-data global existence result.
\end{remark}

\subsection{Comparison with previous work}

\subsubsection{Mathematical comparisons} \label{SSS:MathematicalComparisons}

Our result is an extension of a large and growing hierarchy of stability results for the ${1 + 3-}$dimensional
Minkowski spacetime solution to the Einstein equations, which began with the celebrated work \cite{dCsK1993} of Christodoulou and Klainerman, and which was later replicated by Klainerman and Nicol{\`o} in \cite{sKfN2003} using alternate techniques. Both of these proofs used a manifestly covariant framework for both the formulation of the problem and the derivation of the estimates. However, mathematically speaking, the closest relatives to the present article are the seminal works \cite{hLiR2005} and \cite{hLiR2010}, in which Lindblad and Rodnianski developed a technically simpler framework for showing the stability of the vacuum Minkowski spacetime solution of the Einstein-scalar field system using a \emph{wave coordinate} gauge. Although their decay estimates are not as precise as those of \cite{dCsK1993} and \cite{sKfN2003}, their work was \emph{much} shorter than its predecessors, yet is robust enough to allow for modifications, including the presence of the nonlinear electromagnetic fields examined in this article. We remark that many of the technical results we need are contained in \cite{hLiR2005} and \cite{hLiR2010} and we will often direct the reader to these works for their proofs. 

Other stability results in this vein include \cite{nZ2000}, in which Zipser extended the framework of \cite{dCsK1993} to show the stability of the vacuum Minkowski spacetime solution to the Einstein-Maxwell system, and \cite{lBnZ2009}, in which Bieri weakened the assumptions of \cite{dCsK1993} on the decay of the initial data at infinity. We also mention the works \cite{jL2008} (see also \cite{jL2006}, \cite{jL2009}), in which Loizelet used the framework of \cite{hLiR2005} and \cite{hLiR2010} to demonstrate the stability of the vacuum Minkowski spacetime solution of the Einstein-scalar field-Maxwell system in $1 + n,$ $n \geq 3,$ dimensions. Moreover, in spacetimes of dimension $1 + n,$ with $n \geq 5$ odd, it has been shown \cite{yCBpCjL2006} that the \emph{conformal method} can be used to show the stability of the Einstein-Maxwell system for initial data that coincide with the standard Schwarzschild data outside of a compact set. Roughly speaking, the conformal method is a way of mapping a global existence problem into a local existence problem. Whenever it is available, the method tends to give very precise information concerning the asymptotics of the global solutions. In particular, the results of \cite{yCBpCjL2006} provide a more detailed description of the asymptotics than the results of \cite{jL2008}.

We state with emphasis that the techniques used in this article differ in a fundamental way from those used by Loizelet in \cite{jL2008}. More specifically, in \cite{jL2008}, Loizelet analyzed the familiar linear Maxwell-Maxwell\footnote{Our use of the terminology ``Maxwell-Maxwell'' equations, which are commonly referred to as the ``Maxwell'' equations,
is explained in \cite{jS2010a}.} equations through the use of a four-potential\footnote{Recall that a four-potential is a one-form $A_{\mu}$ such that $\Far_{\mu \nu} = (dA)_{\mu \nu}.$} $A_{\mu}$ satisfying the \emph{Lorenz gauge} condition $(g^{-1})^{\kappa \lambda} \mathscr{D}_{\kappa} A_{\lambda} = 0,$ where $\mathscr{D}$ is the Levi-Civita connection corresponding to $g_{\mu \nu}.$ In Loizelet's analysis of the linear Maxwell-Maxwell equations, the Lorenz gauge leads to a system of linear wave equations for the components $A_{\mu}.$ Furthermore, these equations can be analyzed using the same techniques that are used in the study of the components of the metric (see equation \eqref{E:Reducedh1Summary}) and the scalar field. In particular, in Loizelet's case, Lemma \ref{L:weightedenergy} can be used to deduce suitable weighted energy estimates for the components $\nabla_{\mu} A_{\nu}.$ In contrast, as discussed in \cite{jS2010a}, it is not clear that Lorenz gauge can be used for the kinds of quasilinear electromagnetic field equations \eqref{E:IntrodMis0} studied in this article. More specifically, it is not clear that the Lorenz gauge in general leads to a hyperbolic formulation of the electromagnetic equations that is suitable for deriving the kinds of $L^2$ energy estimates needed for our analysis. For this reason, throughout this article, we work directly with the Faraday tensor. In particular, as described in detail in Section \ref{E:EOVandStress}, we use Christodoulou's geometric framework \cite{dC2000} to generate \emph{energy currents} that can be used to derive the kinds of $L^2$ estimates needed in our analysis. In this way, we prove Lemma \ref{L:weightedenergyFar}, which compensates for the fact that Lemma \ref{L:weightedenergy} is not generally available for controlling the electromagnetic quantities. We remark that there is another advantage to working directly with the Faraday tensor: \emph{our smallness condition for stability depends only on the physical field variables, and not on auxiliary mathematical quantities such as the values achieved by the components $\nabla_{\mu} A_{\nu}.$}

Now roughly speaking, the reason that we are able to prove our main stability result is because in our wave coordinate gauge
(see the discussion in Section \ref{SSS:Settinguptheequations}), the nonlinear terms have a special algebraic structure, which Lindblad and Rodnianski have labeled \cite{hLiR2003} \emph{the weak null condition}. We remark that in order for small-data global existence to hold, it is essential that the quadratic nonlinearities have special structure: John's blow-up result \cite{fJ1981} shows that quadratic perturbations of the linear wave equation in $1 + 3$ dimensional Minkowski space (of which our equations \eqref{E:Reducedh1Intro} below are an example), \emph{do not necessarily} have small-data global existence. Now by definition, a system of PDEs satisfies the weak null condition if the corresponding \emph{asymptotic system} has 
small-data global solutions. The asymptotic system is obtained by discarding cubic and higher order terms, and also derivatives that are tangential to the outgoing Minkowskian null cones (see the discussion in Section \ref{SSS:GeometryandNullDecompositions}); the discarded terms are expected to decay faster than the remaining terms. The general philosophy is that if the asymptotic system has small-data global existence, then one should be hopeful that the original system does too. In \cite{hLiR2010}, Lindblad and Rodnianski showed that the asymptotic system corresponding to the Einstein-scalar field system in wave coordinates has global solutions. Although we do not carry out such an analysis in this article, we remark that it can be checked that the asymptotic system\footnote{To obtain this asymptotic system, one also discards the quadratic terms containing the fast-decaying null components $\alpha[\Far], \rho[\Far]$ and $\sigma[\Far]$ of the Faraday tensor; see Section \ref{SSS:GeometryandNullDecompositions}.} corresponding to the Einstein-nonlinear electromagnetic system in wave coordinates also has global solutions. This was our original motivation for pursuing the present work.

The aforementioned weak null condition is a generalization of the classic \emph{null condition} of Klainerman \cite{sK1986} (see also Christodoulou's work \cite{dC1986}), in which the quadratic nonlinearities are \emph{standard null forms} (which are defined below in the statement of Lemma \ref{L:RicciInWave}). By now, there is a very large body of global existence and almost-global existence results that are based on the analysis of quadratic nonlinearities that satisfy generalizations of Klainerman's null condition. This includes the stability results for the Einstein equations mentioned above, but also many other results; there are far too many to list exhaustively, but we mention the following as examples: \cite{sK2005}, \cite{sKtS1996}, \cite{hL2004}, \cite{hL2008}, \cite{jMcS2007}, \cite{jMmNcS2005b}, \cite{tS1996}, \cite{jS2010a}.

\subsubsection{Connections to the ``divergence'' problem}

One of the most important unresolved issues in physics is that of the so-called ``divergence problem.'' In the setting of classical electrodynamics, this problem manifests itself as the unhappy fact that the familiar linear Maxwell-Maxwell equations with \emph{point charge} sources (i.e., delta function source terms), together with the \emph{Lorentz force law}\footnote{Recall that the Lorentz force is $F_{Lorentz} = q [\Electricfield + v \times \Magneticinduction],$ where $q$ is the charge associated to the point charge, $\Electricfield$ is the electric field, $v$ is the instantaneous point charge velocity, and $\Magneticinduction$ is the magnetic induction field.}, do not comprise a well-defined system of equations. This is because the theory dictates that the Lorentz force at the location of a point charge is ``infinite in all directions,'' so that the charge's motion is ill-defined. A further symptom of the divergence problem in this theory is that the energy of a static point charge is infinite. Moreover, our present-day flagship model of quantum electrodynamics (QED), which is based on a quantization of the classical Maxwell-Dirac field equations, has not yet fixed the crux of the problem; similar manifestations of the divergence problem arise in QED; see \cite{mK2004a}, \cite{mK2004b} for a detailed discussion of these issues.

Now in \cite{mK2004a}, \cite{mK2004b}, Kiessling has taken a preliminary step in the direction of resolving the divergence problem by reconsidering classical electrodynamics. One of Kiessling's primary strategies is to follow the lead of of Max Born \cite{mB1933} by replacing the linear Maxwell-Maxwell equations with a suitable nonlinear system, the hope being that it will be possible to make rigorous mathematical sense of the motion of point charges in the nonlinear theory. As is discussed below, Kiessling's leading candidate is the Maxwell-Born-Infeld (MBI) model \cite{mBlI1934} of classical electromagnetism, a model put forth by Born and Infeld in $1934$ based on Born's earlier ideas. The electromagnetic Lagrangian for this model is 
\begin{align} \label{E:LMBI}
	\Ldual_{(MBI)}  \eqdef \frac{1}{\upbeta^4} - \frac{1}{\upbeta^4} \big(1 + \upbeta^4 \Farinvariant_{(1)} - \upbeta^8 
		\Farinvariant_{(2)}^2 \big)^{1/2} = \frac{1}{\upbeta^4} - \frac{1}{\upbeta^4} \big(\mbox{det}_g(g + \Far) \big)^{1/2},
\end{align}
where $\upbeta > 0$ denotes \emph{Born's ``aether'' constant.} We point out that as verified in e.g. \cite{jS2010a}, this Lagrangian satisfies the assumptions \eqref{E:Ldualassumptions} and \eqref{E:DECL1} - \eqref{E:DECTrace} below, so that the main results of this article apply to the MBI model. Now it turns out that it was not enough for Kiessling to simply replace the linear Maxwell-Maxwell equations with the Maxwell-Born-Infeld equations, for such a modification fails to fix the problem of the Lorentz force being ill-defined at the location of the point charge. On the other hand, in MBI theory on the Minkowski spacetime background, there exist \emph{Lipschitz-continuous} electromagnetic potentials corresponding to single static point charge solutions to the field equations. Kiessling observed that this level of regularity is (just barely) sufficient for a relativistic version of Hamilton-Jacobi theory to be well-defined; he thus proposed a new relativistic Hamilton-Jacobi ``guiding law'' of motion for the point charges (see \cite{mK2004a} for the details).

Kiessling's interest in the Maxwell-Born-Infeld system was further motivated by results contained in \cite{gB1969} and \cite{jP1970}, which show that it is the unique\footnote{More precisely, there is a one-parameter family of such theories indexed by $\upbeta > 0.$} theory of classical electromagnetism that is derivable from an action principle and that satisfies the following $5$ postulates (see also the discussion in \cite{iBB1983}, \cite{mK2004a}):

\begin{enumerate}
	\item The field equations transform covariantly under the Poincar\'e group.
	\item The field equations are covariant under a Weyl (gauge) group.
	\item The electromagnetic energy surrounding a stationary point charge is finite.
	\item The field equations reduce to the linear Maxwell-Maxwell equations in the weak field limit.
	\item The solutions to the field equations are not birefringent.
\end{enumerate}
We remark that the linear Maxwell-Maxwell system satisfies all of the above postulates except for (iii), and that
the MBI system was shown to satisfy (iii) by Born in \cite{mB1933}. Physically, postulate (v) is equivalent to the statement that the ``speed of light propagation'' is independent of the polarization of the wave fields. Mathematically, this is the postulate that there is only a single \emph{null cone}\footnote{In general this ``light cone'' does not have to coincide with the gravitational null cone, although it \emph{does} in the case of the linear Maxwell-Maxwell equations.} associated to the electromagnetic equations; in a typical theory of classical electromagnetism, the causal structure of the electromagnetic equations is more complicated than the structure corresponding to a single null cone (see \cite{jS2010a} for a detailed discussion of this issue in the context of the Maxwell-Born-Infeld equations on the Minkowski spacetime background).

It is here that we can mention the connection of the present article to Kiessling's work. First, as noted in \cite{mK2004a}, 
Kiessling expects that his theory can be generalized to the case of a curved spacetime through a coupling to the Einstein equations. Next, we mention that although the Maxwell-Born-Infeld system is Kiessling's leading candidate for an electromagnetic model, he is also considering other models. In particular, by relaxing postulate (v) above, a relaxation that in principle could be supported by experimental evidence, one is led to consider a larger family of electromagnetic models. Now one basic criterion for any viable electromagnetic model is that small, nearly linear-Maxwellian electromagnetic fields in near-Minkowski vacuums should not lead to a severe breakdown in the structure of spacetime or other degenerate behavior. The present work confirms this criterion for a large family of electromagnetic models coupled to the Einstein equations, including the Maxwell-Born-Infeld system and many other models that fall under the scope of Kiessling's program.

\subsection{Discussion of the analysis}
\subsubsection{The splitting of the spacetime metric and setting up the equations} \label{SSS:Settinguptheequations}

As in the works \cite{hLiR2005} and \cite{hLiR2010}, in order to analyze the spacetime metric,
we split it into the following three pieces:

\begin{subequations}
\begin{align} 
	g_{\mu \nu} & = m_{\mu \nu} + h_{\mu \nu}, && (\mu, \nu = 0,1,2,3), \label{E:gmhexpansion} \\
	h_{\mu \nu} & = h_{\mu \nu}^{(0)} + h_{\mu \nu}^{(1)}, && (\mu, \nu = 0,1,2,3), \label{E:hdefIntro} \\
	h_{\mu \nu}^{(0)} & \eqdef \chi\big(\frac{r}{t}\big)\chi(r)\frac{2M}{r} \delta_{\mu \nu}, \ \Big(h_{\mu \nu}^{(0)}|_{t = 0} = 	
	\chi(r)\frac{2M}{r} \delta_{\mu \nu}, \ \partial_t h_{\mu \nu}^{(0)}|_{t = 0} = 0 \Big), && (\mu, \nu = 0,1,2,3), 	
		\label{E:h0defIntro}
\end{align} 
\end{subequations}
where $m_{\mu \nu} = \mbox{diag}(-1,1,1,1)$ is the Minkowski metric, and the function
$\chi$ plays several roles that will be discussed in Section \ref{SSS:h0}. Above and throughout,
$\chi(z)$ is a fixed cut-off function that satisfies

\begin{align}
	\chi \in C^{\infty}, \qquad \chi \equiv 1 \ \mbox{for} \ z \geq 3/4, \qquad \chi \equiv 0 \ \mbox{for} \ z \leq 1/2.
\end{align}
\textbf{We remark that here and throughout the rest of the article, unless we explicitly indicate otherwise 
(which, as is explained in Section \ref{SS:Indices}, we sometimes do with the use of the symbol $\#$), all indices on all tensors are lowered and raised with the Minkowski metric $m_{\mu \nu} = \mbox{diag}(-1,1,1,1)$ and its inverse $(m^{-1})^{\mu \nu} = \mbox{diag}(-1,1,1,1)$.} Furthermore, as in \cite{hLiR2005} and \cite{hLiR2010}, we work in a wave coordinate system, which is a coordinate system in which the contracted Christoffel symbols 
$\Gamma^{\mu} \eqdef (g^{-1})^{\kappa \lambda} \Gamma_{\kappa \ \lambda}^{\ \mu}$ (see \eqref{E:EMBIChristoffeldef}) of the metric $g_{\mu \nu}$ satisfy

\begin{align} \label{E:Wavecoordinateintro}
	\Gamma^{\mu} & = 0,&& (\mu = 0,1,2,3).
\end{align}
We remark that several equivalent definitions of the wave coordinate condition \eqref{E:Wavecoordinateintro} are discussed in Section \ref{SS:WaveCoordinates}, and that the viability of the wave coordinate gauge for the system \eqref{E:IntroEinstein} - \eqref{E:IntrodMis0} (which is a rather standard result based on the ideas of \cite{CB1952}) is discussed in Section \ref{SS:WaveCoordinatesPreserved}.

As is discussed in detail in Section \ref{SS:ReducedEquations}, in a wave coordinate system $(t,x),$ the equations \eqref{E:IntroEinstein} - \eqref{E:IntrodMis0} are equivalent to the \emph{reduced equations} 

\begin{subequations}
\begin{align}
	\widetilde{\Square}_{g} h_{\mu \nu}^{(1)} & = \mathfrak{H}_{\mu \nu} - \widetilde{\Square}_{g} h_{\mu \nu}^{(0)},&& 
		(\mu, \nu = 0,1,2,3), \label{E:Reducedh1Intro} \\
	\nabla_{\lambda} \Far_{\mu \nu} + \nabla_{\mu} \Far_{\nu \lambda} + \nabla_{\nu} \Far_{\lambda \mu} & = 0,&&
		(\lambda, \mu, \nu = 0,1,2,3), \label{E:ReduceddFis0Intro} \\
	N^{\#\mu \nu \kappa \lambda} \nabla_{\mu} \Far_{\kappa \lambda} & = \mathfrak{F}^{\nu},&& (\nu = 0,1,2,3),
		\label{E:ReduceddMis0Intro} 
\end{align}
\end{subequations}
where $\widetilde{\Square}_{g} = (g^{-1})^{\kappa \lambda} \nabla_{\kappa} \nabla_{\lambda}$ is the reduced wave operator corresponding 
to $g_{\mu \nu},$ $\nabla$ is the Levi-Civita connection corresponding to the \emph{Minkowski metric} $m_{\mu \nu},$
$N^{\#\mu \nu \kappa \lambda} \eqdef \frac{1}{2} \big\lbrace (m^{-1})^{\mu \kappa} (m^{-1})^{\nu \lambda} - (m^{-1})^{\mu \lambda} (m^{-1})^{\nu \kappa} - h^{\mu \kappa} (m^{-1})^{\nu \lambda} + h^{\mu \lambda} (m^{-1})^{\nu \kappa} - (m^{-1})^{\mu \kappa} h^{\nu \lambda} + (m^{-1})^{\mu \lambda} h^{\nu \kappa} \big\rbrace + N_{\triangle}^{\#\mu \nu \kappa \lambda},$ $N_{\triangle}^{\#\mu \nu \kappa \lambda} = O^{\dParameter}\big(|(h,\Far)|^2\big)$ is a quadratic error term that depends on the chosen model of nonlinear electromagnetism, and $\mathfrak{H}_{\mu \nu},$ $\mathfrak{F}^{\nu}$ are inhomogeneous terms that depend in part on on the chosen model of nonlinear electromagnetism.

The question of the stability of the Minkowski spacetime solution to \eqref{E:IntroEinstein} - \eqref{E:IntrodMis0}
has thus been reduced to two subquestions: i) show that the reduced system \eqref{E:Reducedh1Intro} - \eqref{E:ReduceddMis0Intro}, where the unknowns are viewed to be $(h_{\mu \nu}^{(1)},\Far_{\mu \nu}),$ has small-data global existence (if the ADM mass $M$ is sufficiently small); ii) show that the resulting spacetime $(\mathbb{R}^{1+3}, g_{\mu \nu} = m_{\mu \nu} + h_{\mu \nu}^{(0)} + h_{\mu \nu}^{(1)})$ is geodesically complete. The second question is very much related to the first, for as in \cite[Section 16]{hLiR2005}, \cite[Section 9]{jL2008}, 
the question of geodesic completeness can be answered if one has sufficiently detailed information about the asymptotic behavior of $h_{\mu \nu}^{(1)};$ our stability theorem (see Section \ref{S:GlobalExistence}) provides sufficient information.

\subsubsection{The smallness condition}
Our smallness condition the abstract initial data is stated in terms of the ADM mass $M$ and a weighted Sobolev norm
of the abstract initial data $\underline{\nabla}_i \mathring{\underline{h}}^{(1)}_{jk},$ $\mathring{K}_{jk},$
$\mathring{\mathfrak{\Displacement}}_j,$ and $\mathring{\mathfrak{\Magneticinduction}}_k.$ More specifically, in order to deduce global existence, we will require that

\begin{align}
	E_{\dParameter;\upgamma}(0) + M < \varepsilon_{\dParameter},
\end{align}	
where $\varepsilon_{\dParameter} > 0$ is a sufficiently small positive number, $E_{\dParameter;\upgamma}(0) \geq 0$ is defined by

\begin{align}   \label{E:DataNormIntro}
	E_{\dParameter;\upgamma}^2(0) 
	& \eqdef \| \underline{\nabla} \mathring{\underline{h}}^{(1)} \|_{H_{1/2 + \upgamma}^{\dParameter}}^2 
		\ + \ \| \mathring{K} \|_{H_{1/2 + \upgamma}^{\dParameter}}^2 
		\ + \ \| \mathring{\mathfrak{\Displacement}} \|_{H_{1/2 + \upgamma}^{\dParameter}}^2 
		\ + \ \| \mathring{\mathfrak{\Magneticinduction}} \|_{H_{1/2 + \upgamma}^{\dParameter}}^2,
\end{align}
the weighted Sobolev norm $\| \cdot \|_{H_{1/2 + \upgamma}^{\dParameter}}$ is defined in 
Definition \ref{D:HNdeltanorm} below, $0 < \upgamma < 1/2$ is a constant, and $\dParameter \geq 8$ is an integer. The condition $\dParameter \geq 8$ is needed for various weighted Sobolev embedding results, including the weighted Klainerman-Sobolev inequality \eqref{E:KSIntro}, and the results stated in Appendix \ref{A:SobolevMoser}. In the above expressions, $\underline{\nabla}$ is the Levi-Civita connection corresponding to the Euclidean metric\footnote{Throughout the article, we use the symbol 
$\um$ to denote both the Euclidean metric $\um_{jk} \eqdef \mbox{diag}(1,1,1)$ on $\mathbb{R}^3,$ and the 
first fundamental form $\um_{\mu \nu} \eqdef \mbox{diag}(0,1,1,1)$ of the constant time hypersurfaces
$\Sigma_t$ viewed as embedded hypersurfaces of Minkowski spacetime; this double-use of notation should not cause any confusion.} $\um_{jk} \eqdef \mbox{diag}(1,1,1).$ 
Note that the assumed fall-off conditions \eqref{E:h1AbstractDataAsymptotics} - \eqref{E:BdecayAssumption} 
guarantee the existence of a constant $0 < \upgamma < 1/2$ such that $E_{\dParameter;\upgamma}(0) < \infty.$

Although the norm \eqref{E:DataNormIntro} is useful for expressing the small-data global existence condition
in terms of quantities inherent to the data, from the perspective of analysis, a more useful quantity is
the energy ${\mathcal{E}_{\dParameter;\upgamma;\upmu}(t) \geq 0},$ which is defined by

\begin{align} \label{E:EnergyIntro}
	\mathcal{E}_{\dParameter;\upgamma;\upmu}^2(t) & \eqdef \underset{0 \leq \tau \leq t}{\mbox{sup}} 
		\sum_{|I| \leq \dParameter } \int_{\Sigma_{\tau}} 
		\Big\lbrace |\nabla \nabla_{\mathcal{Z}}^I h^{(1)}|^2 + |\Lie_{\mathcal{Z}}^I \Far|^2 \Big\rbrace w(q) \, d^3 x,
\end{align}
where $\nabla$ denotes the Levi-Civita connection corresponding to the \emph{full Minkowski spacetime metric}, 
$q \eqdef |x| - t$ is a null coordinate, the weight function $w(q)$ is defined by

\begin{align} \label{E:weightintro}
	w = w(q) = \left \lbrace
		\begin{array}{lr}
    	1 \ + \ (1 + |q|)^{1 + 2 \upgamma}, &  \mbox{if} \ q > 0, \\
      1 \ + \ (1 + |q|)^{-2 \upmu}, & \mbox{if} \ q < 0,
    \end{array}
  \right.,
\end{align}
$\upgamma$ is from \eqref{E:DataNormIntro}, and $0 < \upmu < 1/2$ is a fixed constant. In the above expression, 
$\mathcal{Z} \eqdef \big\lbrace \partial_{\mu}, x_{\mu} \partial_{\nu} - x_{\nu} \partial_{\mu}, 
x^{\kappa} \partial_{\kappa} \big\rbrace_{0 \leq \mu < \nu \leq 3}$ is a subset of the conformal Killing fields of
Minkowski space, $I$ is a vectorfield multi-index, $\nabla_{\mathcal{Z}}^I$ represents iterated Minkowski covariant differentiation with respect to vectorfields in $\mathcal{Z},$ and $\Lie_{\mathcal{Z}}^I$ represents iterated Lie differentiation with respect to vectorfields in $\mathcal{Z}.$ The significance of the set $\mathcal{Z}$ is that it is needed
for the weighted Klainerman-Sobolev inequality \eqref{E:KSIntro}, which is discussed below.

\begin{remark}\label{R:Roleofmu}
	The presence of the parameter $\upmu > 0$ in \eqref{E:weightintro} might seem unnecessary, since 
	$1 \ + \ (1 + |q|)^{-2 \upmu} \approx 1.$ However, as is explained in Section \ref{SSS:EnergyandStress}, 
	the presence of $\upmu > 0$ ensures that $w'(q) > 0,$ a fact that plays a key role in our energy estimates.
\end{remark}

\subsubsection{Overall strategy of the proof} \label{SS:DiscussionofProof}

The overall strategy is to deduce a hierarchy of Gronwall-amenable inequalities for the energies $\mathcal{E}_{k;\upgamma;\upmu}(t),$ $(0 \leq k \leq \dParameter );$ this is accomplished in \eqref{E:Mainenergyinequalityreexpressed} below. The net effect is that under the assumption $E_{\dParameter;\upgamma}(0) + M \leq \varepsilon,$ we are able to deduce the following a-priori estimate for the solution, which is valid during its classical lifetime:

\begin{align} \label{E:EnergyaprioriIntro}
	\mathcal{E}_{\dParameter;\upgamma;\upmu}(t) \leq c_{\dParameter} \varepsilon (1 + t)^{\widetilde{c}_{\dParameter} \varepsilon}.
\end{align}
In the above inequality, $c_{\dParameter}$ and $\widetilde{c}_{\dParameter}$ are positive constants. Now it is a standard result in the theory
of hyperbolic PDEs that if $\varepsilon$ is sufficiently small, then an a-priori estimate of the form \eqref{E:EnergyaprioriIntro} implies that the solution exists for $(t,x) \in (-\infty, \infty) \times \mathbb{R}^3;$ see Proposition \ref{P:LocalExistence} for more details. Furthermore, as shown in \cite{hLiR2005} and \cite{jL2008}, if $\varepsilon$ is sufficiently small, then it also follows that the spacetime $(\mathbb{R}^{1+3}, g_{\mu \nu} = m_{\mu \nu} + h_{\mu \nu}^{(0)} + h_{\mu \nu}^{(1)})$ is geodesically complete. \textbf{The main goal of this article is therefore to derive \eqref{E:EnergyaprioriIntro}}.

\subsubsection{Geometry and null decompositions} \label{SSS:GeometryandNullDecompositions}

Let us now describe the tools used to derive \eqref{E:EnergyaprioriIntro}. First and foremost, as mentioned above in Section \ref{SSS:MathematicalComparisons}, the reason we are able to prove our stability result is that the reduced equations \eqref{E:Reducedh1Intro} - \eqref{E:ReduceddMis0Intro} have special algebraic structure, and satisfy (in the language of Lindblad and Rodnianski) the \emph{weak null condition}. Now in order to see the special structure of the terms in the reduced equations, we follow the strategy of Lindblad and Rodnianski and decompose them into their \emph{Minkowskian null components}; we refer to this as a \emph{Minkowskian null decomposition}. We emphasize the following point: \textbf{the Minkowskian geometry is not the ``correct'' geometry to use for analyzing the equations, for the actual characteristics of the system 
correspond to the null cones of the spacetime metric $g_{\mu \nu}$ and the characteristics of the nonlinear electromagnetic equations (which in general do not have to coincide with the gravitational null cones). However, the errors that we make in using the Minkowskian geometry (which has the advantage of being simple) for our analysis are controllable.} Let us briefly recall the meaning of a Minkowskian null decomposition; a more detailed description is offered in Section \ref{S:NullFrame}. The notion of a Minkowskian null decomposition is intimately connected to the following spacetime subsets: the \emph{outgoing Minkowskian null cones} $C_{q}^+ \eqdef \lbrace (\tau,y) \ | \ |y| - \tau = q \rbrace,$ the \emph{ingoing Minkowskian null cones} $C_{s}^- \eqdef \lbrace (\tau,y) \ | \ |y| + \tau = s \rbrace,$ the \emph{constant time slices} $\Sigma_t \eqdef \lbrace (\tau,y) \ | \ \tau = t \rbrace,$ and the \emph{Euclidean spheres} $S_{r,t} \eqdef \lbrace (\tau,y) \ | \ t = \tau, |y| = r \rbrace.$ Observe that the \emph{null coordinate} $q \eqdef |x| - t$ associated to the spacetime point with coordinates $(t,x)$ is constant on the outgoing cones, and the null coordinate $s \eqdef |x| + t$ is constant on the ingoing cones. These coordinates will be used throughout the article to discuss the rates of decay of various quantities. With $\omega^j \eqdef x^j/r,$ $(j=1,2,3),$ we also define the \emph{ingoing Minkowskian null geodesic vectorfield} $\uL^{\mu} \eqdef (1,-\omega^1,-\omega^2,-\omega^3),$ which satisfies $m_{\kappa \lambda}\uL^{\kappa} \uL^{\lambda} = 0$ and is tangent to the $C_{s}^-,$ and the \emph{outgoing Minkowskian null geodesic vectorfield} $L^{\mu} \eqdef (1,\omega^1,\omega^2,\omega^3),$ which satisfies $m_{\kappa \lambda} L^{\kappa} L^{\lambda} = 0,$ $m_{\kappa \lambda} \uL^{\kappa} L^{\lambda} = -2$ and is tangent to the $C_{q}^+.$ Furthermore, in a neighborhood of each non-zero spacetime point $p,$ there exists a locally defined pair of $m-$orthonormal vectorfields $e_1, e_2$ that are tangent to the family of Euclidean spheres, and $m-$orthogonal to $\uL$ and $L.$ The set $\mathcal{N} \eqdef \lbrace \uL, L, e_1, e_2 \rbrace,$ which spans the tangent space, is known as a \emph{Minkowskian null frame}. In the discussion that follows, we will also make use of the set $\mathcal{T} \eqdef \lbrace L, e_1, e_2 \rbrace,$ which is the subset consisting of only those frame vectors tangent to the $C_{q}^+,$ and the set $\mathcal{L} \eqdef \lbrace L \rbrace.$

Given any two form $\Far,$ we can decompose it into its Minkowskian null components $\ualpha[\Far],$ $\alpha[\Far],$
$\rho[\Far],$ and $\sigma[\Far],$ where $\ualpha,$ $\alpha$ are two-forms $m-$tangent\footnote{By $m-$tangent, we mean that their vector duals relative to the Minkowski metric are tangent to the $S_{r,t}.$} to the spheres $S_{r,t}$, and $\rho,$ $\sigma$ are scalars. More specifically, we define $\ualpha_A = \Far_{A \uL},$ $\alpha_A = \Far_{AL},$
$\rho = \frac{1}{2} \Far_{\uL L},$ and $\sigma = \Far_{12},$ where $A \in \lbrace 1,2 \rbrace,$ and we have abbreviated
$\Far_{A \uL} \eqdef e_{A}^{\kappa} \uL^{\lambda} \Far_{\kappa \lambda},$ etc. Similarly, we can decompose the tensor
$h_{\mu \nu}$ into its null components $h_{LL},$ $h_{\uL L},$ $h_{LT},$ etc., where $T$ stands for any of the vectors in
$\mathcal{T}.$ We are now ready to discuss one of the major themes running throughout this article: the rates of decay of the various null components of $\Far$ and $h$ are distinguished by the kinds of contractions taken against the null frame vectors. In particular, contractions against $L,e_1,e_2$ are associated with favorable decay, with $L$ being the most favorable, while contractions against $\uL$ are associated with unfavorable decay. Similarly, differentiation in the directions $L,e_1,e_2$ are associated with creating \emph{additional favorable decay} in the null coordinate $s,$ while differentiation in the direction $\uL$ is associated with creating less favorable additional decay in $q$  (see Lemma \ref{L:PointwisetandqWeightedNablainTermsofZestiamtes} for a precise version of this claim).
Equivalently, the operator $\conenabla$ creates favorable decay in $s,$ while $\nabla$ only creates decay in $q.$ Here and throughout, $\conenabla$ is the projection (of the derivative component only) of the Minkowski connection $\nabla$ onto the outgoing Minkowski null cones. From this point of view, the most dangerous terms in the equations are $\ualpha$ and $h_{\uL \uL},$ and the $\partial_q \sim \nabla_{\uL}$ derivatives (see Section \ref{SS:Derivatives}) of these quantities. We recommend that at this point, the reader examine the conclusions of Propositions \ref{P:UpgradedDecayhA} and \ref{P:UpgradedDecayh1A} to get a feel for the kind of decay properties possessed by the various null components.

The main idea behind the Minkowskian null decomposition is that it can be used to show the following fact: \emph{the worst possible combinations of terms, from the point of view of decay rates, are not present in the reduced equations} \eqref{E:Reducedh1Intro} - \eqref{E:ReduceddMis0Intro}. This special algebraic structure, which is of central importance in our small-data global existence proof, is examined in detail in Propositions \ref{P:AlgebraicInhomogeneous} - \ref{P:harmonicgauge} of Section \ref{S:AlgebraicEstimates}. We remark that as revealed in \cite{hLiR2003}, \cite{hLiR2005}, and \cite{hLiR2010}, this special algebraic structure is tensorial in nature.

\subsubsection{Energy inequalities and the canonical stress} \label{SSS:EnergyandStress}

The first major analytical step in deriving the all-important Gronwall-amenable estimate \eqref{E:Mainenergyinequalityreexpressed} is to deduce the energy inequalities of Lemma \ref{L:weightedenergyFar} and Lemma \ref{L:weightedenergy}, which respectively provide $L^2$ estimates for solutions to the electromagnetic \emph{equations of variation} (which are the linearized equations satisfied by the derivatives of solutions $\Far$ to \eqref{E:ReduceddFis0Intro} - \eqref{E:ReduceddMis0Intro}), and $L^2$ estimates for solutions to quasilinear wave equations whose principal operator agrees with that of \eqref{E:Reducedh1Intro} (i.e., $\widetilde{\Square}_g$). As is explained below, such equations come into play because we require $L^2$ estimates for derivatives of $h^{(1)}$ and $\Far$ in order to close our global existence argument. We will comment mainly on the estimates for the electromagnetic equations of variation, since the estimates of Lemma \ref{L:weightedenergy} are perhaps more familiar to the reader, and in any case are explained in detail in \cite[Lemma 6.1 and Proposition 6.2]{hLiR2010}. Our proof of Lemma \ref{L:weightedenergyFar} is based on the construction of a suitable \emph{energy current} $\dot{J}^{\mu} \eqdef - \Stress_{\ \nu}^{\mu}X^{\nu},$ where $\Stress_{\ \nu}^{\mu}$ is the \emph{canonical stress}, which is a tensorfield that depends quadratically on the linearized variables $\dot{\Far}_{\mu \nu},$ $X^{\nu} \eqdef w(q) \delta_0^{\nu},$ $(\nu=0,1,2,3)$ is a ``multiplier vectorfield,'' and $w(q)$ is the weight function defined in \eqref{E:weightintro}. The end result is provided by inequality \eqref{E:FirstweightedenergyFar} below. Although at first glance inequality \eqref{E:FirstweightedenergyFar} below may appear to be a standard energy inequality, one of the most important features of this particular energy current is that it provides the \emph{additional positive} spacetime integral $\int_{t_1}^{t_2} \int_{\Sigma_{\tau}} \big(\dot{\alpha}^2 + \dot{\rho}^2 + \dot{\sigma}^2 \big) w'(q) \,d^3x \, d \tau$ on the left-hand side of \eqref{E:FirstweightedenergyFar}; here, $\dot{\alpha},$ $\dot{\rho},$ and $\dot{\sigma}$ are the ``favorable'' null components of the two-form $\dot{\Far}.$ This additional positive quantity, which is analogous to the quantity $\int_{t_1}^{t_2} \int_{\Sigma_{\tau}} |\conenabla \phi|^2 w'(q) \,d^3x \, d \tau$ on the left-hand side of \eqref{E:Firstweightedenergyscalar} that was exploited by Lindblad and Rodnianski, is one of the key advantages afforded by our use of a weight function of the form \eqref{E:weightintro}. Its availability is directly related to the fact that we have better integrated control over the quadratic terms $\dot{\alpha}^2 + \dot{\rho}^2 + \dot{\sigma}^2$ than we do over the term $\dot{\ualpha}^2.$ The quantity plays a key role in the derivation of the energy inequality \eqref{E:Mainenergyinequalityreexpressed}.

Let us now make a few comments concerning the canonical stress and the construction of the above energy current. A very detailed description is located in \cite{dC2000} and \cite{jS2010a}, so we confine ourselves here to its two most salient features. The canonical stress (see \eqref{E:Stressdef}) plays the role of an energy-momentum-type tensor for the electromagnetic equations of variation. Because these linearized equations depend on the ``background'' $\Far_{\mu \nu}$ in addition to the linearized variables $\dot{\Far}_{\mu \nu},$ it is \emph{not} the case that $\mathscr{D}_{\mu} \Stress_{\ \nu}^{\mu} = 0;$ this is in contrast to the property  $(g^{-1})^{\kappa \lambda}\mathscr{D}_{\kappa} T_{\lambda \nu} = 0$ (see \eqref{E:TisDivergenceFree}) enjoyed by the energy-momentum tensor. However, we now point out the first key property of the canonical stress: $\nabla_{\mu} \Stress_{\ \nu}^{\mu}$ is lower-order in the sense that it does not depend on $\nabla_{\lambda} \dot{\Far}_{\mu \nu};$ by using the equations of variation for substitution, the $\nabla_{\lambda} \dot{\Far}_{\mu \nu}$ terms can be replaced with inhomogeneous terms (see \ref{E:divergenceofStress}). It is already important to appreciate the availability of this non-trivial quadratic quantity whose divergence can be controlled by the background and inhomogeneous terms. For the availability of such a quantity is not a feature inherent to all systems of equations\footnote{It is however a feature inherent to all scalar quasilinear wave equations.}, but is instead related to the symmetry properties of the indices of the principal terms (i.e., the terms on the left-hand side) in equations \eqref{E:EOVdFis0} - \eqref{E:EOVdMis0}, which themselves are related to the fact the original nonlinear equations are derivable from a Lagrangian. 

The second key property enjoyed by the canonical stress is that of positivity upon contraction against certain covector/vector pairs $(\xi, X).$ That is, for certain choices of $(\xi, X),$ the quantity $\Stress_{\ \nu}^{\mu}\xi_{\mu}X^{\nu}$ is a positive definite quadratic form in $\dot{\Far}.$ 
These properties are analogous to (but distinct from) the positivity properties of an energy-momentum tensor satisfying the dominant energy condition, and the positivity properties of the Bel-Robinson tensor (which played a central role in \cite{dCsK1993}). As is explained in \cite{dC2000} and \cite{jS2010a}, the set of pairs leading to integrated positivity is intimately connected to the \emph{hyperbolicity of and the geometry of the electromagnetic equations}, and to the speeds and directions of propagation in the system. In this article, the only pair $(\xi, X)$ that we make use of is $\xi_{\mu} = - \delta_{\mu}^0,$ and $X^{\nu} = w(q) \delta_0^{\nu}.$ The special positivity properties stemming from this choice of $(\xi, X)$ are derived in Lemma \ref{L:weightedenergyFar}.

\subsubsection{Weighted Klainerman-Sobolev inequalities}

Based on the energy inequalities of Proposition \ref{P:weightedenergy}, which are relatively straightforward consequences of
Lemmas \ref{L:weightedenergyFar} and Lemmas \ref{L:weightedenergy}, it is clear that most of the hard work in deriving the estimate \eqref{E:Mainenergyinequalityreexpressed} goes into estimating the integrals of the inhomogeneous terms on the
right-hand sides of \eqref{E:Secondweightedenergyscalar} and \eqref{E:SecondweightedenergyFar}. In particular, we attempt to summarize here the origin of the factors $(1 + \tau)^{-1}$ and $(1 + \tau)^{-1 + C \varepsilon}$ that appear in 
\eqref{E:Mainenergyinequalityreexpressed}, and that are of central importance in our derivation of the fundamental energy inequality \eqref{E:EnergyaprioriIntro}. Roughly speaking, these factors arise from a variety of pointwise decay estimates that we will soon explain. The first tools of interest to us along these lines are the weighted Klainerman-Sobolev inequalities, which allow us to deduce pointwise decay estimates for functions $\phi \in C_0^{\infty}(\mathbb{R}^3),$ in terms of weighted $L^2$ estimates for $\phi$ and its Minkowskian covariant derivatives with respect to vectorfields $Z \in \mathcal{Z}.$ More specifically (see also Appendix \ref{A:WeightedKS}), the weighted Klainerman-Sobolev inequalities state that 

\begin{align} \label{E:KSIntro}
	(1 + t + |x|)[(1 + |q|) w(q)]^{1/2} |\phi(t,x)| & \leq C \sum_{|I| \leq 2} 
	\big\| w^{1/2}(q) \nabla_{\mathcal{Z}}^I \phi(t, \cdot) \big \|_{L^2}, && q \eqdef |x| - t.
\end{align}
We refer to these estimates as ``weak pointwise decay estimates,'' since they have nothing to do
with the special structure of the Einstein-nonlinear electromagnetic equations; a major theme permeating this article is that
in order to close our global existence bootstrap argument, the estimate \eqref{E:KSIntro} need to be upgraded using the special structure of the equations. Inequality \eqref{E:KSIntro} can therefore be viewed as a preliminary estimate that will play a role in the proof of the upgraded estimates.

The form of the inequalities \eqref{E:KSIntro} raises several important issues. First, in order to apply the 
weighted Klainerman-Sobolev inequalities to $h^{(1)},$ we have to achieve $L^2$ control over the quantities $w^{1/2}(q) \nabla_{\mathcal{Z}}^I h^{(1)}.$ To this end, we have to study the equations satisfied by the quantities $\nabla_{\mathcal{Z}}^I h^{(1)}.$ In order to derive these equations, we have to commute the operator $\nabla_{\mathcal{Z}}^I$ through the reduced wave operator term $\widetilde{\Square}_{g} h^{(1)}.$ Lindblad and Rodnianski accomplished this commutation through the use of \emph{modified covariant derivatives} $\hat{\nabla}_Z,$ which are equal to ordinary covariant derivatives plus a scalar multiple (depending on $Z \in \mathcal{Z}$) of the identity; see Definition \ref{D:ModifiedDerivatives}. The main advantage of these operators is that $\hat{\nabla}_Z \Square_{m} - \Square_{m} \nabla_Z = 0,$ where $\Square_{m} \eqdef (m^{-1})^{\kappa \lambda} \nabla_{\kappa} \nabla_{\lambda}$ denotes the wave operator of the Minkowski metric; see Lemma \ref{L:NablaModZLiemodMinkowskiWaveOperatorCommutator}. Therefore, $\nabla_{\mathcal{Z}}^I h^{(1)}$ is a solution to the equation $\widetilde{\Square}_{g} \nabla_{\mathcal{Z}}^I h^{(1)}$ $= \hat{\nabla}_{\mathcal{Z}}^I \widetilde{\Square}_{g} h^{(1)}$ $+ H^{\kappa \lambda} \nabla_{\kappa} \nabla_{\lambda} \nabla_{\mathcal{Z}}^I h^{(1)}$ $- \hat{\nabla}_{\mathcal{Z}}^I \big(H^{\kappa \lambda} \nabla_{\kappa} \nabla_{\lambda} h^{(1)}\big),$ where $\widetilde{\Square}_{g} h^{(1)}$ is equal to the inhomogeneous term on the right-hand side of \eqref{E:Reducedh1Intro} above, and $H^{\mu \nu} \eqdef (g^{-1})^{\mu \nu} - (m^{-1})^{\mu \nu} = - h^{\mu \nu} + O(|h|^2).$ We remark that the analysis of the commutator term $H^{\kappa \lambda} \nabla_{\kappa} \nabla_{\lambda} \nabla_{\mathcal{Z}}^I h^{(1)}$ 
$- \hat{\nabla}_{\mathcal{Z}}^I \big(H^{\kappa \lambda} \nabla_{\kappa} \nabla_{\lambda} h^{(1)}\big),$ which was performed in \cite{hLiR2010} (see also Propositions \ref{P:InhomogeneousTermsNablaZIh1} and Lemmas \ref{L:NablaZIBoxCommutatorIntegrated}), is among the most challenging work encountered. Rather than repeat this analysis and the discussion behind it, which is throughly explained and carried out in \cite{hLiR2010}, we will instead focus on the analogous difficulties that arise in our analysis of $\Far.$ We do, however point out the role that Hardy inequalities of Proposition \ref{P:Hardy} play in the analysis of $h^{(1)}:$ they are used to estimate a weighted $L^2$ norm of $\nabla_{\mathcal{Z}}^I h^{(1)},$ which is \emph{not} directly controlled in $L^2$ by the energy $\mathcal{E}_{\dParameter;\upgamma;\upmu}(t),$ by a weighted $L^2$ norm of $\nabla \nabla_{\mathcal{Z}}^I h^{(1)},$ which \emph{is} controlled in $L^2$ by the energy. The cost of applying this inequality is powers of $1 + |q|,$ which are always sufficiently available thanks to our use of the weight $w(q).$

\subsubsection{The role of Lie derivatives}

The next important issue concerning inequality \eqref{E:KSIntro} is that it is more convenient to work with
Lie derivatives of $\Far$ rather than covariant derivatives of $\Far;$ this claim has already been suggested by the
definition \eqref{E:EnergyIntro} of our energy $\mathcal{E}_{\dParameter;\upgamma;\upmu}(t).$ According to inequality \eqref{E:LieZIinTermsofNablaZI} below, inequality \eqref{E:KSIntro} remains valid if we replace the operators $\nabla_{\mathcal{Z}}^I$ with $\Lie_{\mathcal{Z}}^I.$
However, as in the case of covariant derivatives, we have to study the equations satisfied by the $\Lie_{\mathcal{Z}}^I \Far.$
Now on the one hand, Lemma \ref{L:Liecommuteswithcoordinatederivatives} shows that the operator $\Lie_Z$ can be commuted through the Minkowski connection $\nabla$ in equation \eqref{E:ReduceddFis0Intro}. On the other hand, to commute Lie derivatives through equation \eqref{E:ReduceddMis0Intro}, it is convenient to work with \emph{modified Lie derivatives} $\Liemod_Z,$ which are equal to ordinary Lie derivatives plus a scalar multiple\footnote{The multiple is $2 c_Z,$ where $c_Z$ is the multiple corresponding to the modified covariant derivative $\hat{\nabla}_Z.$} (depending on $Z \in \mathcal{Z}$) of the identity; see Definition \ref{D:ModifiedDerivatives}. Unlike covariant derivatives, these operators have favorable commutation properties with the linear Maxwell term $\nabla_{\mu} \Far^{\mu \nu},$ which is the leading term in \eqref{E:ReduceddMis0Intro}.
More specifically, $\Liemod_{Z} \Big\lbrace \big[(m^{-1})^{\mu \kappa} (m^{-1})^{\nu \lambda} - (m^{-1})^{\mu \lambda} (m^{-1})^{\nu \kappa}\big] \nabla_{\mu} \Far_{\kappa \lambda} \Big\rbrace = \big[(m^{-1})^{\mu \kappa} (m^{-1})^{\nu \lambda} - (m^{-1})^{\mu \lambda} (m^{-1})^{\nu \kappa}\big] \nabla_{\mu} \Lie_Z \Far_{\kappa \lambda};$ see Lemma \ref{L:LiemodZLiemodMaxwellCommutator}. As is captured by Proposition \ref{P:InhomogeneoustermsLieZIFar}, these operators are also useful for differentiating equation \eqref{E:ReduceddMis0Intro}; the error terms generated have a favorable null structure that is captured in Proposition \ref{P:EnergyInhomogeneousTermAlgebraicEstimate}.

\subsubsection{The tensorfield $h_{\mu \nu}^{(0)}$} \label{SSS:h0}

Let us now discuss the ideas behind the Lindblad-Rodnianski splitting of the metric defined in 
\eqref{E:gmhexpansion} - \eqref{E:h0defIntro}. We first note that because of 
the $2M/r$ ADM mass term present in $h_{\mu \nu}^{(0)},$ substituting the tensorfield 
$h_{\mu \nu} \eqdef h_{\mu \nu}^{(0)} + h_{\mu \nu}^{(1)}$ in place of $h_{\mu \nu}^{(1)}$ in the definition of the energy would lead to $\mathcal{E}_{\dParameter;\upgamma;\upmu}(0) = \infty.$ Thus, as a practical matter, the introduction of $h_{\mu \nu}^{(1)}$ allows us to work with a quantity of finite energy. Now according to the discussion in \cite{hLiR2010}, the precise form $h_{\mu \nu}^{(0)} = \chi\big(\frac{r}{t}\big)\chi(r)\frac{2M}{r} \delta_{\mu \nu}$ was determined by making an ``educated'' guess concerning the contribution of the ADM mass term $(2M/r) \delta_{\mu \nu}$ to the solution. The term $h_{\mu \nu}^{(0)}$ manifests itself in the reduced equations as the $\widetilde{\Square}_{g} h_{\mu \nu}^{(0)}$ inhomogeneous term on the right-hand side of the reduced equation \eqref{E:Reducedh1Summary}. Because of the identity $\Square_m (1/r ) = 0$ for $r > 0,$ where $\Square_m = (m^{-1})^{\kappa \lambda} \nabla_{\kappa}\nabla_{\lambda}$ is the Minkowski wave operator, it follows that the main contribution of the term $\widetilde{\Square}_{g} h_{\mu \nu}^{(0)}$ comes from the ``interior'' region $\lbrace (t,x) \mid 1/2 < r/t < 3/4 \rbrace;$ this is because the derivatives of $\chi(z)$ are supported in the interval $[1/2, 3/4].$ Now in the interior region, the quantities $1 + |q|$ and $1 + s$ are uniformly comparable. Thus, the weighted Klainerman-Sobolev inequality \eqref{E:KSIntro} predicts strong decay for the solution in this region, and consequently, one can derive suitable weighted Sobolev bounds for the inhomogeneity $\widetilde{\Square}_{g} h_{\mu \nu}^{(0)};$ see Lemma \ref{L:mathfrakH0partialZIhAproductestimate} for a precise statement of this estimate.

\subsubsection{The wave coordinate condition}

Before expanding our discussion of the pointwise decay estimates, we will discuss the analytic role of the \emph{wave coordinate} condition $\nabla_{\nu} [\sqrt{|\mbox{det}\, g|}(g^{-1})^{\mu \nu}] = 0,$ $(\mu = 0,1,2,3),$ which plays \emph{multiple} roles in this article. First, it hyperbolizes the Einstein equations and allows us to replace
certain unfavorable terms from the equations \eqref{E:IntroEinstein} - \eqref{E:IntroEinstein}
with more favorable ones; the culmination of this procedure is exactly the reduced system \eqref{E:Reducedh1Intro} - \eqref{E:ReduceddMis0Intro}. In addition, the wave coordinate condition allows us to deduce several \emph{independent and improved estimates}, both at both the pointwise level and the $L^2$ level, for the components $h_{LL}$ and $h_{LT}.$ As we will see, these improved estimates are central to the structure of the proof of 
Theorem \ref{T:ImprovedDecay}, and our stability argument would not close without them. More specifically, as shown in \cite{hLiR2010}, a null decomposition of the wave coordinate condition leads to the algebraic inequality

\begin{align} \label{E:wavecoordinatenulldecompintro}
	|\nabla h|_{\mathcal{L}\mathcal{T}} + |\nabla \nabla_Z h|_{\mathcal{L}\mathcal{L}}
	& \lesssim |\conenabla h| + |h||\nabla h|,
\end{align}
where $\conenabla$ is the projection of $\nabla$ (the derivative component only) onto the outgoing Minkowski cones. Note that the right-hand side of \eqref{E:wavecoordinatenulldecompintro} involves only favorable derivatives of $h$ and quadratic error terms, while the left-hand side involves \emph{all} derivatives of $h,$ including the dangerous $\nabla_{\uL}$ derivative. Generalizations of \eqref{E:wavecoordinatenulldecompintro} for $\nabla_{\mathcal{Z}}^I h$ are stated in Proposition \ref{P:harmonicgauge}. We remark that it is important to note in these generalizations that the estimates for 
$|\nabla \nabla_Z h|_{\mathcal{L}\mathcal{L}}$ are stronger than what can be proved for $|\nabla \nabla_Z h|_{\mathcal{L}\mathcal{T}}.$

\subsubsection{Upgraded pointwise decay estimates}

We now discuss the full collection of \emph{upgraded pointwise decay estimates} (see Propositions \ref{P:FLUTTimproveddecay} - \ref{P:UpgradedDecayh1A} below), which are of central importance in closing the global existence bootstrap argument. For, as mentioned above, the pointwise decay estimates \eqref{E:KSIntro} are not sufficient to close the argument. Aside from the components $h_{LL}$ and $h_{LT},$ which are controlled by the wave coordinate condition,
there is a relatively strong coupling between the remaining components of $h$ and the dangerous $\ualpha[\Far]$ component of the Faraday tensor. Therefore, our proofs of the upgraded estimates (and Proposition \ref{P:UpgradedDecayh1A} in particular) have a hierarchal structure; i.e., the order in which they are proved is very important. Although we don't provide a complete description of all of the subtleties of this hierarchy in this introduction, we do provide a preliminary description of some of its salient features. We first emphasize the following important feature: most null components of $h,$ all null components of $\Far,$ and the components $\nabla_Z h_{LL}$ (for $Z \in \mathcal{Z}$) have better $t-$decay properties than their higher-order-derivative counterparts; this is the content of Proposition \ref{P:UpgradedDecayhA}. Roughly speaking, the reason for this discrepancy is that the un-differentiated reduced equations have a more favorable algebraic structure than the differentiated reduced equations. This feature will be particularly important during our global existence argument, for the principal terms (from the point of view of differentiability) in the Leibniz expansion of the operator $\nabla_{\mathcal{Z}}^I$ acting on a quadratic term are of the form $u \nabla_{\mathcal{Z}}^I v,$ and similarly for the operator $\Lie_{\mathcal{Z}}^I.$ Consequently, the strong pointwise decay property of the un-differentiated quantity, which is represented by $u,$ 
is an important ingredient the derivation of the $C\varepsilon \int_{0}^{t} (1 + \tau)^{-1} \mathcal{E}_{k;\upgamma;\upmu}^2(\tau) \, d \tau$ term on the right-hand side of \eqref{E:Mainenergyinequalityreexpressed}. We emphasize that our stability proof would not go through if this term were replaced with $C \varepsilon \int_{0}^{t} (1 + \tau)^{-1 + C \varepsilon} \mathcal{E}_{k;\upgamma;\upmu}^2(\tau) \, d \tau.$

The derivation of the upgraded pointwise decay estimates for the Faraday tensor begins with Proposition \ref{P:EOVNullDecomposition}, which provides a null decomposition of the electromagnetic equations of variation, and 
Proposition \ref{P:EnergyInhomogeneousTermAlgebraicEstimate}, which provides a null decomposition of the inhomogeneous terms
that result when differentiating the reduced electromagnetic equations with modified Lie derivatives.
The net effect is that the null components of the \emph{lower-order} Lie derivatives of $\Far$ satisfy ODEs along ingoing and outgoing cones (see Proposition \ref{P:ODEsNullComponentsLieZIFar}), and furthermore, the inhomogeneous terms appearing on the right-hand side of the ODEs can be inductively controlled (see Proposition \ref{P:UpgradedDecayh1A}). It is important to distinguish between two classes of ODEs that play a role in this analysis. The first class consists of ODEs for the null components $(\dot{\alpha}, \dot{\rho}, \dot{\sigma}) \eqdef (\alpha[\Lie_{\mathcal{Z}}^I \Far], \rho[\Lie_{\mathcal{Z}}^I \Far], \sigma[\Lie_{\mathcal{Z}}^I \Far]),$ and involves differentiation in the direction of the null generators of the \emph{ingoing Minkowskian cones}; i.e., the principal part of the ODEs is $\nabla_{\uL}.$ We remark that this point of view represents a rather crude treatment of equations \eqref{E:uLdotalphaEOVnulldecomp} - \eqref{E:uLdotnablasigmaEOVnulldecomp}, but because of the favorable decay properties of the inhomogeneities, this approach is sufficient to conclude the desired estimates: by integrating back towards the Cauchy hypersurface $\Sigma_0,$ we are able to deduce $t-$decay for $\alpha[\Lie_{\mathcal{Z}}^I \Far],$ $\rho[\Lie_{\mathcal{Z}}^I \Far],$ and $\sigma[\Lie_{\mathcal{Z}}^I \Far]$ from $t-$decay of the inhomogeneous terms at the expense of a loss of decay in $q.$ We remark that the proof of the upgraded estimates for these components happens in two stages. We refer to the first stage, which are proved in Proposition \ref{P:FLUTTimproveddecay}, as the ``initial upgraded'' pointwise decay estimates. These first-stage estimates follow from using the weighted Klainerman-Sobolev estimates to bound the inhomogeneous terms in the ODEs. The second-stage upgraded estimates, which we refer to as ``fully upgraded'' pointwise decay estimates, are proved at the end of
Section \ref{P:UpgradedDecayh1A}, after all of the other upgraded pointwise decay estimates for the remaining components of 
the lower-order derivatives of $h$ and $\Far$ have been proved. For at this point in the upgraded hierarchy, we will have better pointwise control over the inhomogeneous terms in the ODEs than that afforded by the weighted Klainerman-Sobolev estimates.

The next class consists of ODEs for the null components $\dot{\ualpha} \eqdef \ualpha[\Lie_{\mathcal{Z}}^I \Far].$ Notice that (see equation \eqref{E:dotualphaEOVnulldecomp}) unlike the other null components, the $\dot{\ualpha}$ do \emph{not} satisfy an ODE that to $0-$th order involves differentiation in the direction of $\uL.$ Instead, at first sight, it might appear that one should reason in analogy with the first class and view equation \eqref{E:dotualphaEOVnulldecomp} as ODE in the direction of $L$ with inhomogeneous terms. However, the desired decay estimates do \emph{not} close at this level. Instead, one must also consider the effect of the quadratic term $-\angm_{\nu}^{\ \lambda} h^{\mu \kappa} \nabla_{\mu} \dot{\Far}_{\kappa \lambda}.$ A null decomposition of this term reveals that it contains the dangerous term $\frac{1}{4} h_{LL} \nabla_{\uL} \dot{\ualpha}_{\nu},$ which decays too slowly to be treated as an inhomogeneous term in the ODE satisfied by $\dot{\ualpha}.$ To remedy this difficulty, we introduce the vectorfield $\Lambda = L + \frac{1}{4} h_{LL} \uL,$ which can be viewed as a first-order correction to the Minkowski outgoing null direction arising from the presence of a non-zero tensorfield $h$ in the expansion $g_{\mu \nu} = m_{\mu \nu} + h_{\mu \nu}.$ Note that for these upgraded pointwise decay estimates for the lower-order Lie derivatives, we do not bother to correct for the fact that the electromagnetic model is not necessarily the linear Maxwell model; the deviation from the linear Maxwell model comprises cubic terms, which we can treat as small inhomogeneities. We may thus view equation \eqref{E:dotualphaEOVnulldecomp} as ODE in the direction of $\Lambda$ with inhomogeneous terms; this is exactly the point of view emphasized in Proposition \ref{P:ODEsNullComponentsLieZIFar}. Because we have a sufficiently strong independent decay estimates for $h_{LL}$ (this is yet another example of the special role played by the component $h_{LL}$) and also for the inhomogeneities, this approach is sufficient to achieve the desired estimates.

Our analysis of the upgraded pointwise decay estimates for the metric-related quantities $h$ and $h^{(1)}$ closely mirrors the analysis in \cite{hLiR2010}. Hence, we will not discuss them in full detail here, but instead refer the reader to the discussion in \cite{hLiR2010}. The estimates can be divided into three classes, the first one being the estimates \eqref{E:partialhLTpartialZhLLpointwise} and \eqref{E:hLTZhLLpointwise} for 
$|\nabla h|_{\mathcal{L} \mathcal{T}},$ $|\nabla \nabla_{Z} h|_{\mathcal{L} \mathcal{L}},$
$|h|_{\mathcal{L} \mathcal{T}},$  and $|\nabla_{Z} h|_{\mathcal{L} \mathcal{L}}.$ As was suggested above, the first-class estimates are consequences of the additional special algebraic structure that follows from the wave coordinate condition, together with the weighted Klainerman-Sobolev inequality. The second class consists of the estimates \eqref{E:partialhTUpointwise} and \eqref{E:partialhpointwise} for $|\nabla h|_{\mathcal{T} \mathcal{N}}$ and $|\nabla h|.$ These estimates heavily rely on the decay estimates of Lemma \ref{L:scalardecay} and Corollary \ref{C:systemdecay} below, which
were proved in \cite{hLiR2010} and which are independent of the specific structure of the Einstein equations. The lemma and its corollary can be viewed as a second-order counterpart to the ODE estimates for the Faraday tensor discussed in the previous paragraphs. It is important to note that the hypotheses of the lemma and its corollary are satisfied \emph{as a consequence} of the independent upgraded pointwise decay estimates provided by the wave coordinate condition. The third class consists of the estimates \eqref{E:partialZIh1Aupgraded}, \eqref{E:ZIh1Aupgraded}, and \eqref{E:barpartialZIh1Aupgraded} for $|\nabla \nabla_{\mathcal{Z}}^I h^{(1)}|,$ $|\nabla_{\mathcal{Z}}^I h^{(1)}|,$ and $|\conenabla \nabla_{\mathcal{Z}}^I h^{(1)}|$ (related estimates for the tensorfield $h$ also hold). Their derivation is similar in spirit to the derivation of the second-class estimates, but the inductive proof we give is highly coupled to the simultaneous derivation of analogous upgraded pointwise decay estimates for $|\Lie_{\mathcal{Z}}^I\Far|,$ which were discussed two paragraphs ago. 

\subsubsection{The geometry of Lie derivatives}

We make some final comments concerning the relationship between Lie derivatives and covariant derivatives. On the one hand, 
since we differentiate the equations satisfied by $h^{(1)}$ with the operators $\nabla_{\mathcal{Z}}^I,$
our analysis of $h^{(1)}$ naturally allows us to estimate the quantities $|\nabla_{\mathcal{Z}}^I h|,$ 
$|\nabla_{\mathcal{Z}}^I h|_{\mathcal{L}\mathcal{L}}$ and $|\nabla_{\mathcal{Z}}^I h|_{\mathcal{L}\mathcal{T}},$ etc. Furthermore, as discussed above, the quantities $|\nabla_{\mathcal{Z}}^I h|_{\mathcal{L}\mathcal{L}}$ and $|\nabla_{\mathcal{Z}}^I h|_{\mathcal{L}\mathcal{T}}$ have a distinguished role in view of their connection to the wave coordinate condition. One the other hand, because we use modified Lie derivatives to differentiate the electromagnetic equations, we will have to confront the terms $|\Lie_{\mathcal{Z}}^I h|,$ $|\Lie_{\mathcal{Z}}^I h|_{\mathcal{L}\mathcal{L}},$ and $|\Lie_{\mathcal{Z}}^I h|_{\mathcal{L}\mathcal{T}},$ etc. In order to bridge the gap between Lie derivative estimates and covariant derivative estimates, we provide Proposition \ref{P:LievsCovariantLContractionRelation}, the proof of which relies on the special algebraic structure of the vectorfields in $\mathcal{Z}.$ Proposition \ref{P:LievsCovariantLContractionRelation} is an especially important ingredient in the null decomposition estimate \eqref{E:EnergyInhomogeneousTermAlgebraicEstimate}. As an example of the role played by this proposition, we cite the estimate \eqref{E:LieZILLinTermsofNablaZILLLieZJLTPlusJunk}, which reads $|\Lie_{\mathcal{Z}}^I h|_{\mathcal{L}\mathcal{L}} \lesssim |\nabla_{\mathcal{Z}}^I h|_{\mathcal{L}\mathcal{L}} + \underbrace{\sum_{|J| \leq |I|-1} |\nabla_{\mathcal{Z}}^J h|_{\mathcal{L} \mathcal{T}}}_{\mbox{absent if $|I| = 0$}}
+ \underbrace{\sum_{|J'| \leq |I|-2} |\nabla_{\mathcal{Z}}^{J'} h|}_{\mbox{absent if $|I| \leq 1$}}.$ This shows that in the translation from Lie derivatives to covariant derivatives, the error terms that arise in the analysis of the $|\cdot|_{\mathcal{L}\mathcal{L}}$ seminorm are either $1$ degree lower in order \emph{and} controllable by the wave coordinate condition (i.e. the terms with $|J| \leq |I|-1$), or are $2$ degrees lower in order (i.e. the terms with $|J'| \leq |I|-2$). This fact, and others similar to it, play a role in allowing our hierarchy of estimates unfold in a viable order.

\subsection{Outline of the article}

The remainder of the article is organized as follows.

\begin{itemize}
	\item In Section \ref{S:Notation}, we provide for convenience a summary of the notation that is used throughout 
		the article.
	\item In Section \ref{S:ENESinWaveCoordinates}, we discuss the Einstein-nonlinear electromagnetic equations in detail.
		We also introduce our wave coordinate condition and our assumptions on the electromagnetic Lagrangian.
		Next, we derive a reduced system of equations, which is equivalent to the system of interest in our wave coordinate gauge.
		In Section \ref{SS:ReducedEquations}, we summarize the version of the reduced equations that we work with for most of the 
		article.
	\item In Section \ref{S:IVP}, we construct initial data for the reduced system from the 
		abstract initial data in a manner compatible with the wave coordinate condition. We also sketch a proof of the fact
		that the wave coordinate condition is preserved by the flow of the reduced equations.
	\item In Section \ref{S:NullFrame}, we introduce the notion of a Minkowskian null frame and discuss the 
		corresponding null decomposition of various tensorfields. 
	\item In Section \ref{S:DifferentialOperators}, we introduce the differential operators that will be used throughout 
		the remainder of the article, including modified Lie derivatives and modified covariant derivatives with respect to a
		special subset $\mathcal{Z}$ of Minkowskian conformal Killing fields. We also provide a collection of lemmas that
		relate the various operators.
	\item In Section \ref{S:EquationSatisfiedbyNablaZIh1}, we provide a preliminary algebraic expression for the equations
		satisfied by $\nabla_{\mathcal{Z}}^I h^{(1)},$ where $h^{(1)}$ is a solution to the reduced equations.
	\item In Section \ref{E:EOVandStress}, we introduce the electromagnetic equations of variation, which 
		are a linearized version of the electromagnetic equations. We also provide a preliminary algebraic expression 
		for the inhomogeneous terms in the equations of variation
		satisfied by $\Lie_{\mathcal{Z}}^I \Far,$ where $\Far$ is a solution to the reduced equations.
		We then introduce the canonical stress tensor and use it to construct an energy current that will be used to control 
		weighted Sobolev norms of $\Lie_{\mathcal{Z}}^I \Far.$
	\item In Section \ref{S:DecompositionsofElectromagneticEquations}, we perform two decompositions of the electromagnetic 
		equations, including a null decomposition of the electromagnetic equations of variation, and a decomposition of the 
		electromagnetic equations into constraint equations and evolution equations for the Minkowskian one-forms 
		$\Electricfield,$ $\Displacement,$ $\Magneticinduction,$ and $\Magneticfield.$ In order to connect these one-forms 
		to the abstract initial data, we also introduce the geometric electromagnetic one-forms 
		$\mathfrak{\Electricfield},$ $\mathfrak{\Displacement},$ 
		$\mathfrak{\Magneticinduction},$ and $\mathfrak{\Magneticfield}.$
	\item In Section \ref{S:SmallDataAssumptions}, we introduce our smallness condition on the abstract initial data. 
		We then prove that this smallness condition guarantees that the energy $\mathcal{E}_{\dParameter;\upgamma;\upmu}(t)$ of the 
		corresponding solution to the reduced equations is small at $t = 0;$ it is this smallness of 
		$\mathcal{E}_{\dParameter;\upgamma;\upmu}(0)$ that will lead to a global solution of the reduced equations. 
	\item In Section \ref{S:AlgebraicEstimates}, we provide algebraic estimates for the inhomogeneities in the reduced
		equations under the assumption that the wave coordinate condition holds. We also derive differential inequalities for the 
		null components of $\Lie_{\mathcal{Z}}^I \Far,$ and provide algebraic estimates for the corresponding inhomogeneities.
	\item In Section \ref{S:WeightedEnergy}, we prove weighted energy estimates for solutions to the electromagnetic
		equations of variation. We also recall some results of \cite{hLiR2010} that provide analogous weighted energy estimates for 
		both scalar wave equations and tensorial systems of wave equations with principal part $(g^{-1})^{\kappa \lambda} 
		\nabla_{\kappa} \nabla_{\lambda}.$
	\item In Section \ref{S:WaveEquationDecay}, we recall some results of \cite{hLiR2010} that
		provide pointwise decay estimates for both scalar wave equations and tensorial systems of wave equations with principal 
		part $(g^{-1})^{\kappa \lambda} \nabla_{\kappa} \nabla_{\lambda}.$
	\item In Section \ref{S:LocalExistence}, we state a basic local existence result and continuation principle for the
		reduced equations. The continuation principle will be used in Section \ref{S:GlobalExistence} in order to
		deduce small-data global existence for the reduced equations.
	\item In Section \ref{S:DecayFortheReducedEquations}, we introduce our bootstrap assumption on the energy 
		$\mathcal{E}_{\dParameter;\upgamma;\upmu}(t).$ We then use this assumption to deduce a collection of pointwise decay estimates for 
		solutions to the reduced equations under the assumption that the wave coordinate condition holds.
	\item In Section \ref{S:GlobalExistence}, we prove our main stability results. The results are separated into two theorems.
		In Theorem \ref{T:ImprovedDecay}, we use the decay estimates proved in Section \ref{S:DecayFortheReducedEquations} to
		derive a ``strong'' inequality for the energy $\mathcal{E}_{\dParameter;\upgamma;\upmu}(t);$ 
		\textbf{the proof of this theorem is the centerpiece of the article}. Theorem \ref{T:MainTheorem}, which is our main 
		theorem, is then an easy consequence of Theorem \ref{T:ImprovedDecay} and the continuation principle of Section 
		\ref{S:LocalExistence}. Both of these theorems rely upon the assumption that the wave coordinate condition holds.
\end{itemize}

\section{Notation}\label{S:Notation}
For convenience, in this section we collect together some of the important notation that is introduced throughout the article.

\noindent \hrulefill
\ \\

\subsection{Constants}
We use the symbols $c,$ $\widetilde{c},$ $C,$ and $\widetilde{C}$ to denote generic \emph{positive} constants that are free to vary from line to line. In general, they can depend on many quantities, but in the small-solution regime that we consider in this article, they can be chosen uniformly. Sometimes it is illuminating to explicitly indicate one of the quantities $\mathfrak{Q}$ that a constant depends on; we do by writing e.g. $C_{\mathfrak{Q}}.$ If $A$ and $B$ are two quantities, then we often write 
\begin{align*}
	A \lesssim B
\end{align*}
to mean that ``there exists a $C > 0$ such that $A \leq C B.$'' Furthermore, if $A \lesssim B$ and $B \lesssim A,$ then we 
often write

\begin{align*}
	A \approx B.
\end{align*}

\subsection{Indices} \label{SS:Indices}
\begin{itemize}
	\item Lowercase Latin indices $a,b,j,k,$ etc. take on the values $1,2,3.$
	\item Greek indices $\kappa, \lambda, \mu, \nu,$ etc. take on the values $0,1,2,3.$
	\item Uppercase Latin indices $A,B$ etc. take on the values $1,2$ and are used to enumerate
		the two Minkowski-orthogonal null frame vectors tangent to the spheres $S_{r,t}.$
	\item As a convention, the tensorfields $\Far_{\mu \nu},$ $\Max_{\mu \nu},$ $R_{\mu \nu},$ $T_{\mu \nu},$ 
		$\epsilon_{\mu \nu \kappa \lambda	},$ and $N_{\mu \nu \kappa \lambda}$ are assumed to ``naturally'' have all of their 
		indices downstairs, and unless indicated otherwise, all indices on all tensors 
		are lowered and raised with the Minkowski metric $m_{\mu \nu}$ 
		and its inverse $(m^{-1})^{\mu \nu};$ e.g. $T^{\mu \nu} \eqdef (m^{-1})^{\mu \kappa} (m^{-1})^{\nu \lambda}T_{\kappa 
		\lambda}.$
	\item The symbol $\#$ is used to indicate that all indices of a given tensorfield have been raised with $g^{-1};$ e.g.
		$T^{\#\mu \nu} \eqdef (g^{-1})^{\mu \kappa} (g^{-1})^{\nu \lambda}T_{\kappa \lambda}.$
	\item Repeated indices are summed over.
\end{itemize}

\subsection{Coordinates}

\begin{itemize}
	\item $\lbrace x^{\mu} \rbrace_{\mu = 0,1,2,3}$ denotes the wave coordinate system.
	\item $t = x^0,$ $x = (x^1,x^2,x^3).$
	\item $q = r - t,$ $s = r + t$ are the null coordinates of the spacetime point $(t,x),$ where $r = |x|.$
	\item $q_{-} = 0$ if $q \geq 0$ and $q_{-} = |q|$ if $q < 0.$ 
	\item $\omega^j = x^j/r,$ $(j = 1,2,3).$
\end{itemize}

\subsection{Surfaces}
	Relative to the wave coordinate system:
\begin{itemize}
	\item $C_{s}^- \eqdef \lbrace (\tau,y) \ | \ |y| + \tau = s \rbrace$
		are the ingoing Minkowskian null cones.
	\item $C_{q}^+ \eqdef \lbrace (\tau,y) \ | \ |y| - \tau = q \rbrace$ are the outgoing Minkowskian null cones. 
	\item $\Sigma_t \eqdef \lbrace (\tau,y) \ | \ \tau = t \rbrace$ are the constant Minkowskian time slices.
	\item $S_{r,t} \eqdef \lbrace (\tau,y) \ | \ \tau = t, |y| = r \rbrace$ are the Euclidean spheres.
\end{itemize}

\subsection{Metrics and volume forms}
\begin{itemize}
	\item $m_{\mu \nu}$ denotes the standard Minkowski metric on $\mathbb{R}^{1+3};$ in our wave coordinate system,
		$m_{\mu \nu} = \mbox{diag}(-1,1,1,1).$
	\item $\underline{m}$ denotes the Minkowskian first fundamental form of $\Sigma_t;$ in our wave coordinate system, 
		$\underline{m}_{\mu \nu} = \mbox{diag}(0,1,1,1).$
	\item $\angm$ denotes the Minkowskian first fundamental form of $S_{r,t};$ relative to an arbitrary coordinate system, \\
		$\angm_{\mu \nu} = m_{\mu \nu} + \frac{1}{2}\big(L_{\mu} \uL_{\nu} + \uL_{\mu} L_{\nu} \big),$
		where $\uL, L$ are defined in Section \ref{SS:NullFrames}.
	\item $g_{\mu \nu}$ denotes the spacetime metric.
	\item $g_{\mu \nu} = m_{\mu \nu} + h_{\mu \nu}^{(0)} + h_{\mu \nu}^{(1)}$ is the splitting of the spacetime metric into
		the Minkowski metric $m_{\mu \nu},$ the Schwarzschild tail $h_{\mu \nu}^{(0)} 
		= \chi\big(\frac{r}{t}\big)\chi(r)\frac{2M}{r} \delta_{\mu \nu},$ and the remainder $h_{\mu \nu}^{(1)}.$ 
	\item $h_{\mu \nu} = h_{\mu \nu}^{(0)} + h_{\mu \nu}^{(1)}.$
	\item $(g^{-1})^{\mu \nu} = (m^{-1})^{\mu \nu} + H_{(0)}^{\mu \nu} + H_{(1)}^{\mu \nu}$ 
		is the splitting of the inverse spacetime metric into the inverse Minkowski metric 
		$(m^{-1})^{\mu \nu},$ the Schwarzschild tail $H_{(0)}^{\mu \nu} = - \chi\big(\frac{r}{t}\big) \chi(r) \frac{2M}{r} 
		\delta^{\mu \nu},$ and the remainder $H_{(1)}^{\mu \nu}.$ 
	\item $H^{\mu \nu} = H_{(0)}^{\mu \nu} + H_{(1)}^{\mu \nu}.$
	\item $\mathring{\underline{g}}$ denotes the first fundamental form of the Cauchy hypersurface $\Sigma_0$
		relative to the spacetime metric $g.$
	\item $\mathring{\underline{g}}_{jk} = \delta_{jk} + \chi(r)\frac{2M}{r} \delta_{jk} + \mathring{\underline{h}}_{jk}^{(1)}$
		is the splitting of $\mathring{\underline{g}}_{jk}$ into the Schwarzschild tail $\chi(r)\frac{2M}{r} \delta_{jk}$
		and the remainder $\mathring{\underline{h}}_{jk}^{(1)}.$
	\item $\Minkvolume_{\mu \nu \kappa \lambda} = |\mbox{det} \, m|^{1/2} [\mu \nu \kappa \lambda]$ denotes the volume form of
		the Minkowski metric $m;$ $[\mu \nu \kappa \lambda]$ is totally antisymmetric with normalization
		$[0123] = 1;$ $|\mbox{det} \, m|^{1/2} = 1$ in our wave coordinate system.
	\item $\epsilon_{\mu \nu \kappa \lambda} = |\mbox{det} \, g|^{1/2} [\mu \nu \kappa \lambda]$ denotes the volume form of
		the spacetime metric $g.$
	\item $\epsilon^{\# \mu \nu \kappa \lambda} = -|\mbox{det} \, g|^{-1/2} [\mu \nu \kappa \lambda]$ denotes the volume form of
		the spacetime metric $g$ with all of the indices raised with $g^{-1}.$
	\item $\uvolume_{\nu \kappa \lambda} = [0 \nu \kappa \lambda]$ denotes the Euclidean volume
		form of the surfaces $\Sigma_t$ viewed as embedded submanifolds of Minkowski spacetime equipped with
		the wave coordinate system. 
	\item $\uvolume_{ijk}= [ijk]$ denotes the Euclidean volume form of the surfaces $\Sigma_t$ viewed as abstract $3-$manifolds.
	\item $\angupsilon_{\mu \nu} = \Minkvolume_{\mu \nu \kappa \lambda} \uL^{\kappa} L^{\lambda}$ denotes the Euclidean
		volume form of the spheres $S_{r,t}.$
\end{itemize}

\subsection{Hodge duals} \label{SS:Hodge}
	For an arbitrary two-form $\Far_{\mu \nu}:$
\begin{itemize}
	\item $\Fardual_{\mu \nu} = \frac{1}{2} g_{\mu \mu'} g_{\nu \nu'} 
		\epsilon^{\# \mu' \nu' \kappa \lambda} \Far_{\kappa \lambda}
		= - \frac{1}{2} |\mbox{det} \, g|^{-1/2} g_{\mu \mu'} g_{\nu \nu'}
		[\mu' \nu' \kappa \lambda] \Far_{\kappa \lambda}$ denotes the Hodge dual
		of $\Far_{\mu \nu}$ with respect to the spacetime metric $g_{\mu \nu}.$
	\item $\FarMinkdual_{\mu \nu} = \frac{1}{2} \Minkvolume_{\mu \nu}^{\ \ \kappa \lambda} \Far_{\kappa \lambda}
	 = - \frac{1}{2} |\mbox{det} \, m|^{-1/2} 
		m_{\mu \mu'} m_{\nu \nu'} [\mu' \nu' \kappa \lambda]\Far_{\kappa \lambda}$ 
		denotes the Hodge dual of $\Far_{\mu \nu}$ with respect to the Minkowski metric $m_{\mu \nu}.$ In our wave coordinate 
		system, $|\mbox{det} \, m|^{-1/2} = 1.$
\end{itemize}

\subsection{Derivatives} \label{SS:Derivatives}
\begin{itemize}
	\item $\nabla$ denotes the Levi-Civita connection corresponding to $m.$
	\item $\mathscr{D}$ denotes the Levi-Civita connection corresponding to $g.$
	\item $\mathring{\underline{\mathscr{D}}}$ denotes the Levi-Civita connection corresponding to $\mathring{\underline{g}}.$
	\item $\unabla$ denotes the Levi-Civita connection corresponding to $\underline{m}.$
	\item $\angn$ denotes the Levi-Civita connection corresponding to $\angm.$
	\item $\conenabla$ denotes the projection of $\nabla$ onto the outgoing Minkowski null cones; 
		i.e., $\conenabla_{\mu} = \coneproject_{\mu}^{\ \kappa} \nabla_{\kappa},$ where
		$\coneproject_{\mu}^{\ \nu} = \delta_{\mu}^{\nu} + \frac{1}{2}L_{\mu} \uL^{\nu}$ projects vectors $X^{\mu}$ onto
		the outgoing Minkowski null cones.
	\item In our wave coordinate system $\lbrace x^{\mu} \rbrace_{\mu=0,1,2,3},$ 
		$\partial_{\mu} = \frac{\partial}{\partial x^{\mu}},$ $\nabla_{\mu} = \nabla_{\frac{\partial}{\partial x^{\mu}}}.$
	\item In our wave coordinate system, $\partial_r = \omega^a \partial_a$ denotes the radial 
		derivative, where $\omega^j = x^j/r.$
	\item In our wave coordinate system,
		$\partial_s \eqdef \frac{1}{2}(\partial_r + \partial_t), \partial_q \eqdef \frac{1}{2}(\partial_r - \partial_t)$
		denote the null derivatives; $\partial_q$ denotes partial differentiation at fixed $s$ and fixed angle $x/|x|,$
		while $\partial_s$ denotes partial differentiation at fixed $q$ and fixed angle $x/|x|,$
	\item If $X$ is a vectorfield and $\phi$ is a function, then $X \phi = X^{\kappa} \partial_{\kappa} \phi.$
	\item $\nabla_X$ denotes the differential operator $X^{\kappa} \nabla_{\kappa}.$
	\item $\underline{\nabla}_X$ denotes the differential operator $X^{\kappa} 
		\underline{\nabla}_{\kappa}.$
	\item $\angn_X$ denotes the differential operator $X^{\kappa} \angn_{\kappa}.$
	\item $\Lie_X$ denotes the Lie derivative with respect to the vectorfield $X.$
	\item $[X,Y]^{\mu} = (\Lie_X Y)^{\mu} = X^{\kappa}\partial_{\kappa}Y^{\mu} - Y^{\kappa}\partial_{\kappa} X^{\mu}$
		denotes the Lie bracket of the vectorfields $X$ and $Y.$
	\item For $Z \in \mathcal{Z},$ $\hat{\nabla}_Z = \nabla_Z + c_Z$ denotes the modified covariant derivative,
		where the constant $c_Z$ is defined in Section \ref{SS:Killingnotation}.
	\item For $Z \in \mathcal{Z},$ $\Liemod_Z = \Lie_Z + 2c_Z$ denotes the modified Lie derivative,
		where the constant $c_Z$ is defined in Section \ref{SS:Killingnotation}.
	\item $\nabla^I U,$ $\unabla^I U,$ $\nabla_{\mathcal{Z}}^I U,$ $\hat{\nabla}_{\mathcal{Z}}^I U,$ 
		$\Lie_{\mathcal{Z}}^I U,$ and $\Liemod_{\mathcal{Z}}^I U$ respectively 
		denote an $|I|^{th}$ order iterated Minkowski covariant derivative, iterated Euclidean (spatial) covariant derivative,
		iterated Minkowski $\mathcal{Z}-$covariant derivative, iterated modified Minkowski $\mathcal{Z}-$covariant derivative,
		iterated $\mathcal{Z}-$Lie  derivative, and iterated modified $\mathcal{Z}-$Lie derivative of the tensorfield $U.$
	\item $\Square_m = (m^{-1})^{\kappa \lambda} \nabla_{\kappa} \nabla_{\lambda}$ denotes the standard
		Minkowski wave operator.
	\item $\widetilde{\Square}_g = (g^{-1})^{\kappa \lambda} \nabla_{\kappa} \nabla_{\lambda}$ denotes the 
		reduced wave operator corresponding to the spacetime metric $g.$ Note that $\nabla$ is the \emph{Minkowskian} connection.
\end{itemize}

\subsection{Minkowskian conformal Killing fields} \label{SS:Killingnotation} \ \\
Relative to the wave coordinate system $\lbrace x^{\mu} \rbrace_{\mu=0,1,2,3} = (t,x):$
\begin{itemize}
	\item $\partial_{\mu} = \frac{\partial}{\partial x^{\mu}},$ $(\mu=0,1,2,3),$ denotes a translation vectorfield.
	\item $\Omega_{jk} = x_j \frac{\partial}{\partial x^k} - x_k \frac{\partial}{\partial x^j},$ 
		$(1 \leq j < k \leq 3),$ denotes a rotation vectorfield.
	\item $\Omega_{0j} = -t \frac{\partial}{\partial x^j} - x_j \frac{\partial}{\partial t},$ 
		$(j=1,2,3),$ denotes a Lorentz boost vectorfield.
	\item $S = x^{\kappa} \frac{\partial}{\partial x^{\kappa}}$ denotes the scaling vectorfield.
	\item $\mathcal{O} = \big\lbrace \Omega_{jk} \big\rbrace_{1 \leq j < k \leq 3}$ are the rotational
		Minkowskian Killing fields.
	\item $\mathcal{Z} = \big\lbrace \frac{\partial}{\partial x^{\mu}}, 
		\Omega_{\mu \nu}, S \big\rbrace_{0 \leq \mu < \nu \leq 3}.$
	\item For $Z \in \mathcal{Z},$ $^{(Z)}\pi_{\mu \nu} = \nabla_{\mu} Z_{\nu} + \nabla_{\nu} Z_{\mu} = c_Z m_{\mu \nu}$
		is the Minkowskian deformation tensor of $Z,$	where $c_Z$ is a constant. 
	\item Commutation properties with the Maxwell-Maxwell term: \\
		$\Liemod_{\mathcal{Z}}^I \Big\lbrace \big[(m^{-1})^{\mu \kappa} (m^{-1})^{\nu \lambda} - (m^{-1})^{\mu \lambda} 
			(m^{-1})^{\nu \kappa}\big] \nabla_{\mu} \Far_{\kappa \lambda} \Big\rbrace
		= \big[(m^{-1})^{\mu \kappa} (m^{-1})^{\nu \lambda} - (m^{-1})^{\mu \lambda} (m^{-1})^{\nu \kappa}\big] 
			\nabla_{\mu} \Lie_{\mathcal{Z}}^I \Far_{\kappa \lambda}.$
	\item Commutation properties with the Minkowski wave operator \\
		$\Square_m = (m^{-1})^{\kappa \lambda} \nabla_{\kappa} \nabla_{\lambda}:$ \\
		$[\Square_m, \partial_{\mu}] = [\Square_m,$ $\Omega_{\mu \nu}] = 0,$ \ $[\Square_m, S] = 2 \Square_m,$ \
		$[\nabla_Z, \Square_m] = - c_Z \Square_m,$ \ $\Square_m \nabla_Z \phi = \nablamod \Square_m \phi.$ 
\end{itemize}

\subsection{Minkowskian null frames} \label{SS:NullFrames}
\begin{itemize}
	\item $\uL = \partial_t - \partial_r$ denotes the Minkowskian null geodesic vectorfield transversal to the $C_{q}^+;$ it 	
		generates the cones $C_{s}^-.$
	\item $L = \partial_t + \partial_r$ denotes the Minkowskian null geodesic vectorfield generating the cones $C_q^+.$
	\item $e_A, \ A = 1,2$ denotes Minkowski-orthonormal vectorfields spanning the tangent space of the spheres $S_{r,t}.$  
	\item The set $\mathcal{L} \eqdef \lbrace L \rbrace$ contains only $L.$
	\item The set $\mathcal{T} \eqdef \lbrace L, e_1, e_2 \rbrace$ denotes the frame vector fields tangent
		to the $C_q^+.$
	\item The set $\mathcal{N} \eqdef \lbrace \uL,L, e_1, e_2 \rbrace$ denotes the entire Minkowski null frame.
\end{itemize}

\subsection{Minkowskian null frame decomposition}

\begin{itemize}
	\item For an arbitrary vectorfield $X$ and frame vector $N \in \mathcal{N},$ we define
		$X_N = X_{\kappa} N^{\kappa},$ where $X_{\mu} = m_{\mu \kappa} X^{\kappa}.$
	\item For an arbitrary vectorfield $X = X^{ \kappa}\partial_{\kappa} = X^{L} L + X^{\uL} \uL
		+ X^{A} e_A,$ where \\
		$X^{L} = - \frac{1}{2}X_{\uL},$ $X^{\uL} = - \frac{1}{2}X_{L},$ $X^A = X_A.$
	\item For an arbitrary pair of vectorfields $X,Y:$ \\
		$m(X,Y) = m_{\kappa \lambda} X^{\kappa} X^{\lambda} 
		= X^{\kappa}Y_{\kappa} = -\frac{1}{2}X_{L}Y_{\uL} - \frac{1}{2}X_{\uL}Y_{L} + X_A Y_A.$
\end{itemize}

If $\Far_{\mu \nu}$ is any two-form, its Minkowskian null components are:

\begin{itemize}
	\item $\ualpha_{\mu} = \angm_{\mu}^{\ \nu} \Far_{\nu \lambda} \uL^{\lambda}.$
	\item $\alpha_{\mu} = \angm_{\mu}^{\ \nu} \Far_{\nu \lambda} L^{\lambda}.$
	\item $\rho = \frac{1}{2} \Far_{\kappa \lambda}\uL^{\kappa} L^{\lambda}.$
	\item $\sigma = \frac{1}{2} \angupsilon^{\kappa \lambda} \Far_{\kappa \lambda}.$
\end{itemize}

\subsection{Electromagnetic decompositions}

If $\Far_{\mu \nu}$ is any two-form, $\Maxdual_{\mu \nu} = g_{\mu \kappa} g_{\nu \lambda} 
\frac{\partial \Ldual}{\partial \Far_{\kappa \lambda}},$
and $\hat{N}^{\mu}$ is the future-directed unit $g-$normal to $\Sigma_t,$ then its electromagnetic components are:

\begin{itemize}
	\item $\mathfrak{\Electricfield}_{\mu} = \Far_{\mu \kappa}\hat{N}^{\kappa}.$ 
	\item $\mathfrak{\Magneticinduction}_{\mu} = - \Fardual_{\mu \kappa}\hat{N}^{\kappa}.$ 
	\item $\mathfrak{\Displacement}_{\mu} = - \Maxdual_{\mu \kappa} \hat{N}^{\kappa}.$  	
	\item $ \mathfrak{\Magneticfield}_{\mu} = - \Max_{\mu \kappa}\hat{N}^{\kappa}.$
\end{itemize}

If $\Far_{\mu \nu}$ is any two-form, then relative to the wave coordinate system, its Minkowskian
electromagnetic components are:

\begin{itemize}
	\item $\Electricfield_{\mu} = \Far_{\mu 0}.$  
	\item $\Magneticinduction_{\mu} = - \FarMinkdual_{\mu 0}.$ 
	\item $\Displacement_{\mu} = - 	\MaxMinkdual_{\mu 0}.$  	
	\item $\Magneticfield_{\mu} = - \Max_{\mu 0}.$
\end{itemize}

\subsection{Seminorms and energies}

For an arbitrary type $\binom{0}{2}$ tensorfield $P_{\mu \nu},$ and $\mathcal{V}, \mathcal{W} \in \lbrace \mathcal{L}, \mathcal{T},\mathcal{N} \rbrace:$

\begin{itemize}
	\item $|P|_{\mathcal{V} \mathcal{W}} = \sum_{V \in \mathcal{V}, W \in \mathcal{W}} |V^{\kappa} W^{\lambda} 
		P_{\kappa \lambda}|.$
	\item $|\nabla P|_{\mathcal{V} \mathcal{W}} = \sum_{N \in \mathcal{N}, V \in \mathcal{V}, W \in \mathcal{W}} 
		|V^{\kappa} W^{\lambda} N^{\gamma} \nabla_{\gamma} P_{\kappa \lambda}|.$
	\item $|\conenabla P|_{\mathcal{V} \mathcal{W}} = \sum_{T \in \mathcal{T}, V \in \mathcal{V}, W \in \mathcal{W}} 
		|V^{\kappa} W^{\lambda} T^{\gamma} \nabla_{\gamma} P_{\kappa \lambda}|.$
	\item $|P| = |P|_{\mathcal{N} \mathcal{N}}.$
	\item $|\nabla P| = |\nabla P|_{\mathcal{N} \mathcal{N}}.$
	\item $|\conenabla P| = |\conenabla P|_{\mathcal{N} \mathcal{N}}.$
	\item We use similar notation for an arbitrary tensorfield $U$ of type $\binom{n}{m}.$
\end{itemize}

For an arbitrary tensorfield $U$ defined on the Euclidean space $\Sigma_0$ with Euclidean coordinate system 
$x = (x^1,x^2,x^3):$

\begin{itemize}
	\item $\| U \|_{L^2}^2 = \int_{x \in \mathbb{R}^3} |U(x)|^2 \, d^3 x$ is the square of the standard spatial $L^2$ norm of $U.$
	\item $\| U \|_{L^{\infty}} = \mbox{ess} \sup_{x \in \mathbb{R}^3} |U(x)|$ is the standard spatial $L^{\infty}$ norm of $U.$
	\item $\| U \|_{H_{\eta}^{\dParameter}}^2 = \sum_{|I| \leq \dParameter } \int_{x \in \mathbb{R}^3} 
		(1 + |x|^2)^{(\eta + |I|)} |\unabla^I U(x)|^2 \, d^3 x$ is the square of a weighted Sobolev norm of $U.$ 
	\item $\| U \|_{C_{\eta}^{\dParameter}}^2 =  
		\sum_{|I| \leq \dParameter } \mbox{ess} \sup_{x \in \mathbb{R}^3} (1 + |x|^2)^{(\eta + |I|)} |\unabla^I U(x)|^2$ 
		is the square of a weighted pointwise norm of $U.$
\end{itemize}

For arbitrary abstract initial data 
$(\mathring{\underline{h}}^{(1)}_{jk}, \mathring{K}_{jk}, \mathring{\mathfrak{\Displacement}}_j, \mathring{\mathfrak{\Magneticinduction}}_j)$ on the manifold $\mathbb{R}^3:$

\begin{itemize}
	\item $E_{\dParameter;\upgamma}^2(0) 
		= \| \underline{\nabla} \mathring{\underline{h}}^{(1)} \|_{H_{1/2 + \upgamma}^{\dParameter}}^2 
		\ + \ \| \mathring{K} \|_{H_{1/2 + \upgamma}^{\dParameter}}^2 
		\ + \ \| \mathring{\mathfrak{\Displacement}} \|_{H_{1/2 + \upgamma}^{\dParameter}}^2 
		\ + \ \| \mathring{\mathfrak{\Magneticinduction}} \|_{H_{1/2 + \upgamma}^{\dParameter}}^2$
		is the square of the norm of the abstract initial data.
\end{itemize}

For an arbitrary symmetric type $\binom{0}{2}$ tensorfield $h_{\mu \nu}^{(1)}$ and an arbitrary two-form $\Far_{\mu \nu}:$

\begin{itemize}
	\item $\mathcal{E}_{\dParameter;\upgamma;\upmu}^2(t) = \underset{0 \leq \tau \leq t}{\mbox{sup}} 
		\sum_{|I| \leq \dParameter } \int_{\Sigma_{\tau}} 
		\Big\lbrace |\nabla\nabla_{\mathcal{Z}}^I h^{(1)}|^2 + |\Lie_{\mathcal{Z}}^I \Far|^2 \Big\rbrace w(q) \, d^3 x$
		is the square of the energy of the pair $(h_{\mu \nu}^{(1)}, \Far_{\mu \nu}).$ 
\end{itemize}

\subsection{$O^{\dParameter}()$ and $o^{\dParameter}()$} \label{SS:Oando}
	
\begin{itemize}
	\item Given an $\dParameter -$times continuously differentiable 
	function $f(\mathfrak{Q}_1,\cdots,\mathfrak{Q}_m)$ depending on the tensorial quantities
	$\mathfrak{Q}_1,\cdots,\mathfrak{Q}_m,$ we write 
	$f(\mathfrak{Q}_1,\cdots,\mathfrak{Q}_m)= O^{\dParameter}\big(|\mathfrak{Q}_1|^{p_1} \cdots |\mathfrak{Q}_k|^{p_k};
	\mathfrak{Q}_{k+1},\cdots,\mathfrak{Q}_m\big)$ if we can decompose \\
	$f(\mathfrak{Q}_1,\cdots,\mathfrak{Q}_m) 
	= \sum_{i = 1}^n p_i(\mathfrak{Q}_1,\cdots,\mathfrak{Q}_k) \widetilde{f}_i(\mathfrak{Q}_1,\cdots,\mathfrak{Q}_m),$
	where $n$ is a positive integer, each $p_i(\mathfrak{Q}_1,\cdots,\mathfrak{Q}_k)$ is a polynomial in the components of 
	$\mathfrak{Q}_1,\cdots,\mathfrak{Q}_k$
	that satisfies $|p_i(\mathfrak{Q}_1,\cdots,\mathfrak{Q}_k)| \lesssim |\mathfrak{Q}_1|^{p_1} \cdots |\mathfrak{Q}_k|^{p_k}$
	in a neighborhood of the origin, and $\widetilde{f}_i(\cdot)$ is $\dParameter -$times continuously 
	differentiable in a neighborhood of the origin. 
	\item Given an $\dParameter -$times continuously differentiable function $f(x),$ we write $f(x) = o^{\dParameter}(r^{-a})$ if 
	$\lim_{r \to \infty} \frac{|\unabla^I f(x)|}{r^{a + |I|}} = 0$ for $|I| \leq \dParameter .$ 
\end{itemize}

\subsection{Fixed constants} \label{SS:FixedConstants}
The fixed constants $\dParameter,$ $\updelta,$ $\upgamma,$ $\upmu,$ $\upgamma',$ $\upmu'$ are subject to the following constraints:

\begin{itemize}
	\item To prove our global stability theorem, we assume that $\dParameter$ is an integer satisfying $\dParameter \geq 8.$ 
	\item $0 < \updelta < \frac{1}{4}.$
	\item $0 < \updelta < \upgamma < 1/2.$
	\item $0 < \upgamma' < \upgamma - \updelta.$
	\item $0 < \updelta < \upmu' < \frac{1}{2}.$
	\item $0 < \upmu < \frac{1}{2} - \upmu'.$
\end{itemize}

\subsection{Weights} \label{SS:Weights}

\begin{itemize}
	\item $w = w(q) = \left \lbrace
		\begin{array}{lr}
    	1 \ + \ (1 + |q|)^{1 + 2 \upgamma}, &  \mbox{if} \ q > 0, \\
      1 \ + \ (1 + |q|)^{-2 \upmu}, & \mbox{if} \ q < 0,
    \end{array}
  \right.$ is the energy estimate weight function.
 	\item $\varpi = \varpi(q) = \left \lbrace
		\begin{array}{lr}
    	(1 + |q|)^{1 + \upgamma'}, &  \mbox{if} \ q > 0, \\
      (1 + |q|)^{1/2 - \upmu'}, & \mbox{if} \ q < 0,
    \end{array}
  \right.$ is the decay estimate weight function.
\end{itemize}

\section{The Einstein-Nonlinear Electromagnetic System in Wave Coordinates} \label{S:ENESinWaveCoordinates}
In this section, we discuss equations \eqref{E:IntroEinstein} - \eqref{E:IntrodMis0} in detail. We also discuss our assumptions on the electromagnetic Lagrangian and introduce our wave coordinate gauge. We then derive a reduced system of equations, which are equivalent to \eqref{E:IntroEinstein} - \eqref{E:IntrodMis0} in the wave coordinate gauge. Finally, we summarize the results by providing a version of the reduced equations that will be used throughout the remainder of the article. In particular, in this version, we distinguish between principal terms, which require a careful treatment, and ``error terms,'' which are, from the point of view of decay rates, relatively easy to estimate.

\noindent \hrulefill
\ \\

In this article, we consider the $1+3-$dimensional electro-gravitational system \eqref{E:IntroEinstein} - \eqref{E:IntrodMis0}, which we restate here for convenience:

\begin{subequations}
\begin{align} 
	R_{\mu \nu} - \frac{1}{2}g_{\mu \nu} R & = T_{\mu \nu}, && (\mu, \nu = 0,1,2,3), 
		\label{E:IntroEinsteinagain} \\
	(d \Far)_{\lambda \mu \nu} & = 0, && (\lambda, \mu, \nu = 0,1,2,3), \label{E:IntrodFaris0again} \\
	(d \Max)_{\lambda \mu \nu} & = 0, && (\lambda, \mu, \nu = 0,1,2,3). \label{E:IntrodMis0again}
\end{align}
\end{subequations}
We remark that the spacetimes we consider will always have the manifold structure $I \times \mathbb{R}^3$ for some ``time'' interval $I.$ The energy-momentum tensor $T_{\mu \nu}$ is given below in \eqref{E:electromagnetictensorloweroinTermsofLagrangian}, while $\Max_{\mu \nu}$ is related to $(g_{\mu \nu}, \Far_{\mu \nu})$ via the constitutive relation \eqref{E:Maxdualdef}. The precise forms of $T_{\mu \nu}$ and $\Max_{\mu \nu}$ depend on the chosen model of electromagnetism, which, as discussed in detail in Section \ref{SS:Lagrangianformluationofnonlinearelectromagnetism}, we assume is a Lagrangian-derived model subject to the restrictions \eqref{E:Ldualassumptions}, \eqref{E:DECL1} - \eqref{E:DECTrace} below. We recall (see e.g. \cite{dC2008}, \cite{rW1984}) the following relationships between the \emph{spacetime metric} $g_{\mu \nu},$ the \emph{Riemann curvature tensor}\footnote{Under our sign convention, $\mathscr{D}_{\mu} \mathscr{D}_{\nu} X_{\kappa} - \mathscr{D}_{\nu} \mathscr{D}_{\mu} X_{\kappa} = R_{\mu \nu \kappa}^{\ \ \ \ \lambda} X_{\lambda}.$}, $R_{\mu \kappa \nu}^{\ \ \ \ \lambda},$ the \emph{Ricci tensor} $R_{\mu \nu},$ the \emph{scalar curvature} $R,$ and the \emph{Christoffel symbols} $\Gamma_{\mu \ \nu}^{\ \kappa},$ which are valid in an arbitrary coordinate system on $\mathbb{R}^{1+3}:$

\begin{subequations}
\begin{align}
R_{\mu \kappa \nu}^{\ \ \ \ \lambda} & \eqdef 
	\partial_{\kappa} \Gamma_{\mu \ \nu}^{\ \lambda}
	- \partial_{\mu} \Gamma_{\kappa \ \nu}^{\ \lambda}
  + \Gamma_{\kappa \ \beta}^{\ \lambda} \Gamma_{\mu \ \nu}^{\ \beta}
	- \Gamma_{\mu \ \beta}^{\ \lambda} \Gamma_{\kappa \ \nu}^{\ \beta},	
	\label{E:EMBIRiemanndef} \\
R_{\mu \nu} & \eqdef R_{\mu \kappa \nu}^{\ \ \ \ \kappa} 
	= \partial_{\kappa} \Gamma_{\mu \ \nu}^{\ \kappa}
	- \partial_{\mu} \Gamma_{\kappa \ \nu}^{\ \kappa}
  + \Gamma_{\kappa \ \lambda}^{\ \kappa} \Gamma_{\mu \ \nu}^{\ \lambda}
	- \Gamma_{\mu \ \kappa}^{\ \lambda} \Gamma_{\lambda \ \nu}^{\ \kappa},	
	\label{E:Riccidef} \\
R & \eqdef (g^{-1})^{\kappa \lambda} R_{\kappa \lambda}, \label{E:EMBIRdef} \\
\Gamma_{\mu \ \nu}^{\ \kappa} & \eqdef \frac{1}{2} (g^{-1})^{\kappa \lambda}(\partial_{\mu} g_{\lambda \nu} 
	+ \partial_{\nu} g_{\mu \lambda} - \partial_{\lambda} g_{\mu \nu}). \label{E:EMBIChristoffeldef}
\end{align}
\end{subequations}
We also recall the following symmetry properties:

\begin{align}
	R_{\mu \nu} = R_{\nu \mu}, \\ 
	\Gamma_{\mu \ \nu}^{\ \kappa}= \Gamma_{\nu \ \mu}^{\ \kappa}.
\end{align}

We note for future use that taking the trace with respect to $g$ of each side of \eqref{E:IntroEinsteinagain} implies that

\begin{align}
	R = -(g^{-1})^{\kappa \lambda} T_{\kappa \lambda}.
\end{align}
Hence, \eqref{E:IntroEinsteinagain} is equivalent to

\begin{align}
	R_{\mu \nu} & = T_{\mu \nu} - \frac{1}{2} g_{\mu \nu} (g^{-1})^{\kappa \lambda} T_{\kappa \lambda}  	
	\tag{\ref{E:IntroEinsteinagain}'}.
\end{align}
Furthermore, we note that the twice-contracted Bianchi identities (see e.g. \cite{rW1984}) are the relation (see Section \ref{SS:Indices}
concerning our use of the notation $\#$)

\begin{align}
	\mathscr{D}_{\mu} \big(R^{\# \mu \nu} - \frac{1}{2}(g^{-1})^{\mu \nu} R \big) = 0, && (\nu =0,1,2,3),
\end{align}
so that by \eqref{E:IntroEinsteinagain}, $T_{\mu \nu}$ necessarily satisfies the following divergence-free condition:

\begin{align} \label{E:EMBITconservation}
	\mathscr{D}_{\mu} T^{\# \mu \nu} & = 0, && (\nu =0,1,2,3).
\end{align}
In the above expressions, $\mathscr{D}$ denotes the Levi-Civita connection corresponding to $g_{\mu \nu}.$

\subsection{Wave coordinates} \label{SS:WaveCoordinates}

In this article, we use the framework developed in \cite{hLiR2005}, \cite{hLiR2010} and work in a \emph{wave coordinate} system, which is defined to be a coordinate system in which

\begin{subequations}
\begin{align} \label{E:wavecoordinategauge1}
	\Gamma^{\mu} \eqdef (g^{-1})^{\kappa \lambda} \Gamma_{\kappa \ \lambda}^{\ \mu} = 0, && (\mu = 0,1,2,3).
\end{align}
The condition \eqref{E:wavecoordinategauge1} is also known as \emph{harmonic gauge} or \emph{de Donder gauge}. It is easy to check that the condition \eqref{E:wavecoordinategauge1} is equivalent to the conditions

\begin{align}
	g_{\mu \nu} (g^{-1})^{\kappa \lambda} \Gamma_{\kappa \ \lambda}^{\ \nu} & = 0, && (\mu = 0,1,2,3), 
		\label{E:wavecoordinategauge2} \\
	(g^{-1})^{\kappa \lambda} \partial_{\kappa} g_{\lambda \mu} - \frac{1}{2} (g^{-1})^{\kappa \lambda} \partial_{\mu} g_{\kappa 
		\lambda} & = 0, && (\mu = 0,1,2,3), \label{E:wavecoordinategauge3} \\
	\partial_{\nu} [\sqrt{|\mbox{det} \, g|}(g^{-1})^{\mu \nu}] & = 0, && (\mu = 0,1,2,3). \label{E:wavecoordinategauge4}
\end{align}
\end{subequations}
We also note that condition \eqref{E:wavecoordinategauge4} follows from the identity

\begin{align} \label{E:ContractedChristoffelIdendity}
	\Gamma^{\mu} \eqdef (g^{-1})^{\kappa \lambda} \Gamma_{\kappa \ \lambda}^{\ \mu}
	= - \frac{1}{\sqrt{|\mbox{det} \, g|}} \partial_{\nu} [\sqrt{|\mbox{det} \, g|}(g^{-1})^{\mu \nu}], && (\mu = 0,1,2,3),
\end{align}
which holds in any coordinate system. Furthermore, if the wave coordinate system is also interpreted to be a coordinate system in which the Minkowski metric takes the form $m_{\mu \nu} = \mbox{diag}(-1,1,1,1),$ then all coordinate derivatives $\partial$ can be interpreted as covariant
derivatives $\nabla,$ where $\nabla$ is the Levi-Civita connection corresponding to the Minkowski metric. \textbf{Throughout the article, we will often take this point of view, because it allows for a covariant interpretation of all of our equations}.

We remark that the use of wave coordinates in the context of the Einstein equations goes back at least to the work \cite{tD1921} of de Donder. However, the role of wave coordinates in the context of the local aspects of the initial-value problem formulation of the Einstein equations was realized to its fullest extent by Choquet-Bruhat in \cite{CB1952}. See Section \ref{SS:WaveCoordinatesPreserved} for further discussion on the viability of using wave coordinates to analyze the system \eqref{E:IntroEinsteinagain} - \eqref{E:IntrodMis0again}.

\subsection{The Lagrangian formulation of nonlinear electromagnetism} \label{SS:Lagrangianformluationofnonlinearelectromagnetism}

In this section, we recall some standard facts concerning a classical electromagnetic field theory in a Lorentzian spacetime 
$(\mathbb{R}^{1+3},g_{\mu \nu}).$ Our goal is to explain the origin of the equations \eqref{E:IntrodFaris0again} - \eqref{E:IntrodMis0again}. We remark that for our purposes in this section, we may assume that the spacetime is known. The fundamental quantity in such a classical electromagnetic field theory is the \emph{Faraday tensor} $\Far_{\mu \nu},$ an anti-symmetric type $\binom{0}{2}$ tensorfield (i.e., a two-form). We assume the \emph{Faraday-Maxwell law}, which is the postulate that $\Far_{\mu \nu}$ is closed:

\begin{align} \label{E:dFis0}
	(d \Far)_{\lambda \mu \nu} = 0,&& (\lambda, \mu, \nu = 0,1,2,3),
\end{align}
where $d$ denotes the exterior derivative operator. 

We restrict our attention to covariant theories of nonlinear electromagnetism arising from a Lagrangian $\mathscr{L}.$  In such a theory, the Hodge dual\footnote{For brevity, we often refer to $\Ldual$ as the Lagrangian.} $\Ldual$ of $\mathscr{L}$ is a scalar-valued function of the two invariants of the Faraday tensor, which we denote by $\Farinvariant_{(1)}$ and $\Farinvariant_{(2)}:$ 

\begin{subequations}
\begin{align}
	\Ldual & = \Ldual(\Farinvariant_{(1)},\Farinvariant_{(2)}), \\
	 \Farinvariant_{(1)} & = \Farinvariant_{(1)}[\Far] \eqdef \frac{1}{2} (g^{-1})^{\kappa \mu} (g^{-1})^{\lambda \nu} 
	 	\Far_{\kappa \lambda} \Far_{\mu \nu}, \label{E:firstinvariant} \\
	\Farinvariant_{(2)} & = \Farinvariant_{(2)}[\Far] \eqdef \frac{1}{4} (g^{-1})^{\kappa \mu} (g^{-1})^{\lambda \nu} 
		\Far_{\kappa \lambda} \Fardual_{\mu \nu}
		= \frac{1}{8} \epsilon^{\# \kappa \lambda \mu \nu} \Far_{\kappa \lambda} \Far_{\mu \nu}. \label{E:secondinvariant}
\end{align}
\end{subequations}
Throughout the article, we use $\star$ to denote the Hodge duality operator corresponding to the spacetime metric 
$g_{\mu \nu}:$ 

\begin{align}	
	\Fardual^{\#\mu \nu} & \eqdef \frac{1}{2} \epsilon^{\#\mu \nu \kappa \lambda} \Far_{\kappa \lambda}.
\end{align}
Here, $\epsilon^{\#\mu \nu \kappa \lambda}$ is totally anti-symmetric with normalization $\epsilon^{\#0123}= 
-|\mbox{det} \, g|^{-1/2},$ while $\epsilon_{\mu \nu \kappa \lambda}$ is totally anti-symmetric with normalization $\epsilon_{0123}= |\mbox{det} \, g|^{1/2}.$ See Section \ref{SS:Indices} concerning our use of the notation $\#.$ 

We now introduce the \emph{Maxwell tensor} $\Max_{\mu \nu},$ a two-form whose Hodge dual $\Maxdual_{\mu \nu}$ is defined by

\begin{align} \label{E:Maxdualdef}
	\Maxdual^{\# \mu \nu} & \eqdef \frac{\partial \Ldual}{\partial \Far_{\mu \nu}}.
\end{align}
We also postulate that $\Max_{\mu \nu}$ is closed:

\begin{align} \label{E:dMis0}
	(d \Max)_{\lambda \mu \nu} = 0,&& (\lambda, \mu, \nu = 0,1,2,3).
\end{align}
Taken together, \eqref{E:dFis0} and \eqref{E:dMis0} are the electromagnetic equations for $\Far_{\mu \nu}$ corresponding to $\Ldual.$ 

We remark for future use that it can be easily checked that equation \eqref{E:dFis0} is equivalent to any of

\begin{subequations}
\begin{align} 
	\mathscr{D}_{\lambda} \Far_{\mu \nu} + \mathscr{D}_{\nu} \Far_{\lambda \mu} + \mathscr{D}_{\mu} \Far_{\nu \lambda} &= 0,
		&& (\lambda, \mu, \nu = 0,1,2,3), \label{E:dFis0Dversion} \\
	\nabla_{\lambda} \Far_{\mu \nu} + \nabla_{\nu} \Far_{\lambda \mu} + \nabla_{\mu} \Far_{\nu \lambda} &= 0,
		&& (\lambda, \mu, \nu = 0,1,2,3), \label{E:dFis0nablaversion} \\
	\mathscr{D}_{\mu} \Fardual^{\# \mu \nu} & = 0,&& (\nu = 0,1,2,3), \label{E:DdivergenceofFardualis0} \\
	\nabla_{\mu} \FarMinkdual^{\mu \nu} & = 0,&& (\nu = 0,1,2,3), \label{E:DdivergenceofFarMinkdualis0}
\end{align}
\end{subequations}
and that equation \eqref{E:dMis0} is equivalent to any of

\begin{subequations}
\begin{align} 
	\mathscr{D}_{\lambda} \Max_{\mu \nu} + \mathscr{D}_{\nu} \Max_{\lambda \mu} + \mathscr{D}_{\mu} \Max_{\nu \lambda} &= 0,
		&& (\lambda, \mu, \nu = 0,1,2,3), \label{E:dMis0Dversion} \\
	\nabla_{\lambda} \Max_{\mu \nu} + \nabla_{\nu} \Max_{\lambda \mu} + \nabla_{\mu} \Max_{\nu \lambda} &= 0,
		&& (\lambda, \mu, \nu = 0,1,2,3), \label{E:dMis0nablaversion} \\
	\mathscr{D}_{\mu} \Maxdual^{\# \mu \nu} & = 0,&& (\nu = 0,1,2,3), \label{E:Euler-Lagrange} \\
	\nabla_{\mu} \MaxMinkdual^{\mu \nu}& = 0,&& (\nu = 0,1,2,3). \label{E:DdivergenceofMaxMinkdualis0}
\end{align}
\end{subequations}
In the above formulas, $\ostar$ denotes the Hodge duality operator corresponding to the Minkowski metric $m_{\mu \nu};$ this operator is defined in Section \ref{SS:Hodge}.

We state as a lemma the following identities, which will be used for various computations. We leave the proof
as a simple exercise for the reader.

\begin{lemma} \label{L:electromagneticidentities}
	\textbf{(Identities)}
	The following identities hold:
	
	\begin{subequations}
	\begin{align}
		\frac{\partial |\mbox{det} \, g|}{\partial g_{\mu \nu}} & = |\mbox{det} \, g| (g^{-1})^{\mu \nu}, \\
		\frac{\partial (g^{-1})^{\kappa \lambda}}{g_{\mu \nu}} & = -(g^{-1})^{\kappa \mu} (g^{-1})^{\lambda \nu}, \\
		\Farinvariant_{(2)}^2 & = |\mbox{det} \, \Far| |\mbox{det} \, g|^{-1}, \\
		(g^{-1})^{\kappa \lambda }\Far_{\mu \kappa} \Fardual_{\nu \lambda} & = \Farinvariant_{(2)} g_{\mu \nu}, \\
		\frac{\partial \Farinvariant_{(1)}}{\partial g_{\mu \nu}} & = 
			- g_{\kappa \lambda} \Far^{\#\mu \kappa}\Far^{\#\nu \lambda}, \\
		\frac{\partial \Farinvariant_{(2)}}{\partial g_{\mu \nu}} & = - \frac{1}{2} \Farinvariant_{(2)} (g^{-1})^{\mu \nu}, \\
		\frac{\partial \Farinvariant_{(1)}}{\partial \Far_{\mu \nu}} & = 2 \Far^{\#\mu \nu}, 
			\label{E:partialfirstinvariantpartialFarmunu} \\
		\frac{\partial \Farinvariant_{(2)}}{\partial \Far_{\mu \nu}} & = \Fardual^{\#\mu \nu}, \\
		\frac{\partial \Far^{\#\mu \nu}}{\partial \Far_{\kappa \lambda}} & 
			= (g^{-1})^{\mu \kappa} (g^{-1})^{\nu \lambda} - (g^{-1})^{\mu \lambda}(g^{-1})^{\nu \kappa}, \\
		\frac{\partial \Fardual^{\#\mu \nu}}{\partial \Far_{\kappa \lambda}} & = 
			\epsilon^{\#\mu \nu \kappa \lambda}, \label{E:partialFardualmunupartialFarkappalambda} \\ 
		\mathscr{D}_{\mu} \Farinvariant_{(1)} & = \Far^{\# \kappa \lambda} \mathscr{D}_{\mu} \Far_{\kappa \lambda}, 
			\qquad (\mu = 0,1,2,3), \label{E:DmuFarinvariant1} \\
		\mathscr{D}_{\mu} \Farinvariant_{(2)} & = \frac{1}{2} \Fardual^{\# \kappa \lambda} \mathscr{D}_{\mu} \Far_{\kappa \lambda},
			\qquad (\mu = 0,1,2,3). \label{E:DmuFarinvariant2}
	\end{align}
	\end{subequations}

\end{lemma}

\hfill $\qed$

\subsection{Assumptions on the electromagnetic Lagrangian}

The familiar linear Maxwell-Maxwell equations correspond to the Lagrangian

\begin{align} \label{E:LlinearMaxwell}
	\Ldual_{(Maxwell)} & = - \frac{1}{2} \Farinvariant_{(1)},
\end{align}
which by \eqref{E:Maxdualdef} and \eqref{E:partialfirstinvariantpartialFarmunu} leads to the relationship

\begin{align}
	\Max_{\mu \nu}^{(Maxwell)} = \Fardual_{\mu \nu}.
\end{align}
Roughly speaking, we will assume that our electromagnetic Lagrangian is a covariant perturbation of $\Ldual_{(Maxwell)}.$
More precisely, we make the following assumptions concerning our Lagrangian $\Ldual.$\vspace{.5in}

\begin{center}
	\textbf{\LARGE Assumptions on the electromagnetic Lagrangian}
\end{center}
We assume that in a neighborhood of $(0,0),$ $\Ldual$ is an $\dParameter + 2-$times (where $\dParameter \geq 8$) continuously differentiable function of the invariants $(\Farinvariant_{(1)},\Farinvariant_{(2)})$ that can be expanded as follows:

\begin{subequations}
\begin{align} \label{E:Ldualassumptions} 
	\Ldual = \Ldual_{(Maxwell)} + O^{\dParameter+2}\big(|(\Farinvariant_{(1)},\Farinvariant_{(2)})|^2\big).
\end{align}
The notation $O^{\dParameter+2}(\cdots)$ is defined in Section \ref{SS:Oando}.

We also assume that the corresponding energy-momentum tensor $T_{\mu \nu},$ which is defined below in \eqref{E:electromagnetictensorupper},
satisfies the \emph{dominant energy condition}, which is the assumption that 

\begin{align} \label{E:DEC}
	T_{\kappa \lambda} X^{\kappa} Y^{\lambda} \geq 0
\end{align}
\end{subequations}
whenever the following conditions are satisfied:

\begin{itemize}
	\item $X,$ $Y$ are both timelike (i.e., $g_{\kappa \lambda}X^{\kappa}X^{\lambda} < 0,$ $g_{\kappa \lambda}Y^{\kappa}Y^{\lambda} < 0)$ 
	\item $X,$ $Y$ are $g-$future-directed.
\end{itemize}	

As discussed in e.g. \cite{gGcH2001}, sufficient conditions for the dominant energy condition to hold are

\begin{subequations}
\begin{align}
	\frac{\partial \Ldual}{ \partial \Farinvariant_{(1)}} & < 0, \label{E:DECL1} \\
	\Ldual - \Farinvariant_{(1)}\frac{\partial \Ldual}{\partial \Farinvariant_{(1)}}
		- \Farinvariant_{(2)}\frac{\partial \Ldual}{\partial \Farinvariant_{(2)}} & \leq 0. \label{E:DECTrace}
\end{align}
\end{subequations}
We remark that it is straightforward to verify the sufficiency of these conditions
using equation \eqref{E:AlternateelectromagnetictensorloweroinTermsofLagrangian} below, and that condition
\eqref{E:DECTrace} is equivalent to the non-positivity of the trace of the energy-momentum tensor corresponding to $\Ldual.$
Furthermore, we recall that the trace vanishes in the case of the linear Maxwell-Maxwell model.

\vspace{.5in}	

\begin{remark}
	We make the $\dParameter +2-$times differentiability assumption because we will need to differentiate the equations
	\eqref{E:EquationSatisfiedbyMaxdualChainruleExpandedfirstNversion} below $\dParameter $ times in order to prove our
	main stability theorem.
\end{remark}

We will now derive an equivalent version of the electromagnetic equations that will be used throughout the remainder of the article. The final form, which is valid only in a wave coordinate system, is given below in Lemma \ref{L:MBIAsystem}. To begin, we use \eqref{E:Maxdualdef}, Lemma \ref{L:electromagneticidentities}, and the chain rule to compute that

\begin{align} \label{E:MaxdualintermsofLagrangian}
	\Maxdual^{\#\mu \nu} & = 2\frac{\partial \Ldual}{\partial \Farinvariant_{(1)}} \Far^{\#\mu \nu}
		+ \frac{\partial \Ldual}{\partial \Farinvariant_{(2)}}\Fardual^{\#\mu \nu}.
\end{align}
We then use \eqref{E:DdivergenceofFardualis0}, \eqref{E:Euler-Lagrange}, and \eqref{E:MaxdualintermsofLagrangian} to compute that the following equation holds:

\begin{align} \label{E:FirstEquationSatisfiedbyMaxdual}
	- 2 \frac{\partial \Ldual}{\partial \Farinvariant_{(1)}} \mathscr{D}_{\mu} \Far^{\#\mu \nu}
	- 2 \Far^{\#\mu \nu} \mathscr{D}_{\mu} \Big(\frac{\partial \Ldual}{\partial \Farinvariant_{(1)}}\Big)
	- \Fardual^{\#\mu \nu} \mathscr{D}_{\mu} \Big(\frac{\partial \Ldual}{\partial \Farinvariant_{(2)}}\Big) & = 0.
\end{align}
Furthermore, using the chain rule and the fact that $\mathscr{D}_{\mu} \phi = \nabla_{\mu} \phi$ 
for scalar-valued functions $\phi,$ it follows from \eqref{E:FirstEquationSatisfiedbyMaxdual} that

\begin{align} \label{E:EquationSatisfiedbyMaxdualChainruleExpanded}
	- 2 \frac{\partial \Ldual}{\partial \Farinvariant_{(1)}} \mathscr{D}_{\mu} \Far^{\#\mu \nu}
	& - \Big(2 \Far^{\#\mu \nu} \frac{\partial^2 \Ldual}{\partial \Farinvariant_{(1)}^2}
		+ \Fardual^{\#\mu \nu} \frac{\partial^2 \Ldual}{\partial \Farinvariant_{(1)} \partial \Farinvariant_{(2)}} 
		\Big) \nabla_{\mu} \Farinvariant_{(1)}  \\
	& - \Big(2 \Far^{\#\mu \nu} \frac{\partial^2 \Ldual}{\partial \Farinvariant_{(1)} \partial 
		\Farinvariant_{(2)}}+ \Fardual^{\#\mu \nu} \frac{\partial^2 \Ldual}{\partial \Farinvariant_{(2)}^2} \Big) 
		\nabla_{\mu} \Farinvariant_{(2)} = 0. \notag
\end{align}

We note for future use that equation \eqref{E:EquationSatisfiedbyMaxdualChainruleExpanded} can be expressed as

\begin{align} \label{E:EquationSatisfiedbyMaxdualChainruleExpandedfirstNversion}
	N^{\#\mu \nu \kappa \lambda} \mathscr{D}_{\mu} \Far_{\kappa \lambda} = 0,&& (\nu = 0,1,2,3),
\end{align}
where the tensorfield $N^{\#\mu \nu \kappa \lambda}$ is defined by

\begin{align} \label{E:firstNdef}
	N^{\#\mu \nu \kappa \lambda} 
	& \eqdef - \frac{\partial \Ldual}{\partial \Farinvariant_{(1)}} \Big\lbrace(g^{-1})^{\mu \kappa} (g^{-1})^{\nu \lambda} - (g^{-1})^{\mu \lambda} 
		(g^{-1})^{\nu \kappa} \Big\rbrace 
		\ - \ 2 \frac{\partial^2 \Ldual}{\partial \Farinvariant_{(1)}^2} \Far^{\#\mu \nu} \Far^{\# \kappa \lambda} \\
	& \ \ - \ \frac{\partial^2 \Ldual}{\partial \Farinvariant_{(1)} \partial 
		\Farinvariant_{(2)}} \Big\lbrace \Far^{\#\mu \nu} \Fardual^{\# \kappa \lambda} 
		\ + \ \Fardual^{\#\mu \nu} \Fardual^{\# \kappa \lambda} \Big\rbrace
		\ - \ \frac{1}{2} \frac{\partial^2 \Ldual}{\partial \Farinvariant_{(2)}^2} \Fardual^{\#\mu \nu} \Fardual^{\# \kappa \lambda}. \notag
\end{align}
We also note that $N^{\#\mu \nu \kappa \lambda}$ has the following symmetry properties, which will play an important role during our construction of suitable energies for $\Far_{\mu \nu}$ (and in particular during our proof of Lemma \ref{L:DivergenceofStress}):

\begin{subequations}
\begin{align}
	N^{\#\nu \mu \kappa \lambda} & = - N^{\#\mu \nu \kappa \lambda},&& (\kappa, \lambda, \mu, \nu = 0,1,2,3), 	
		\label{E:Nminussignproperty1} \\
	N^{\#\mu \nu \lambda \kappa } & = - N^{\#\mu \nu \kappa \lambda},&& (\kappa, \lambda, \mu, \nu = 0,1,2,3), 	
		\label{E:Nminussignproperty2} \\
	N^{\#\kappa \lambda \mu \nu } & = N^{\#\mu \nu \kappa \lambda},&& (\kappa, \lambda, \mu, \nu = 0,1,2,3). 	
		\label{E:Nsymmetryproperty}
\end{align}
\end{subequations}
The moral reason that the above properties are satisfied is that $N^{\#\mu \nu \kappa \lambda}$ is closely related to the Hessian of $\Ldual:$

\begin{align} \label{E:NisHessian}
	N^{\#\mu \nu \kappa \lambda} = - \frac{1}{2} \frac{\partial^2 \Ldual}{\partial \Far_{\mu \nu} \partial \Far_{\kappa \lambda}}
		\ + \ \frac{1}{2} \frac{\partial \Ldual}{\partial \Farinvariant_{(2)}}\epsilon^{\#\mu \nu \kappa \lambda}.
\end{align}
We have added the last term on the right-hand side of \eqref{E:NisHessian} in order to cancel a term appearing in the Hessian; this is permissible because equation \eqref{E:dFis0Dversion} implies that this term does not contribute to equation \eqref{E:EquationSatisfiedbyMaxdualChainruleExpandedfirstNversion}.

Our next goal is to formulate a ``reduced'' electromagnetic equation that is equivalent to equation \eqref{E:EquationSatisfiedbyMaxdualChainruleExpandedfirstNversion} in a wave coordinate system, and 
to decompose the reduced equation into the principal terms and error terms of an equation involving the Minkowski connection $\nabla.$ This is accomplished in Lemma \ref{L:MBIAsystem} below. Before proving this lemma, we first provide the following preliminary lemma, whose simple proof is left to the reader.

\begin{lemma} \label{L:gmhexpansions} \textbf{(Expansions)}
	Assume that the electromagnetic Lagrangian $\Ldual$ satisfies \eqref{E:Ldualassumptions}.
	Then in terms of the expansion $h_{\mu \nu} \eqdef g_{\mu \nu} - m_{\mu \nu}$ from \eqref{E:gmhexpansion}, and with 
	$H^{\mu \nu} \eqdef (g^{-1})^{\mu \nu} - (m^{-1})^{\mu \nu},$ we have that:
	
	\begin{subequations}
	\begin{align}
		H^{\mu \nu} & = - h^{\mu \nu} + O^{\infty}(|h|^2) 
			= - h^{\mu \nu} \ + \ O^{\infty}(|H|^2),&& \label{E:Hintermsofh} \\
		\nabla_{\lambda} (g^{-1})^{\mu \nu} & = - (g^{-1})^{\mu \mu'} (g^{-1})^{\nu \nu'} \nabla_{\lambda} h_{\mu' \nu'}
			= - (m^{-1})^{\mu \mu'} (m^{-1})^{\nu \nu'} \nabla_{\lambda} h_{\mu' \nu'} \ + \ O^{\infty}(|h||\nabla h|),&&
			\label{E:derivativeginversegmhexpansion} \\
		|\mbox{det} \, g| & =  1 \ + \ (m^{-1})^{\kappa \lambda} h_{\kappa \lambda} \ + \ O^{\infty}(|h|^2)
			= 1 \ - \ m_{\kappa \lambda} H^{\kappa \lambda} + O^{\infty}(|H|^2),&& \\
		|\mbox{det} \, g|^{1/2} & =  1 + \frac{1}{2} (m^{-1})^{\kappa \lambda} h_{\kappa \lambda} \ + \ O^{\infty}(|h|^2)
			= 1 \ - \ \frac{1}{2} m_{\kappa \lambda} H^{\kappa \lambda} + O^{\infty}(|H|^2),&& \\
		|\mbox{det} \, g|^{-1/2} & =  1 \ - \ \frac{1}{2}(m^{-1})^{\kappa \lambda} h_{\kappa \lambda} \ + \ O^{\infty}(|h|^2)
			= 1 \ + \ \frac{1}{2} m_{\kappa \lambda} H^{\kappa \lambda} \ + \ O^{\infty}(|H|^2),&& \label{E:detgminusonehalfexpansion} \\
		\epsilon^{\#\mu \nu \kappa \lambda} & = - \big(1 + O^{\infty}(|h|) \big) [\mu \nu \kappa \lambda],&& 
			\label{E:volumeformraisedhexpansion} \\
		\epsilon_{\mu \nu \kappa \lambda} & = \big(1 + O^{\infty}(|h|) \big) [\mu \nu \kappa \lambda],&& 
			\label{E:volumeformloweredhexpansion} \\
		\Far^{\#\mu \nu} & = \Far^{\mu \nu}
			\ + \ O^{\infty}(|h||\Far|) \eqdef (m^{-1})^{\mu \kappa}(m^{-1})^{\nu \lambda} \Far_{\kappa \lambda} \ + \ O^{\infty}(|h||\Far|) ,&& \\
		\Fardual^{\# \mu \nu} & 	= \FarMinkdual^{\mu \nu} \ + \ O^{\infty}(|h||\Far|)
			\eqdef - \frac{1}{2} [\mu \nu \kappa \lambda] \Far_{\kappa \lambda} \ + \ O^{\infty}(|h||\Far|), && 
			\label{E:FargDualIntermsofFarmDual}\\
		\Farinvariant_{(1)} & = \frac{1}{2} (m^{-1})^{\kappa \mu} (m^{-1})^{\lambda \nu} \Far_{\kappa \lambda} \Far_{\mu \nu} 
			\ + \ O^{\infty}(|h||\Far|^2),&& \\
		\Farinvariant_{(2)} & = - \frac{1}{8} [\mu \nu \kappa \lambda ] \Far_{\mu \nu} \Far_{\kappa \lambda} 
			\ + \ O^{\infty}(|h||\Far|^2),&& \\
		\Ldual & = - \frac{1}{4} (m^{-1})^{\eta \kappa} (m^{-1})^{\zeta \lambda} \Far_{\kappa \lambda} \Far_{\eta \zeta} 
			\ + \ O^{\dParameter+2}(|h||\Far|^2) \ + \ O^{\dParameter+2}(|\Far|^3;h),&& \\
		\nabla \Farinvariant_{(i)} & = O^{\infty}(|\Far||\nabla \Far|) 
			\ + \ O^{\infty}(|\nabla h||\Far|^2;h) \ + \ O^{\infty}(|h||\Far||\nabla \Far|), && (i = 1,2), \\
	\Max_{\mu \nu} & = \FarMinkdual_{\mu \nu} \ + \ O^{\dParameter+1}(|h||\Far|) 
		\ + \ O^{\dParameter+1}(|\Far|^3;h).&& \label{E:MaxintermsofFarMinkDualPlusError}
	\end{align}
	\end{subequations}
	
	In formulas \eqref{E:volumeformraisedhexpansion} - \eqref{E:volumeformloweredhexpansion}, 
	$[\mu \nu \kappa \lambda]$ is totally anti-symmetric with normalization $[0123] = 1,$ while $\star$ denotes the Hodge duality 
	operator corresponding to the spacetime metric $g_{\mu \nu},$ and $\ostar$ denotes the Hodge duality operator corresponding 
	to the Minkowski metric $m_{\mu \nu}.$ 	Furthermore, the notation $O(\cdots)$ is defined in Section \ref{SS:Oando}.
\end{lemma}

\hfill $\qed$

\subsection{The reduced electromagnetic equations} \label{SS:ReducedElectromagnetic}

In this section, we provide the aforementioned decomposition of the reduced electromagnetic equations.

\begin{lemma} \label{L:MBIAsystem} \textbf{(The reduced electromagnetic equations)}
	Assume that the wave coordinate condition \eqref{E:wavecoordinategauge1} holds.
	Then in terms of the expansion \eqref{E:gmhexpansion}, the system of electromagnetic equations  
	\eqref{E:dFis0}, \eqref{E:EquationSatisfiedbyMaxdualChainruleExpandedfirstNversion} is equivalent to the 
	following reduced system of equations:
	
	\begin{subequations}
	\begin{align}
		\nabla_{\lambda} \Far_{\mu \nu} + \nabla_{\mu} \Far_{\nu \lambda} + \nabla_{\nu} \Far_{\lambda \mu} & = 0, 
			&& (\lambda, \mu, \nu = 0,1,2,3), \label{E:dFis0Expansion} \\
		N^{\#\mu \nu \kappa \lambda} \nabla_{\mu} \Far_{\kappa \lambda} & =
			\mathscr{Q}_{(2;\Far)}^{\nu}(\nabla h, \Far) + O^{\dParameter}(|h||\nabla h||\Far|) + O^{\dParameter}(|\nabla h||\Far|^2;h), && (\nu = 0,1,2,3), 
			\label{E:dMis0Expansion}
	\end{align}
	\end{subequations}
	where
	
	\begin{align}
		N^{\#\mu \nu \kappa \lambda} & = \frac{1}{2} \Big\lbrace (m^{-1})^{\mu \kappa} (m^{-1})^{\nu \lambda} - 
			(m^{-1})^{\mu \lambda} (m^{-1})^{\nu \kappa}
			- h^{\mu \kappa} (m^{-1})^{\nu \lambda} + h^{\mu \lambda} (m^{-1})^{\nu \kappa}
			- (m^{-1})^{\mu \kappa} h^{\nu \lambda} + (m^{-1})^{\mu \lambda} h^{\nu \kappa}\Big\rbrace \label{E:secondNdef}  \\
			& \ \ + N_{\triangle}^{\#\mu \nu \kappa \lambda}, \notag \\
		\mathscr{Q}_{(2;\Far)}^{\nu}(\nabla h, \Far) & = 
			(m^{-1})^{\mu \kappa} (m^{-1})^{\nu \nu'} (m^{-1})^{\lambda \lambda'} (\nabla_{\mu} h_{\nu' \lambda'}) 
			\Far_{\kappa \lambda}. \label{E:Q2Farfirstdef}
	\end{align}
	
	Furthermore, 
	
	\begin{align} \label{E:NtriangleErrorProperties}
		N_{\triangle}^{\#\mu \nu \kappa \lambda} = O^{\dParameter}\big(|(h,\Far)|^2\big),
	\end{align}
	and like $N^{\#\mu \nu \kappa \lambda},$
	the tensorfield $N_{\triangle}^{\#\mu \nu \kappa \lambda}$ also 
	possesses the symmetry properties \eqref{E:Nminussignproperty1} - \eqref{E:Nsymmetryproperty}.
\end{lemma}

\begin{remark}
	Equations \eqref{E:secondNdef} - \eqref{E:secondNdef} are valid only in a wave coordinate system. Hence, we
	refer to the system \eqref{E:secondNdef} - \eqref{E:secondNdef} as the ``reduced'' electromagnetic equations.
\end{remark}	

\begin{proof}
We use the assumption \eqref{E:Ldualassumptions} and the Leibniz rule to expand \eqref{E:EquationSatisfiedbyMaxdualChainruleExpanded} and apply the results of Lemma \ref{L:gmhexpansions}, arriving at the following expansion:

\begin{align} \label{E:MBIEuler-Lagrangeexpanded}
	\mathscr{D}_{\mu} \Far^{\#\mu \nu} 
		\ + \ \widetilde{N}^{\mu \nu \kappa \lambda} \nabla_{\mu} \Far_{\kappa \lambda} =  
		O^{\dParameter}(|h||\nabla h||\Far|) \ + \ O^{\dParameter}(|\nabla h||\Far|^2;h),
\end{align}
where $\widetilde{N}^{\mu \nu \kappa \lambda} = O^{\dParameter}\big(|(h,\Far)|^2\big).$ Let us now decompose the $\mathscr{D}_{\mu} \Far^{\#\mu \nu}$ term. Using the anti-symmetry of $\Far^{\#\mu \nu},$ the symmetry of the Christoffel symbol $\Gamma_{\mu \ \lambda}^{\ \nu}$ under the exchanges $\mu \leftrightarrow \lambda,$ the identity 
$\Gamma_{\kappa \ \mu}^{\ \kappa} = \frac{1}{\sqrt{|\mbox{det $g$}|}} \nabla_{\mu} (\sqrt{|\mbox{det} \, g|}),$ and the wave coordinate condition $\nabla_{\mu}\big[\sqrt{|\mbox{det} \, g|} (g^{-1})^{\mu \kappa} \big] = 0,$ $(\kappa = 0,1,2,3),$ 
we have that

\begin{align} \label{E:divergenceFarardayexpansion1}
	\mathscr{D}_{\mu} \Far^{\#\mu \nu} & = \nabla_{\mu} \Far^{\#\mu \nu} \ + \ \Gamma_{\kappa \ \mu}^{\ \kappa} \Far^{\#\mu \nu}
		\ + \ \Gamma_{\mu \ \lambda}^{\ \nu}\Far^{\#\mu \lambda} \\
	& = \nabla_{\mu} \big[(g^{-1})^{\mu \kappa} (g^{-1})^{\nu \lambda} \Far_{\kappa \lambda}\big] 
			\ + \ \Big[\frac{1}{\sqrt{|\mbox{det} \, g|}} \nabla_{\mu} (\sqrt{|\mbox{det} \, g|})\Big] 
			 (g^{-1})^{\mu \kappa} (g^{-1})^{\nu \lambda} \Far_{\kappa \lambda} \notag \\
	& = \frac{1}{\sqrt{|\mbox{det} \, g|}} 
		\nabla_{\mu} \big[\sqrt{|\mbox{det} \, g|} (g^{-1})^{\mu \kappa} (g^{-1})^{\nu \lambda} \Far_{\kappa \lambda}\big]
	= (g^{-1})^{\mu \kappa} (g^{-1})^{\nu \lambda} \nabla_{\mu} \Far_{\kappa \lambda} 
		\ + \ \big[(g^{-1})^{\mu \kappa} \nabla_{\mu} (g^{-1})^{\nu \lambda} \big] \Far_{\kappa \lambda}. \notag
\end{align}
Using \eqref{E:Hintermsofh}, we conclude that the term $(g^{-1})^{\mu \kappa} (g^{-1})^{\nu \lambda} \nabla_{\mu} \Far_{\kappa \lambda}$ on the right-hand side of \eqref{E:divergenceFarardayexpansion1} can be expressed as

\begin{align}
	& \frac{1}{2} \Big\lbrace (m^{-1})^{\mu \kappa} (m^{-1})^{\nu \lambda} - (m^{-1})^{\mu \lambda} (m^{-1})^{\nu \kappa}
		- h^{\mu \kappa} (m^{-1})^{\nu \lambda} + h^{\mu \lambda} (m^{-1})^{\nu \kappa} - (m^{-1})^{\mu \kappa} h^{\nu \lambda} + 
		(m^{-1})^{\mu \lambda} h^{\nu \kappa}\Big\rbrace \nabla_{\mu} \Far_{\kappa \lambda} \\ 
	& \ \ + \ O^{\dParameter}(|h^2|) \nabla_{\mu} \Far_{\kappa \lambda}, \notag
\end{align}	
where the term in braces is equal to the term in braces on the right-hand side of \eqref{E:secondNdef}. 

Similarly, using \eqref{E:derivativeginversegmhexpansion}, we conclude that the term $\big[(g^{-1})^{\mu \kappa} \nabla_{\mu} (g^{-1})^{\nu \lambda} \big] \Far_{\kappa \lambda}$ on the right-hand side of \eqref{E:divergenceFarardayexpansion1} is equal to $-\mathscr{Q}_{(2;\Far)}^{\nu}(\nabla h, \Far)$ + $O^{\dParameter}(|h||\nabla h||\Far|),$ where $\mathscr{Q}_{(2;\Far)}^{\nu}(\nabla h, \Far)$ is defined in \eqref{E:Q2Farfirstdef}. Combining these expansions with \eqref{E:MBIEuler-Lagrangeexpanded}, we arrive at \eqref{E:dMis0Expansion} - \eqref{E:NtriangleErrorProperties}.
	
The fact that $N_{\triangle}^{\#\mu \nu \kappa \lambda}$ possesses the symmetry properties \eqref{E:Nminussignproperty1} - \eqref{E:Nsymmetryproperty} follows trivially from the fact that both  $N^{\#\mu \nu \kappa \lambda}$ and the term in braces on the right-hand side of \eqref{E:secondNdef} both satisfy these properties.

\end{proof}

\begin{remark} \label{R:ReducedElectromagneticInhomogeneous}
	
	With the help of the identity \eqref{E:ContractedChristoffelIdendity}, the above proof shows that the
	reduced equation \eqref{E:dMis0Expansion} is obtained by adding the inhomogeneous term 
	$- \Gamma^{\kappa} (g^{-1})^{\nu \lambda} \Far_{\kappa \lambda}$ to the right-hand side of 
	equation \eqref{E:EquationSatisfiedbyMaxdualChainruleExpandedfirstNversion}:

	\begin{align} \label{E:ModifiedElectromagneticEquationWithInhomogeneousTerm}
		N^{\#\mu \nu \kappa \lambda} \mathscr{D}_{\mu} \Far_{\kappa \lambda} 
		= - \Gamma^{\kappa} (g^{-1})^{\nu \lambda} \Far_{\kappa \lambda}.
	\end{align}
	We will make use of this fact in Section \ref{SS:WaveCoordinatesPreserved}.
	
\end{remark}

\subsection{The energy-momentum tensor} \label{SS:EMT}
In this section, we discuss the energy-momentum tensor $T_{\mu \nu}$ appearing on the right-hand side of \eqref{E:IntroEinsteinagain}. We recall that the energy-momentum tensor for an electromagnetic Lagrangian field theory is defined as follows:

\begin{align} \label{E:electromagnetictensorupper}
	T^{\#\mu \nu} & \eqdef  2 \frac{\partial \Ldual}{ \partial g_{\mu \nu}} + (g^{-1})^{\mu \nu} \Ldual.
\end{align}	
It follows trivially from the definition \eqref{E:electromagnetictensorupper} that $T_{\mu \nu}$ is symmetric:

\begin{align}
	T_{\mu \nu} & = T_{\nu \mu}, && (\mu,\nu = 0,1,2,3).
\end{align}
Furthermore, we recall that if $\Far_{\mu \nu}$ is a solution to the electromagnetic equations \eqref{E:IntrodFaris0again} - 
\eqref{E:IntrodMis0again}, then

\begin{align} \label{E:TisDivergenceFree}
	\mathscr{D}_{\mu} T^{\#\mu \nu} & = 0,&& (\nu = 0,1,2,3).
\end{align}

For the class of electromagnetic energy-momentum tensors considered in this article, we can 
use the chain rule and Lemma \ref{L:electromagneticidentities} to express $T_{\mu \nu}$ as follows:

\begin{subequations}
\begin{align} \label{E:electromagnetictensorloweroinTermsofLagrangian}
	T_{\mu \nu} & = - 2 \frac{\partial \Ldual}{ \partial \Farinvariant_{(1)}} (g^{-1})^{\kappa \lambda}
		\Far_{\mu \kappa} \Far_{\nu \lambda} 
		\ - \ \Farinvariant_{(2)} \frac{\partial \Ldual}{ \partial \Farinvariant_{(2)}} g_{\mu \nu} 
		\ + \ g_{\mu \nu} \Ldual \\
		& = - 2 \frac{\partial \Ldual}{ \partial \Farinvariant_{(1)}} T_{\mu \nu}^{(Maxwell)}  
		\ + \ \frac{1}{4} T g_{\mu \nu}, \label{E:AlternateelectromagnetictensorloweroinTermsofLagrangian}
\end{align}
\end{subequations}
where

\begin{align}
	T_{\mu \nu}^{(Maxwell)} \eqdef (g^{-1})^{\kappa \lambda} \Far_{\mu \kappa} \Far_{\nu \lambda}
	 \ - \ \frac{1}{2} \Farinvariant_{(1)} g_{\mu \nu}
\end{align}
is the energy-momentum tensor corresponding to the linear Maxwell-Maxwell equations, and

\begin{align}
	(g^{-1})^{\kappa \lambda} T_{\kappa \lambda} = 4 \Big(\Ldual - \Farinvariant_{(1)}\frac{\partial \Ldual}{\partial 
		\Farinvariant_{(1)}} - \Farinvariant_{(2)}\frac{\partial \Ldual}{\partial \Farinvariant_{(2)}} \Big)
\end{align}
is the trace of $T_{\mu \nu}$ with respect to $g_{\mu \nu}.$ Furthermore, from \eqref{E:electromagnetictensorloweroinTermsofLagrangian} and the expansions of Lemma \ref{L:gmhexpansions}, it follows that

\begin{align} \label{E:TLowerExpansion}
	T_{\mu \nu} 
	& = (m^{-1})^{\kappa \lambda}\Far_{\mu \kappa} \Far_{\nu \lambda} 
		\ - \ \frac{1}{4}m_{\mu \nu} (m^{-1})^{\kappa \eta} (m^{-1})^{\lambda \zeta} \Far_{\kappa \lambda} \Far_{\eta \zeta} \\
	& \ \ + \ O^{\dParameter+1}(|h||\Far|^2) \ + \ O^{\dParameter+1}(|\Far|^3;h). \notag
\end{align}

We now compute the right-hand side of (\ref{E:IntroEinsteinagain}'). First, taking the trace of \eqref{E:TLowerExpansion} with respect to $g,$ we compute that

\begin{align} \label{E:TtraceExpansion}
	(g^{-1})^{\kappa \lambda} T_{\kappa \lambda} = O^{\dParameter+1}(|h||\Far|^2) \ + \ O^{\dParameter+1}(|\Far|^3;h).
\end{align}
Combining \eqref{E:TLowerExpansion} and \eqref{E:TtraceExpansion}, and using the expansion \eqref{E:gmhexpansion},
we have that the right-hand side of (\ref{E:IntroEinsteinagain}') can be expressed as follows:

\begin{align} \label{E:MBIrighthandsidefieldequations}
	T_{\mu \nu} \ - \ \frac{1}{2} g_{\mu \nu} (g^{-1})^{\kappa \lambda} T_{\kappa \lambda} & = (m^{-1})^{\kappa \lambda} 
		\Far_{\mu \kappa} \Far_{\nu \lambda} 
		\ - \ \frac{1}{4}m_{\mu \nu} (m^{-1})^{\kappa \eta} (m^{-1})^{\lambda \zeta} \Far_{\kappa \lambda} \Far_{\eta \zeta} 
		\ + \ O^{\dParameter+1}(|h||\Far|^2) + O^{\dParameter+1}(|\Far|^3;h). 	
\end{align}

To conclude this section, we note for future use that if $\Far_{\mu \nu}$ is a solution to the inhomogeneous system

\begin{subequations}
	\begin{align}
		\nabla_{\lambda} \Far_{\mu \nu} + \nabla_{\mu} \Far_{\nu \lambda} + \nabla_{\nu} \Far_{\lambda \mu} & = 0, 
			&& (\lambda, \mu, \nu = 0,1,2,3), \label{E:dFis0DivergenceofT} \\
		N^{\#\mu \nu \kappa \lambda} \mathscr{D}_{\mu} \Far_{\kappa \lambda} & = \mathfrak{I}^{\nu}, && (\nu = 0,1,2,3), 
			\label{E:dMisnot0DivertenceofT}
	\end{align}
	\end{subequations}
then with the help of Lemma \ref{L:electromagneticidentities}, it can be shown that the following identity holds:

\begin{align} \label{E:DivergenceofTidentity}
	(g^{-1})^{\kappa \lambda} \mathscr{D}_{\kappa} T_{\lambda \nu} & = \mathfrak{I}^{\kappa} \Far_{\nu \kappa},&& (\nu = 0,1,2,3).
\end{align}
Equation \eqref{E:TisDivergenceFree} corresponds to the special case $\mathfrak{I}^{\nu} = 0,$ 
$(\nu = 0,1,2,3).$

\subsection{The modified Ricci tensor} \label{SS:Rmunuinwave}

Throughout the remainder of this article, we perform the standard wave coordinate system procedure (see e.g. \cite{rW1984}) of replacing the Ricci tensor $R_{\mu \nu}$ in the Einstein's field equation \eqref{E:IntroEinsteinagain} with a modified Ricci tensor $\widetilde{R}_{\mu \nu}.$ As we will soon see, this replacement transforms equations \eqref{E:IntroEinsteinagain} into a system of quasilinear wave equations.

\begin{definition}
We define the \emph{modified Ricci tensor} $\widetilde{R}_{\mu \nu}$ of the metric $g_{\mu \nu}$ as follows:

\begin{align} \label{E:ModifiedRicci}
	\widetilde{R}_{\mu \nu} \eqdef R_{\mu \nu} 
	 \ - \ \frac{1}{2} \big\lbrace  g_{\kappa \nu} \mathscr{D}_{\mu} \Gamma^{\kappa}
	 	+ g_{\kappa \mu} \mathscr{D}_{\nu} \Gamma^{\kappa} \big\rbrace 
		\ + \ u_{\mu \nu \kappa}(g,g^{-1},\partial g) \Gamma^{\kappa},
\end{align}
where the Ricci tensor $R_{\mu \nu}$ is defined in \eqref{E:Riccidef}, and the ``gauge term''
$u_{\mu \nu \kappa}(g,g^{-1},\partial g)\Gamma^{\kappa}$ is a smooth function of $g,$ $g^{-1},$ and $\partial g$ 
that will be discussed in Lemma \ref{L:RicciInWave}. We remark that for purposes of covariant differentiation by $\mathscr{D}$ in equation \eqref{E:ModifiedRicci}, the $\Gamma^{\mu}$ are treated as the components of a vectorfield.
\end{definition}

In the next lemma, we perform an algebraic decomposition of the modified Ricci tensor.

\begin{lemma} \cite[Lemmas 3.1 and 3.2]{hLiR2005} \label{L:RicciInWave} \textbf{(Decomposition of the modified Ricci tensor)}
	For a suitable choice of the gauge term $u_{\mu \nu \kappa}(g,g^{-1},\partial g)\Gamma^{\kappa},$
	the modified Ricci tensor $\widetilde{R}_{\mu \nu}$ of the metric $g_{\mu \nu} = m_{\mu \nu} + h_{\mu \nu}$ can be decomposed 
	as follows:
	
\begin{align}  \label{E:RicciInWave}
	\widetilde{R}_{\mu \nu} & = - \frac{1}{2} \Big\lbrace \widetilde{\Square}_g g_{\mu \nu} - 
		\mathscr{P}(\nabla_{\mu}h, \nabla_{\nu}h) 
		- \mathscr{Q}_{\mu \nu}^{(1;h)}(\nabla h, \nabla h) \Big\rbrace \ + \ O^{\infty}(|h||\nabla h|^2),
\end{align}
where

\begin{align}
	\widetilde{\Square}_g \eqdef (g^{-1})^{\kappa \lambda} \nabla_{\kappa} \nabla_{\lambda}
\end{align}
is the \textbf{reduced wave operator} corresponding to $g_{\mu \nu},$ and the quadratic terms $\mathscr{P}(\nabla_{\mu} \cdot, \nabla_{\nu} \cdot),$ 
$\mathscr{Q}_{\mu \nu}^{(1;h)}(\cdot,\cdot)$ are defined by their action on tensorfields $\Pi_{\mu \nu},$ $\Theta_{\mu \nu},$ and $h_{\mu \nu}$ 
as follows:

\begin{align}
	\mathscr{P}(\nabla_{\mu} \Pi, \nabla_{\nu} \Theta) 
	& \eqdef \frac{1}{4}(\nabla_{\mu} \Pi_{\kappa}^{\ \kappa})(\nabla_{\nu} \Theta_{\lambda}^{\ \lambda})
		\ - \ \frac{1}{2}(\nabla_{\mu} \Pi^{\kappa \lambda})(\nabla_{\nu} \Theta_{\kappa \lambda}), 
		\label{E:PNullform} \\
	\mathscr{Q}_{\mu \nu}^{(1;h)}(\nabla h, \nabla h) \label{E:hAddedUpNullForms}
		& \eqdef (m^{-1})^{\lambda \lambda'} \mathscr{Q}_0(\nabla h_{\lambda \mu}, \nabla h_{\lambda' \nu}) 
			\ - \ (m^{-1})^{\kappa \kappa'}(m^{-1})^{\lambda \lambda'} 
			\mathscr{Q}_{\kappa \lambda'}(\nabla h_{\lambda \mu}, \nabla 	h_{\kappa' \nu}) \\ 
		& \ \ +  \ (m^{-1})^{\kappa \kappa'}(m^{-1})^{\lambda \lambda'} 
			\mathscr{Q}_{\mu \kappa}(\nabla h_{\kappa' \lambda'}, \nabla h_{\lambda \nu}) 
			\ + \ (m^{-1})^{\kappa \kappa'}(m^{-1})^{\lambda \lambda'} 
			\mathscr{Q}_{\nu \kappa}(\nabla h_{\kappa' \lambda'}, \nabla h_{\lambda \mu}) \notag \\
		& \ \ + \ \frac{1}{2} (m^{-1})^{\kappa \kappa'}(m^{-1})^{\lambda \lambda'}
			\mathscr{Q}_{\lambda' \mu}(\nabla h_{\kappa \kappa'},\nabla h_{\lambda \nu}) 
			\ + \ \frac{1}{2} (m^{-1})^{\kappa \kappa'}(m^{-1})^{\lambda \lambda'}
			\mathscr{Q}_{\lambda' \nu}(\nabla h_{\kappa \kappa'},\nabla h_{\lambda \mu}). \notag 
\end{align}

The bilinear forms $\mathscr{Q}_0(\cdot,\cdot)$ and $\mathscr{Q}_{\mu \nu}(\cdot,\cdot),$ which appear on the right-hand side of \eqref{E:hAddedUpNullForms}, are known as the \textbf{standard null forms}. They are defined through their action on the derivatives of scalar-valued functions $\psi,$ $\chi$ by

\begin{subequations}
\begin{align}
	\mathscr{Q}_0(\nabla \psi,\nabla \chi) & \eqdef (m^{-1})^{\kappa \lambda} (\nabla_{\kappa} \psi)(\nabla_{\lambda} \chi)
		\label{E:StandardNullForm0}, \\
	\mathscr{Q}_{\mu \nu}(\nabla \psi, \nabla \chi) & \eqdef (\nabla_{\mu} \psi)(\nabla_{\nu} \chi) 
		- (\nabla_{\nu} \psi)(\nabla_{\mu} \chi). \label{E:StandardNullFormmunu}
\end{align}
\end{subequations}

\end{lemma}

\begin{proof}
	This decomposition is carried out in Lemmas 3.1 and 3.2 of \cite{hLiR2005}.
\end{proof}

We conclude this section by observing that (\ref{E:IntroEinsteinagain}'), \eqref{E:MBIrighthandsidefieldequations}, and \eqref{E:RicciInWave}
together imply that under the wave coordinate condition \eqref{E:wavecoordinategauge1}, and under the assumption \eqref{E:Ldualassumptions} on the Lagrangian, the Einstein field equation \eqref{E:IntroEinsteinagain} is equivalent to the following equation:

\begin{align} \label{E:WaveCoordinateVersionofEinsteinFieldEquation}
	\widetilde{\Square}_g g_{\mu \nu} & = \mathscr{P}(\nabla_{\mu}h, \nabla_{\nu}h) 
		\ + \ \mathscr{Q}_{\mu \nu}^{(1;h)}(\nabla h, \nabla h) 
		\ - \ 2 (m^{-1})^{\kappa \lambda} \Far_{\mu \kappa} \Far_{\nu \lambda} 
		\ + \ \frac{1}{2}m_{\mu \nu} (m^{-1})^{\kappa \eta} (m^{-1})^{\lambda \zeta} \Far_{\kappa \lambda} \Far_{\eta \zeta} \\
	& \ \ + \ O^{\infty}(|h||\nabla h|^2) \ + \ O^{\dParameter+1}(|h||\Far|^2) \ + \ O^{\dParameter+1}(|\Far|^3;h). \notag 
\end{align}

\subsection{Summary of the reduced system} \label{SS:ReducedEquations}
In this section, we summarize the above results by stating the form of the reduced Einstein nonlinear-electromagnetic system system that we work with for most of the remainder of the article, namely equations \eqref{E:Reducedh1Summary} - \eqref{E:ReduceddMis0Summary}; the derivation of this version of the reduced equations follows easily from the previous results of this section Section \ref{S:ENESinWaveCoordinates}. We remind the reader that the reduced equations are obtained by 
by adding the inhomogeneous term $-\Gamma^{\kappa} (g^{-1})^{\nu \lambda} \Far_{\kappa \lambda}$ to the right-hand side of 
equation \eqref{E:EquationSatisfiedbyMaxdualChainruleExpandedfirstNversion}
and by substituting the modified Ricci tensor in place of the Ricci tensor in equation \eqref{E:IntroEinsteinagain}, and that in a wave coordinate system, the reduced system is equivalent to the system \eqref{E:IntroEinsteinagain} - \eqref{E:IntrodMis0again}.

\begin{center}
	{\LARGE \textbf{The Reduced System}}
\end{center}

The reduced system (where $g_{\mu \nu} = m_{\mu \nu} + h_{\mu \nu}^{(0)} + h_{\mu \nu}^{(1)}$ and the unknowns are viewed to be $(h_{\mu \nu}^{(1)}, \Far_{\mu \nu})$) can be expressed as

\begin{subequations}
\begin{align}
	\widetilde{\Square}_{g} h_{\mu \nu}^{(1)} & = \mathfrak{H}_{\mu \nu} - \widetilde{\Square}_{g} h_{\mu \nu}^{(0)},&&
	(\mu, \nu = 0,1,2,3),  \label{E:Reducedh1Summary} \\
	\nabla_{\lambda} \Far_{\mu \nu} + \nabla_{\mu} \Far_{\nu \lambda} + \nabla_{\nu} \Far_{\lambda \mu} & = 0,&& (\lambda, \mu, \nu = 0,1,2,3), 
		\label{E:ReduceddFis0Summary} \\
	N^{\#\mu \nu \kappa \lambda} \nabla_{\mu} \Far_{\kappa \lambda} & = \mathfrak{F}^{\nu},&& (\nu = 0,1,2,3),
		\label{E:ReduceddMis0Summary} 
\end{align}
\end{subequations}
where $\widetilde{\Square}_g \eqdef (g^{-1})^{\kappa \lambda} \nabla_{\kappa} \nabla_{\lambda}$ is the reduced wave operator corresponding to 
$g_{\mu \nu}.$ 

The quantities $\mathfrak{H}_{\mu \nu}, N^{\#\mu \nu \kappa \lambda},$ and $\mathfrak{F}^{\nu}$ can be decomposed 
into principal terms and error terms (which are denoted with a ``$\triangle$'') as follows:

\begin{subequations}
\begin{align}
	\mathfrak{H}_{\mu \nu} & = \mathscr{P}(\nabla_{\mu} h, \nabla_{\nu} h) \ + \ \mathscr{Q}_{\mu \nu}^{(1;h)}(\nabla h, \nabla h)
		 \ + \ \mathscr{Q}_{\mu \nu}^{(2;h)}(\Far, \Far) \ + \ \mathfrak{H}_{\mu \nu}^{\triangle},  
		 \label{E:ReducedhInhomogeneous} \\
	\mathfrak{F}^{\nu} & =  \mathscr{Q}_{(2;\Far)}^{\nu}(\nabla h, \Far) \ + \ \mathfrak{F}_{\triangle}^{\nu}, 
		\label{E:EMBIFarInhomogeneous} \\
	N^{\#\mu \nu \kappa \lambda} & = \frac{1}{2} \Big\lbrace (m^{-1})^{\mu \kappa} (m^{-1})^{\nu \lambda} - (m^{-1})^{\mu 
		\lambda} (m^{-1})^{\nu \kappa}
			\ - \ h^{\mu \kappa} (m^{-1})^{\nu \lambda} + h^{\mu \lambda} (m^{-1})^{\nu \kappa}
			\ - \ (m^{-1})^{\mu \kappa} h^{\nu \lambda} + (m^{-1})^{\mu \lambda} h^{\nu \kappa} \Big\rbrace
			\ + \ N_{\triangle}^{\#\mu \nu \kappa \lambda}, \label{E:NSummarydef}
\end{align}
where $\mathscr{P}(\nabla_{\mu} h, \nabla_{\nu} h)$ is defined in \eqref{E:PNullform}, 
$\mathscr{Q}_{\mu \nu}^{(1;h)}(\nabla h, \nabla h)$ is defined in \eqref{E:hAddedUpNullForms}, and

\begin{align}
	\mathscr{Q}_{\mu \nu}^{(2;h)}(\Far, \Gar) 
		& = -2(m^{-1})^{\kappa \lambda} \Far_{\mu \kappa} \Gar_{\nu \lambda} \ + \ \frac{1}{2}m_{\mu \nu}(m^{-1})^{\kappa 
		\lambda}(m^{-1})^{\lambda \kappa}\Far_{\kappa \lambda} \Gar_{\kappa \lambda}, \label{E:Q2h} \\
	\mathscr{Q}_{(2;\Far)}^{\nu}(\nabla h, \Far) & = (m^{-1})^{\mu \kappa} (m^{-1})^{\lambda \lambda'} (m^{-1})^{\nu \nu'} 
		(\nabla_{\mu} h_{\nu' \lambda'}) \Far_{\kappa \lambda}, \label{E:Q2Far} \\
	\mathfrak{H}_{\mu \nu}^{\triangle} & = O^{\infty}(|h||\nabla h|^2) \ + \ O^{\dParameter+1}(|h||\Far|^2) \ + \ O^{\dParameter+1}(|\Far|^3;h), 
		\label{E:HtriangleSmallAlgebraic} \\
	\mathfrak{F}_{\triangle}^{\nu} & =  O^{\dParameter}(|h||\nabla h||\Far|) \ + \ O^{\dParameter}(|\nabla h||\Far|^2;h),
		\label{E:FtriangleSmallAlgebraic} \\
	N_{\triangle}^{\#\mu \nu \kappa \lambda} & = O^{\dParameter}\big(|(h,\Far)|^2\big). 
		\label{E:NtriangleSmallAlgebraic} 
\end{align}
\end{subequations}

Furthermore, the left-hand side of \eqref{E:ReduceddMis0Summary} can be expressed as

\begin{subequations}
\begin{align}  \label{E:NNullFormDecomposition}
	N^{\#\mu \nu \kappa \lambda} \nabla_{\mu} \Far_{\kappa \lambda} 
		& = \frac{1}{2} \big[ (m^{-1})^{\mu \kappa} (m^{-1})^{\nu \lambda} - (m^{-1})^{\mu \lambda} (m^{-1})^{\nu \kappa} 
			\big] \nabla_{\mu} \Far_{\kappa \lambda} 
		\ - \ \mathscr{P}_{(\Far)}^{\nu}(h, \nabla \Far) 
		\ - \ \mathscr{Q}_{(1;\Far)}^{\nu}(h, \nabla \Far)
		\ + \ N_{\triangle}^{\#\mu \nu \kappa \lambda} \nabla_{\mu} \Far_{\kappa \lambda}, 
\end{align}
where	
		
\begin{align}
	\mathscr{P}_{(\Far)}^{\nu}(h, \nabla \Far) 
	& = (m^{-1})^{\mu \mu'}(m^{-1})^{\kappa \kappa'} (m^{-1})^{\nu \lambda} h_{\mu' \kappa'}\nabla_{\mu} \Far_{\kappa \lambda},
		\label{E:PFar} \\
	\mathscr{Q}_{(1;\Far)}^{\nu}(h, \nabla \Far) 
	& = (m^{-1})^{\mu \kappa} (m^{-1})^{\nu \nu'} (m^{-1})^{\lambda \lambda'} h_{\nu' \lambda'}\nabla_{\mu} \Far_{\kappa \lambda}. \label{E:Q1Far}
\end{align}
\end{subequations}

\section{The Initial Value Problem} \label{S:IVP}
In this section, we discuss the abstract initial data and the constraint equations for the Einstein-nonlinear electromagnetic system. We then use the abstract initial data to construct initial data for the reduced equations that satisfy the wave coordinate condition at $t=0.$ Finally, we sketch a proof of the well-known fact that the wave coordinate condition is preserved by the solutions to the reduced equations that are launched by this data; this result shows that the wave coordinate gauge is a viable gauge for studying the Einstein-nonlinear electromagnetic system.

\noindent \hrulefill
\ \\

\subsection{The abstract initial data} \label{SS:AbstractData}

The initial value problem formulation of the Einstein equations goes back to the seminal work \cite{CB1952} by Choquet-Bruhat. In this article, initial data for the Einstein-nonlinear electromagnetic system consist of the $3-$dimensional manifold $\Sigma_0 = \mathbb{R}^3$ together with the following fields on $\Sigma_0:$ 
a Riemannian metric $\mathring{\underline{g}}_{jk},$ a covariant two-tensor $\mathring{K}_{jk},$ and a pair of one-forms $\mathring{\mathfrak{\Displacement}}_{j}, \mathring{\mathfrak{\Magneticinduction}}_{j}.$ After we construct the ambient Lorentzian spacetime $(\mathfrak{M},g_{\mu \nu})$, $\mathring{\underline{g}}_{jk}$ and $\mathring{K}_{jk}$ will respectively be the first and second fundamental forms of $\Sigma_0,$ while $\mathring{\mathfrak{\Displacement}}_{j}, \mathring{\mathfrak{\Magneticinduction}}_{j},$ which are defined below in Section \ref{SS:EBDH}, will be an electromagnetic decomposition of $\Far_{\mu \nu}|_{\Sigma_0}$ into a pair of one-forms that are both $m-$tangent and $g-$tangent to $\Sigma_0.$ 

It is well-known that one cannot consider arbitrary data for the Einstein-nonlinear electromagnetic system. The data are subject to the following constraints: 

\begin{subequations}
\begin{align} 
	\mathring{\underline{R}} - \mathring{K}_{ab} \mathring{K}^{ab} + 
		\big[(\mathring{\underline{g}}^{-1})^{ab} \mathring{K}_{ab} \big]^2 & = 
		2T(\hat{N},\hat{N})|_{\Sigma_0},&& \label{E:Gauss} \\
	(\mathring{\underline{g}}^{-1})^{ab} \mathring{\underline{\mathscr{D}}}_a \mathring{K}_{bj} - 
		(\mathring{\underline{g}}^{-1})^{ab} \mathring{\underline{\mathscr{D}}}_j \mathring{K}_{ab} & = 
		T(\hat{N},\frac{\partial}{\partial x^j})|_{\Sigma_0},&&  (j=1,2,3), \label{E:Codazzi} 
\end{align}
\end{subequations}

\begin{subequations}
\begin{align}
	(\mathring{\underline{g}}^{-1})^{ab} \underline{\mathring{\mathscr{D}}}_a \mathring{\mathfrak{\Displacement}}_b & = 0, 
		\label{E:DivergenceD0} \\
	(\mathring{\underline{g}}^{-1})^{ab} \underline{\mathring{\mathscr{D}}}_a \mathring{\mathfrak{\Magneticinduction}}_b & = 0, 
		\label{E:DivergenceB0}
\end{align}
\end{subequations}
where $\mathring{\underline{\mathscr{D}}}$ is the Levi-Civita connection corresponding to 
$\mathring{\underline{g}}_{jk},$ $\mathring{\underline{R}}$ is the scalar curvature of $\mathring{\underline{g}}_{jk},$  $T_{\mu \nu}$ is defined in \eqref{E:electromagnetictensorloweroinTermsofLagrangian}, and $\hat{N}^{\mu}$ is the future-directed unit $g-$normal to $\Sigma_0.$ The right-hand sides of \eqref{E:Gauss} - \eqref{E:Codazzi} can (in principle) be computed in terms of and $\mathring{\underline{g}}_{jk},$ $\mathring{\mathfrak{\Displacement}}_j,$ and $\mathring{\mathfrak{\Magneticinduction}}_j$ with the help of the relations \eqref{E:AbstractEBDHinertialcomponents}, which connect these quantities to $\Far_{\mu \nu}|_{\Sigma_0}.$
In equations \eqref{E:Gauss} - \eqref{E:Codazzi}, indices are lowered and raised with the Riemannian metric $\mathring{\underline{g}}_{jk}$ and its inverse $(\mathring{\underline{g}}^{-1})^{jk}.$ The constraints \eqref{E:Gauss} - \eqref{E:Codazzi} are manifestations of the \emph{Gauss} and \emph{Codazzi} equations respectively. These equations relate the geometry of the ambient spacetime $(\mathfrak{M},g_{\mu \nu})$ (which has to be constructed) to the geometry inherited by an embedded Riemannian hypersurface (which will be $(\Sigma_0,\mathring{\underline{g}}_{jk})$ after construction). Without providing the rather standard details (see e.g. \cite{dC2008}), we remark that they are consequences of the following assumptions:

\begin{itemize}
	\item $\Sigma_0$ is a submanifold of the spacetime manifold $\mathfrak{M}$
	\item $\mathring{\underline{g}}_{jk}$ is the first fundamental form of $\Sigma_0,$ and
		$\mathring{K}_{jk}$ is the second fundamental form of $\Sigma_0$
	\item The Einstein-nonlinear electromagnetic system is satisfied along $\Sigma_0$
	\item Along $\Sigma_0$ (viewed as a subset of $\mathfrak{M}$),
	$\mathfrak{\Magneticinduction}_{\mu} = - \Fardual_{\mu \kappa}\hat{N}^{\kappa}$ 
	and $\mathfrak{\Displacement}_{\mu} = - \Maxdual_{\mu \kappa} \hat{N}^{\kappa},$
	where $\hat{N}^{\mu}$ is the future-directed unit $g-$normal to $\Sigma_0.$
	
\end{itemize}


We recall that under the above assumptions, $\mathring{\underline{g}}$ and $\mathring{K}$ are defined by

\begin{align}
	\mathring{\underline{g}}|_p(X,Y) & = g|_p(X,Y), && \forall X,Y \in T_p \Sigma_0, \\
	\mathring{K}|_p(X,Y) & = g|_p(\mathscr{D}_X \hat{N},Y), && \forall X,Y \in T_p \Sigma_0, 
\end{align}
where $\hat{N}$ is the future-directed unit $g-$normal\footnote{Under the above assumptions, it follows that at every point $p \in \Sigma_0,$ $\hat{N}^{\mu} = (A^{-1},0,0,0).$} to $\Sigma_0$ at $p,$ and $\mathscr{D}$ is the Levi-Civita connection corresponding to $g.$ Furthermore, if $X,Y$ are vectorfields tangent to $\Sigma_0,$ then

\begin{align}
	\mathscr{D}_X Y = \mathring{\underline{\mathscr{D}}}_X Y + \mathring{K}(X,Y) \hat{N}. 
\end{align}

We also remind the reader that our stability theorem requires that the abstract initial data decay according to the rates
\eqref{E:metricdataexpansion} - \eqref{E:BdecayAssumption}.

\subsection{The initial data for the reduced equations} \label{SS:ReducedData}

We assume that we are given ``abstract'' initial data $(\mathring{\underline{g}}_{jk}, \mathring{K}_{jk},
\mathring{\mathfrak{\Displacement}}_j, \mathring{\mathfrak{\Magneticinduction}}_j),$ $(j,k=1,2,3),$ on the manifold $\mathbb{R}^3$
for the Einstein equations as discussed in the previous section. In this section, we will use this data to construct data $(g_{\mu \nu}|_{t=0},$ $\partial_t g_{\mu \nu}|_{t=0},$ $\Far_{\mu \nu}|_{t=0}),$ $(\mu, \nu = 0,1,2,3)$ for the reduced equations \eqref{E:Reducedh1Summary} - \eqref{E:ReduceddMis0Summary} that satisfy the wave coordinate condition $\Gamma^{\mu}|_{t=0} = 0.$ We begin by recalling that $\chi(z)$ is a fixed cut-off function with the following properties:

\begin{align} \label{E:chidef}
	\chi \in C^{\infty}, \qquad \chi \equiv 1 \ \mbox{for} \ z \geq 3/4, \qquad \chi \equiv 0 \ \mbox{for} \ z \leq 1/2.
\end{align}
We then define the function $A(x^1,x^2,x^3) \geq 0$ by

\begin{align} \label{E:aSquareddef}
	A^2 & \eqdef 1 - \frac{2M}{r} \chi(r), && r \eqdef |x|.
\end{align}

We define the data for the spacetime metric $g_{\mu \nu}$ by

\begin{subequations}
\begin{align}
	g_{00}|_{t=0} & = -A^2, &&  g_{0j}|_{t=0} = 0, && g_{jk}|_{t=0} = \mathring{\underline{g}}_{jk}, \label{E:ReducedMetricData} \\
	\partial_t g_{00}|_{t=0} & = 2A^3 (\mathring{\underline{g}}^{-1})^{ab} \mathring{K}_{ab},
		&& \partial_t g_{0j}|_{t=0} = A^2 (\mathring{\underline{g}}^{-1})^{ab} \partial_a \mathring{\underline{g}}_{bj}
		- \frac{1}{2} A^2 (\mathring{\underline{g}}^{-1})^{ab} \partial_j \mathring{\underline{g}}_{ab} - A \partial_j A, 
		&& \partial_t g_{jk}|_{t=0} = 2A \mathring{K}_{jk}, \label{E:ReducedMetricTimeDerivativeData}
\end{align}
\end{subequations}
and the data for the Faraday tensor $\Far_{\mu \nu}$ by

\begin{subequations}
\begin{align}
	\Far_{j0}|_{t=0} & = \mathring{\Electricfield}_j, \\
 		\Far_{jk}|_{t=0} & = [ijk] \mathring{\Magneticinduction}_i. 
\end{align}
\end{subequations}
The one-forms $\mathring{\Electricfield}_j$ and $\mathring{\Magneticinduction}_j$
can be expressed in terms of $\mathring{\underline{h}}_{jk}$ and the one-forms
$\mathring{\mathfrak{\Displacement}}_j$ and $\mathfrak{\mathring{\Magneticinduction}}_j$ appearing in the constraint equations \eqref{E:DivergenceD0} -  \eqref{E:DivergenceB0} by using the relations \eqref{E:AbstractEBDHinertialcomponents} and \eqref{E:EBDHinertialcomponents} below. The precise form of this relations depends on the choice of Lagrangian $\Ldual,$ but in the small-data regime, the estimates \eqref{E:ElectricfieldDataInTermsofIntrinsic} \eqref{E:IntialInductionintermsofInitialQuantities}, and \eqref{E:IntialDisplacementintermsofInitialQuantities} hold.

We now state the main result of this section, which is a lemma showing that the wave coordinate condition is satisfied at $t=0.$

\begin{lemma} \label{L:Gammamuare0initially} \textbf{(Wave coordinate condition holds at $t=0$)}
	Suppose that the initial data $(g_{\mu \nu}|_{t=0}, \partial_t g_{\mu \nu}|_{t=0}),$ $(\mu, \nu = 
	0,1,2,3),$ for the reduced equations are constructed from abstract initial data $(\mathring{\underline{g}}_{jk}, 
	\mathring{K}_{jk}),$ $(j,k=1,2,3)$ as described above. Then the 
	wave coordinate condition holds initially:

\begin{align} \label{E:Gammamuare0initially}
	& \Gamma^{\mu}|_{t=0},&& (\mu= 0,1,2,3).
\end{align}
\end{lemma}
\begin{proof}
	Lemma \ref{L:Gammamuare0initially} follows from the expression \eqref{E:wavecoordinategauge3}, 
	the definitions \eqref{E:ReducedMetricData} - \eqref{E:ReducedMetricTimeDerivativeData}, and simple calculations. 
\end{proof}

Note that the above definitions induce the following data for the spacetime metric ``remainder'' piece
$h_{\mu \nu}^{(1)},$ which is defined by \eqref{E:gmhexpansion} - \eqref{E:h0defIntro}:

\begin{subequations}
\begin{align}
	h_{00}^{(1)}|_{t=0} & = 0, && h_{0j}^{(1)}|_{t=0} = 0,
		&& h_{jk}^{(1)}|_{t=0} = \mathring{\underline{h}}_{jk}^{(1)}, \\
	\partial_t h_{00}^{(1)}|_{t=0} & = 2A^3 (\mathring{\underline{g}}^{-1})^{ab} \mathring{K}_{ab},
		&& \partial_t h_{0j}^{(1)}|_{t=0} = A^2 (\mathring{\underline{g}}^{-1})^{ab} \partial_a \mathring{\underline{g}}_{bj}
		- \frac{1}{2} A^2 (\mathring{\underline{g}}^{-1})^{ab} \partial_j \mathring{\underline{g}}_{ab} - A \partial_j A,
		&& \partial_t h_{jk}^{(1)}|_{t=0} = 2A \mathring{K}_{jk}.
\end{align}
\end{subequations}
Similarly, the following data are induced in $h_{\mu \nu} = h_{\mu \nu}^{(0)} + h_{\mu \nu}^{(1)},$ which is defined in \eqref{E:hdefIntro}:

\begin{subequations}
\begin{align}
	h_{00}|_{t=0} & = \chi(r) \frac{2M}{r}, && h_{0j}|_{t=0} = 0,
		&& h_{jk}|_{t=0} = \mathring{\underline{h}}_{jk}^{(1)} + \chi(r) \frac{2M}{r}, \label{E:InducedhData} \\
	\partial_t h_{00}|_{t=0} & = 2A^3 (\mathring{\underline{g}}^{-1})^{ab} \mathring{K}_{ab},
		&& \partial_t h_{0j}|_{t=0} = A^2 (\mathring{\underline{g}}^{-1})^{ab} \partial_a \mathring{\underline{g}}_{bj}
		- \frac{1}{2} A^2 (\mathring{\underline{g}}^{-1})^{ab} \partial_j \mathring{\underline{g}}_{ab} - A \partial_j A,
		&& \partial_t h_{jk}|_{t=0} = 2A \mathring{K}_{jk}.
\end{align}
\end{subequations}
We will make use of these facts in our proof of Proposition \ref{P:SmallNormImpliesSmallEnergy} below.

\subsection{Preservation of the wave coordinate gauge} \label{SS:WaveCoordinatesPreserved}
In this section, we sketch a proof of the fact that if the reduced data are constructed from abstract data 
as described in Section \ref{SS:ReducedData}, then the wave coordinate condition $\Gamma^{\mu} = 0$ is preserved by the flow of the reduced equations. 
This result requires the assumption that the abstract data satisfy the constraints \eqref{E:Gauss} - \eqref{E:DivergenceB0}. To simplify the discussion, we assume in this section that the data are smooth. However, the result also holds in the regularity class we use during our global existence proof. We remark that this result is quite standard, and that we have included it only for convenience.

\begin{proposition} \label{P:PreservationofWaveCoordianteGauge}
	\textbf{(Preservation of the wave coordinate gauge)}
	Suppose that $(g_{\mu \nu}|_{t=0}, \partial_t g_{\mu \nu}|_{t=0}, \Far_{\mu \nu}|_{t=0}),$ 
	$(\mu, \nu = 0,1,2,3),$ are smooth initial data for the 
	reduced equations \eqref{E:Reducedh1Summary} - \eqref{E:ReduceddMis0Summary} constructed from abstract initial data 
	satisfying the constraints \eqref{E:Gauss} - \eqref{E:DivergenceB0} as described in Section \ref{SS:ReducedData}. In 
	particular, by Lemma \ref{L:Gammamuare0initially}, the wave coordinate condition 
	$\Gamma^{\mu}|_{t=0}$ holds. Assume further that the reduced data are small enough so that they lie within the regime of 
	hyperbolicity\footnote{Since our electromagnetic equations are perturbations of the linear Maxwell-Maxwell equations, there 
	will always be such a regime.} of 
	the reduced equations. Let $(g_{\mu \nu},\Far_{\mu \nu})$ be the corresponding solution to the reduced equations
	that is launched by the data. Then $\Gamma^{\mu} \equiv 0$ holds in the entire maximal globally hyperbolic development of the
	data\footnote{Roughly speaking, this is the largest possible solution that is uniquely determined by the data.}.
\end{proposition}

\noindent{\textit{Sketch of proof}:} Our goal is to show that whenever we have a smooth solution to the 
reduced equations \eqref{E:Reducedh1Summary} - \eqref{E:ReduceddMis0Summary}, the corresponding $\Gamma^{\mu}$ satisfy a system of wave equations with principal part equal to $(g^{-1})^{\kappa \lambda} \partial_{\kappa} \partial_{\lambda}$ and with trivial initial data $\Gamma^{\mu}|_{t=0} = \partial_t \Gamma^{\mu}|_{t=0} = 0.$ The conclusion that $\Gamma^{\mu} \equiv 0$ in the maximal globally hyperbolic development of the data then follows from a standard uniqueness theorem based on energy estimates (see e.g. \cite{lH1997}, \cite{cS2008} for the ideas on how to prove such a theorem). To derive the equations satisfied by the $\Gamma^{\mu},$ we will view $\Gamma^{\mu}$ as a vectorfield for purposes of covariant differentiation. We apply $(g^{-1})^{\nu \lambda} \mathscr{D}_{\lambda}$ to each side of equation \eqref{E:Modified}, use the Bianchi identity $(g^{-1})^{\nu \lambda} \mathscr{D}_{\lambda} \big(R_{\mu \nu} - \frac{1}{2}R g_{\mu \nu}\big) = 0,$ the fact that 
$(g^{-1})^{\nu \lambda} \mathscr{D}_{\lambda} T_{\mu \nu} = - \Gamma^{\kappa} (g^{-1})^{\beta \lambda} \Far_{\kappa \lambda} \Far_{\mu \beta}$ (see Remark \ref{R:ReducedElectromagneticInhomogeneous} and \eqref{E:DivergenceofTidentity}), the curvature relation $\mathscr{D}_{\mu} \mathscr{D}_{\kappa} \Gamma^{\kappa} = \mathscr{D}_{\kappa} \mathscr{D}_{\mu} \Gamma^{\kappa}
- R_{\mu \kappa} \Gamma^{\kappa},$ and expand the covariant derivatives in terms of coordinate derivatives and Christoffel symbols to deduce that the $\Gamma^{\mu}$ are solutions to the following \emph{hyperbolic} system of wave equations:

\begin{align} \label{E:Gammawaveequationcoordinates}
	(g^{-1})^{\kappa \lambda} \partial_{\kappa} \partial_{\lambda} \Gamma^{\mu} & =
		A_{\ \ \lambda}^{\mu \kappa}(g,g^{-1},\partial g) \partial_{\kappa} \Gamma^{\lambda}
		+ B_{\ \kappa}^{\mu} (g,g^{-1},\partial g,\Far)\Gamma^{\kappa}, && (\mu=0,1,2,3),
\end{align}
where the $A_{\ \ \lambda}^{\mu \kappa}\big(g(t,x),g^{-1}(t,x),\partial g(t,x)\big)$ and 
$B_{\ \kappa}^{\mu}\big(g(t,x),g^{-1}(t,x), \partial g(t,x),\Far(t,x)\big)$ are smooth functions of $(t,x).$ 

To complete our sketch of the proof, it remains to show that $\partial_t \Gamma^{\mu}|_{t=0} = 0.$ We first recall 
(see Remark \ref{R:ReducedElectromagneticInhomogeneous}) that equation \eqref{E:RicciInWave} is obtained by adding the gauge term $- \frac{1}{2} \big\lbrace  g_{\kappa \nu} \mathscr{D}_{\mu} \Gamma^{\kappa} + g_{\kappa \mu} \mathscr{D}_{\nu} \Gamma^{\kappa} \big\rbrace \ + \ u_{\mu \nu \kappa}(g,g^{-1},\partial g) \Gamma^{\kappa}$ to the expression \eqref{E:Riccidef} for $R_{\mu \nu}.$ Consequently, it follows that for a solution to the reduced equations \eqref{E:Reducedh1Summary} - \eqref{E:ReduceddMis0Summary}, we have that

\begin{align} \label{E:Modified}
	R_{\mu \nu} \ - \ \frac{1}{2}R g_{\mu \nu}  
		\ - \ T_{\mu \nu} 
		& = \frac{1}{2} \big\lbrace  g_{\kappa \nu} \mathscr{D}_{\mu} \Gamma^{\kappa}
	 		+ g_{\kappa \mu} \mathscr{D}_{\nu} \Gamma^{\kappa} \big\rbrace
		\ - \ u_{\mu \nu \kappa}(g,g^{-1},\partial g)\Gamma^{\kappa}&& \\
		& \ \ - \ g_{\mu \nu} \mathscr{D}_{\lambda} \Gamma^{\lambda}
	 		\ + \ \frac{1}{2} g_{\mu \nu} (g^{-1})^{\kappa \lambda} u_{\kappa \lambda \delta}(g,g^{-1},\partial g)\Gamma^{\delta}, 
			&&(\mu,\nu=0,1,2,3). \notag
\end{align}	
The left-hand side of \eqref{E:Modified} is simply the difference of the left and right sides of the Einstein equations \eqref{E:IntroEinstein}. Since the abstract initial data $(\mathring{\underline{g}}_{jk}, \mathring{K}_{jk},
\mathring{\mathfrak{\Displacement}}_j, \mathring{\mathfrak{\Magneticinduction}}_j),$ $(j,k=1,2,3),$ are assumed to satisfy the constraint equations \eqref{E:Gauss} - \eqref{E:Codazzi}, it follows that the left-hand side of \eqref{E:Modified} is equal to $0$ at $t=0$ after contracting\footnote{In fact, one derives the constraint equations by assuming that these contractions are $0$ at $t=0.$} against $\hat{N}^{\mu} \hat{N}^{\nu}$ or $\hat{N}^{\mu} X^{\nu},$ where $\hat{N}^{\mu}$ is the future-directed unit $g-$normal to $\Sigma_0$ and $X^{\mu}$ is any vector tangent to $\Sigma_0.$

Recalling that $\hat{N}^{\mu}|_{t=0} = A^{-1}\delta_0^{\mu},$ and choosing $X^{\nu} = \delta_j^{\nu},$
it therefore follows that the right-hand side must also be equal to $0$ at $t=0$ upon contraction:

\begin{subequations}
\begin{align} \label{E:InitialGammaupderivativeconditionNormalNormal}
	\bigg\lbrace g_{\kappa 0} \mathscr{D}_t \Gamma^{\kappa}
	 		\ - \ u_{0 0 \kappa}(g,g^{-1},\partial g)\Gamma^{\kappa}
		\ - \ g_{0 0} \mathscr{D}_{\lambda} \Gamma^{\lambda}
	 		\ + \ \frac{1}{2} g_{0 0} (g^{-1})^{\kappa \lambda} u_{\kappa \lambda \delta}(g,g^{-1},\partial g) \Gamma^{\delta} 
	 		\bigg\rbrace\Big|_{t = 0} = 0, &&  \\
		\bigg\lbrace \frac{1}{2} \big[  g_{\kappa j} \mathscr{D}_t \Gamma^{\kappa}
	 		+ g_{\kappa 0} \mathscr{D}_j \Gamma^{\kappa} \big]
		\ - \ u_{0 j \kappa}(g,g^{-1},\partial g)\Gamma^{\kappa}
		\ - \ g_{0j} \mathscr{D}_{\lambda} \Gamma^{\lambda}
	 		\ + \ \frac{1}{2} g_{0 j} (g^{-1})^{\kappa \lambda} u_{\kappa \lambda \delta}(g,g^{-1},\partial g)\Gamma^{\delta} 
	 		\bigg\rbrace\Big|_{t = 0} = 0, && (j=1,2,3). 		
			\label{E:InitialGammaupderivativeconditionNormalTangential}
\end{align}
\end{subequations}
Expanding the covariant differentiation in \eqref{E:InitialGammaupderivativeconditionNormalNormal} - \eqref{E:InitialGammaupderivativeconditionNormalTangential} in terms of coordinate derivatives and Christoffel symbols, and
using \eqref{E:ReducedMetricData} - \eqref{E:ReducedMetricTimeDerivativeData} plus the fact that the initial data were constructed so as to satisfy $\Gamma^{\mu}|_{t=0} = 0,$ it is easy to check that $\partial_t \Gamma^{\mu}$ must \emph{also necessarily} be trivial at $t=0:$

\begin{align}
	\partial_t \Gamma^{\mu}|_{t=0} & = 0, && (\mu = 0,1,2,3).
\end{align}
This completes our sketch of the proposition.

\hfill $\qed$

\section{Geometry and the Minkowskian Null Frame} \label{S:NullFrame}

In this section, we introduce the families of ingoing Minkowskian null cones $C_{s}^-,$  
outgoing Minkowskian light cones $C_{q}^+,$ constant Minkowskian time slices $\Sigma_t,$ and Euclidean spheres $S_{r,t}.$
We then discuss the well-known notion of a Minkowskian null frame, which allows us to geometrically decompose the tangent space 
as a direct sum $T|_p \mathbb{R}^{1+3} =  \mbox{span} \lbrace \uL|_p \rbrace \, \oplus \, \mbox{span} \lbrace L|_p \rbrace \, \oplus \, T|_p S_{r,t}.$ These decompositions allow us to geometrically decompose tensorfields. In Section \ref{SS:NullComponents}, we provide a full description of the null decomposition of a two-form $\Far$ into its \emph{Minkowskian null components}. This decomposition will be essential to our subsequent analysis of the decay properties of the Faraday tensor. In Section \ref{SS:NullDecompElectromagnetic}, we will derive equations for these null components under the assumption that $\Far$ is a solution to the reduced electromagnetic equations \eqref{E:ReduceddFis0Summary} - 
\eqref{E:ReduceddMis0Summary}. In Section \ref{S:DecayFortheReducedEquations}, we will use the equations for the null components to deduce ``upgraded'' pointwise decay estimates for the lower-order Lie derivatives of $\Far;$ these estimates are essential for closing our global existence bootstrap argument in Section \ref{S:GlobalExistence}.

\noindent \hrulefill
\ \\

\subsection{The Minkowskian null frame} \label{SS:NullFrame}

Before proceeding, we introduce the subsets $C_{q}^+,$ $C_{s}^-,$ $\Sigma_t,$ $S_{r,t}.$

\begin{definition}
	In our wave coordinate system $(t,x),$ we define the \emph{outgoing Minkowski 
	null cones} $C_{q}^+,$ \emph{ingoing Minkowski null cones} $C_{s}^-,$ 
	the \emph{constant Minkowskian time slices} $\Sigma_t$ and the Euclidean spheres 
	$S_{r,t}$ as follows:
	
	\begin{subequations}
	\begin{align}
		C_{q}^+ & \eqdef \lbrace (\tau,y) \ | \ |y| - \tau = q \rbrace, \\
		C_{s}^- & \eqdef \lbrace (\tau,y) \ | \ |y| + \tau = s \rbrace, \\
		\Sigma_t & \eqdef \lbrace (\tau,y) \ | \ \tau = t \rbrace, \\ 
		S_{r,t} & \eqdef \lbrace (\tau,y) \ | \ \tau = t, |y| = r \rbrace, 
	\end{align}
	\end{subequations}
	In the above formulas, $y \eqdef (y^1,y^2,y^3),$ and 
	$|y| \eqdef \big[(y^1)^2 + (y^2)^2 + (y^2)^2 \big]^{1/2}.$
\end{definition}

We also introduce the following vectorfields, which play a fundamental role throughout this article.

\begin{definition} \label{D:uLLdef}
	We define the \emph{ingoing Minkowski-null geodesic vectorfield} $\uL$ and the 
	\emph{outgoing Minkowski-null geodesic vectorfield} $L$ by
	
	\begin{subequations}
	\begin{align}
		\uL^{\mu} & = (1,-\omega^1,-\omega^2,-\omega^3), \label{E:uLdef} \\
		L^{\mu} & = (1,\omega^1,\omega^2,\omega^3), \label{E:Ldef}
	\end{align}
	\end{subequations}
	where $\omega^j \eqdef x^j/r.$ By ``Minkowski-null,'' we mean that $m(\uL,\uL) = m(L,L) = 0.$
	Note that $\uL$ is tangent to the ingoing cones $C_{s}^-,$ that $L$ is tangent to the outgoing cones $C_{q}^+,$
	and that $\uL,$ $L$ are both $m-$orthogonal to the $S_{r,t}.$ By ``geodesic,'' we mean that
	$\nabla_{\uL} \uL = \nabla_L L = 0.$
\end{definition}

Note that

\begin{subequations}
\begin{align}
	\uL & = \partial_t - \partial_r, \\
	L & = \partial_t + \partial_r.
\end{align}
\end{subequations}

We now recall the definitions of the Minkowskian first fundamental forms of the surfaces $\Sigma_t$ and $S_{r,t}.$

\begin{definition} \label{D:FirstFundamental}
	The \emph{Minkowskian first fundamental} forms of the surfaces $\Sigma_t$ and $S_{r,t}$ are respectively defined to be the 
	following intrinsic metrics:
	
	\begin{subequations}
	\begin{align}
		\um_{\mu \nu} & \eqdef \mbox{diag}(0,1,1,1), \label{E:FirstFundSigmatDef} \\
		\angm_{\mu \nu} & \eqdef m_{\mu \nu} + \frac{1}{2}(\uL_{\mu} L_{\nu} + L_{\mu} \uL_{\nu}). \label{E:angmdef}
	\end{align}
	\end{subequations}
\end{definition}
\noindent Recall that $\um|_p(X,Y) = m|_p(X,Y)$ for $X,Y \in T|_p \Sigma_t,$ and $\angm(X,Y) = m(X,Y)$ for $X,Y \in T|_p S_{r,t}.$
Note also that the tensorfields $\um_{\mu}^{\ \nu}$ and $\angm_{\mu}^{\ \nu}$ respectively $m-$orthogonally project onto the $\Sigma_t$ and the $S_{r,t}.$ 

We now defined a related tensorfield corresponding to the outgoing Minkowski null cones $C_{q}^+.$

\begin{definition} \label{D:ConeProjection}
	The tensorfield $\coneproject_{\mu}^{\ \nu},$ which $m-$orthogonally projects vectors $X^{\mu}$ onto 
	the outgoing cones $C_{q}^+,$ can be expressed as follows:
	
	\begin{align} \label{E:ConeProjection}
		\coneproject_{\mu}^{\ \nu} & \eqdef \delta_{\mu}^{\nu} + \frac{1}{2} L_{\mu} \uL^{\nu}. 
	\end{align}
\end{definition}

Furthermore, we recall the definitions of the Minkowskian volume forms of Minkowski space and of the surfaces $\Sigma_t$ and $S_{r,t}.$

\begin{definition}
	The \emph{Minkowskian volume forms} of Minkowski spacetime, the surfaces $\Sigma_t,$ and the Euclidean spheres $S_{r,t}$
	are respectively defined relative to our wave coordinate system as follows:
	
	\begin{subequations}
	\begin{align}
		\Minkvolume_{\mu \nu \kappa \lambda} & \eqdef [\mu \nu \kappa \lambda], \label{E:Minkvolumedef} \\
		\uvolume_{\nu \kappa \lambda} & \eqdef \Minkvolume_{0 \nu \kappa \lambda}, \label{E:Sigmatvolumedef} \\
		\angupsilon_{\mu \nu} & \eqdef \Minkvolume_{\mu \nu \kappa \lambda} \uL^{\kappa} L^{\lambda},	\label{E:Spheresvolumedef}
	\end{align}
	\end{subequations}
	where $[\mu \nu \kappa \lambda]$ is totally anti-symmetric with normalization $[0123] = 1.$
\end{definition}

We also recall what it means for a spacetime tensorfield to be $m-$tangent to the surfaces $\Sigma_t$ or $S_{r,t}.$

\begin{definition} \label{D:Tangency}
	Let $U$ be a type $\binom{n}{m}$ spacetime tensorfield. We say that $U$ is $m-$tangent to the time slices $\Sigma_t$ if 
	
	\begin{align}
		U_{\mu_1 \cdots \mu_m}^{\ \ \ \ \ \ \ \nu_1 \cdots \nu_n} 
		= \um_{\mu_1}^{\ \mu'_1} \cdots \um_{\mu_m}^{\ \mu'_m} 
		\um_{\nu'_1}^{\ \nu_1}
		\cdots \um_{\nu'_n}^{\ \nu_1} U_{\mu'_1 \cdots \mu'_m}^{\ \ \ \ \ \ \ 
		\nu'_1 \cdots \nu'_n}.
	\end{align}
		Equivalently, $U$ is $m-$tangent to the $\Sigma_t$ if and only if every wave coordinate component of $U$ containing
		a $0$ index vanishes.
		
	Similarly, we say that $U$ is $m-$tangent to the spheres $S_{r,t}$ if
	
	\begin{align}
		U_{\mu_1 \cdots \mu_m}^{\ \ \ \ \ \ \ \nu_1 \cdots \nu_n} 
		= \angm_{\mu_1}^{\ \mu'_1} \cdots \angm_{\mu_m}^{\ \mu'_m} \angm_{\nu'_1}^{\ \nu_1}
		\cdots \angm_{\nu'_n}^{\ \nu_1} U_{\mu'_1 \cdots \mu'_m}^{\ \ \ \ \ \ \ \nu'_1 
		\cdots \nu'_n}.
	\end{align}
	Equivalently, $U$ is $m-$tangent to the spheres $S_{r,t}$ if and only if any contraction of any index of $U$ with either $\uL$ or $L$ vanishes.
	
\end{definition}

We are now ready to introduce the notion of a \emph{Minkowskian null frame}. We complement the vectorfields $\uL,$ $L$ with a locally-defined pair of $m-$orthogonal vectorfields $e_1,$ $e_2$ that are  tangent to the spheres $S_{r,t},$ and therefore $m-$orthogonal to $\uL,$ $L.$ The resulting collection of vectorfields $\mathcal{N} \eqdef \lbrace \uL,L,e_1,e_2 \rbrace$ is known as \emph{Minkowskian null frame}. It spans the tangent space $T|_p \mathbb{R}^{1+3}$ at each point $p$ where it is defined.

We leave the proof of the following lemma, which summarizes some of the important properties of the geometric quantities introduced in this section, as an exercise for the reader.

\begin{lemma} \label{L:PropertiesofuLandL} \textbf{(Null frame field properties)}
The following identities hold:

\begin{subequations}
\begin{align}
	\nabla_L L & = \nabla_{\uL} \uL = 0,&& \label{E:LanduLaregeodesic} \\
	\nabla_L \uL & = \nabla_{\uL} L = 0,&& \label{E:nablaLuLis0} \\
	L^{\kappa} \uL_{\kappa} & = - 2,&& \label{E:LunderlineLcontracted} \\
	e_A^{\kappa} L_{\kappa} & = e_A^{\kappa} \uL_{\kappa} = 0,&& (A = 1,2),  \label{E:LunderlineLnormaltospheres} \\
	m_{\kappa \lambda} e_A^{\kappa} e_B^{\lambda} & = \delta_{AB},&& (A,B = 1,2),
\end{align}
\end{subequations}

\begin{align} \label{E:NablaLangmandNablauLangmVanish}
	\nabla_{\uL} \angm_{\mu \nu} = \nabla_L \angm_{\mu \nu} & = 0,&& (\mu,\nu =0,1,2,3),
\end{align}

\begin{align}
	\nabla_{\uL} \angupsilon_{\mu \nu} = \nabla_L \angupsilon_{\mu \nu} & = 0,&& (\mu,\nu =0,1,2,3). 
\end{align}
See Definition \ref{D:NablaXDef} concerning our use of notation in these formulas.
\end{lemma}

\hfill $\qed$

Later in the article, we will see that the decay rates of the null components (see Section \ref{SS:NullComponents}) 
of $\Far$ will be distinguished according to the kinds of contractions of $\Far$ taken against $\uL, L, e_1,$ and $e_2.$ With these ideas in mind, we introduce the following sets of vectorfields:

\begin{subequations}
\begin{align} \label{E:Framefieldsubsets}
	\mathcal{L} \eqdef \lbrace L \rbrace, & & \mathcal{T} \eqdef \lbrace L, e_1, e_2 \rbrace,
	& & \mathcal{N} \eqdef \lbrace \uL, L, e_1, e_2 \rbrace.
\end{align}
\end{subequations}
In order to measure the size of the contractions of various tensors and their covariant derivatives against vectors
belonging to the sets $\mathcal{L}, \mathcal{T}, \mathcal{N},$ we introduce the following definitions. 

\begin{definition} \label{D:contractionnomrs}

If $\mathcal{V}, \mathcal{W}$ denote any two of the above sets, and $P$ is a type $\binom{0}{2}$
tensor, then we define the following pointwise seminorms:

\begin{subequations}

\begin{align}
	|P|_{\mathcal{V} \mathcal{W}} & \eqdef \sum_{V \in \mathcal{V}, W \in \mathcal{W}} |V^{\kappa} W^{\lambda} P_{\kappa 
		\lambda}|, \label{E:contractionnorm} \\
	|\nabla P|_{\mathcal{V} \mathcal{W}} & \eqdef \sum_{N \in \mathcal{N}, V \in \mathcal{V}, W \in \mathcal{W}} 
		|V^{\kappa} W^{\lambda} N^{\gamma} \nabla_{\gamma} P_{\kappa \lambda}|, \\
	|\conenabla P|_{\mathcal{V} \mathcal{W}} & \eqdef \sum_{T \in \mathcal{T}, V \in \mathcal{V}, W \in \mathcal{W}} 
		|V^{\kappa} W^{\lambda} T^{\gamma} \nabla_{\gamma} P_{\kappa \lambda}|.
\end{align}
\end{subequations}

We often use the abbreviations $|P| \eqdef |P|_{\mathcal{N} \mathcal{N}},$ 
$|\nabla P| \eqdef |\nabla P|_{\mathcal{N} \mathcal{N}},$
and $|\conenabla P| \eqdef |\conenabla P|_{\mathcal{N} \mathcal{N}}.$

\end{definition}

The above definition generalizes in an obvious way to arbitrary type $\binom{n}{m}$ tensorfields
$U_{\mu_1 \cdots \mu_m}^{\ \ \ \ \ \ \ \nu_1 \cdots \nu_n}.$ Observe that for any such tensorfield, the following
inequalities hold \emph{in our wave coordinate system}:

\begin{align} \label{E:AlternateNablaNorm}
	|U| & \approx \sum_{\mu_1,\cdots,\mu_m,\nu_1,\cdots,\nu_n = 0}^3
	|U_{\mu_1 \cdots \mu_m}^{\ \ \ \ \ \ \ \nu_1 \cdots \nu_n}|.
\end{align}

\subsection{Minkowskian null frame decomposition of a tensorfield}

For an arbitrary vectorfield $X$ and frame vector $N \in \mathcal{N},$ we define 

\begin{align} \label{E:XlowerUdef}
	X_N & \eqdef X_{\kappa} N^{\kappa}, \ \mbox{where} \ X_{\mu} \eqdef m_{\mu \kappa} X^{\kappa}. 
\end{align}
The components $X_N$ are known as the \emph{Minkowskian null components} of $X.$ In the sequel, we will abbreviate

\begin{align}
	X_A \eqdef X_{e_A}, && \nabla_A \eqdef \nabla_{e_A}, \ \mbox{etc}.
\end{align}
It follows from \eqref{E:XlowerUdef} that
\begin{align}
	X & = X^{\kappa} \partial_{\kappa} = X^{L} L + X^{\uL} \uL
		+ X^{A} e_A,  \label{X:AbstractLuLAdecomp} \\
	X^{L} & = - \frac{1}{2}X_{\uL}, && X^{\uL} = - \frac{1}{2}X_L, && X^A = X_A.
		\label{E:XupperUdef}
\end{align}
Furthermore, it is easy to check that
\begin{align} \label{E:XYnullframeinnerproduct}
	m(X,Y) \eqdef m_{\kappa \lambda}X^{\kappa}X^{\lambda} = X^{\kappa}Y_{\kappa} = -\frac{1}{2}X_{L}Y_{\uL} - 
	\frac{1}{2}X_{\uL}Y_{L} + \delta^{AB} X_A Y_B. 
\end{align}

The above null decomposition of a vectorfield generalizes in the obvious way to higher order tensorfields. In the next section,
we provide a detailed version of the null decomposition of two-forms $\Far,$ since this decomposition is needed
for our derivation of decay estimates later in the article; see e.g.
Proposition \ref{P:EOVNullDecomposition} and Proposition \ref{P:EnergyInhomogeneousTermAlgebraicEstimate}.

\subsection{The detailed Minkowskian null decomposition of a two-form} \label{SS:NullComponents}

\begin{definition} \label{D:null}
Given any two-form $\Far,$ we define its \emph{Minkowskian null components} to be the following pair of one-forms
$\ualpha_{\mu},$ $\alpha_{\mu},$ and the following pair of scalar $\rho,$ $\sigma:$

\begin{subequations}
\begin{align}
	\ualpha_{\mu} & \eqdef \angm_{\mu}^{\ \nu} \Far_{\nu \lambda} \uL^{\lambda},&& (\mu = 0,1,2,3),
		\label{E:ualphadef} \\
	\alpha_{\mu} & \eqdef \angm_{\mu}^{\ \nu} \Far_{\nu \lambda} L^{\lambda},&& (\mu = 0,1,2,3), \\
	\rho & \eqdef \frac{1}{2} \Far_{\kappa \lambda}\uL^{\kappa} L^{\lambda},&& \\
	\sigma & \eqdef \frac{1}{2} \angupsilon^{\kappa \lambda} \Far_{\kappa \lambda}.&& \label{E:sigmadef}
\end{align}
\end{subequations}

\end{definition}

It is a simple exercise to check that $\ualpha_{\mu}$ and $\alpha_{\mu}$ are $m-$tangent to the spheres $S_{r,t}:$

\begin{subequations}
\begin{align}
	\ualpha_{\kappa}\uL^{\kappa} & = 0, & \ualpha_{\kappa}L^{\kappa} & = 0, \\
	\alpha_{\kappa}\uL^{\kappa} & = 0, & \alpha_{\kappa}L^{\kappa} & = 0.
\end{align}
\end{subequations}
Furthermore, relative to the null frame $\mathcal{N} \eqdef \lbrace \uL, L, e_1, e_2 \rbrace,$ we have that

\begin{subequations}
\begin{align}
	\underline{\alpha}_A & = \Far_{A \uL},&& (A = 1,2), \\
	\alpha_A & = \Far_{AL},&& (A = 1,2), \\
	\rho & = \frac{1}{2} \Far_{\uL L},&& \\
	\sigma & = \Far_{12}.&&
\end{align}
\end{subequations}

In terms of the seminorms introduced in Definition \ref{D:contractionnomrs}, it follows that

\begin{subequations}
\begin{align}
	|\Far| & \approx |\Far|_{\mathcal{N}\mathcal{N}} \approx |\ualpha| + |\alpha| + |\rho| + |\sigma|, \\
	|\Far|_{\mathcal{L}\mathcal{N}} & \approx |\alpha| + |\rho|, \\
	|\Far|_{\mathcal{T} \mathcal{T}} & \approx |\alpha| + |\sigma|. 
\end{align}
\end{subequations}

The null components of $\FarMinkdual$ (the Minkowskian Hodge duality operator $\ostar$ is defined in Section \ref{SS:Hodge}) can be expressed in terms of the above null components of $\Far.$ Denoting the null components\footnote{We use the symbol $\odot$ in order to avoid confusion with the Minkowskian Hodge duality operator $\ostar;$ i.e., it is not true that ${^{\ostar \hspace{-.03in}}(\ualpha[\Far])} = \ualpha[\FarMinkdual].$} of $\FarMinkdual$ by $\ualphadot, \alphadot, \rhodot, \sigmadot,$ we leave it as a simple exercise to the reader to check that  

\begin{subequations}
\begin{align}
	\ualphadot_A & = - \underline{\alpha}^B \angupsilon_{BA},&& (A = 1,2), \label{E:Fardualalpha} \\
	\alphadot_A & = \alpha^B \angupsilon_{BA},&& (A = 1,2), \\
	\rhodot & = \sigma,&& \\  
	\sigmadot & = - \rho.&& \label{E:Fardualsigma}
\end{align}
\end{subequations}

\section{Differential Operators} \label{S:DifferentialOperators}
In this section, we introduce a collection of differential operators that will be used throughout the remainder of the article.
In order to define these operators, we also introduce subsets $\mathcal{O}$ and $\mathcal{Z}$ of Minkowskian conformal Killing fields. Finally, we prove a collection of lemmas that expose useful properties of these operators, and that illustrate various relationships between them.

\noindent \hrulefill
\ \\

\subsection{Covariant derivatives}

As previously mentioned, throughout the article, $\nabla$ denotes the Levi-Civita connection of the Minkowski metric $m.$
Let $\um$ and $\angm$ be the first fundamental forms of the $\Sigma_t$ and $S_{r,t}$ as defined in
Definition \ref{D:FirstFundamental}, and let $\unabla,$ $\angn$ be their corresponding Levi-Civita connections. We state as a lemma the following well-known identities, which relates the connections $\unabla,$ $\angn$ to $\nabla$ through the first fundamental forms.

\begin{lemma} \textbf{(Relationships between connections)}
	If $U$ is any type $\binom{n}{m}$ tensorfield $m-$tangent to the $\Sigma_t,$ then

\begin{align} 
		\unabla_{\lambda} 
			U_{\mu_1 \cdots \mu_m}^{\ \ \ \ \ \ \ \nu_1 \cdots \nu_n} & = \um_{\lambda}^{\ \lambda'}
			\um_{\mu_1}^{\ \mu'_1} \cdots 
			\um_{\mu_m}^{\ \mu'_m} \um_{\nu'_1}^{\ \nu_1} \cdots \um_{\nu'_n}^{\ \nu_n} 
			\nabla_{\lambda'} U_{\mu'_1 \cdots \mu'_m}^{\ \ \ \ 
			\ \ \ \nu'_1 \cdots \nu'_n}, 
			\label{E:SigmatIntrinsicintermsofExtrinsic} 
\end{align}

Similarly, if If $U$ is any type $\binom{n}{m}$ tensorfield $m-$tangent to $S_{r,t}$ then

\begin{align}
		\angn_{\lambda}
			U_{\mu_1 \cdots \mu_m}^{\ \ \ \ \ \ \ \nu_1 \cdots \nu_n} & = \angm_{\lambda}^{\ \lambda'}
			\angm_{\mu_1}^{\ \mu'_1} \cdots 
			\angm_{\mu_m}^{\ \mu'_m} \angm_{\nu'_1}^{\ \nu_1} \cdots \angm_{\nu'_n}^{\ \nu_n} 
			\nabla_{\lambda'} U_{\mu'_1 \cdots \mu'_m}^{\ \ \ \ 
			\ \ \ \nu'_1 \cdots \nu'_n}. \label{E:SphereIntrinsicintermsofExtrinsic}
\end{align}

\end{lemma}

We recall the following fundamental properties of the connections $\nabla,$ $\unabla,$ and $\angn:$

\begin{subequations}
\begin{align}
	\nabla_{\lambda} m_{\mu \nu} & = 0 = \nabla_{\lambda} (m^{-1})^{\mu \nu},&& (\lambda, \mu, \nu = 0,1,2,3), 
		\label{E:nablamis0} \\
	\unabla_{\lambda} \um_{\mu \nu} & = 0,&& (\lambda, \mu, \nu = 0,1,2,3), \\
	\angn_{\lambda} \angm_{\mu \nu} & = 0,&& (\lambda, \mu, \nu = 0,1,2,3). \label{E:angnablaangmis0}
\end{align}
\end{subequations}

We will also make use of the projection of the operator $\nabla$ onto the favorable directions, i.e., the directions tangent to the outgoing Minkowski cones $C_q^+.$ 

\begin{definition} \label{D:ConeProjectedDerivative}
	If $U$ is any type $\binom{n}{m}$ spacetime tensorfield, then we define the projected Minkowskian 
	covariant derivative $\conenabla U$ by
	
	\begin{align} \label{E:ConeProjectedDerivative}
		\conenabla_{\lambda} U_{\mu_1 \cdots \mu_m}^{\ \ \ \ \ \ \ \nu_1 \cdots \nu_n} 
		& = \coneproject_{\lambda}^{\ \lambda'} \nabla_{\lambda'}U_{\mu_1 \cdots \mu_m}^{\ \ \ \ 
			\ \ \ \nu_1 \cdots \nu_n},
\end{align}
where the projection $\coneproject_{\mu}^{\ \nu}$ is defined in \eqref{E:ConeProjection}.

\end{definition}

\begin{remark}
	Note that only the $\lambda$ component is projected onto the outgoing cones, so that the tensorfield
	$\conenabla_{\lambda} U_{\mu_1 \cdots \mu_m}^{\ \ \ \ \ \ \ \nu_1 \cdots \nu_n}$ need not be $m-$tangent to 
	the outgoing Minkowski cones.
\end{remark}

\begin{definition} \label{D:NablaXDef}
	If $X$ is any vectorfield, then we define the covariant derivative operators $\nabla_X$ 
	and $\angn_X$ by
	
	\begin{subequations}
	\begin{align}
		\nabla_X & \eqdef X^{\kappa} \nabla_{\kappa}, \\
		\angn_{X} & \eqdef X^{\kappa} \angn_{\kappa}.
	\end{align}
	\end{subequations}
	
\end{definition}

\subsection{Minkowskian conformal Killing fields} \label{SS:ConformalKillingFields}
In this section, we introduce the special set of vectorfields $\mathcal{Z}$ that appears in the definition \eqref{E:EnergyIntro} of our energy $\mathcal{E}_{\dParameter;\upgamma;\upmu}(t),$ and in the weighted Klainerman-Sobolev inequality
\eqref{E:KSIntro}. We begin by recalling that a \emph{Minkowskian conformal Killing field} is a vectorfield $Z$ such that

\begin{align} \label{E:ConformalKillingDef}
	\nabla_{\mu} Z_{\nu} + \nabla_{\nu} Z_{\mu} = {^{(Z)}\phi} m_{\mu \nu}
\end{align}
for some function ${^{(Z)}\phi}(t,x).$ The tensorfield 

\begin{align} \label{E:MinkowskianDeformationTensordef}
	{^{(Z)}\pi_{\mu \nu}} \eqdef \nabla_{\mu} Z_{\nu} + \nabla_{\nu} Z_{\mu}
\end{align}
is known as the \emph{Minkowskian deformation tensor} of $Z.$ If ${^{(Z)}\pi_{\mu \nu}} = 0,$ then $Z$ is known as a \emph{Minkowskian Killing field}.
We also recall that the conformal Killing fields of the Minkowski metric $m_{\mu \nu}$ form a Lie Algebra generated under the Lie bracket $[\cdot,\cdot]$ (see \eqref{E:bracket}) that is generated by the following $15$ vectorfields (see e.g. \cite{dC2008}): 

\begin{enumerate}
	\item the four \emph{translations} $\partial_{\mu} = \frac{\partial}{\partial x^{\mu}}, \qquad (\mu =0,1,2,3),$ 
	\item the three \emph{rotations} $\Omega_{jk} \eqdef x_j \frac{\partial}{\partial x^k} - x_k \frac{\partial}{\partial x^j}, 
		\qquad (1 \leq j < k \leq 3),$
	\item the three \emph{Lorentz boosts} $\Omega_{jk} \eqdef -t \frac{\partial}{\partial x^j} - x_j \frac{\partial}{\partial t}, 
		\qquad (j = 1,2,3),$
	\item the \emph{scaling} vectorfield $S \eqdef x^{\kappa} \frac{\partial}{\partial x^{\kappa}},$	
	\item the four \emph{acceleration} vectorfields $K_{(\mu)} \eqdef - 2 x_{\mu} S 
	+ g_{\kappa \lambda}x^{\kappa} x^{\lambda} \frac{\partial}{\partial x^{\mu}}, \qquad \mu = (0,1,2,3).$ 
\end{enumerate}
It can be checked that the translations, rotations, and Lorentz boosts are in fact Killing fields of $m_{\mu \nu}.$

Two subsets of the above conformal Killing fields will play a prominent role in the remainder of the article, namely the rotations $\mathcal{O}$ and a larger set $\mathcal{Z},$ which are defined by

\begin{subequations}
\begin{align}
	\mathcal{O} & \eqdef \big\lbrace \Omega_{jk} \big\rbrace_{1 \leq j < k \leq 3}, \label{E:Rotationsetdef} \\
	\mathcal{Z} & \eqdef \big\lbrace \frac{\partial}{\partial x^{\mu}}, \Omega_{\mu \nu}, S \big\rbrace_{0 \leq \mu < \nu \leq 3}.  \label{E:Zsetdef} 
\end{align}
\end{subequations}
The vectorfields in $\mathcal{Z}$ satisfy a strong version of the relation \eqref{E:ConformalKillingDef}. That is,
if $Z \in \mathcal{Z},$ then

\begin{align} \label{E:CovariantDerivativesofZareConstant}
	\nabla_{\mu} Z_{\nu} = {^{(Z)}c}_{\mu \nu},
\end{align}
where the components ${^{(Z)}c}_{\mu \nu}$ are \emph{constants} in our wave coordinate system. In particular, we compute for future use that

\begin{subequations}
\begin{align}
	\nabla_{\mu} S_{\nu} & = m_{\mu \nu}, \label{E:ScalingCovariantDerivative} \\
	\nabla_{\mu} (\Omega_{\kappa \lambda})_{\nu} & = m_{\mu \kappa} m_{\nu \lambda} - m_{\mu \lambda} m_{\nu \kappa}.
	\label{E:SpacetimeRotationsCovariantDerivative}
\end{align}
\end{subequations}
We note in addition that if $Z \in \mathcal{Z},$ then there exists a \emph{constant} $c_Z$ such that

\begin{align} \label{E:ZDeformationTensorinTermsofcZm}
	\nabla_{\mu} Z_{\nu} + \nabla_{\nu} Z_{\mu} = c_Z m_{\mu \nu}.
\end{align}
Furthermore, by contracting each side of \eqref{E:ZDeformationTensorinTermsofcZm} against $(m^{-1})^{\mu \nu},$ it follows that

\begin{align}
	c_Z = \frac{1}{4} {^{(Z)}\pi_{\kappa}^{\ \kappa}} = \frac{1}{2} {^{(Z)}c}_{\kappa}^{\ \kappa}.
\end{align}

\subsection{Lie derivatives}
As mentioned in Section \ref{SS:DiscussionofProof}, it is convenient to use Lie derivatives to differentiate 
the electromagnetic equations \eqref{E:ReduceddFis0Summary} - \eqref{E:ReduceddMis0Summary}. In this section,
we recall some basic facts concerning Lie derivatives.

We recall that if $X,Y,$ are any pair of vectorfields, then relative to an arbitrary coordinate system,
their \emph{Lie bracket} $[X,Y]$ can be expressed as

\begin{align} \label{E:bracket}
	[X,Y]^{\mu} & = X^{\kappa} \partial_{\kappa} Y^{\mu} - Y^{\kappa} \partial_{\kappa} X^{\mu}.
\end{align}
Furthermore, we have that 

\begin{align} \label{E:LieXY}
	\Lie_X Y = [X,Y], 
\end{align}	
where $\Lie$ denotes the \emph{Lie derivative operator}. Given a type $\binom{0}{m}$ tensorfield $U,$ and vectorfields $Y_{(1)}, \cdots Y_{(m)},$ the Leibniz rule for $\Lie$ implies that \eqref{E:LieXY} generalizes as follows:

\begin{align} \label{E:Liederivativebracketexpression}
	(\Lie_X U)(Y_{(1)}, \cdots, Y_{(m)}) & = X \lbrace U(Y_{(1)}, \cdots, Y_{(m)}) \rbrace
		- \sum_{i=1}^n U(Y_{(1)}, \cdots, Y_{(i-1)}, [X,Y_{(i)}], Y_{(i+1)}, \cdots, Y_{(m)}).
\end{align} 

Using Lemma \ref{L:Liederivativeintermsofnabla} below, we see that
the left-hand side of \eqref{E:ZDeformationTensorinTermsofcZm} is equal
the Lie derivative of the Minkowski metric. It therefore follows that if $Z \in \mathcal{Z},$ then

\begin{subequations}
\begin{align}
	\Lie_Z m_{\mu \nu} & = c_Z m_{\mu \nu}, \label{E:LieZonmlower} \\
	(\Lie_Z m^{-1})^{\mu \nu} & = - c_Z (m^{-1})^{\mu \nu}, \label{E:LieZonmupper}
\end{align}
\end{subequations}
where the constant $c_Z$ is defined in \eqref{E:ZDeformationTensorinTermsofcZm}. 

\subsection{Modified covariant and modified Lie derivatives}

It will be convenient for us to work with \emph{modified Minkowski covariant derivatives} $\nablamod_{Z}$ and
\emph{modified Lie derivatives}\footnote{Note that these are not the same
modified Lie derivatives that appear in \cite{lBnZ2009}, \cite{dCsK1993}, and \cite{nZ2000}.} $\Liemod_Z.$

\begin{definition} \label{D:ModifiedDerivatives}
	For $Z \in \mathcal{Z},$ we define the modified Minkowski covariant derivative $\nablamod_Z$ by
	
	\begin{align} \label{E:Covariantmoddef}
		\nablamod_Z \eqdef \nabla_Z + c_Z,
	\end{align}
	where $c_Z$ denotes the constant from \eqref{E:ZDeformationTensorinTermsofcZm}.
	
	For each vectorfield $Z \in \mathcal{Z},$ we define the modified Lie derivative $\Liemod_Z$ by
	\begin{align} \label{E:Liemoddef}
		\Liemod_Z \eqdef \Lie_Z + 2c_Z,
	\end{align}
	where $c_Z$ denotes the constant from \eqref{E:ZDeformationTensorinTermsofcZm}.
\end{definition}
The crucial features of the above definitions are captured by Lemmas \ref{L:NablaModZLiemodMinkowskiWaveOperatorCommutator} and
\ref{L:LiemodZLiemodMaxwellCommutator} below. The first shows that for each $Z \in \mathcal{Z},$
$\nablamod_Z \Square_m \phi = \Square_m \nabla_Z \phi,$ where $\Square_m = (m^{-1})^{\kappa \lambda}\nabla_{\kappa} \nabla_{\lambda}$ is the \emph{Minkowski} wave operator. The second shows that
 $\Liemod_Z \Big\lbrace \big[(m^{-1})^{\mu \kappa} (m^{-1})^{\nu \lambda} - (m^{-1})^{\mu \lambda}(m^{-1})^{\nu \kappa} \big] \nabla_{\mu} \Far_{\kappa \lambda}\Big\rbrace$ $= \big[(m^{-1})^{\mu \kappa} (m^{-1})^{\nu \lambda} - (m^{-1})^{\mu \lambda}(m^{-1})^{\nu \kappa}\big] \nabla_{\mu} 
\Lie_Z \Far_{\kappa \lambda}.$ Furthermore, Lemma \ref{L:Liecommuteswithcoordinatederivatives} shows that $\Lie_Z \nabla_{[\lambda}\Far_{\mu \nu]}$ = $\nabla_{[\lambda}\Lie_Z \Far_{\mu \nu]},$ where $[\cdots]$ denotes anti-symmetrization.
These commutation identities suggest that the operators $\Liemod_Z$ and $\nablamod_Z$ are
potentially useful operators for differentiating equations \eqref{E:Reducedh1Summary} and 
\eqref{E:ReduceddFis0Summary} - \eqref{E:ReduceddMis0Summary} respectively. This suggestion is borne out in
Propositions \ref{P:EnergyInhomogeneousTermAlgebraicEstimate} and \ref{P:DIPointwise}, which show that the
inhomogeneous terms generated by differentiating the nonlinear equations have a special algebraic structure, a structure that will be exploited during our global existence bootstrap argument.

\subsection{Vectorfield algebra}

We introduce here some notation that will allow us to compactly express iterated derivatives. If $\mathcal{A}$
is one of the sets from \eqref{E:Rotationsetdef} - \eqref{E:Zsetdef}, then we label the vectorfields in $\mathcal{A}$ as $Z^{\iota_{1}}, \cdots, Z^{\iota_{d}},$ where $d$ is the cardinality of $\mathcal{A}.$ Then for any multi-index 
$I = (\iota_1, \cdots, \iota_k)$ of length $k,$ where each $\iota_i \in \lbrace 1,2,\cdots, d \rbrace,$ we make the
following definition.

\begin{definition} \label{D:iterated}
The iterated derivative operators are defined by

\begin{subequations}
\begin{align} 
	\nabla_{\mathcal{A}}^I & \eqdef \nabla_{Z^{\iota_1}} \circ \cdots \circ \nabla_{Z^{\iota_k}},  \label{E:iteratedCovariant} \\
	\nablamod_{\mathcal{A}}^I & \eqdef \nablamod_{Z^{\iota_1}} \circ \cdots \circ \nablamod_{Z^{\iota_k}}, \\
	\Lie_{\mathcal{A}}^I & \eqdef \Lie_{Z^{\iota_1}} \circ \cdots \circ \Lie_{Z^{\iota_k}}, \label{E:iteratedLie} \\
	\Liemod_{\mathcal{A}}^I & \eqdef \Liemod_{Z^{\iota_1}} \circ \cdots \circ \Liemod_{Z^{\iota_k}}, 
\end{align}
\end{subequations}
etc. 
\end{definition}

Similarly, if $I = (\mu_1, \cdots, \mu_k)$ is a coordinate multi-index of length $k,$ where 
$\mu_1, \cdots, \mu_k \in \lbrace 0,1,2,3 \rbrace,$ and $U$ is a tensorfield, then
we use shorthand notation such as

\begin{align}
	\nabla^I U \eqdef \nabla_{\mu_1} \cdots \nabla_{\mu_k} U,
\end{align}
etc.

Under the above conventions, the Leibniz rule can be written as e.g.

\begin{align}
	\Lie_{\mathcal{Z}}^I (UV) = \sum_{I_1 + I_2 = I} (\Lie_{\mathcal{Z}}^I U)(\Lie_{\mathcal{Z}}^I V),
\end{align}
etc., where by a sum over $I_1 + I_2 = I,$ we mean a sum over all order preserving partitions of the index $I$ into two 
multi-indices; i.e., if $I = (\iota_1, \cdots, \iota_k),$ then $I_1 = (\iota_{i_1}, \cdots, \iota_{i_a}), 
I_2 = (\iota_{i_{a+1}}, \cdots, \iota_{i_k}),$ where $i_1, \cdots, i_k$ is any re-ordering of the integers
$1,\cdots,k$ such that $i_1 < \cdots < i_a,$ and $i_{a+1} < \cdots < i_k.$

The next standard lemma provides a useful expression relating Lie derivatives to covariant derivatives.

\begin{lemma} \cite{rW1984} \label{L:Liederivativeintermsofnabla} \textbf{(Lie derivatives in terms of covariant derivatives)}
Let $X$ be a vectorfield, and let $U$ be a tensorfield of type $\binom{n}{m}.$ Then $\Lie_X U$ can be expressed in terms of covariant derivatives of $U$ and $X$ as follows:

\begin{align} \label{E:Liederivativeintermsofnabla} 
	\Lie_X U_{\mu_1 \cdots \mu_m}^{\ \ \ \ \ \ \ \ \nu_1 \cdots \nu_n} =
		\nabla_X U_{\mu_1 \cdots \mu_m}^{\ \ \ \ \ \ \ \ \nu_1 \cdots \nu_n} 
	& + U_{\kappa \mu_2 \cdots \mu_m}^{\ \ \ \ \ \ \ \ \nu_1 \cdots \nu_n}\nabla_{\mu_1}X^{\kappa} 
		+ \cdots + U_{\mu_1 \cdots \mu_{m-1} \kappa}^{\ \ \ \ \ \ \ \ \ \nu_1 \cdots \nu_n}\nabla_{\mu_m}X^{\kappa} \\
	& - U_{\mu_1 \cdots \mu_m}^{\ \ \ \ \ \ \ \ \kappa \cdots \nu_n} \nabla_{\kappa}X^{\nu_1}
	 	- \cdots - U_{\mu_1 \cdots \mu_m}^{\ \ \ \ \ \ \ \ \nu_1 \cdots \nu_{n-1} \kappa} \nabla_{\kappa}X^{\nu_n}. \notag
\end{align}

\end{lemma}

\hfill $\qed$

The next lemma shows that the operators $\Lie_Z$ and $\Liemod_Z$ commute with $\nabla$ if $Z \in \mathcal{Z}.$

\begin{lemma} \label{L:Liecommuteswithcoordinatederivatives} \textbf{($\Lie_Z$ and $\nabla$ commute)}
Let $\nabla$ denote the Levi-Civita connection corresponding to the Minkowski metric $m,$ and let $I$ be a 
$\mathcal{Z}-$multi-index. Let $\Liemod_{\mathcal{Z}}^I$ be the iterated modified Lie derivative from Definitions \ref{D:ModifiedDerivatives} and \ref{D:iterated}. Then

\begin{align} \label{E:Liecommuteswithcoordinatederivatives}
	[\nabla, \Lie_{\mathcal{Z}}^I] & = 0, && [\nabla, \Liemod_{\mathcal{Z}}^I] = 0.
\end{align}

In an arbitrary coordinate system, equations \eqref{E:Liecommuteswithcoordinatederivatives} are equivalent to the following relations, which hold for all type $\binom{n}{m}$ tensorfields $U:$

\begin{align}
	\nabla_{\mu}\big\lbrace \Lie_{\mathcal{Z}}^I U_{\mu_1 \cdots \mu_m}^{\ \ \ \ \ \ \ \ \nu_1 \cdots \nu_n} \big\rbrace 
	& = \Lie_{\mathcal{Z}}^I \big\lbrace \nabla_{\mu}U_{\mu_1 \cdots \mu_m}^{\ \ \ \ \ \ \ \ \nu_1 \cdots \nu_n} \big\rbrace, \\
	\nabla_{\mu}\big\lbrace \Liemod_{\mathcal{Z}}^I U_{\mu_1 \cdots \mu_m}^{\ \ \ \ \ \ \ \ \nu_1 \cdots \nu_n} \big\rbrace 
	& = \Liemod_{\mathcal{Z}}^I \big\lbrace \nabla_{\mu} U_{\mu_1 \cdots \mu_m}^{\ \ \ \ \ \ \ \ \nu_1 \cdots \nu_n} \big\rbrace. \notag
\end{align}

\end{lemma}

\begin{proof}
	The relation \eqref{E:Liecommuteswithcoordinatederivatives} can be shown via induction in $|I|$ using
	\eqref{E:Liederivativeintermsofnabla} and the fact that $\nabla\nabla Z = 0.$ 
\end{proof}

The next lemma captures the commutation properties of vectorfields $Z \in \mathcal{Z}.$

\begin{lemma} \label{L:ConformalKillingFieldCommuatators} \cite[pg. 139]{dCsK1990} 
\textbf{(Lie bracket relations)}
	Relative to the wave coordinate system
	$\lbrace x^{\mu} \rbrace_{\mu = 0,1,2,3},$ the vectorfields belonging to the subset $\mathcal{Z} \eqdef \big\lbrace 
	\frac{\partial}{\partial x^{\mu}}, \Omega_{\mu \nu}, S \big\rbrace_{0 \leq \mu < \nu \leq 3}$ of the Minkowskian conformal 
	Killing fields satisfy the following commutation relations, where ${^{(Z)} c_{\mu}^{\ \kappa}}$ is defined in 
	\eqref{E:CovariantDerivativesofZareConstant}:
	
	\begin{subequations}
	\begin{align}
		\left[\frac{\partial}{\partial x^{\mu}}, \frac{\partial}{\partial x^{\nu}}\right] & = 0
			= {^{(\frac{\partial}{\partial x^{\nu}})} c_{\mu}^{\ \kappa}} \frac{\partial}{\partial x^{\kappa}}, 
			&& (\mu, \nu = 0,1,2,3), \label{E:translationscommute} \\
		\left[\frac{\partial}{\partial x^{\lambda}}, \Omega_{\mu \nu}\right] & = m_{\lambda \mu}\frac{\partial}{\partial x^{\nu}} 
			- m_{\lambda \nu} \frac{\partial}{\partial x^{\mu}}
			= {^{(\Omega_{\mu \nu})} c_{\lambda}^{\ \kappa}} \frac{\partial}{\partial x^{\kappa}}, && (\lambda, \mu, \nu = 0,1,2,3), 
				\\
		\left[\frac{\partial}{\partial x^{\mu}}, S \right] & = \frac{\partial}{\partial x^{\mu}}
			= {^{(S)} c_{\mu}^{\ \kappa}} \frac{\partial}{\partial x^{\kappa}}, && (\mu = 0,1,2,3), \\
		\left[\Omega_{\kappa \lambda}, \Omega_{\mu \nu} \right] & = m_{\kappa \mu} \Omega_{\nu \lambda} 
			- m_{\kappa \nu} \Omega_{\mu \lambda} + m_{\lambda \mu} \Omega_{\kappa \nu}
			- m_{\lambda \nu} \Omega_{\kappa \mu}, && (\kappa, \lambda, \mu, \nu = 0,1,2,3), \\
		\left[\Omega_{\mu \nu}, S\right] & = 0, && (\mu, \nu = 0,1,2,3).
	\end{align}	
	\end{subequations}
\end{lemma}

\hfill $\qed$

We now state the following simple commutation lemma.

\begin{lemma} \label{L:NablapartialmuNablaZCommutatorExpression} 
\textbf{($\nabla_Z,$ $\nabla_{\frac{\partial}{\partial x^{\mu}}}$ commutation relations)}
Let $Z \in \mathcal{Z}.$ Then relative to the wave coordinate system
$\lbrace x^{\mu} \rbrace_{\mu = 0,1,2,3},$ the differential operators $\nabla_{\frac{\partial}{\partial x^{\mu}}}$ and $\nabla_{Z}$ satisfy the following commutation relations:

\begin{align} \label{E:NablapartialmuNablaZCommutatorExpression}
	[\nabla_{\frac{\partial}{\partial x^{\mu}}}, \nabla_Z] = {^{(Z)} c_{\mu}^{\ \kappa}} \frac{\partial}{\partial x^{\kappa}},
\end{align}
where ${^{(Z)} c_{\mu}^{\ \kappa}}$ is defined in \eqref{E:CovariantDerivativesofZareConstant}.

\end{lemma}

\begin{proof}
The relation \eqref{E:NablapartialmuNablaZCommutatorExpression} follows from Lemma \ref{L:ConformalKillingFieldCommuatators} 
and the identity $[\nabla_X, \nabla_Y] = \nabla_{[X,Y]},$ which holds for all pairs of vectorfields $X,Y;$ this identity holds
because of the torsion-free property of the connection $\nabla$ and because the Riemann curvature tensor of the Minkowski metric $m_{\mu \nu}$ completely vanishes.
\end{proof}

The next lemma shows that the operators $\nabla$ and $\nabla_{\mathcal{Z}}^I$ commute up to lower-order terms.

\begin{lemma} \label{L:NablaZICommutesWithCovariantDerivativePlusErrorTerms} \textbf{($\nabla$ and 
$\nabla_{\mathcal{Z}}^I$ commutation inequalities)}
Let $U$ be a type $\binom{n}{m}$ tensorfield, and let $I$ be a $\mathcal{Z}-$multi-index. Then the following inequality holds:

\begin{align} \label{E:NablaZICommutesWithCovariantDerivativePlusErrorTerms}
	|\nabla_{\mathcal{Z}}^I \nabla U| & \lesssim |\nabla\nabla_{\mathcal{Z}}^I U|
	\ + \ \sum_{|J| \leq |I| - 1} |\nabla\nabla_{\mathcal{Z}}^J U|.
\end{align}

\end{lemma}

\begin{proof}
	Using \eqref{E:AlternateNablaNorm}, we have that
	
	\begin{align} \label{E:NablaZINablaPNormapproximatelyNablaZISumOverNablamuPNorm}
		|\nabla_{\mathcal{Z}}^I \nabla U| & \approx \sum_{\mu = 0}^3 
			|\nabla_{\mathcal{Z}}^I \nabla_{\frac{\partial}{\partial x^{\mu}}}U|.
	\end{align}
	We therefore repeatedly apply Lemma \ref{L:NablapartialmuNablaZCommutatorExpression} to deduce that
	there exist constants $C_{I;J}^{\nu}$ such that
	
	\begin{align} \label{E:NablapartialmuNablaZICommutatorExpression}
			\nabla_{\mathcal{Z}}^I \nabla_{\frac{\partial}{\partial x^{\mu}}}U
			& = \nabla_{\frac{\partial}{\partial x^{\mu}}} \nabla_{\mathcal{Z}}^I U
				\ + \ \sum_{|J| \leq |I| - 1} \sum_{\nu = 0}^3 C_{I;J}^{\nu}
				\nabla_{\frac{\partial}{\partial x^{\nu}}} \nabla_{\mathcal{Z}}^J U.
	\end{align}
	Inequality \eqref{E:NablaZICommutesWithCovariantDerivativePlusErrorTerms}
	now follows from applying \eqref{E:AlternateNablaNorm} to each side of \eqref{E:NablapartialmuNablaZICommutatorExpression}.

\end{proof}

The next lemma captures some important differential identities.

\begin{lemma} \label{L:LieDerivativeCommuatorsVanish} \textbf{(Geometric differential identities)}
	 Let $\uL,$ $L$ be the Minkowski-null geodesic vectorfields defined in \eqref{E:uLdef} - \eqref{E:Ldef}, and let
	 $O \in \mathcal{O}.$ Then the vectorfields, $\uL,$ $L,$ $O$ mutually commute:

\begin{align}
	[\uL, L] & = 0, && [\uL,O] = 0, && [L, O] = 0. \label{E:RotationLuLBracketis0} 
\end{align}
	
	Furthermore, let $\Minkvolume_{\kappa \lambda \mu \nu},$ $\angm_{\mu \nu}$ denote the volume forms defined in
	\eqref{E:Minkvolumedef} and \eqref{E:Spheresvolumedef}. Then

	\begin{subequations}
	\begin{align}
		\Lie_O \Minkvolume_{\kappa \lambda \mu \nu} & = 0, \label{E:LieO4epsilonis0} \\
		\Lie_O \angm_{\mu \nu} & = 0, \label{E:LieOanggis0} \\
		\Lie_O \angupsilon_{\mu \nu} & = 0. \label{E:LieO2epsilonis0}
	\end{align}
	\end{subequations}
	
\end{lemma}

\begin{proof}
	\eqref{E:RotationLuLBracketis0} can be checked via simple calculations using the definitions \eqref{E:uLdef} - \eqref{E:Ldef} of $\uL$ and $L,$ the 
	definitions of the rotations $O \in \mathcal{O}$ given at the beginning of Section \ref{SS:ConformalKillingFields}, and the 	  Lie bracket formula \eqref{E:bracket}.
	\eqref{E:LieO4epsilonis0} follows from the well-known identity $\Lie_X \Minkvolume_{\kappa \lambda \mu \nu} = \frac{1}{2} {^{(X)}\pi_{\ \beta}^{\beta}}
	\Minkvolume_{\kappa \lambda \mu \nu},$ where ${^{(X)}\pi_{\mu \nu}}$ is defined in \eqref{E:MinkowskianDeformationTensordef}, together with
	the fact that $\Lie_O m_{\mu \nu} = {^{(O)}\pi_{\mu \nu}} = 0$ (i.e., that $O$ is a Killing field of $m_{\mu \nu}$). 
	\eqref{E:LieOanggis0} and \eqref{E:LieO2epsilonis0} then follow from definitions \eqref{E:angmdef}, 
	\eqref{E:Spheresvolumedef}, and \eqref{E:RotationLuLBracketis0} - \eqref{E:LieO4epsilonis0}.
\end{proof}

The next lemma shows that the modified covariant derivatives $\nablamod_{\mathcal{Z}}^I$ have favorable commutation properties
with the Minkowski wave operator.

\begin{lemma} \label{L:NablaModZLiemodMinkowskiWaveOperatorCommutator} \textbf{($\nablamod_{\mathcal{Z}}^I$ and
	$\Square_m$ commutation properties)}
	Let $I$ be a $\mathcal{Z}-$multi-index, and let 
	$\phi$ be any function. Let $\nablamod_{\mathcal{Z}}^I$ be the iterated modified Minkowski covariant derivative operator from 
	Definitions \ref{D:ModifiedDerivatives} and \ref{D:iterated}, and let 
	$\Square_m \eqdef (m^{-1})^{\kappa \lambda} \nabla_{\kappa} \nabla_{\lambda}$ denote the Minkowski wave operator.
	Then
	
	\begin{align} \label{E:NablaModZLiemodMinkowskiWaveOperatorCommutator}
		\nablamod_{\mathcal{Z}}^I \Square_m \phi & = \Square_m \nabla_{\mathcal{Z}}^I \phi.
	\end{align}
\end{lemma}

\begin{proof}
	Using the symmetry of the tensorfield $\nabla_{\kappa} \nabla_{\lambda} \phi,$ together with \eqref{E:nablamis0}, 
	\eqref{E:ZDeformationTensorinTermsofcZm}, and definition \eqref{E:Covariantmoddef}, we compute that
	
	\begin{align}
		\Square_m \nabla_Z \phi = (m^{-1})^{\kappa \lambda} \nabla_{\kappa} \nabla_{\lambda} 
			\big(Z^{\zeta} \nabla_{\zeta} \phi \big)
		& = \nabla_Z \Square_m \phi  
			\ + \ 2(\nabla^{\kappa} Z^{\lambda}) \nabla_{\lambda} \nabla_{\kappa} \phi \\
		& = \nabla_Z \Square_m \phi 
			\ + \ (\nabla^{\kappa} Z^{\lambda} + \nabla^{\lambda} Z^{\kappa}) \nabla_{\kappa} \nabla_{\lambda} \phi 
			 \notag \\
		& = \nabla_Z \Square_m \phi \ + \ c_Z \Square_m \phi \notag \\
		& \eqdef \nablamod_Z \Square_m \phi. \notag 
	\end{align}
	This proves \eqref{E:NablaModZLiemodMinkowskiWaveOperatorCommutator} in the case $|I| = 1.$ The general case
	now follows inductively.
	
\end{proof}

The next lemma shows that the modified Lie derivative $\Lie_{\mathcal{Z}}^I$ operator has favorable commutation properties with the linear Maxwell term $\nabla_{\mu} \Far^{\mu \nu} = \frac{1}{2}\big[(m^{-1})^{\mu \kappa} (m^{-1})^{\nu \lambda} - (m^{-1})^{\mu \lambda} (m^{-1})^{\nu \kappa}\big] \nabla_{\mu} \Far_{\kappa \lambda}.$

\begin{lemma} \label{L:LiemodZLiemodMaxwellCommutator} \textbf{(Commutation properties of $\Liemod_{\mathcal{Z}}^I$ with linear 
	Maxwell term)}
	Let $I$ be a $\mathcal{Z}-$multi-index, and let $\Far$ be a
	two-form. Let $\Liemod_{\mathcal{Z}}^I$ be the iterated modified Lie derivative from Definitions 
	\ref{D:ModifiedDerivatives} and \ref{D:iterated}. Then
	
	\begin{align} \label{E:LiemodZLiemodMaxwellCommutator}
		\Liemod_{\mathcal{Z}}^I \Big\lbrace & \big[(m^{-1})^{\mu \kappa} (m^{-1})^{\nu \lambda} - 
			(m^{-1})^{\mu \lambda} (m^{-1})^{\nu \kappa}\big] \nabla_{\mu} \Far_{\kappa \lambda} \Big\rbrace \\
		& = \big[(m^{-1})^{\mu \kappa} (m^{-1})^{\nu \lambda} 
			- (m^{-1})^{\mu \lambda} (m^{-1})^{\nu \kappa}\big] \nabla_{\mu} 
			\Lie_{\mathcal{Z}}^I \Far_{\kappa \lambda}. \notag
	\end{align}
\end{lemma}

\begin{proof}
	Let $Z \in \mathcal{Z}.$ By the Leibniz rule, \eqref{E:LieZonmupper}, and Lemma \ref{L:Liecommuteswithcoordinatederivatives}, 
	we have that 
	
	\begin{align}
		\Lie_Z \Big\lbrace & \big[(m^{-1})^{\mu \kappa} (m^{-1})^{\nu \lambda} - (m^{-1})^{\mu \lambda} 
			(m^{-1})^{\nu \kappa}\big] \nabla_{\mu} \Far_{\kappa \lambda} \Big\rbrace \\
		& = -2c_Z \big[(m^{-1})^{\mu \kappa} (m^{-1})^{\nu \lambda} - (m^{-1})^{\mu \lambda} 
			(m^{-1})^{\nu \kappa}\big] \nabla_{\mu} \Far_{\kappa \lambda} \notag \\
		& \ \ + \big[(m^{-1})^{\mu \kappa} (m^{-1})^{\nu \lambda} - (m^{-1})^{\mu \lambda} 
			(m^{-1})^{\nu \kappa}\big] \nabla_{\mu} \Lie_Z \Far_{\kappa \lambda}. \notag
	\end{align}
	It thus follows from Definition \ref{D:ModifiedDerivatives} that 
	
	\begin{align}
		\Liemod_{Z}\Big\lbrace & \big[(m^{-1})^{\mu \kappa} (m^{-1})^{\nu \lambda} - (m^{-1})^{\mu \lambda} 	
			(m^{-1})^{\nu \kappa}\big] \nabla_{\mu} \Far_{\kappa \lambda} \Big\rbrace \\
		& = \big[(m^{-1})^{\mu \kappa} (m^{-1})^{\nu \lambda} - (m^{-1})^{\mu \lambda} 
			(m^{-1})^{\nu \kappa}\big] \nabla_{\mu} \Lie_Z \Far_{\kappa \lambda}. \notag
	\end{align}
	This implies \eqref{E:LiemodZLiemodMaxwellCommutator} in the case $|I| = 1.$ The general case now follows inductively.
\end{proof}

The next lemma shows that some of the differential operators we have introduced commute with the null decomposition of a two-form.

\begin{lemma} \label{L:LieRotationcommuteswithnulldecomp} \textbf{(Differential operators that commute with the null 
	decomposition)}
	Let $\Far$ be a two-form and let $\ualpha,$ $\alpha,$ $\rho,$ and $\sigma$ be its null components.
	Let $O \in \mathcal{O}$ be any of the rotational Minkowskian Killing fields $\Omega_{jk},$ 
	$(1 \leq j < k \leq 3).$ Then $\Lie_O \ualpha[\Far] = \ualpha[\Lie_O \Far],$
	$\Lie_O \alpha[\Far] = \alpha[\Lie_O \Far],$ $\Lie_O \rho[\Far] = \rho[\Lie_O \Far],$ 
	and $\Lie_O \sigma[\Far] = \sigma[\Lie_O \Far].$ An analogous result
	holds the operators $\nabla_{\uL}$ and $\nabla_L;$ i.e.,
	$\Lie_O, \nabla_{\uL},$ and $\nabla_L$ commute with the null decomposition of $\Far.$
\end{lemma}

\begin{proof}
	Lemma \ref{L:LieRotationcommuteswithnulldecomp} follows from 
	Lemma \ref{L:PropertiesofuLandL}, \eqref{E:nablamis0}, and Lemma \ref{L:LieDerivativeCommuatorsVanish}.
\end{proof}

The next lemma shows that \emph{weighted} covariant derivatives can be estimated by covariant derivatives with respect to
vectorfields $Z \in \mathcal{Z}.$

\begin{lemma} \label{L:PointwisetandqWeightedNablainTermsofZestiamtes} \cite[Lemma 5.1]{hLiR2010}
	\textbf{(Weighted pointwise differential operator inequalities)}
	For any tensorfield $U$ and any two-tensor $\Pi,$ we have the following pointwise estimates:
	
	\begin{subequations}
	\begin{align}
		(1 + t + |q|) |\conenabla U| \ + \ (1 + |q|)|\nabla U| & \lesssim \sum_{|I| \leq 1} |\nabla_{\mathcal{Z}}^I U|, && 
			\label{E:WeightedDerivativesinTermsofNablaZI} \\
		|\conenabla^2 U| \ + \ r^{-1} |\conenabla U| & \lesssim r^{-1}(1 + t + |q|)^{-1}
			\sum_{|I| \leq 2} |\nabla_{\mathcal{Z}}^I U|, && |\conenabla^2 U| 
			\eqdef |\conenabla \conenabla U|, \\
		|\Pi^{\kappa \lambda} \nabla_{\kappa} \nabla_{\lambda} U| & \lesssim \Big\lbrace(1 + t + |q|)^{-1} |\Pi| 
			\ + \ (1 + |q|)^{-1} |\Pi|_{\mathcal{L} \mathcal{L}}\Big\rbrace \sum_{|I| \leq 1} |\nabla\nabla_{\mathcal{Z}}^I U|.
			&&
	\end{align}
	\end{subequations}
\end{lemma}

\hfill $\qed$

The next lemma shows that rotational Lie derivatives can be used to approximate weighted $S_{r,t}-$intrinsic
covariant derivatives. 

\begin{lemma} \label{L:RotationalLieDerivativesinTermsofrWeightedAngularDerivatives} \cite[Lemma 7.0.17]{jS2010a} 
		\textbf{(Weighted covariant derivatives approximated by rotational Lie derivatives)}
		Let $U$ be any tensorfield $m-$tangent to the spheres $S_{r,t}$ and $k \geq 0$ be any integer.
		Then with $r \eqdef |x|,$ we have that
		
		\begin{align} \label{E:RotationalLieDerivativesinTermsofrWeightedAngularDerivatives}
			\sum_{|I| \leq k} r^{|I|} |\angn^I U| \approx \sum_{|I| \leq k} |\Lie_{\mathcal{O}}^I U|.
		\end{align}
\end{lemma}

\begin{corollary} \label{C:rWeightedAngularDerivativesinTermsofLieDerivatives}
	Let $\Far$ be a two-form, and let $\ualpha[\Far],$ $\alpha[\Far],$ $\rho[\Far],$ $\sigma[\Far]$
	denote its null components. Then with $r = |x|,$ we have that
	
	\begin{align} \label{E:rWeightedAngularDerivativesinTermsofLieDerivatives}
		r |\angn \ualpha[\Far]| & \lesssim \sum_{|I| \leq 1} |\ualpha[\Lie_{\mathcal{Z}}^I \Far]|.
	\end{align}
	Furthermore, analogous inequalities hold for $\alpha[\Far],$ $\rho[\Far],$ and $\sigma[\Far].$

\end{corollary}

\begin{proof}
	Inequality \eqref{E:rWeightedAngularDerivativesinTermsofLieDerivatives} follows from 
	Lemma \ref{L:LieRotationcommuteswithnulldecomp} and 
	Lemma \ref{L:RotationalLieDerivativesinTermsofrWeightedAngularDerivatives}.
\end{proof}

Finally, the following proposition provides pointwise inequalities relating various 
Lie and covariant derivative operators under various contraction seminorms.

\begin{proposition} \label{P:LievsCovariantLContractionRelation}
	\textbf{(Lie derivative and covariant derivative inequalities)}
	Let $U$ be a tensorfield. Then
	
	\begin{align} \label{E:LieZIinTermsofNablaZI}
		\sum_{|I| \leq k} |\Lie_{\mathcal{Z}}^I U| \approx \sum_{|I| \leq k} |\nabla_{\mathcal{Z}}^I U|.
	\end{align}
	
	Furthermore, let $P$ be a symmetric or an anti-symmetric type $\binom{0}{2}$ tensorfield. Then the following inequalities 
	hold:
	
	\begin{subequations}
	\begin{align}
		\sum_{|I| \leq k} |\nabla\Lie_{\mathcal{Z}}^I P| 
		& \lesssim \sum_{|I| \leq k} |\nabla\nabla_{\mathcal{Z}}^I P|,
			\label{E:NablaLieZIinTermsofNablaNablaZI} \\
		\sum_{|I| \leq k} |\conenabla \Lie_{\mathcal{Z}}^I P| 
		& \lesssim \sum_{|I| \leq k} |\conenabla \nabla_{\mathcal{Z}}^I P|,
			\label{E:TangentialDerivativesLieZIvsTangentialDerivativesNablaZI}
	\end{align}
		
	\begin{align} 
		|\Lie_{\mathcal{Z}}^I P|_{\mathcal{L}\mathcal{L}}
		& \lesssim |\nabla_{\mathcal{Z}}^I P|_{\mathcal{L}\mathcal{L}}
			\ + \ \underbrace{\sum_{|J| \leq |I|-1} |\nabla_{\mathcal{Z}}^J P|_{\mathcal{L} \mathcal{T}}}_{\mbox{absent if $|I| =
			0$}}
			\ + \ \underbrace{\sum_{|J'| \leq |I|-2} |\nabla_{\mathcal{Z}}^{J'} P|}_{\mbox{absent if $|I| \leq 1$}}, 
			\label{E:LieZILLinTermsofNablaZILLLieZJLTPlusJunk} \\
		|\nabla \Lie_{\mathcal{Z}}^I P|_{\mathcal{L}\mathcal{L}}
		& \lesssim |\nabla \nabla_{\mathcal{Z}}^I P|_{\mathcal{L}\mathcal{L}}
			\ + \ \underbrace{\sum_{|J| \leq |I|-1} |\nabla_{\mathcal{Z}}^J P]|_{\mathcal{L} \mathcal{T}}}_{\mbox{absent if 
			$|I| = 0$}}
			\ + \ \underbrace{\sum_{|J'| \leq |I|-2} |\nabla\nabla_{\mathcal{Z}}^{J'} P|}_{\mbox{absent if $|I| \leq 1$}},
			\label{E:NablaLieZILLinTermsofNablaNablaZILLNablaLieZJLTPlusJunk}
	\end{align}

	\begin{align} \label{E:NablaFarGoodInTermsofqWeightedLieZIFarGood}
		|\nabla P|_{\mathcal{L} \mathcal{N}} 
		+ |\nabla P|_{\mathcal{T} \mathcal{T}}
		& \lesssim (1 + |q|)^{-1} \sum_{|I| \leq 1} \big(|\Lie_{\mathcal{Z}}^I P|_{\mathcal{L} \mathcal{N}} 
			+ |\Lie_{\mathcal{Z}}^I P|_{\mathcal{T} \mathcal{T}} \big)
			\ + \ (1 + t + |q|)^{-1} \sum_{|I| \leq 1} |\Lie_{\mathcal{Z}}^I P|.
	\end{align}
	\end{subequations}
	
\end{proposition}

\begin{proof}
	Inequality \eqref{E:LieZIinTermsofNablaZI} follows inductively using 
	\eqref{E:CovariantDerivativesofZareConstant} and \eqref{E:Liederivativeintermsofnabla}. 
	
	To prove the remaining inequalities, for each $Z \in \mathcal{Z},$ we define the contraction operator $\mathcal{C}_{Z}$ by
	
	\begin{align} \label{E:ContractionOperatorDef}
		(\mathcal{C}_{Z} P)_{\mu \nu} \eqdef P_{\kappa \nu}{^{(Z)}c_{\mu}^{\ \kappa}} \ + \ P_{\mu \kappa}{^{(Z)}c_{\nu}^{\ 
		\kappa}},
	\end{align}
	where the covariantly constant tensorfield ${^{(Z)}c_{\mu}^{\ \kappa}}$ is defined in 
	\eqref{E:CovariantDerivativesofZareConstant}. It follows from definition \eqref{E:ContractionOperatorDef} and 
	Lemma \ref{L:Liederivativeintermsofnabla} that
	
	\begin{align}
		\Lie_{\mathcal{Z}} P = \nabla_Z P \ + \ \mathcal{C}_{Z} P.
	\end{align}
	Since each $Z \in \mathcal{Z}$ is a conformal Killing field, and since 
	$L^{\mu}L^{\nu} m_{\mu \nu} = 0,$ it follows that
	$L^{\mu}L_{\nu}{^{(Z)} c_{\mu}^{\ \nu}} = 0.$ Also using the fact that each ${^{(Z)}c_{\mu}^{\ \nu}}$ is a constant, we have 
	that
	
	\begin{align}
		|\mathcal{C}_{Z} P|_{\mathcal{L} \mathcal{L}} & \lesssim |P|_{\mathcal{L} \mathcal{T}}
			\label{E:LowerOrderLieTermLLContractioninTermsofLTContraction}, \\
		|\mathcal{C}_{Z} P| & \lesssim |P|. \label{E:LowerOrderLieTermBounded}
	\end{align}

If $I = (\iota_1, \cdots, \iota_k)$ is a $\mathcal{Z}-$multi-index with $1 \leq |I| = k,$ then using the fact that the components ${^{(Z)}c_{\mu}^{\ \kappa}}$ are constants, we have that

\begin{align} \label{E:LieZIConformalKillingFieldNablaZIRelation}
	\Lie_{\mathcal{Z}}^I P \eqdef \Lie_{Z^{\iota_1}} \circ \cdots \circ \Lie_{Z^{\iota_k}} P 
	& = (\nabla_{Z^{\iota_1}} \ + \ \mathcal{C}_{Z^{\iota_1}}) \circ \cdots 
		\circ (\nabla_{Z^{\iota_k}} \ + \ \mathcal{C}_{Z^{\iota_k}}) P \\
	& = \nabla_{\mathcal{Z}}^I P + 
		\sum_{i=1}^{k} \mathcal{C}_{Z^{\iota_i}} \circ \nabla_{Z^{\iota_1}} \circ \cdots  	
		\circ \nabla_{Z^{\iota_{i-1}}} \circ \nabla_{Z^{\iota_{i+1}}} \circ \cdots \circ 
		\nabla_{Z^{\iota_k}} P \notag \\
	& \ \ + \ \underbrace{\mathop{\sum_{I_1 + I_2 = I}}_{|I|_2 \leq k - 2}
		\mathcal{C}_{\mathcal{Z}}^{I_1} \nabla_{\mathcal{Z}}^{I_2} P}_{\mbox{absent if $k = 1$}}. \notag
\end{align}
Inequality \eqref{E:NablaLieZIinTermsofNablaNablaZI} now follows from applying $\nabla$ to each side of \eqref{E:LieZIConformalKillingFieldNablaZIRelation}, from using the fact that the operator $\nabla$ commutes through the operators $\mathcal{C}_Z,$ and from \eqref{E:LowerOrderLieTermBounded}. Inequality \eqref{E:TangentialDerivativesLieZIvsTangentialDerivativesNablaZI} follows from similar reasoning.
Inequalities \eqref{E:LieZILLinTermsofNablaZILLLieZJLTPlusJunk} and \eqref{E:NablaLieZILLinTermsofNablaNablaZILLNablaLieZJLTPlusJunk} also follow from similar reasoning, together with \eqref{E:LowerOrderLieTermLLContractioninTermsofLTContraction}.

To prove \eqref{E:NablaFarGoodInTermsofqWeightedLieZIFarGood}, we first observe that
by \eqref{E:WeightedDerivativesinTermsofNablaZI}, \eqref{E:LieZIinTermsofNablaZI}, 
and \eqref{E:TangentialDerivativesLieZIvsTangentialDerivativesNablaZI}, we have that

\begin{align} \label{E:PrelminaryEstimateNablaFarGoodInTermsofqWeightedLieZIFarGood}
	|\nabla P|_{\mathcal{L} \mathcal{N}} 
		\ + \ |\nabla P|_{\mathcal{T} \mathcal{T}}
	& \lesssim |\nabla_{\uL} P|_{\mathcal{L} \mathcal{N}} 
		\ + \ |\nabla_{\uL} P|_{\mathcal{T} \mathcal{T}}
		\ + \ |\conenabla P| \\
	& \lesssim |\nabla_{\uL} P|_{\mathcal{L} \mathcal{N}} 
		\ + \ |\nabla_{\uL} P|_{\mathcal{T} \mathcal{T}} 
		\ + \ (1 + t + |q|)^{-1} \sum_{|I| \leq 1} |\Lie_{\mathcal{Z}}^I P|. \notag
\end{align}
Therefore, from \eqref{E:PrelminaryEstimateNablaFarGoodInTermsofqWeightedLieZIFarGood}, we see that to prove \eqref{E:NablaFarGoodInTermsofqWeightedLieZIFarGood}, it suffices to prove that the following inequality holds for any
symmetric or anti-symmetric type $\binom{0}{2}$ tensorfield $P:$

\begin{align} \label{E:uLDerivativeofFarGoodInTermsofqWeightedLieZIFarGood}
		|\nabla_{\uL} P|_{\mathcal{L} \mathcal{N}} + |\nabla_{\uL} P|_{\mathcal{T} \mathcal{T}}
			& \lesssim (1 + |q|)^{-1} \sum_{|I| \leq 1} \big(|\Lie_{\mathcal{Z}}^I P|_{\mathcal{L} \mathcal{N}} 
			\ + \ |\Lie_{\mathcal{Z}}^I P|_{\mathcal{T} \mathcal{T}} \big).
\end{align}
To this end, we use the vectorfields $S = x^{\kappa} \partial_{\kappa},$ $\Omega_{0j} = - t \partial_j - x_j \partial_t$ to decompose

\begin{align}
	\uL & = - q^{-1}(S + \omega^a \Omega_{0a}), \ \omega^a \eqdef x^a/r,
\end{align}
which implies that

\begin{align} \label{E:qNablauLFarinTermsofNablaZFar}
	- q \nabla_{\uL} P_{\mu \nu} = \nabla_S P_{\mu \nu} \ + \ \omega^a \nabla_{\Omega_{0a}} P_{\mu \nu}.
\end{align}

Using \eqref{E:ScalingCovariantDerivative}, \eqref{E:SpacetimeRotationsCovariantDerivative}, 
and \eqref{E:Liederivativeintermsofnabla}, we compute that

\begin{align}
	\nabla_S P_{\mu \nu} & = \Lie_{S} P_{\mu \nu} \ - \ 2 P_{\mu \nu}, \\
	\omega^a \nabla_{\Omega_{0a}} P_{\mu \nu}  
	& = \omega^a \Lie_{\Omega_{0a}} P_{\mu \nu}
		\ - \ \frac{1}{2} \Big\lbrace \uL_{\mu} L^{\kappa} P_{\kappa \nu} 
		- L_{\mu} \uL^{\kappa} P_{\kappa \nu} 
		+ \uL_{\nu} L^{\kappa} P_{\mu \kappa} 
		- L_{\nu} \uL^{\kappa} P_{\mu \kappa}\Big\rbrace.
\end{align}
Combining these two identities with \eqref{E:qNablauLFarinTermsofNablaZFar}, we conclude that

\begin{align} \label{E:qWeigthedNablauLFarinTermsofLieDerivatives}
	- q \nabla_{\uL} P_{\mu \nu} = \Lie_{S} P_{\mu \nu} \ + \ \omega^a \Lie_{\Omega_{0a}} P_{\mu \nu}
	- 2 P_{\mu \nu} \ - \ \frac{1}{2} \Big\lbrace \uL_{\mu} L^{\kappa} P_{\kappa \nu} 
		- L_{\mu} \uL^{\kappa} P_{\kappa \nu} 
		+ \uL_{\nu} L^{\kappa} P_{\mu \kappa} 
		- L_{\nu} \uL^{\kappa} P_{\mu \kappa}\Big\rbrace.
\end{align}
Contracting \eqref{E:qWeigthedNablauLFarinTermsofLieDerivatives} against the sets
$\mathcal{L}\mathcal{N}$ and $\mathcal{T}\mathcal{T},$ it follows that

\begin{align} \label{E:NablauLFarGoodComponentsinTermsofLieZIFarGoodComponentsLargeq}
	|q||\nabla_{\uL} P|_{\mathcal{L} \mathcal{N}} \ + \ |q||\nabla_{\uL} P|_{\mathcal{T} \mathcal{T}}
	& \lesssim \sum_{|I| \leq 1}
		\big(|\Lie_{\mathcal{Z}}^I P|_{\mathcal{L} \mathcal{N}} + |\Lie_{\mathcal{Z}}^I P|_{\mathcal{T} \mathcal{T}}\big).
\end{align}
Furthermore, by decomposing

\begin{align}
	\uL = \partial_t - \partial_r = \partial_t - \omega^a \partial_a,
\end{align}
and using the fact that ${^{(\frac{\partial}{\partial t})}c_{\mu}^{\ \nu}} = 
{^{(\frac{\partial}{\partial x^j})}c_{\mu}^{\ \nu}} = 0$ 
(where ${^{(Z)}c}_{\mu \nu}$ is defined in \eqref{E:CovariantDerivativesofZareConstant}), it follows that

\begin{align} \label{E:NablauLFarinTermsofLieDerivativesSmallq}
	\nabla_{\uL} P_{\mu \nu} = \Lie_{\frac{\partial}{\partial t}} P_{\mu \nu} 
	\ - \ \omega^a \Lie_{\frac{\partial}{\partial x^a}}P_{\mu \nu}.
\end{align}
Contracting \eqref{E:NablauLFarinTermsofLieDerivativesSmallq} against the sets
$\mathcal{L}\mathcal{N}$ and $\mathcal{T}\mathcal{T},$ we have that

\begin{align} \label{E:NablauLFarGoodComponentsinTermsofLieZIFarGoodComponentsSmallq}
	|\nabla_{\uL} P|_{\mathcal{L} \mathcal{N}} + |\nabla_{\uL} P|_{\mathcal{T} \mathcal{T}}
	& \lesssim 	\sum_{|I| = 1}
		\big(|\Lie_{\mathcal{Z}}^I P|_{\mathcal{L} \mathcal{N}} + |\Lie_{\mathcal{Z}}^I P|_{\mathcal{T} \mathcal{T}} \big).
\end{align}
Adding \eqref{E:NablauLFarGoodComponentsinTermsofLieZIFarGoodComponentsLargeq} and \eqref{E:NablauLFarGoodComponentsinTermsofLieZIFarGoodComponentsSmallq}, we arrive at inequality
\eqref{E:uLDerivativeofFarGoodInTermsofqWeightedLieZIFarGood}. This completes our proof of \eqref{E:NablaFarGoodInTermsofqWeightedLieZIFarGood}.

\end{proof}

\section{The Reduced Equation Satisfied by \texorpdfstring{$\nabla_{\mathcal{Z}}^I h^{(1)}$}{the Derivatives of the Metric Remainder Piece}} \label{S:EquationSatisfiedbyNablaZIh1}
In this short section, we assume that $h_{\mu \nu}^{(1)}$ is a solution to the reduced equation
\eqref{E:Reducedh1Summary}. We provide a proposition that gives a preliminary description of the inhomogeneities
in the equation satisfied by $\nabla_{\mathcal{Z}}^I h_{\mu \nu}^{(1)}.$

\noindent \hrulefill
\ \\

\begin{proposition} \label{P:InhomogeneousTermsNablaZIh1}
\textbf{(Inhomogeneities for $\nabla_{\mathcal{Z}}^I h_{\mu \nu}^{(1)}$)}
Suppose that $h_{\mu \nu}^{(1)}$ is a solution to the reduced equation \eqref{E:Reducedh1Summary}, and 
let $I$ be any $\mathcal{Z}-$multi-index. Then $\nabla_{\mathcal{Z}}^I h_{\mu \nu}^{(1)}$ is a solution to the inhomogeneous system

\begin{align}
	\widetilde{\Square}_{g} \nabla_{\mathcal{Z}}^I h_{\mu \nu}^{(1)} & = \mathfrak{H}_{\mu \nu}^{(1;I)}, 
		\label{E:InhomogeneousTermsNablaZIh1} \\
	\mathfrak{H}_{\mu \nu}^{(1;I)} & = \nablamod_{\mathcal{Z}}^I \mathfrak{H}_{\mu \nu} 
		- \nablamod_{\mathcal{Z}}^I \widetilde{\Square} h_{\mu \nu}^{(0)} 	
		- \Big\lbrace \nablamod_{\mathcal{Z}}^I \widetilde{\Square}_{g} h_{\mu \nu}^{(1)} - \widetilde{\Square}_{g} 
		\nabla_{\mathcal{Z}}^I h_{\mu \nu}^{(1)} \Big\rbrace  \label{E:mathfrakH1Idef} \\
	& = \nablamod_{\mathcal{Z}}^I \mathfrak{H}_{\mu \nu} 
		- \nablamod_{\mathcal{Z}}^I \widetilde{\Square} h_{\mu \nu}^{(0)} 	
		- \Big\lbrace \nablamod_{\mathcal{Z}}^I \big(H^{\kappa \lambda} \nabla_{\kappa} \nabla_{\lambda} h_{\mu \nu}^{(1)}\big) 
		- H^{\kappa \lambda} \nabla_{\kappa} \nabla_{\lambda}
		\nabla_{\mathcal{Z}}^I h_{\mu \nu}^{(1)} \Big\rbrace. \notag	
\end{align}

\end{proposition}

\begin{proof}
	Proposition \ref{P:InhomogeneousTermsNablaZIh1} follows from differentiating each side of \eqref{E:Reducedh1Summary}
	with modified covariant derivatives $\nablamod_{\mathcal{Z}}^I$ and applying Lemma 
	\ref{L:NablaModZLiemodMinkowskiWaveOperatorCommutator}.
\end{proof}

\section{The Equations of Variation, the Canonical Stress, and Electromagnetic Energy Currents} \label{E:EOVandStress}
In this section, we introduce the electromagnetic equations of variation, which are linearized versions of the 
reduced electromagnetic equations. The significance of the equations of variation is the following: if $\Far$ is a solution to the reduced electromagnetic equations \eqref{E:ReduceddFis0Summary} - \eqref{E:ReduceddMis0Summary}, then $\Lie_{\mathcal{Z}}^I \Far$ is a solution to the equations of variation. We then provide a preliminary description of the structure of the inhomogeneous terms in the equations of variation satisfied by $\Lie_{\mathcal{Z}}^I \Far.$ Additionally, we introduce the canonical stress tensorfield and use it to construct energy currents, which are vectorfields that will be used in Section \ref{S:WeightedEnergy} to derive weighted energy estimates for solutions to the equations of variation.

\noindent \hrulefill
\ \\

\subsection{Equations of variation}

The equations of variation in the unknowns $\dot{\Far}_{\mu \nu}$ are the linearization of \eqref{E:ReduceddFis0Summary} - \eqref{E:ReduceddMis0Summary} around a background $(h_{\mu \nu}, \Far_{\mu \nu}).$ More specifically, the equations of variation are the
system

\begin{subequations} 
\begin{align}
	\nabla_{\lambda} \dot{\Far}_{\mu \nu} + \nabla_{\mu} \dot{\Far}_{\nu \lambda} + \nabla_{\nu} \dot{\Far}_{\lambda \mu}
		& = \dot{\mathfrak{F}}_{\lambda \mu \nu},&& (\lambda, \mu, \nu = 0,1,2,3), \label{E:EOVdFis0} \\
	N^{\# \mu \nu \kappa \lambda} \nabla_{\mu} \dot{\Far}_{\kappa \lambda} & = \dot{\mathfrak{F}}^{\nu},&& (\nu = 0,1,2,3), 
		\label{E:EOVdMis0} 
\end{align}
\end{subequations}
where $N^{\# \mu \nu \kappa \lambda}$ is the $(h_{\mu \nu}, \Far_{\mu \nu})-$dependent tensorfield defined in \eqref{E:NSummarydef}, and $\dot{\mathfrak{F}}_{\lambda \mu \nu},$ $\dot{\mathfrak{F}}^{\nu}$ are inhomogeneous terms that need to be specified. In this article, the equations of variation will arise when we differentiate the reduced equations \eqref{E:ReduceddFis0Summary} - \eqref{E:ReduceddMis0Summary} with modified Lie derivatives. In this case, the quantities 
$\Lie_{\mathcal{Z}}^I \Far_{\mu \nu}$ will play the role of $\dot{\Far}.$ The next proposition, which is a companion of Proposition \ref{P:InhomogeneousTermsNablaZIh1}, provides a preliminary expression of the inhomogeneous terms that arise in the study of the equations of variation satisfied by $\Lie_{\mathcal{Z}}^I \Far_{\mu \nu}.$

\begin{proposition} \label{P:InhomogeneoustermsLieZIFar}
	\textbf{(Inhomogeneities for $\Lie_{\mathcal{Z}}^I \Far_{\mu \nu}$)}
	If $\Far_{\mu \nu}$ is a solution to the reduced electromagnetic equations 
	\eqref{E:ReduceddFis0Summary} - \eqref{E:ReduceddMis0Summary} and $I$ is a
	$\mathcal{Z}-$multi-index, then $\dot{\Far}_{\mu \nu} \eqdef \Lie_{\mathcal{Z}}^I \Far_{\mu \nu}$ is a solution to the 
	equations of variation \eqref{E:EOVdFis0} - \eqref{E:EOVdMis0} (corresponding to the background 
	$(h_{\mu \nu}, \Far_{\mu \nu})$) with inhomogeneous terms $\dot{\mathfrak{F}}_{\lambda \mu \nu} \eqdef
	\mathfrak{F}_{\lambda \mu \nu}^{(I)}$ and $\dot{\mathfrak{F}}^{\nu} \eqdef \mathfrak{F}_{(I)}^{\nu},$ where
	
	\begin{subequations}
	\begin{align}
		\mathfrak{F}_{\lambda \mu \nu}^{(I)} & = 0, && (\lambda, \mu, \nu = 0,1,2,3), \label{E:EOVInhomogeneousTermsJvanish} \\
		\mathfrak{F}_{(I)}^{\nu} & = \Liemod_{\mathcal{Z}}^I \mathfrak{F}^{\nu}
			+ \Big\lbrace N^{\# \mu \nu \kappa \lambda}\nabla_{\mu} \Lie_{\mathcal{Z}}^I\Far_{\kappa \lambda}
			- \Liemod_{\mathcal{Z}}^I \big(N^{\# \mu \nu \kappa \lambda}\nabla_{\mu}\Far_{\kappa \lambda}\big)\Big\rbrace,
			&& (\nu = 0,1,2,3).
			\label{E:LiemodZIdifferentiatedEOVInhomogeneousterms}
	\end{align}
	\end{subequations}
	
	Furthermore, there exist constants $\widetilde{C}_{1;I_1,I_2}, \widetilde{C}_{2;I_1,I_2},
	\widetilde{C}_{\mathscr{P};I_1,I_2}, \widetilde{C}_{\mathfrak{F}_{\triangle};J}, \widetilde{C}_{\dParameter_{\triangle};I_1,I_2}$ such 
	that
	
	\begin{subequations}
	\begin{align}		
		\Liemod_{\mathcal{Z}}^I \mathfrak{F}^{\nu} 
		& = \sum_{|I_1| + |I_2| \leq |I|} \widetilde{C}_{2;I_1,I_2}
				\mathscr{Q}_{(2;\Far)}^{\nu}(\nabla\Lie_{\mathcal{Z}}^{I_1} h, \Lie_{\mathcal{Z}}^{I_2} \Far) 
			 \label{E:LiemodZIFExpanded} \\
		& \ \ + \sum_{|J| \leq |I|} \widetilde{C}_{\mathfrak{F}_{\triangle};J} 
			\Lie_{\mathcal{Z}}^J \mathfrak{F}_{\triangle}^{\nu}, \notag \\
		N^{\# \mu \nu \kappa \lambda}\nabla_{\mu} \Lie_{\mathcal{Z}}^I \Far_{\kappa \lambda}
		- \Liemod_{\mathcal{Z}}^I \big(N^{\# \mu \nu \kappa \lambda}\nabla_{\mu}\Far_{\kappa \lambda}\big) 
		& = \mathop{\sum_{|I_1| + |I_2| \leq |I|}}_{|I_2| \leq |I| - 1} \widetilde{C}_{\mathscr{P};I_1,I_2}
			\mathscr{P}_{(\Far)}^{\nu}(\Lie_{\mathcal{Z}}^{I_1} h, \nabla \Lie_{\mathcal{Z}}^{I_2}\Far) 
		 	\label{E:LiemodZINnablaFarCommutatorTerms} \\
		& \ \ + \mathop{\sum_{|I_1| + |I_2| \leq |I|}}_{|I_2| \leq |I| - 1} \widetilde{C}_{1;I_1,I_2}
			\mathscr{Q}_{(1;\Far)}^{\nu}(\Lie_{\mathcal{Z}}^{I_1} h, \nabla \Lie_{\mathcal{Z}}^{I_2}\Far)
			\notag \\
		& \ \ + \mathop{\sum_{|I_1| + |I_2| \leq |I|}}_{|I_2| \leq |I| - 1} \widetilde{C}_{\dParameter_{\triangle};I_1,I_2}
			(\Lie_{\mathcal{Z}}^{I_1} N_{\triangle}^{\# \mu \nu \kappa \lambda}) 
			\nabla_{\mu} \Lie_{\mathcal{Z}}^{I_2} \Far_{\kappa \lambda}. \notag
	\end{align}
	\end{subequations}
	In the above formulas, $\mathfrak{F}_{\triangle}^{\nu}$ and $N_{\triangle}^{\# \mu \nu \kappa \lambda}$ are the error terms
	appearing in \eqref{E:FtriangleSmallAlgebraic} and \eqref{E:NtriangleSmallAlgebraic} respectively, 
	while $\mathscr{P}_{(\Far)}^{\nu}(\cdot, \cdot)$ and $\mathscr{Q}_{(i;\Far)}^{\nu}(\cdot, \cdot),$ $(i=1,2),$  
	$(\nu = 0,1,2,3),$ are the quadratic forms defined in \eqref{E:PFar}, \eqref{E:Q1Far}, and \eqref{E:Q2Far} respectively. 
\end{proposition}

\begin{proof}
	
	To prove \eqref{E:EOVInhomogeneousTermsJvanish}, we first recall equation \eqref{E:ReduceddFis0Summary}, which states 
	that $\Far_{\mu \nu}$ is a solution to $\nabla_{[\kappa} \Far_{\mu \nu]} = 0,$ where $[\cdots]$ denotes anti-symmetrization.
	Using \eqref{E:Liecommuteswithcoordinatederivatives} it therefore follows that
	
	\begin{align}
		0 = \Lie_{\mathcal{Z}}^I \nabla_{[\lambda} \Far_{\mu \nu]} = \nabla_{[\lambda} \Lie_{\mathcal{Z}}^I \Far_{\mu \nu]},
	\end{align}
	which is the desired result. 
	
	To derive \eqref{E:LiemodZIdifferentiatedEOVInhomogeneousterms}, we simply differentiate each side of
	\eqref{E:EOVdMis0} with $\Liemod_{\mathcal{Z}}^I$ to conclude that 
	$\Liemod_{\mathcal{Z}}^I \big( N^{\# \mu \nu \kappa \lambda} 
	\nabla_{\mu} \Far_{\kappa \lambda} \big) = \Liemod_{\mathcal{Z}}^I \mathfrak{F}^{\nu}.$ Trivial algebraic manipulation 
	then leads to the fact that $N^{\# \mu \nu \kappa \lambda} 
	\nabla_{\mu} \Lie_{\mathcal{Z}}^I\Far_{\kappa \lambda} = \mathfrak{F}_{(I)}^{\nu},$ where $\mathfrak{F}_{(I)}^{\nu}$ 
	is defined by \eqref{E:LiemodZIdifferentiatedEOVInhomogeneousterms}.
	
 	Equation \eqref{E:LiemodZIFExpanded} follows from \eqref{E:EMBIFarInhomogeneous},
 	the Definition \ref{D:ModifiedDerivatives} of $\Liemod_Z,$ and Lemma \ref{L:nullformvectorfieldcommutation}, which is proved
 	in Section \ref{S:UsefulLemmas}.
	
	To prove \eqref{E:LiemodZINnablaFarCommutatorTerms}, we first recall equation \eqref{E:NNullFormDecomposition}:
	
	\begin{align} \label{E:NNullFormDecompositionAgain}
		N^{\# \mu \nu \kappa \lambda} \nabla_{\mu} \Far_{\kappa \lambda} 
		& = \frac{1}{2} \big[ (m^{-1})^{\mu \kappa} (m^{-1})^{\nu \lambda} - (m^{-1})^{\mu \lambda} (m^{-1})^{\nu \kappa} 
			\big] \nabla_{\mu} \Far_{\kappa \lambda} 
			\ - \ \mathscr{P}_{(\Far)}^{\nu}(h, \nabla \Far) 
			\ - \ \mathscr{Q}_{(1;\Far)}^{\nu}(h, \nabla \Far)
			\ + \ N_{\triangle}^{\# \mu \nu \kappa \lambda} \nabla_{\mu} \Far_{\kappa \lambda}.
	\end{align}	
	The commutator term arising from the $\big[ (m^{-1})^{\mu \kappa} (m^{-1})^{\nu \lambda} 
		- (m^{-1})^{\mu \lambda} (m^{-1})^{\nu \kappa} \big] \nabla_{\mu} \Far_{\kappa \lambda}$ term on the right-hand side of 
	\eqref{E:NNullFormDecompositionAgain} vanishes. More specifically, we use \eqref{E:LiemodZLiemodMaxwellCommutator} to conclude that
	
	\begin{align} \label{E:LieMinkowskiCommutatorTermVanishes}
			\big[(m^{-1})^{\mu \kappa} (m^{-1})^{\nu \lambda} - (m^{-1})^{\mu \lambda} (m^{-1})^{\nu \kappa} \big]  
				\nabla_{\mu} \Lie_{\mathcal{Z}}^I \Far_{\kappa \lambda}
			\ - \ \Liemod_{\mathcal{Z}}^I \Big\lbrace 
				\big[(m^{-1})^{\mu \kappa} (m^{-1})^{\nu \lambda} 
				- (m^{-1})^{\mu \lambda} (m^{-1})^{\nu \kappa}\big] \nabla_{\mu}\Far_{\kappa \lambda} \Big\rbrace = 0.
	\end{align}
	Therefore, it follows from \eqref{E:NNullFormDecompositionAgain} and \eqref{E:LieMinkowskiCommutatorTermVanishes} that
	
		\begin{align}		
			N^{\# \mu \nu \kappa \lambda}\nabla_{\mu} \Lie_{\mathcal{Z}}^I \Far_{\kappa \lambda}
			\ - \ \Liemod_{\mathcal{Z}}^I \big(N^{\# \mu \nu \kappa \lambda}\nabla_{\mu}\Far_{\kappa \lambda}\big) 
			& = \Liemod_{\mathcal{Z}}^I \mathscr{P}_{(\Far)}^{\nu}(h, \nabla \Far)
			\ - \ \mathscr{P}_{(\Far)}^{\nu}(h, \nabla\Lie_{\mathcal{Z}}^I\Far)
			\label{E:ProofLiemodZINnablaFarCommutatorTerms} \\
		& \ \ + \ \Liemod_{\mathcal{Z}}^I \mathscr{Q}_{(1;\Far)}^{\nu}(h, \nabla \Far)
			\ - \ \mathscr{Q}_{(1;\Far)}^{\nu}(h, \nabla\Lie_{\mathcal{Z}}^I\Far) \notag \\
		& \ \ + \ N_{\triangle}^{\# \mu \nu \kappa \lambda} \nabla_{\mu} \Lie_{\mathcal{Z}}^I 
			\Far_{\kappa \lambda} \ - \ \Liemod_{\mathcal{Z}}^I (N_{\triangle}^{\# \mu \nu \kappa \lambda} 
			\nabla_{\mu} \Far_{\kappa \lambda}). \notag
		\end{align}
		The expression \eqref{E:LiemodZINnablaFarCommutatorTerms} now follows from \eqref{E:ProofLiemodZINnablaFarCommutatorTerms},
		the Leibniz rule, the Definition \ref{D:ModifiedDerivatives} of $\Liemod_Z,$ 
		Lemma \ref{L:Liecommuteswithcoordinatederivatives}, and Lemma \ref{L:nullformvectorfieldcommutation}.

\end{proof}

\subsection{The canonical stress} \label{SS:CanonicalStress}

The notion of the \emph{canonical stress tensorfield} $\Stress_{\ \nu}^{\mu}$ in the context of PDE energy estimates
was introduced by Christodoulou in \cite{dC2000}. As explained in Section \ref{SSS:EnergyandStress}, from the point of view of energy estimates, it plays the role of an energy-momentum-type tensor for the equations of variation. Its two key properties are i) its divergence is lower-order (in the sense of the number of derivatives falling on the variations $\dot{\Far}_{\mu \nu}$); and ii) contraction against a suitable (covector, vector) pair $(\xi_{\mu}, X^{\nu})$ leads to a positive energy density that can be used achieve $L^2$ control of solutions $\dot{\Far}_{\mu \nu}$ to the equations of variation. 
As we will see, property i) is captured by Lemma \ref{L:DivergenceofStress} and \eqref{E:currentdivergence}, while property ii)
is captured by \eqref{E:dotJ0estimate}, \eqref{E:FirstweightedenergyFar}, and \eqref{E:SecondweightedenergyFar}. In order to understand the origin of the canonical stress, we first introduce Christodoulou's \emph{linearized Lagrangian} \cite{dC2000}.

\begin{definition}

Given an electromagnetic Lagrangian $\mathscr{L}[\cdot]$ (as described in Section \ref{SS:Lagrangianformluationofnonlinearelectromagnetism}) and a ``background'' $(h_{\mu \nu}, \Far_{\mu \nu}),$ we define the linearized Lagrangian by

\begin{align} \label{E:LinearizedLagrangian}
	\dot{\mathscr{L}} = \dot{\mathscr{L}}[\dot{\Far};h,\Far] 
		\eqdef \frac{1}{2} \frac{\partial^2 \Ldual[h,\Far]}{\partial \Far_{\zeta \eta} \partial \Far_{\kappa 
		\lambda}} \dot{\Far}_{\zeta \eta} \dot{\Far}_{\kappa \lambda} 
		= - \frac{1}{4}N^{\#\zeta \eta \kappa \lambda} \dot{\Far}_{\zeta \eta} \dot{\Far}_{\kappa \lambda},
\end{align}
where $N^{\#\zeta \eta \kappa \lambda}$ is the $(h_{\mu \nu}, \Far_{\mu \nu})-$dependent tensorfield defined in \eqref{E:firstNdef}.

\end{definition}

The merit of the above definition is the following: the principal part (from the point of view of number of derivatives) of the Euler-Lagrange equations (assuming that we view $(h,\Far)$ as a background, $\dot{\Far}$ to be the unknowns, and that 
an appropriately defined action\footnote{A suitable action $\mathcal{A}_{\mathfrak{C}}[\dot{\Far}]$ is e.g. of the form  $\mathcal{A}_{\mathfrak{C}}[\dot{\Far}] \eqdef \int_{\mathfrak{C} \Subset \mathfrak{M}} \dot{\mathscr{L}}[\dot{\Far};h,\Far] \, d^4x,$ where $\mathfrak{C}$ is a compact subset of spacetime.} is stationary with respect to closed variations of $\dot{\Far}$) corresponding to $\dot{\mathscr{L}}[\dot{\Far};h,\Far]$ is identical to the principal part of the electromagnetic equations of variation \eqref{E:EOVdMis0}; i.e., $\dot{\mathscr{L}}[\dot{\Far};h,\Far]$ generates the linearized equations.

\begin{definition}
Given a linearized Lagrangian $\dot{\mathscr{L}}[\dot{\Far};h,\Far],$ the canonical stress tensorfield $\Stress_{\ \nu}^{\mu}$ is defined as follows:

\begin{align} \label{E:Stressdef}
	\Stress_{\ \nu}^{\mu} \eqdef - 2\frac{\partial \dot{\mathscr{L}}}{\partial \dot{\Far}_{\mu \zeta}}\dot{\Far}_{\nu \zeta}
		\ + \ \delta_{\nu}^{\mu} \dot{\mathscr{L}} 
		=	N^{\# \mu \zeta \kappa \lambda} \dot{\Far}_{\kappa \lambda} \dot{\Far}_{\nu \zeta}
			\ - \ \frac{1}{4} \delta_{\nu}^{\mu} N^{\#\zeta \eta \kappa \lambda} \dot{\Far}_{\zeta \eta} \dot{\Far}_{\kappa \lambda},
\end{align}
where $N^{\# \mu \nu \kappa \lambda}$ is defined in \eqref{E:firstNdef}.
\end{definition}
\noindent Note that in contrast to the energy-momentum tensor $T_{\mu \nu},$ 
$\Stress_{\mu \nu} \eqdef m_{\mu \kappa} \Stress_{\ \nu}^{\kappa}$ is in general not symmetric.

Because of our assumption \eqref{E:Ldualassumptions} concerning the Lagrangian, $\Stress_{\ \nu}^{\mu}$ is equal to the 
energy-momentum tensor (in $\dot{\Far}$) for the linear Maxwell-Maxwell equations in Minkowski space, plus small corrections. More specifically, it follows from definition \ref{E:Stressdef} and the decomposition \eqref{E:NSummarydef} that

\begin{align} \label{E:StressExpansion}
	\Stress_{\ \nu}^{\mu} & = \overbrace{\dot{\Far}^{\mu \zeta} \dot{\Far}_{\nu \zeta} 
		\ - \ \frac{1}{4} \delta_{\nu}^{\mu} 
		\dot{\Far}_{\zeta \eta}\dot{\Far}^{\zeta \eta}}^{\mbox{terms from linear Maxwell-Maxwell equations 
		in Minkowski spacetime}} \\
		& \ \ \overbrace{- \ h^{\mu \kappa} \dot{\Far}_{\kappa \zeta} \dot{\Far}_{\nu}^{\ \zeta}
			\ - \ h^{\kappa \lambda} \dot{\Far}_{\ \kappa}^{\mu} \dot{\Far}_{\nu \lambda}
			\ + \ \frac{1}{2} \delta_{\nu}^{\mu} h^{\kappa \lambda} \dot{\Far}_{\kappa \eta}\dot{\Far}_{\lambda}^{\ \eta}}^{\mbox{corrections to Minkowskian 
			linear Maxwell-Maxwell equations arising from $h$}}
			\notag \\
		& \ \ + \ \underbrace{N_{\triangle}^{\# \mu \zeta \kappa \lambda} \dot{\Far}_{\kappa \lambda} \dot{\Far}_{\nu \zeta}
			\ - \ \frac{1}{4} \delta_{\nu}^{\mu} N_{\triangle}^{\#\zeta \eta \kappa \lambda} \dot{\Far}_{\zeta \eta} 
			\dot{\Far}_{\kappa \lambda}.}_{\mbox{error terms}} \notag
\end{align}

The next lemma captures the lower-order divergence property enjoyed by $\Stress_{\ \nu}^{\mu}.$ 

\begin{lemma} \label{L:DivergenceofStress} \textbf{(Divergence of the canonical stress)}
Let $\dot{\Far}_{\mu \nu}$ be a solution to the equations of variation \eqref{E:EOVdFis0} - \eqref{E:EOVdMis0}
corresponding to the background $(h_{\mu \nu}, \Far_{\mu \nu}),$ and let $\dot{\mathfrak{F}}_{\lambda \mu \nu},$ $\dot{\mathfrak{F}}^{\nu}$ be the inhomogeneous terms from the right-hand sides of \eqref{E:EOVdFis0} - \eqref{E:EOVdMis0}. Let $\Stress_{\ \nu}^{\mu}$
be the canonical stress tensorfield defined in \eqref{E:Stressdef}. Then

\begin{align} \label{E:divergenceofStress}
	\nabla_{\mu} \Stress_{\ \nu}^{\mu} & = - \ \frac{1}{2} N^{\#\zeta \eta \kappa \lambda} \dot{\Far}_{\zeta \eta} 
		\dot{\mathfrak{F}}_{\nu \kappa \lambda}
		\ + \ \dot{\Far}_{\nu \eta} \dot{\mathfrak{F}}^{\eta}
		\ + \ (\nabla_{\mu}N^{\# \mu \zeta \kappa \lambda}) \dot{\Far}_{\kappa \lambda} \dot{\Far}_{\nu \zeta}
		\ - \ \frac{1}{4} (\nabla_{\nu}N^{\#\zeta \eta \kappa \lambda}) \dot{\Far}_{\zeta \eta} \dot{\Far}_{\kappa \lambda}, \\
	& = \ - \frac{1}{2} N^{\#\zeta \eta \kappa \lambda} \dot{\Far}_{\zeta \eta} \dot{\mathfrak{F}}_{\nu \kappa \lambda}
		\ + \ \dot{\Far}_{\nu \eta} \dot{\mathfrak{F}}^{\eta} \notag \\
	& \ \ - \ (\nabla_{\mu} h^{\mu \kappa}) \dot{\Far}_{\kappa \zeta} \dot{\Far}_{\nu}^{\ \zeta}
			\ - \ (\nabla_{\mu} h^{\kappa \lambda}) \dot{\Far}_{\ \kappa}^{\mu} \dot{\Far}_{\nu \lambda}
			\ + \ \frac{1}{2} (\nabla_{\nu} h^{\kappa \lambda}) \dot{\Far}_{\kappa \eta}\dot{\Far}_{\lambda}^{\ \eta}
		\notag \\
	& \ \ + \ (\nabla_{\mu}N_{\triangle}^{\# \mu \zeta \kappa \lambda}) \dot{\Far}_{\kappa \lambda} \dot{\Far}_{\nu \zeta}
		\ - \ \frac{1}{4} (\nabla_{\nu}N_{\triangle}^{\#\zeta \eta \kappa \lambda}) \dot{\Far}_{\zeta \eta} \dot{\Far}_{\kappa 
		\lambda}. \notag
\end{align}

\end{lemma}

\begin{proof}
	To obtain \eqref{E:divergenceofStress}, we use use the equations \eqref{E:EOVdFis0} - \eqref{E:EOVdMis0}, 
	together with the properties \eqref{E:Nminussignproperty1} - \eqref{E:Nsymmetryproperty}, which are also satisfied by
	the tensorfield $N_{\triangle}^{\# \mu \zeta \kappa \lambda}.$
\end{proof}

\subsection{Electromagnetic energy currents}

In this section, we introduce the energy current that will be used to derive the 
weighted energy estimate \eqref{E:FirstweightedenergyFar} for a solution $\dot{\Far}$ to the equations
of variation \eqref{E:EOVdFis0} - \eqref{E:EOVdMis0}.

\begin{definition} \label{D:Jdotdef}
	Let $h_{\mu \nu}$ be a symmetric type $\binom{0}{2}$ tensorfield, and let
	$\Far_{\mu \nu},$ $\dot{\Far}_{\mu \nu}$ be a pair of two-forms. Let $w(q)$ be the weight defined in \eqref{E:weight},
	and let $X^{\nu} \eqdef w(q)\delta_0^{\nu}$ be the ``multiplier'' vectorfield.
	We define the \emph{energy current} $\dot{J}_{(h,\Far)}^{\mu}[\dot{\Far}]$ corresponding to the variation 
	$\dot{\Far}_{\mu \nu}$ and the background $(h_{\mu \nu}, \Far_{\mu \nu})$ to be the vectorfield
	
	\begin{align} \label{E:Jdotdef}
		\dot{J}_{(h,\Far)}^{\mu}[\dot{\Far}] & \eqdef - \Stress_{\ \nu}^{\mu} X^{\nu} = - w(q) \Stress_{\ 0}^{\mu},
	\end{align}
	where $\Stress_{\ \nu}^{\mu}$ is the canonical stress tensorfield from \eqref{E:Stressdef}.
\end{definition}

\begin{lemma} \label{L:currentproperties} \textbf{(Positivity of $\dot{J}_{(h,\Far)}^0$)}
	Let $\dot{J}_{(h,\Far)}^{\mu}[\dot{\Far}]$ be the energy current defined in \eqref{E:Jdotdef}. Then
	
	\begin{align} \label{E:dotJ0estimate}
		\dot{J}_{(h,\Far)}^0 & = \frac{1}{2} |\dot{\Far}|^2 w(q) 
		\ + \ \Big\lbrace O^{\infty}(|h|;\Far) + O^{\dParameter}\big(|(h,\Far)|^2 \big) \Big\rbrace|\dot{\Far}|^2 w(q).
	\end{align}
	
	Furthermore, if $\dot{\Far}_{\mu \nu}$ is a solution to the equations of variation \eqref{E:EOVdFis0} - \eqref{E:EOVdMis0}
	with inhomogeneous terms $\dot{\mathfrak{F}}_{\lambda \mu \nu} \equiv 0,$
	then the Minkowskian divergence of $\dot{J}_{(h,\Far)}$ can be expressed as follows:
	
	\begin{align} \label{E:currentdivergence}
		\nabla_{\mu} \dot{J}_{(h,\Far)}^{\mu} & = - \ \frac{1}{2} w'(q) (\dot{\alpha}^2 + \dot{\rho}^2 + \dot{\sigma}^2) \\
		& \ \ - \ w(q) \Big \lbrace \dot{\Far}_{0 \eta} \dot{\mathfrak{F}}^{\eta} 
			- (\nabla_{\mu} h^{\mu \kappa}) \dot{\Far}_{\kappa \zeta} \dot{\Far}_{0}^{\ \zeta}
			- (\nabla_{\mu} h^{\kappa \lambda}) \dot{\Far}_{\ \kappa}^{\mu} \dot{\Far}_{0 \lambda}
			+ \frac{1}{2} (\nabla_{t} h^{\kappa \lambda}) \dot{\Far}_{\kappa \eta}\dot{\Far}_{\lambda}^{\ \eta} \Big \rbrace \notag  \\
		& \ \ - \ w'(q) \Big\lbrace - L_{\mu}h^{\mu \kappa} \dot{\Far}_{\kappa \zeta} \dot{\Far}_{0}^{\ \zeta} 
			- L_{\mu} h^{\kappa \lambda} \dot{\Far}_{\ \kappa}^{\mu} \dot{\Far}_{0 \lambda} 
			- \frac{1}{2} h^{\kappa \lambda} \dot{\Far}_{\kappa \eta} \dot{\Far}_{\lambda}^{\ \eta} \Big\rbrace \notag \\
		& \ \ - \ w(q) \Big \lbrace (\nabla_{\mu}N_{\triangle}^{\# \mu \zeta \kappa \lambda}) \dot{\Far}_{\kappa \lambda} 
			\dot{\Far}_{0 \zeta} - \frac{1}{4} (\nabla_{t} N_{\triangle}^{\#\zeta \eta \kappa \lambda}) \dot{\Far}_{\zeta \eta} 
			\dot{\Far}_{\kappa \lambda} \Big \rbrace \notag \\ 
		& \ \ - \ w'(q) \Big\lbrace L_{\mu} N_{\triangle}^{\# \mu \zeta \kappa \lambda} 
			\dot{\Far}_{\kappa \lambda} \dot{\Far}_{0 \zeta} + \frac{1}{4} N_{\triangle}^{\#\zeta \eta \kappa \lambda} 
			\dot{\Far}_{\zeta \eta} \dot{\Far}_{\kappa \lambda} \Big\rbrace,  \notag
	\end{align}
	where $\dot{\alpha} \eqdef \alpha[\dot{\Far}],$ $\dot{\rho} \eqdef \rho[\dot{\Far}],$ 
	and $\dot{\sigma} \eqdef \sigma[\dot{\Far}]$ are the ``favorable'' Minkowskian null components of $\dot{\Far}$
	defined in Section \ref{SS:NullComponents}.
	
\end{lemma}

\begin{remark}
	The term $\frac{1}{2} w'(q) (\dot{\alpha}^2 + \dot{\rho}^2 + \dot{\sigma}^2)$ appearing on the right-hand side of
	of \eqref{E:currentdivergence} is of central importance for closing the bootstrap argument during
	our global existence proof. It manifests itself as the additional positive space-time integral
	$\int_{0}^{t} \int_{\Sigma_{\tau}} \big(|\dot{\Far}|_{\mathcal{L} \mathcal{N}}^2 
	+ |\dot{\Far}|_{\mathcal{T} \mathcal{T}}^2 \big) w'(q) \,d^3x \, d \tau$ on the left-hand side of 
	\eqref{E:FirstweightedenergyFar} below, and provides a means
	for controlling some of the spacetime integrals that emerge in Section \ref{SS:MainTheoremFarInhomogeneities}. 
\end{remark}

\begin{proof}
	\eqref{E:dotJ0estimate} follows from \eqref{E:StressExpansion}, simple calculations, and \eqref{E:NtriangleSmallAlgebraic}.
	
	To prove \eqref{E:currentdivergence}, we first recall that since $q = r - t,$ it follows that $\nabla_{\mu} q = L_{\mu},$
	where $L$ is defined in \eqref{E:Ldef}. Hence, we have that $\nabla_{\mu} w(q) = w'(q) L_{\mu}.$ 
	Using this fact, \eqref{E:StressExpansion}, and \eqref{E:divergenceofStress}, we calculate that
	
	\begin{align}
		\nabla_{\mu} \dot{J}_{(h,\Far)}^{\mu} & = - \ w(q) \dot{\Far}_{0 \eta} \dot{\mathfrak{F}}^{\eta} \\
		& \ \ - \ w(q) \Big \lbrace - (\nabla_{\mu} h^{\mu \kappa}) \dot{\Far}_{\kappa \zeta} \dot{\Far}_{0}^{\ \zeta}
			- (\nabla_{\mu} h^{\kappa \lambda}) \dot{\Far}_{\ \kappa}^{\mu} \dot{\Far}_{0 \lambda}
			+ \frac{1}{2} (\nabla_{t} h^{\kappa \lambda}) \dot{\Far}_{\kappa \eta}\dot{\Far}_{\lambda}^{\ \eta}
			\Big \rbrace \notag \\
		& \ \ - \ w(q) \Big \lbrace (\nabla_{\mu}N_{\triangle}^{\# \mu \zeta \kappa \lambda}) \dot{\Far}_{\kappa \lambda} 
			\dot{\Far}_{0 \zeta} - \frac{1}{4} (\nabla_{t} N_{\triangle}^{\#\zeta \eta \kappa \lambda}) \dot{\Far}_{\zeta \eta} 
			\dot{\Far}_{\kappa \lambda} \Big \rbrace \notag \\
		& \ \ - \ w'(q) \Big\lbrace \underbrace{L_{\mu} \dot{\Far}^{\mu \zeta} \dot{\Far}_{0 \zeta} 
			+ \frac{1}{4} \dot{\Far}_{\kappa \lambda} \dot{\Far}^{\kappa \lambda}}_{\frac{1}{2}(\dot{\alpha}^2 + \dot{\rho}^2 + 
				\dot{\sigma}^2)} \Big \rbrace \notag \\
		& \ \ - \ w'(q) \Big\lbrace - L_{\mu}h^{\mu \kappa} \dot{\Far}_{\kappa \zeta} \dot{\Far}_{0}^{\ \zeta} 
			- L_{\mu} h^{\kappa \lambda} \dot{\Far}_{\ \kappa}^{\mu} \dot{\Far}_{0 \lambda} 
			- \frac{1}{2} h^{\kappa \lambda} \dot{\Far}_{\kappa \eta} \dot{\Far}_{\lambda}^{\ \eta} \Big\rbrace \notag \\
		& \ \ - \ w'(q) \Big\lbrace L_{\mu} N_{\triangle}^{\# \mu \zeta \kappa \lambda} \dot{\Far}_{\kappa \lambda} \dot{\Far}_{0 \zeta}
			+ \frac{1}{4} N_{\triangle}^{\#\zeta \eta \kappa \lambda} \dot{\Far}_{\zeta \eta} 
			\dot{\Far}_{\kappa \lambda} \Big\rbrace.  \notag \\
	\end{align}
	The expression \eqref{E:currentdivergence} thus follows.
\end{proof}

\section{Decompositions of the Electromagnetic Equations} \label{S:DecompositionsofElectromagneticEquations}

In this section we perform two decompositions of the electromagnetic equations. The first is a null decomposition of the equations of variation, which will be used in Section \ref{S:DecayFortheReducedEquations} to derive pointwise decay estimates for the lower-order Lie derivatives of $\Far_{\mu \nu}.$ The second is a decomposition of the electromagnetic equations into constraint and evolution equations for the Minkowskian one-forms $\Electricfield_{\mu},$ $\Magneticinduction_{\mu},$ which are 
respectively known as the electric field and magnetic induction. This decomposition will be used in Section \ref{S:SmallDataAssumptions} to prove that our smallness condition on the abstract data necessarily implies a smallness condition on the initial energy $\mathcal{E}_{\dParameter;\upgamma;\upmu}(0)$ of the corresponding solution to the reduced equations. We remark that the Minkowskian one-forms $\Displacement_{\mu},$ $\Magneticfield_{\mu},$ which are respectively known as the electric displacement and the magnetic field, and also the geometric electromagnetic one-forms $\mathfrak{\Electricfield}_{\mu},$ $\mathfrak{\Magneticinduction}_{\mu},$ $\mathfrak{\Displacement}_{\mu},$ $\mathfrak{\Magneticfield}_{\mu}$ will play a role in the discussion.

\noindent \hrulefill
\ \\

\subsection{The Minkowskian null decomposition of the electromagnetic equations of variation} \label{SS:NullDecompElectromagnetic} 
In this section, we decompose the equations of variation into equations for the null components of $\dot{\Far}.$ The main advantage of our decomposition, which is given in Proposition \ref{P:EOVNullDecomposition}, is the following: the terms in each equation can be separated into two classes: i) a derivative of a null component in a ``nearly-Minkowski-null'' direction\footnote{By ``nearly-Minkowski-null,'' we mean vectors that are nearly parallel to $\uL$ or $L,$ with some corrections coming from the presence of a non-zero $h$ in the case of the null component $\dot{\ualpha}.$} (which appears on the left-hand side of the inequality); and ii) some error terms (which appear on the right-hand side of the inequality). Although from the point of view of differentiability the error terms are not lower-order, it will turn out that they are lower-order in terms of decay rates. In this way, the equations can be viewed as \emph{ordinary differential equations} for the null components of $\dot{\Far}$ with inhomogeneous terms; this point of view is fully realized in Proposition \ref{P:ODEsNullComponentsLieZIFar}. The key point is that the ODEs we derive will be amenable to Gronwall estimates: in Section \ref{S:DecayFortheReducedEquations}, we will use this line of argument to derive pointwise decay estimates for the null components of the lower-order Lie derivatives of a solution $\Far$ to the electromagnetic equations \eqref{E:ReduceddFis0Summary} - \eqref{E:ReduceddMis0Summary}. These estimates will be an improvement over what can be deduced from the weighted Klainerman-Sobolev inequality \eqref{E:PhiKlainermanSobolev} alone.

We begin the analysis by using \eqref{E:NSummarydef} to write the equations of variation \eqref{E:EOVdFis0} - \eqref{E:EOVdMis0}
in the following form:

\begin{subequations} 
\begin{align}
	\nabla_{\lambda} \dot{\Far}_{\mu \nu} + \nabla_{\mu} \dot{\Far}_{\nu \lambda} + \nabla_{\nu} \dot{\Far}_{\lambda \mu}
		= 0, \label{E:EOVdFis0nullsection} 
\end{align}
\begin{align}
	\bigg\lbrace \frac{1}{2} \big[(m^{-1})^{\mu \kappa} (m^{-1})^{\nu \lambda} - (m^{-1})^{\mu \lambda} (m^{-1})^{\nu \kappa}
		& - h^{\mu \kappa} (m^{-1})^{\nu \lambda} + h^{\mu \lambda} (m^{-1})^{\nu \kappa} 
		\label{E:EOVdMis0nullsection} \\
	& - (m^{-1})^{\mu \kappa} h^{\nu \lambda} + (m^{-1})^{\mu \lambda} h^{\nu \kappa} \big]
		+ N_{\triangle}^{\# \mu \nu \kappa \lambda} \bigg\rbrace \nabla_{\mu} \dot{\Far}_{\kappa \lambda} 
		= \dot{\mathfrak{F}}^{\nu}. \notag  
\end{align}
\end{subequations}

In our calculations below, we will make use of the identities

\begin{align} \label{E:nablaALnablaAuL}
	\nabla_A \uL = -r^{-1} e_A, && \nabla_A L = r^{-1} e_A, 
\end{align}
which can be directly calculated in our wave coordinate system using \eqref{E:uLdef} - \eqref{E:Ldef}. We will also make use
of the identity 

\begin{align} \label{E:nablaAeBcovriantintrintermsofinsicnablaAeBcovraint}
	\angn_A e_B & = \nabla_A e_B \ + \ \frac{1}{2} m(\nabla_A e_B, \uL) L \ + \ \frac{1}{2} m(\nabla_A e_B, L) \uL \\
	& =  \nabla_A e_B \ - \ \frac{1}{2} m(e_B, \nabla_A \uL) L \ - \ \frac{1}{2} m(e_B, \nabla_A L)\uL \notag \\
	& = \nabla_A e_B \ + \ \frac{1}{2}r^{-1} \delta_{AB} (L - \uL), \notag
\end{align}
which follows from \eqref{E:SphereIntrinsicintermsofExtrinsic} and \eqref{E:nablaALnablaAuL}.

Furthermore, if $U$ is a type $\binom{0}{m}$ tensorfield, and $X_{(i)},$ $(1 \leq i \leq m),$ and $Y$ are vectorfields, then 
by the Leibniz rule, we have that

\begin{align} \label{E:nablaLeibnizrule}
	\nabla_Y \big\lbrace U(X_{(1)}, \cdots, X_{(m)}) \big\rbrace 
	& = (\nabla_Y U)(X_{(1)} \cdots, X_{(m)}) 
		\ + \ U(\nabla_Y X_{(1)}, X_{(2)}, \cdots, X_{(m)}) \ + \ \cdots \ + \ U(X_{(1)}, X_{(2)}, \cdots, \nabla_Y X_{(m)}).
\end{align}
Similarly, if $U$ is $m-$tangent to the spheres $S_{r,t},$ then

\begin{align} \label{E:instrinsicnablaLiebnizrule}
	\angn_A (U_{B_{(1)}, \cdots, B_{(m)}}) & \eqdef \angn_{e_A} \big\lbrace U(e_{B_{(1)}}, \cdots, e_{B_{(m)}}) \big\rbrace \\
	& = (\angn_A U)(e_{B_{(1)}}, \cdots, e_{B_{(m)}}) \ + \ U(\angn_A e_{B_{(1)}}, e_{B_{(2)}}, \cdots, e_{B_{(m)}}) \notag \\
	& \ \ + \ \cdots \ + \ U(e_{B_{(1)}}, e_{B_{(2)}}, \cdots, \angn_A e_{B_{(m)}}). \notag
\end{align}

Applying \eqref{E:nablaLeibnizrule} and \eqref{E:instrinsicnablaLiebnizrule} to $\Far,$ and using \eqref{E:nablaALnablaAuL}, \eqref{E:nablaAeBcovriantintrintermsofinsicnablaAeBcovraint}, and \eqref{E:Fardualalpha} - \eqref{E:Fardualsigma},
we compute (as in \cite[pg. 161]{dCsK1990}) the following identities, which we state as a lemma.

\begin{lemma} \label{L:FarDerivativeNullComponentRelations} \textbf{(Contracted derivatives expressed in terms of the null components)}
Let $\Far$ be a two-form, and let $\ualpha,$ $\alpha,$ $\rho,$ and $\sigma$ be its null components. Then the following identities hold:

\begin{subequations}
\begin{align}
	\nabla_A \Far_{B \uL} & = \angn_A \ualpha_B \ - \ r^{-1}(\rho \delta_{AB} + \sigma \angupsilon_{AB}), 
		\label{E:nablaAFarBuLintermsofintrinsic} \\
	\nabla_A \Far_{BL} & = \angn_A \alpha_B \ - \ r^{-1}(\rho \delta_{AB} - \sigma \angupsilon_{AB}), \\
	\nabla_A \FarMinkdual_{B \uL} & = - \ \angupsilon_{CB}\angn_A \ualpha_C 
		\ - \ r^{-1}(\sigma \delta_{AB} - \rho \angupsilon_{AB}), \\
	\nabla_A \FarMinkdual_{BL} & = \angupsilon_{CB}\angn_A \alpha_C \ - \ r^{-1}(\sigma \delta_{AB} 
		\ + \ \rho \angupsilon_{AB}), \\
	\frac{1}{2} \nabla_A \Far_{\uL L} & = \angn_A \rho \ + \ \frac{1}{2} r^{-1}(\ualpha_A + \alpha_A), \\
	\frac{1}{2} \nabla_A \FarMinkdual_{\uL L} & = \angn_A \sigma \ + \ \frac{1}{2} r^{-1}(-\angupsilon_{BA}\ualpha_B 
		\ + \ \angupsilon_{BA} \alpha_B), \\
	\nabla_A \Far_{BC} & = \angupsilon_{BC} \Big\lbrace \angn_A \sigma + \frac{1}{2}r^{-1}
		(-\angupsilon_{DA}\ualpha_D + \angupsilon_{DA} \alpha_D) \Big\rbrace. \label{E:nablaAFarBCintermsofintrinsic}
\end{align}
\end{subequations}
Note that in all of our expressions, contractions are taken after differentiating; e.g., 
$\nabla_A \Far_{BL} \eqdef e_A^{\mu} e_B^{\kappa} \uL^{\lambda} \nabla_{\mu} \Far_{\kappa \lambda}.$

\end{lemma}

\begin{remark} \label{R:FarDerivativeNullComponentRelations}
	
	The identities in Lemma \ref{L:FarDerivativeNullComponentRelations} can be 
	reinterpreted as identities for spacetime tensors that are $m-$tangent to the spheres $S_{r,t}.$ That is,
	they can be rephrased in terms of our wave coordinate frame with the help of the projection $\angm_{\mu}^{\ \nu}$ 
	and the spherical volume form $\angupsilon_{\mu}^{\ \nu}$ defined in \eqref{E:angmdef} and \eqref{E:Spheresvolumedef} 
	respectively. For example, equation \eqref{E:nablaAFarBuLintermsofintrinsic} is equivalent 
	to the following equation:
	
	\begin{align}
		\angm_{\mu}^{\ \mu'} \angm_{\nu}^{\ \nu'} \uL^{\kappa} \nabla_{\mu}\Far_{\nu' \kappa}
		& = \angm_{\mu}^{\ \mu'} \angm_{\nu}^{\ \nu'} \nabla_{\mu'} \ualpha_{\nu'} 
			\ - \ r^{-1}(\rho \angm_{\mu \nu} \ + \ \sigma \angupsilon_{\mu \nu}).
	\end{align}
	We will use the spacetime coordinate frame version of the identities in our proof of Proposition \ref{P:EOVNullDecomposition}.
\end{remark}

We now derive equations for the null components of a solution $\dot{\Far}$ to \eqref{E:EOVdFis0nullsection} - \eqref{E:EOVdMis0nullsection}.

\begin{proposition} \label{P:EOVNullDecomposition}
\textbf{(Minkowskian null decomposition of the equations of variation)}
Let $\dot{\Far}$ be a solution to the equations of variation \eqref{E:EOVdFis0nullsection} - \eqref{E:EOVdMis0nullsection}, and let $\dot{\ualpha} \eqdef \ualpha[\dot{\Far}],$ $\dot{\alpha} \eqdef \alpha[\dot{\Far}],$ $\dot{\rho} \eqdef \rho[\dot{\Far}],$ $\dot{\sigma} \eqdef \sigma[\dot{\Far}]$ denote its Minkowskian null components.
Assume that the source term $\dot{\mathfrak{F}}_{\lambda \mu \nu}$ on the right-hand side of \eqref{E:EOVdFis0nullsection} vanishes.\footnote{By Proposition \ref{P:InhomogeneoustermsLieZIFar}, this assumption holds for the variations of interest in this article.} Then the following equations are verified by the null components:

\begin{subequations}
\begin{align}
	& \nabla_L \dot{\ualpha}_{\nu} \ + \ r^{-1}\dot{\ualpha}_{\nu} \ + \ \angm_{\nu}^{\ \kappa} \nabla_{\kappa} \dot{\rho} 
		\ - \ \angupsilon_{\nu}^{\ \kappa} \nabla_{\kappa} \dot{\sigma} && \label{E:dotualphaEOVnulldecomp} \\
	& \ \ - \ \overbrace{\angm_{\nu}^{\ \lambda} h^{\mu \kappa} \nabla_{\mu} \dot{\Far}_{\kappa \lambda}}^{\eqdef \angm_{\nu 
			\lambda} \mathscr{P}_{(\Far)}^{\lambda}(h, \nabla \dot{\Far})}
		\ - \ \overbrace{\angm_{\nu \nu'} (m^{-1})^{\mu \kappa} h^{\nu' \lambda} \nabla_{\mu} \dot{\Far}_{\kappa 
			\lambda}}^{\eqdef \angm_{\nu \lambda} \mathscr{Q}_{(1;\Far)}^{\lambda}(h, \nabla \dot{\Far})} 
		\ + \ \angm_{\nu \nu'} N_{\triangle}^{\# \mu \nu' \kappa \lambda} \nabla_{\mu} \dot{\Far}_{\kappa \lambda}
		&& = \ \angm_{\nu \nu'} \dot{\mathfrak{F}}^{\nu'}, \notag \\
	& \nabla_{\uL} \dot{\alpha}_{\nu} \ - \ r^{-1}\dot{\alpha}_{\nu} \ - \ \angm_{\nu}^{\ \kappa} \nabla_{\kappa} \dot{\rho} 
		\ - \ \angupsilon_{\nu}^{\ \kappa} \nabla_{\kappa} \dot{\sigma}&& \label{E:uLdotalphaEOVnulldecomp} \\
	& \ \ - \ \overbrace{\angm_{\nu}^{\ \lambda} h^{\mu \kappa} \nabla_{\mu} \dot{\Far}_{\kappa \lambda}}^{\eqdef \angm_{\nu 
		\lambda} \mathscr{P}_{(\Far)}^{\lambda}(h, \nabla \dot{\Far})}
		\ - \ \overbrace{\angm_{\nu \nu'} (m^{-1})^{\mu \kappa} h^{\nu' \lambda} \nabla_{\mu} \dot{\Far}_{\kappa \lambda}}^{\eqdef 
			\angm_{\nu \lambda} \mathscr{Q}_{(1;\Far)}^{\lambda}(h, \nabla \dot{\Far})} 
		\ + \ \angm_{\nu \nu'} N_{\triangle}^{\# \mu \nu' \kappa \lambda} \nabla_{\mu} \dot{\Far}_{\kappa \lambda}
		&& = \ \angm_{\nu \nu'} \dot{\mathfrak{F}}^{\nu'}, \notag  \\
	& \nabla_{\uL} \dot{\rho} \ + \ \angm^{\mu \nu} \nabla_{\mu} \dot{\ualpha}_{\nu} \ - \ 2 r^{-1} \dot{\rho}
		\ - \ \overbrace{\uL^{\lambda} h^{\mu \kappa} \nabla_{\mu} \dot{\Far}_{\kappa \lambda}}^{
		\uL_{\lambda} \mathscr{P}_{(\Far)}^{\lambda}(h, \nabla \dot{\Far})} && \label{E:uLdotnablarhoEOVnulldecomp} \\
	& \ \ - \ \overbrace{\uL_{\nu} (m^{-1})^{\mu \kappa} h^{\nu \lambda} \nabla_{\mu} \dot{\Far}_{\kappa \lambda}}^{
		\uL_{\nu} \mathscr{Q}_{(1;\Far)}^{\nu}(h, \nabla \dot{\Far})} 
		\ + \ \uL_{\nu} N_{\triangle}^{\# \mu \nu \kappa \lambda} \nabla_{\mu} \dot{\Far}_{\kappa \lambda}
		&& = \ \uL_{\lambda} \dot{\mathfrak{F}}^{\lambda}, \notag \\
	& \nabla_{\uL} \dot{\sigma} \ - \ 2 r^{-1} \dot{\sigma} 
		\ + \ \angupsilon^{\mu \nu} \nabla_{\mu} \dot{\ualpha}_{\nu} && = \ 0, \label{E:uLdotnablasigmaEOVnulldecomp} \\
	& \nabla_L \dot{\rho} \ - \ \angm^{\mu \nu} \nabla_{\mu} \dot{\alpha}_{\nu} \ + \ 2 r^{-1} \dot{\rho}&&
		\label{E:nablaLdotrhoEOVnulldecomp} \\
	& \ \ + \ \overbrace{L^{\lambda} h^{\mu \kappa} \nabla_{\mu} \dot{\Far}_{\kappa \lambda}}^{L_{\lambda} 	
		\mathscr{P}_{(\Far)}^{\lambda}(h, \nabla \dot{\Far})}
		\ + \ \overbrace{L_{\nu} (m^{-1})^{\mu \kappa} h^{\nu \lambda} \nabla_{\mu} \dot{\Far}_{\kappa \lambda}}^{
			L_{\nu} \mathscr{Q}_{(1;\Far)}^{\nu}(h, \nabla \dot{\Far})}  
		\ - \ L_{\nu} N_{\triangle}^{\# \mu \nu \kappa \lambda} \nabla_{\mu} \dot{\Far}_{\kappa \lambda}
		&& = \ - L_{\lambda} \dot{\mathfrak{F}}^{\lambda}, \notag \\
	& \nabla_L \dot{\sigma} \ + \ 2 r^{-1} \dot{\sigma} \ + \ \angupsilon^{\mu \nu} \nabla_{\mu} \dot{\ualpha}_{\nu} && = \ 0. 
		\label{E:nablaLdotsigmaEOVnulldecomp}
\end{align}
In the above expressions, the quadratic terms $\mathscr{P}_{(\Far)}^{\lambda}(h, \nabla \dot{\Far})$ 
and $\mathscr{Q}_{(1;\Far)}^{\lambda}(h, \nabla \dot{\Far})$ are as defined in Section \ref{SS:ReducedEquations}.
\end{subequations}

\begin{remark} \label{R:FavorableAngular}
Note that in the above equations, we have that e.g. $\angm_{\nu}^{\ \kappa} \nabla_{\kappa} = \angm_{\nu}^{\ \kappa} \angn_{\kappa}$ and $\angupsilon_{\nu}^{\ \kappa} \nabla_{\kappa} = \angupsilon_{\nu}^{\ \kappa} \angn_{\kappa},$ so that these operators only involve favorable angular derivatives.
\end{remark}

\end{proposition}

\begin{proof}
	To obtain \eqref{E:dotualphaEOVnulldecomp} and \eqref{E:uLdotalphaEOVnulldecomp}, we contract \eqref{E:EOVdFis0nullsection} 
	against $\uL^{\lambda} L^{\mu} e_A^{\nu},$ \eqref{E:EOVdMis0nullsection} against $(e_A)_{\nu},$
	and use Lemma \ref{L:FarDerivativeNullComponentRelations} plus Remark \ref{R:FarDerivativeNullComponentRelations}
	to deduce that
	
	\begin{align}
		& \nabla_L \ualpha_{\nu} \ - \ \nabla_{\uL} \alpha_{\nu} \ + \ 2 \angm_{\nu}^{\ \nu'} \nabla_{\nu'} \rho 
			\ + \ r^{-1}(\ualpha_{\nu} + \alpha_{\nu}) && = \ 0, \\
		& \nabla_L \ualpha_{\nu} \ + \ \nabla_{\uL} \alpha_{\nu} \ - \ 2 \angupsilon_{\nu}^{\ \kappa} \nabla_{\kappa} \sigma
			\ + \ r^{-1}(\ualpha_{\nu} - \alpha_{\nu}) &&  \\
		& \ \ - \ 2 \angm_{\nu}^{\ \lambda} h^{\mu \kappa} \nabla_{\mu} \dot{\Far}_{\kappa \lambda} 
			\ - \ 2 \angm_{\nu \nu'} (m^{-1})^{\mu \kappa} h^{\nu' \lambda} \nabla_{\mu} \dot{\Far}_{\kappa \lambda}
			\ + \ \angm_{\nu \nu'} N_{\triangle}^{\# \mu \nu' \kappa \lambda} \nabla_{\mu} \dot{\Far}_{\kappa \lambda}
			&& = \ 2 \angm_{\nu \nu'} \dot{\mathfrak{F}}^{\nu'}. \notag
	\end{align}
	Adding the two above equations gives \eqref{E:dotualphaEOVnulldecomp}, while subtracting the first from the second
	gives \eqref{E:uLdotalphaEOVnulldecomp}.
	
	Similarly, to deduce \eqref{E:uLdotnablasigmaEOVnulldecomp},
	we contract \eqref{E:EOVdFis0nullsection} 
	against $\uL^{\lambda} e_A^{\ \mu} e_B^{\ \nu},$
	and then contract against $\angupsilon_{AB};$ to deduce \eqref{E:nablaLdotsigmaEOVnulldecomp}, we contract
	\eqref{E:EOVdFis0nullsection} against $L^{\lambda} e_A^{\ \mu} e_B^{\ \nu},$
	and then against $\angupsilon_{AB};$ to deduce \eqref{E:uLdotnablarhoEOVnulldecomp}, we contract 
	\eqref{E:EOVdMis0nullsection} against $\uL_{\nu};$
	and to deduce \eqref{E:nablaLdotrhoEOVnulldecomp}, we contract
	\eqref{E:EOVdMis0nullsection} against $-L_{\nu}.$
	
\end{proof}

\subsection{Electromagnetic one-forms} \label{SS:EBDH}
In this section, we introduce the one-forms $\mathfrak{\Electricfield},$ $\mathfrak{\Magneticinduction},$ 
$\mathfrak{\Displacement},$ and $\mathfrak{\Magneticfield},$ which are derived from a geometric decomposition of
$\Far$ with the help of the spacetime metric $g_{\mu \nu}.$ We also introduce the one-forms $\Electricfield,$ $\Magneticinduction,$ $\Displacement,$ and $\Magneticfield,$ which are derived from a Minkowskian decomposition of $\Far.$ We then derive an equivalent version of the electromagnetic equations, namely constraint and electromagnetic evolution equations for the Minkowskian one-forms. These quantities play a role only in Section \ref{S:SmallDataAssumptions}, where they are used to connect the smallness of the abstract initial data to the smallness of the energy of the corresponding reduced solution at $t=0.$ Furthermore, we show that the abstract one-forms $\mathring{\mathfrak{\Displacement}}, \mathring{\mathfrak{\Magneticinduction}},$ satisfy the constraints \eqref{E:DivergenceD0Intro} - \eqref{E:DivergenceB0Intro}
if and only if the corresponding Minkowskian one-forms $\mathring{\Displacement}, \mathring{\Magneticinduction},$ satisfy a Minkowskian version of the constraints.

We will perform our electromagnetic decompositions of the equations with the help of two versions of the electromagnetic equations, namely \eqref{E:dFis0Dversion}, \eqref{E:dMis0Dversion} and \eqref{E:dFis0nablaversion}, \eqref{E:dMis0nablaversion}. We restate them here for convenience:

\begin{subequations}
	\begin{align}
		\mathscr{D}_{\lambda} \Far_{\mu \nu} + \mathscr{D}_{\mu} \Far_{\nu \lambda} + \mathscr{D}_{\nu} \Far_{\lambda \mu} & = 0,&&
			(\lambda, \mu, \nu = 0,1,2,3), \label{E:dFis0ElectrodecompMathscrD} \\
		\mathscr{D}_{\lambda} \Max_{\mu \nu} + \mathscr{D}_{\mu} \Max_{\nu \lambda} + \mathscr{D}_{\nu} \Max_{\lambda \mu} & = 0,&&
			(\lambda, \mu, \nu = 0,1,2,3), \label{E:dMis0ElectrodecompMathscrD}
	\end{align}
\end{subequations}
and

\begin{subequations}
	\begin{align}
		\nabla_{\lambda} \Far_{\mu \nu} + \nabla_{\mu} \Far_{\nu \lambda} + \nabla_{\nu} \Far_{\lambda \mu} & = 0,&&
			(\lambda, \mu, \nu = 0,1,2,3), \label{E:dFis0Electrodecomp} \\
		\nabla_{\lambda} \Max_{\mu \nu} + \nabla_{\mu} \Max_{\nu \lambda} + \nabla_{\nu} \Max_{\lambda \mu} & = 0,&&
			(\lambda, \mu, \nu = 0,1,2,3). \label{E:dMis0Electrodecomp}
	\end{align}
\end{subequations}

Before decomposing the equations, we first define the aforementioned geometric electromagnetic one-forms.

\begin{definition} \label{D:IntrinsicOneForms}
	Let $\hat{N}^{\mu} = \hat{N}^{\mu}(t,x)$ denote the future-directed unit $g-$normal to the hypersurface $\Sigma_t.$ 
	Then in components relative to an arbitrary coordinate system, we define the following one-forms:

	\begin{align} \label{E:AbstractEBDHinertialcomponents}
	\mathfrak{\Electricfield}_{\mu} & = \Far_{\mu \kappa}\hat{N}^{\kappa}, & \mathfrak{\Magneticinduction}_{\mu} & = 
		- \Fardual_{\mu \kappa}\hat{N}^{\kappa}, & \mathfrak{\Displacement}_{\mu} & = - 	
	\Maxdual_{\mu \kappa} \hat{N}^{\kappa},  	& \mathfrak{\Magneticfield}_{\mu} & = - \Max_{\mu \kappa}\hat{N}^{\kappa}.
		\end{align}
	\textbf{Note that in the above expressions, $\star$ denotes the Hodge duality operator corresponding to the spacetime metric 
	$g.$}
\end{definition}

We now define the Minkowskian electromagnetic one-forms.

\begin{definition} \label{D:EBDHinertialcomponents}
In components relative to the wave coordinate coordinate system $\lbrace x^{\mu} \rbrace_{\mu = 0,1,2,3},$ we define the \emph{electric field} $\Electricfield,$ the \emph{magnetic induction} $\Magneticinduction,$ the \emph{electric displacement} $\Displacement,$ and the \emph{magnetic field} $\Magneticfield$ by

\begin{align} \label{E:EBDHinertialcomponents}
	\Electricfield_{\mu} & = \Far_{\mu 0}, & \Magneticinduction_{\mu} & = - \FarMinkdual_{\mu 0}, & \Displacement_{\mu} & = - 	
	\MaxMinkdual_{\mu 0},  	& \Magneticfield_{\mu} & = - \Max_{\mu 0}.
\end{align}
	\textbf{Note that in the above expressions, $\ostar$ denotes the Hodge duality operator corresponding to the 
	Minkowski metric $m.$}
\end{definition}

Observe that \eqref{E:EBDHinertialcomponents} implies that

\begin{align} \label{E:FarspatialintermsofB}
	\Far_{jk} & = [ijk]\Magneticinduction_i,&& \Magneticinduction_j = \frac{1}{2} [jab] \Far_{ab},&& 
	\Displacement_j = \frac{1}{2} [jab] \Max_{ab}, && (j,k = 1,2,3).
\end{align}

\begin{remark}
	Our definition of $\Magneticinduction$ coincides with the one commonly found in the physics literature, but it has the 
	opposite sign convention of the definition given in \cite{dCsK1990}.
\end{remark}

It follows from the anti-symmetry of $\Far_{\mu \nu}$ and $\Max_{\mu \nu}$ that $\Electricfield_{\mu},$ $\Magneticinduction_{\mu},$ $\Displacement_{\mu},$ and $\Magneticfield_{\mu}$ are $m-$tangent to the hyperplanes $\Sigma_t;$ i.e., we have that $\Electricfield_0 = \Magneticinduction_0 = \Displacement_0 = \Magneticfield_0 = 0.$ We may therefore view these four quantities as one-forms that are intrinsic to $\Sigma_t.$ Similarly, we have that
$\mathfrak{\Electricfield}_{\mu}\hat{N}^{\mu} = \mathfrak{\Magneticinduction}_{\mu}\hat{N}^{\mu}=
\mathfrak{\Displacement}_{\mu}\hat{N}^{\mu}  = \mathfrak{\Magneticfield}_{\mu}\hat{N}^{\mu} = 0.$

Using Definition \ref{D:EBDHinertialcomponents}, the assumption \eqref{E:Ldualassumptions} on the Lagrangian, 
\eqref{E:MaxintermsofFarMinkDualPlusError}, \eqref{E:FarspatialintermsofB}, 
and \eqref{E:DintermsofEBh} - \eqref{E:HintermsofEBh} it follows that

\begin{subequations}
\begin{align}
	\Displacement & = \Electricfield \ + \ O^{\dParameter+1}\big(|h||(\Electricfield, \Magneticinduction)| \big)
		\ + \ O^{\dParameter+1}\big(|(\Electricfield, \Magneticinduction)|^3;h \big), \label{E:DintermsofEBh} \\
	\Magneticfield  & = \Magneticinduction \ + \ O^{\dParameter+1}\big(|h||(\Electricfield, \Magneticinduction)| \big)
		\ + \ O^{\dParameter+1}\big(|(\Electricfield, \Magneticinduction)|^3;h \big), \label{E:HintermsofEBh} \\
	\Electricfield  & = \Displacement \ + \ O^{\dParameter+1}\big(|h||(\Displacement, \Magneticinduction)| \big)
		\ + \ O^{\dParameter+1}\big(|(\Displacement, \Magneticinduction)|^3;h \big), \label{E:EintermsofDBh} \\
	\Magneticfield & = \Magneticinduction \ + \ O^{\dParameter+1}\big(|h||(\Displacement, \Magneticinduction)| \big)
		\ + \ O^{\dParameter+1}\big(|(\Displacement, \Magneticinduction)|^3;h \big).
\end{align}
\end{subequations}

Furthermore, recalling that $\hat{N}^{\nu}|_{\Sigma_0} = A \delta_0^{\nu},$ where $A \eqdef \sqrt{1 - \frac{2M}{r}\chi(r)},$ 
and using \eqref{E:FargDualIntermsofFarmDual}, and \eqref{E:InducedhData}, it follows that

	\begin{align} \label{E:ElectricfieldDataInTermsofIntrinsic}
		\mathring{\Electricfield} & = \mathring{\Displacement} \ + \ 
		O^{\dParameter+1}\big(|\mathring{h}^{(1)}||(\mathring{\Displacement}, 
			\mathring{\Magneticinduction})|;\chi(r)M/r\big) 
			\ + \ O^{\dParameter+1}\big(|\chi(r)M/r||(\mathring{\Displacement}, \mathring{\Magneticinduction})|;\mathring{h}^{(1)} \big)
			\ + \ O^{\dParameter+1}\big(|(\mathring{\Displacement}, \mathring{\Magneticinduction})|^3;\mathring{h}^{(1)};\chi(r)M/r\big)
	\end{align}
	and

	\begin{subequations}
		\begin{align}
		\mathring{\Magneticinduction} & = \mathring{\mathfrak{\Magneticinduction}} 
			\ + \ O^{\dParameter+1}\big(|\chi(r)M/r||\mathring{\mathfrak{\Magneticinduction}}|\big)
			\ + \ O^{\dParameter+1}\big(|\mathring{h}^{(1)}||(\mathring{\mathfrak{\Magneticinduction}},  
			\mathring{\mathfrak{\Displacement}})|;|\chi(r)M/r|\big), 
			\label{E:IntialInductionintermsofInitialQuantities} \\
		\mathring{\Displacement} & = \mathring{\mathfrak{\Displacement}} 
			\ + \ O^{\dParameter+1}\big(|\mathring{h}^{(1)}||\chi(r)M/r||\mathring{\mathfrak{\Displacement}}|\big)
			\ + \ O^{\dParameter+1}\big(|(\mathring{\mathfrak{\Magneticinduction}},  
			\mathring{\mathfrak{\Displacement}})|;|\chi(r)M/r|\big), 
			\label{E:IntialDisplacementintermsofInitialQuantities} \\
		\mathring{\mathfrak{\Magneticinduction}} & = \mathring{\Magneticinduction}
			\ + \ O^{\dParameter+1}\big(|\chi(r)M/r||\mathring{\Magneticinduction}|\big)
			\ + \ O^{\dParameter+1}\big(|\mathring{h}^{(1)}||(\mathring{\Magneticinduction}, \mathring{\Displacement})|;|\chi(r)M/r|\big), \\
		\mathring{\mathfrak{\Displacement}} & = \mathring{\Displacement} 
			\ + \ O^{\dParameter+1}\big(|\chi(r)M/r||\mathring{\Displacement}|\big)
			\ + \ O^{\dParameter+1}\big(|\mathring{h}^{(1)}||(\mathring{\Magneticinduction},  \mathring{\Displacement})|;|\chi(r)M/r|\big).
			\label{E:InitialIntrinsicDisplacementIntermsofInitialQuantities}
	\end{align}
	\end{subequations}

\begin{remark}
	Logically speaking, the ADM mass $M$ (and hence also the coordinates of the unit normal vector $\hat{N}|_{\Sigma_0}$)
	is only well-defined \emph{after} one has solved the abstract Einstein constraint equations \eqref{E:GaussIntro} - 
	\eqref{E:DivergenceB0Intro}. Thus, the
	relations \eqref{E:ElectricfieldDataInTermsofIntrinsic} - \eqref{E:IntialDisplacementintermsofInitialQuantities}
	should be thought of as defining $\mathring{\Electricfield}_j, \mathring{\Magneticinduction}_j, 
	\mathring{\Displacement}_j$ in terms of $\mathring{\mathfrak{\Electricfield}}_j, \mathring{\mathfrak{\Magneticinduction}}_j$ 
	and \emph{not} the other way around.
\end{remark}

The main goal of this section is to deduce the following proposition, which is a decomposition of the electromagnetic equations into \emph{constraint} equations and \emph{evolution} equations.

\begin{proposition} \label{P:ElectromagneticDecomp}
\textbf{(Electromagnetic constraint and evolution equations)}
Under the assumption \eqref{E:Ldualassumptions} on $\Ldual,$ the electromagnetic equations \eqref{E:dFis0Electrodecomp} - \eqref{E:dMis0Electrodecomp} are equivalent to the following pairs of constraint equations and evolution equations:

\begin{center} 
	\textbf{Constraint Equations}
\end{center}
\begin{subequations}
\begin{align}
	(\um^{-1})^{ab} \unabla_a \Displacement_b & = 0, \label{E:Dconstraint} \\ 
	(\um^{-1})^{ab} \unabla_a \Magneticinduction_b & = 0, \label{E:Bconstraint}
\end{align}
\end{subequations}

\vspace{.5in}

\begin{center} 
	\textbf{Evolution Equations}
\end{center}
\begin{subequations}
\begin{align}
	\partial_t \Magneticinduction_j & = - [jab] \unabla_a \Electricfield_b, 
		\label{E:partialtBisolated} \\
	\partial_t \Electricfield_j & = [jab] \unabla_a \Magneticinduction_b 
		\ + \ O^{\dParameter}\big(|h||\unabla(\Electricfield,\Magneticinduction)|;(\Electricfield,\Magneticinduction) \big)
		\ + \ O^{\dParameter}\big(|(\Electricfield,\Magneticinduction)|^2|\unabla(\Electricfield,\Magneticinduction)|;h \big)
		\ + \ O^{\dParameter}\big(|\nabla h| |(\Electricfield, \Magneticinduction)|; h \big). 
			\label{E:partialtEisolated} 
\end{align}
\end{subequations}

Furthermore, if the one-forms $\mathring{\Displacement}, \mathring{\Magneticinduction}$ are related to the one-forms
$\mathring{\mathfrak{\Displacement}}, \mathring{\mathfrak{\Magneticinduction}}$ as implicitly determined by the 
relations \eqref{E:AbstractEBDHinertialcomponents} - \eqref{E:EBDHinertialcomponents} (together with the fact that
$\hat{N}^{\mu} = A^{-1} \delta_0^{\mu}$), then equations \eqref{E:Dconstraint} - \eqref{E:Bconstraint} hold 
for $\mathring{\Displacement}, \mathring{\Magneticinduction}$ (i.e., along $\Sigma_0$) if and only if the following equations hold:

\begin{center} 
	\textbf{Abstract Constraint Equations}
\end{center}

\begin{subequations}
\begin{align}
	(\mathring{\underline{g}}^{-1})^{ab} \mathring{\underline{\mathscr{D}}}_a \mathring{\mathfrak{\Displacement}}_b & = 0, 
		\label{E:AbstractDconstraint} \\
	(\mathring{\underline{g}}^{-1})^{ab} \mathring{\underline{\mathscr{D}}}_a \mathring{\mathfrak{\Magneticinduction}}_b & = 0.
		\label{E:AbstractBconstraint} 
\end{align}
\end{subequations}
In the above expressions, $\mathring{\underline{g}}_{jk}$ is the first-fundamental form of $\Sigma_0,$ and $\mathring{\underline{\mathscr{D}}}$ is the Levi-Civita connection corresponding to $\mathring{\underline{g}}_{jk}.$

\end{proposition}

\begin{remark}
In equations \eqref{E:Dconstraint} - \eqref{E:Bconstraint}, $(\um^{-1})^{ab} \unabla_a$ is the standard Euclidean divergence operator, while in equations \eqref{E:partialtBisolated} - \eqref{E:partialtEisolated}, $[jab] \unabla_a$ is the standard Euclidean curl operator.
\end{remark}

\begin{remark}
	Using equations \eqref{E:partialtBisminuscurlE} - \eqref{E:partialtDiscurlH}, it is easy to check that if a classical 
	solution to the evolution equations satisfies the constraints at $t=0,$  then it necessarily satisfies the constraints 
	\eqref{E:Dconstraint} - \eqref{E:Bconstraint} at all later times (as long as it persists).
\end{remark}

\begin{proof}
We first show that \eqref{E:Dconstraint} follows from either \eqref{E:dFis0ElectrodecompMathscrD} or \eqref{E:dFis0Electrodecomp}, and that \eqref{E:Dconstraint} holds if and only if \eqref{E:AbstractDconstraint} holds. To this end, we first note that since $\hat{N}^{\mu}$ is the future-directed unit $g-$normal to $\Sigma_t$ and 
$g_{\mu \nu} = \mathring{\underline{g}}_{\mu \nu} - \hat{N}_{\mu} \hat{N}_{\nu}$ along $\Sigma_0,$ 
the following identities hold for any one-form $X_{\mu}$ $g-$tangent to $\Sigma_0$ and any two-form $P_{\mu \nu}:$

\begin{align}
	(\mathring{\underline{g}}^{-1})^{ab} \mathring{\underline{\mathscr{D}}}_a X_b 
	& = (g^{-1})^{\kappa \lambda} \mathscr{D}_{\kappa} X_{\lambda} - X_{\lambda} \hat{N}^{\kappa} 
		\mathscr{D}_{\kappa} \hat{N}^{\lambda}, \label{E:IntrinsicExtrinsicDivergenceIdentity} \\
	(g^{-1})^{\kappa \lambda} P_{\lambda \nu} \mathscr{D}_{\kappa} \hat{N}^{\nu} 
	& = P_{\lambda \nu} \hat{N}^{\nu} \hat{N}^{\kappa} \mathscr{D}_{\kappa} \hat{N}^{\lambda}. 
		\label{E:TwoFormhatNContractionIdentity}
\end{align}

Using \eqref{E:IntrinsicExtrinsicDivergenceIdentity} and \eqref{E:TwoFormhatNContractionIdentity}
with $X_{\mu} \eqdef \mathfrak{\Magneticinduction}_{\mu}$ and $P_{\mu \nu} \eqdef \Fardual_{\mu \nu}$
we compute that the following identities hold along $\Sigma_0:$

\begin{align} \label{E:DivergenceIntrinsicBintermsofFarandhatN}
	(\mathring{\underline{g}}^{-1})^{ab} \mathring{\underline{\mathscr{D}}}_{a} \mathfrak{\Magneticinduction}_b 
	& = (g^{-1})^{\kappa \lambda} \mathscr{D}_{\kappa} \mathfrak{\Magneticinduction}_{\lambda}	
			- \mathfrak{\Magneticinduction}_{\lambda} \hat{N}^{\kappa} \mathscr{D}_{\kappa} \hat{N}^{\lambda} 
			 \\
	& = - (g^{-1})^{\kappa \lambda} \mathscr{D}_{\kappa} (\Fardual_{\lambda \nu} \hat{N}^{\nu})	
			+ \Fardual_{\lambda \nu} \hat{N}^{\nu} \hat{N}^{\kappa} \mathscr{D}_{\kappa} \hat{N}^{\lambda}
			\notag \\
	& = - \frac{1}{2}g_{\nu \nu'} \hat{N}^{\nu'} \epsilon^{\# \mu \nu \kappa \lambda} \mathscr{D}_{\mu} \Far_{\kappa \lambda}. 
		\notag 
\end{align}
Identities analogous to \eqref{E:DivergenceIntrinsicBintermsofFarandhatN} hold if we make the replacements 
$\Big(\mathring{\underline{g}}^{-1}, g, \mathring{\underline{\mathscr{D}}},
\mathscr{D}, \star, \hat{N}^{\mu}, \epsilon^{\# \mu \nu \kappa \lambda}, \mathfrak{\Magneticinduction} \Big) \rightarrow 
\Big(\um^{-1}, m, \unabla, \nabla, \ostar, \hat{n}^{\mu}, \Minkvolume^{\mu \nu \kappa \lambda}, \Magneticinduction \Big),$
where $\hat{n}^{\mu}(t,x)$ is the future-directed unit Minkowskian unit normal to $\Sigma_t.$ Now by \eqref{E:DivergenceIntrinsicBintermsofFarandhatN} and the Minkowskian analogy of \eqref{E:DivergenceIntrinsicBintermsofFarandhatN}, equations \eqref{E:Dconstraint} and
\eqref{E:AbstractDconstraint} follow from either \eqref{E:dFis0ElectrodecompMathscrD} or \eqref{E:dFis0Electrodecomp},
since either \eqref{E:dFis0ElectrodecompMathscrD} or \eqref{E:dFis0Electrodecomp} are sufficient to guarantee that the right-hand side of \eqref{E:DivergenceIntrinsicBintermsofFarandhatN} is $0.$ Furthermore, since
$g_{\nu \nu'} \hat{N}^{\nu'}$ and $m_{\nu \nu'} \hat{n}^{\nu'}$ are proportional along $\Sigma_0,$ since $\epsilon^{\# \mu \nu \kappa \lambda}$ and $\upsilon^{\mu \nu \kappa \lambda}$ are proportional, and since the Christoffel symbols of $\mathscr{D}$ and $\nabla$ are symmetric in their two lower indices, it follows that

\begin{align} \label{E:IntrinsicConstraintEquivalenttoMinkowskiConstraint}
	g_{\nu \nu'} \hat{N}^{\nu'} \epsilon^{\# \mu \nu \kappa \lambda} \mathscr{D}_{\mu} \Far_{\kappa \lambda}|_{\Sigma_0} & = 0 \\
		& \iff \notag \\
	m_{\nu \nu'} \hat{n}^{\nu'} \upsilon^{\mu \nu \kappa \lambda} \nabla_{\mu} \Far_{\kappa \lambda}|_{\Sigma_0} & = 0. \notag
\end{align}
Hence, \eqref{E:Bconstraint} holds along $\Sigma_0$ if and only if \eqref{E:AbstractBconstraint} holds along $\Sigma_0.$ 
The derivation of \eqref{E:Dconstraint} and \eqref{E:AbstractDconstraint} along $\Sigma_0$ from \eqref{E:dMis0ElectrodecompMathscrD} or \eqref{E:dMis0Electrodecomp} and the proof of the equivalence of \eqref{E:Dconstraint} and \eqref{E:AbstractDconstraint} along $\Sigma_0$ are similar.

We now set $\lambda = 0, \mu = a, \nu = b$ in \eqref{E:dFis0Electrodecomp}, then contract against the Euclidean volume form $[jab]$ use \eqref{E:EBDHinertialcomponents} - \eqref{E:FarspatialintermsofB} to deduce that

\begin{align} \label{E:partialtBisminuscurlE}
	\partial_t \Magneticinduction_j = - [jab] \unabla_a \Electricfield_b.
\end{align}
Similarly, we set $\lambda = 0, \mu = a, \nu = b$ in \eqref{E:dMis0Electrodecomp}, contract against $[jab],$
and use \eqref{E:EBDHinertialcomponents} - \eqref{E:FarspatialintermsofB} to deduce that

\begin{align} \label{E:partialtDiscurlH}
	\partial_t \Displacement_j = [jab] \unabla_a \Magneticfield_b.
\end{align}
Finally, we use \eqref{E:partialtBisminuscurlE}, \eqref{E:partialtDiscurlH},
and \eqref{E:DintermsofEBh} - \eqref{E:HintermsofEBh}, to deduce \eqref{E:partialtBisolated} - \eqref{E:partialtEisolated}. 

\end{proof}

\section{The Smallness Condition on the Abstract Data} \label{S:SmallDataAssumptions}
\setcounter{equation}{0}

In this section, we assume that we are given abstract initial data $(\mathring{\underline{g}}_{jk} = \delta_{jk} + \mathring{\underline{h}}_{jk}^{(0)} + \mathring{\underline{h}}_{jk}^{(1)} ,\mathring{K}_{jk},\mathring{\mathfrak{\Displacement}}_j,\mathring{\mathfrak{\Magneticinduction}}_j),$ $(j,k=1,2,3),$
on the manifold $\mathbb{R}^3$ satisfying the constraint equations \eqref{E:Gauss} - \eqref{E:DivergenceB0}. Our goal is to describe in detail the smallness condition on $(\mathring{\underline{h}}_{jk}^{(0)}, \mathring{\underline{h}}_{jk}^{(1)} ,\mathring{K}_{jk},\mathring{\mathfrak{\Displacement}}_j,\mathring{\mathfrak{\Magneticinduction}}_j)$
that will lead to global existence for the reduced system \eqref{E:Reducedh1Summary} - \eqref{E:ReduceddMis0Summary}, under the assumption that its initial data $(g_{\mu \nu}|_{t=0},\partial_t g_{\mu \nu}|_{t=0},\Far_{\mu \nu}|_{t=0}),$ $(\mu, \nu = 0,1,2,3),$
are constructed from the abstract initial data as described in Section \ref{SS:ReducedData}. Recall that our global existence argument is heavily based on the analysis of $\mathcal{E}_{\dParameter;\upgamma;\upmu}(t),$ which is the energy defined in \eqref{E:EnergyIntro}. In particular, $\mathcal{E}_{\dParameter;\upgamma;\upmu}(0)$ must be sufficiently small in order for us to close the argument. The energy depends on \emph{both normal and tangential} Minkowskian covariant derivatives of the quantities $(\nabla_{\lambda} h_{\mu \nu}^{(1)},\Far_{\mu \nu})$ at $t=0.$ 
On the other hand, our smallness condition will be expressed in terms of the ADM mass $M$ and $E_{\dParameter;\upgamma}(0),$ which is a weighted Sobolev norm of $(\unabla_i \mathring{\underline{h}}_{jk}^{(1)} ,\mathring{K}_{jk},\mathring{\mathfrak{\Displacement}}_j,\mathring{\mathfrak{\Magneticinduction}}_j)$
depending only on \emph{tangential} derivatives of the abstract data. More specifically, our smallness condition is expressed in terms of the weighted Sobolev norms $\| \cdot \|_{H_{1/2 + \upgamma}^{\dParameter}}$ introduced in Definition \ref{D:HNdeltanorm}. The main result of this section is contained in Proposition \ref{P:SmallNormImpliesSmallEnergy}, which shows that if $E_{\dParameter;\upgamma}(0) + M$ is sufficiently small and $(h_{\mu \nu}^{(1)},\Far_{\mu \nu})$ is the corresponding solution to the reduced equations, then $\mathcal{E}_{\dParameter;\upgamma;\upmu}(0) \lesssim E_{\dParameter;\upgamma}(0) + M.$ Thus, Proposition \ref{P:SmallNormImpliesSmallEnergy} allows us to deduce the smallness of $\mathcal{E}_{\dParameter;\upgamma;\upmu}(0)$ from the smallness of quantities that depend exclusively on the abstract initial data.

\noindent \hrulefill
\ \\

We begin by introducing the weighted Sobolev norm discussed in the above paragraph.

\begin{definition} \label{D:HNdeltanorm}
	Let $U(x)$ be a tensorfield defined along the Euclidean space $\mathbb{R}^3.$
	Then for any integer $\dParameter \geq 0,$ and any real number $\eta,$ we define the $H_{\eta}^{\dParameter}$ norm of $U$ by
	
	\begin{align} \label{E:HNdeltanorm}
		\| U \|_{H_{\eta}^{\dParameter}}^2 \eqdef \sum_{|I| \leq \dParameter } \int_{x \in \mathbb{R}^3} (1 + 
			|x|^2)^{(\eta + |I|)} |\underline{\nabla}^I U(x)|^2 \, d^3 x.
	\end{align}
\end{definition}

We also introduce the following norm, which can be controlled in terms of a suitable
$H_{\eta}^{\dParameter}$ norm via a Sobolev embedding result; see Proposition \ref{P:SobolevEmbeddingHNdeltaCNprimedeltamprime}.

\begin{definition} \label{D:CNdeltanorm}
	Let $U(x)$ be a tensorfield defined along the Euclidean space $\mathbb{R}^3.$ 
	Then for any integer $\dParameter \geq 0,$ and any real number $\eta,$ we define the $C_{\eta}^{\dParameter}$ norm of $U$ 
	by
	
	\begin{align} \label{E:CNdeltanorm}
		\| U \|_{C_{\eta}^{\dParameter}}^2 \eqdef 	\| U \|_{C_{\eta}^{\dParameter}}^2 
		\eqdef \sum_{|I| \leq \dParameter } \mbox{ess} \sup_{x \in \mathbb{R}^3} (1 + |x|^2)^{(\eta + |I|)} 
		|\underline{\nabla}^I U(x)|^2.
	\end{align}
\end{definition}

We are now ready to introduce our norm $E_{\dParameter;\upgamma}(0) \geq 0$ on the abstract initial data. Recall that as discussed in Section \ref{SS:AbstractData}, the data are the following four fields on $\mathbb{R}^3:$ 
$(\mathring{\underline{g}}_{jk} = \delta_{jk} + \underbrace{\mathring{\underline{h}}_{jk}^{(0)} + \mathring{\underline{h}}_{jk}^{(1)}}_{\mathring{\underline{h}}_{jk}}, \mathring{K}_{jk}, \mathring{\mathfrak{\Displacement}}_j, \mathring{\mathfrak{\Magneticinduction}}_j),$ $(j,k = 1,2,3).$

\begin{definition} \label{D:DataNorm}
The norm $E_{\dParameter;\upgamma}(0) \geq 0$ of the abstract initial data is defined by

\begin{align}   \label{E:DataNorm}
	E_{\dParameter;\upgamma}^2(0) 
	& \eqdef \| \underline{\nabla} \mathring{\underline{h}}^{(1)} \|_{H_{1/2 + \upgamma}^{\dParameter}}^2 
		\ + \ \| \mathring{K} \|_{H_{1/2 + \upgamma}^{\dParameter}}^2 
		\ + \ \| \mathring{\mathfrak{\Displacement}} \|_{H_{1/2 + \upgamma}^{\dParameter}}^2 
		\ + \ \| \mathring{\mathfrak{\Magneticinduction}} \|_{H_{1/2 + \upgamma}^{\dParameter}}^2.
\end{align}

\end{definition}

\vspace{.5in}
\begin{center}
	\textbf{\LARGE The Smallness Condition}
\end{center}

\bigskip

Our smallness condition for global existence is

\begin{align} \label{E:NormSmallnessCondition}
	E_{\dParameter;\upgamma}(0) + M \leq \varepsilon_{\dParameter},
\end{align}
where $\varepsilon_{\dParameter}$ is a sufficiently small positive number.
\vspace{.5in}

Recall that the energy $\mathcal{E}_{\dParameter;\upgamma;\upmu}(t) \geq 0$ is defined by

\begin{align} \label{E:EnergydefSmallnessSection}
	\mathcal{E}_{\dParameter;\upgamma;\upmu}^2(t) & \eqdef \underset{0 \leq \tau \leq t}{\mbox{sup}} 
		\sum_{|I| \leq \dParameter } \int_{\Sigma_{\tau}} 
		\Big\lbrace |\nabla\nabla_{\mathcal{Z}}^I h^{(1)}|^2 + |\Lie_{\mathcal{Z}}^I \Far|^2 \Big\rbrace w(q) \, d^3 x,
\end{align}
where $\nabla$ denotes the \emph{full Minkowski spacetime} covariant derivative operator, and the weight $w(q)$ is defined in \eqref{E:weight}. The next proposition, which is the main result of this section, shows that the smallness of $\mathcal{E}_{\dParameter;\upgamma;\upmu}(0)$ follows from the smallness of $E_{\dParameter;\upgamma}(0) + M.$

\begin{proposition} \label{P:SmallNormImpliesSmallEnergy}
	\textbf{(The smallness of the initial energy)}
	Let $(\mathring{\underline{g}}_{jk}  \delta_{jk} + \mathring{\underline{h}}_{jk}^{(0)} + \mathring{\underline{h}}_{jk}^{(1)}, 
	\mathring{K}_{jk}, \mathring{\mathfrak{\Displacement}_j}, \mathring{\mathfrak{\Magneticinduction}}_j),$ be abstract initial 
	data on the manifold $\mathbb{R}^3$ for the Einstein-nonlinear electromagnetic system \eqref{E:IntroEinstein} - 
	\eqref{E:IntrodMis0} and	assume that the abstract initial data 
	are asymptotically flat in the sense that \eqref{E:metricdataexpansion} - \eqref{E:BdecayAssumption} 
	hold. Let $(g_{\mu \nu}|_{t=0} = m_{\mu \nu} + h_{\mu \nu}^{(0)}|_{t=0} + h_{\mu \nu}^{(1)}|_{t=0}, 
	\partial_t g_{\mu \nu}|_{t=0} = \partial_t h_{\mu \nu}^{(0)}|_{t=0} + \partial_th_{\mu \nu}^{(1)}|_{t=0}, 
	\Far_{\mu \nu}|_{t=0}),$ $(\mu, \nu = 0,1,2,3),$ be the corresponding initial data 
	for the reduced system \eqref{E:Reducedh1Summary} - \eqref{E:ReduceddMis0Summary} as defined in Section \ref{SS:ReducedData},
	and let $(g_{\mu \nu} = m_{\mu \nu} + h_{\mu \nu}^{(0)} + h_{\mu \nu}^{(1)}, 
	\Far_{\mu \nu})$ be the solution to the reduced system launched by this data. Let $\dParameter \geq 8$ be an integer.
	Then there exists a constant $\varepsilon_0 > 0$ independent of $\dParameter $ and constants $C_{\dParameter} > 0,$ $\widetilde{C}_{\dParameter} > 0$
	such that if $E_{\dParameter;\upgamma}(0) + M \leq \varepsilon \leq \varepsilon_0,$ 
	then 
	
	\begin{align} \label{E:SmallNormImpliesSmallEnergy}
		\mathcal{E}_{\dParameter;\upgamma;\upmu}(0) 
		& \leq C_{\dParameter} \big\lbrace E_{\dParameter;\upgamma}(0) + M \big\rbrace
			\leq \widetilde{C}_{\dParameter} \varepsilon.
	\end{align}
	
\end{proposition}

\begin{remark} \label{R:Nomu}
	Note that $q \geq 0$ holds at $t=0.$ Therefore, $\mathcal{E}_{\dParameter;\upgamma;\upmu}(0)$ does not depend on the constant
	$\upmu.$
\end{remark}

The proof of Proposition \ref{P:SmallNormImpliesSmallEnergy} will be given at the end of this section. 
We first establish some technical lemmas.

\begin{lemma} \label{L:EnergyWithFarReplacedbyElectricandMagnetic} \textbf{(Energy in terms of the 
	spacetime metric remainder piece, the electric field, and the magnetic induction)}
	Let $\Far_{\mu \nu}$ be a two-form, let the pair of one-forms $(\Electricfield_{\mu},\Magneticinduction_{\mu})$ be its
	Minkowskian electromagnetic decomposition as defined in Section \ref{SS:EBDH},
	and let $h_{\mu \nu}^{(1)}$ be an arbitrary type $\binom{0}{2}$ tensorfield. Let
	$\mathcal{E}_{\dParameter;\upgamma;\upmu}(t)$ be the energy defined in \eqref{E:EnergydefSmallnessSection}.
	Then
	
	\begin{align} \label{E:EnergyWithFarReplacedbyElectricandMagnetic}
		\mathcal{E}_{\dParameter;\upgamma;\upmu}^2(t) & \approx \underset{0 \leq \tau \leq t}{\sup} 
			\sum_{|I| \leq \dParameter } \int_{\Sigma_{\tau}} 
			\Big\lbrace |\nabla\nabla_{\mathcal{Z}}^I h^{(1)}|^2 
			+ |\nabla_{\mathcal{Z}}^I \Electricfield|^2 + |\nabla_{\mathcal{Z}}^I \Magneticinduction|^2 \Big\rbrace w(q) \, d^3 x.
	\end{align}

\end{lemma}

\begin{proof}
	\eqref{E:EnergyWithFarReplacedbyElectricandMagnetic} easily follows from the identity
	$|\nabla_{\mathcal{Z}}^I \Far|^2 = 2 |\nabla_{\mathcal{Z}}^I \Electricfield|^2 + 2 |\nabla_{\mathcal{Z}}^I 
	\Magneticinduction|^2,$ the verification of which we leave to the reader.
\end{proof}

\begin{lemma} \label{L:NablaZIUL2InTermsofrWeightedPartialtkUnderlineNablaJNorms}
	The following estimates hold for any sufficiently differentiable spacetime tensorfield $U(t,x)$
	defined in a neighborhood of $\Sigma_0 \eqdef \lbrace(t,x) \mid t=0 \rbrace,$ where $w(q)$ is the weight
	defined in \eqref{E:weight}:
	
	\begin{align} \label{E:NablaZIUL2InTermsofrWeightedPartialtkUnderlineNablaJNorms}
		\Big(\sum_{|I| \leq \dParameter } w^{1/2}(q)|\nabla_{\mathcal{Z}}^I U| \Big)|_{\Sigma_0} 
		& \approx \Big(\sum_{|I| \leq \dParameter } (1 + r)^{1/2 + \upgamma + |I|} |\nabla^I U| \Big)\big|_{\Sigma_0} \\
		& \approx \Big( \sum_{|J| + k \leq \dParameter } (1 + r)^{1/2 + \upgamma + |J| + k} 
		|\underline{\nabla}^J \partial_t^k U|\Big)\big|_{\Sigma_0}. \notag 
	\end{align}
	
	The same estimates hold if $\nabla_{\mathcal{Z}}^I$ is replaced with $\Lie_{\mathcal{Z}}^I.$ The notation
	$\big|_{\Sigma_0}$ is meant to indicate that the estimates only hold along $\Sigma_0.$

\end{lemma}

\begin{proof}
	By iterating the identity $\frac{\partial}{\partial x^{\mu}} = \frac{x^{\kappa} \Omega_{\kappa \mu} + x_{\mu}S}{qs},$ 
	and noting that $q = r = s$ along $\Sigma_0,$ it follows that
	
	\begin{align} \label{E:TranslationaldervativeslessthanqweightedLieZKnorm}
		(1 + r)^{|I|} |\nabla^I U| \lesssim \sum_{|J| \leq |I|} |\nabla_{\mathcal{Z}}^J U|.
	\end{align}
	It thus follows that
	
	\begin{align}
		\Big(\sum_{|I| \leq \dParameter } (1 + r)^{1/2 + \upgamma + |I|} |\nabla^I U|\Big)\big|_{\Sigma_0}
		& \lesssim \Big(\sum_{|I| \leq \dParameter } w^{1/2}(q)|\nabla_{\mathcal{Z}}^I U| \Big)\big|_{\Sigma_0}.
	\end{align}
	On the other hand, the opposite inequality follows easily from expanding the operator $\nabla_{\mathcal{Z}}^I$ 
	and using the Leibniz rule plus \eqref{E:CovariantDerivativesofZareConstant}. This proves the first $\approx$ in 
	\eqref{E:NablaZIUL2InTermsofrWeightedPartialtkUnderlineNablaJNorms}. The second $\approx$ is trivial. We have
	thus established \eqref{E:NablaZIUL2InTermsofrWeightedPartialtkUnderlineNablaJNorms}.
	To establish the same estimates with the operator $\Lie_{\mathcal{Z}}^I$ in place of $\nabla_{\mathcal{Z}}^I,$
	we simply use \eqref{E:LieZIinTermsofNablaZI}.

\end{proof}

\begin{corollary} \label{C:InitialEnergyInTermsofTangentialandTimeDerivatives}
	Under the assumptions of Lemma \ref{L:EnergyWithFarReplacedbyElectricandMagnetic}, we have that
	
	\begin{align}   \label{E:InitialEnergyInTermsofTangentialandTimeDerivatives}
		\mathcal{E}_{\dParameter;\upgamma;\upmu}^2(0) 
		& \approx \sum_{k + |I| \leq \dParameter } 
			\int_{\mathbb{R}^3} (1 + |x|)^{1 + 2\upgamma + 2|I|} \Big\lbrace |\partial_t^k \underline{\nabla}^I 
			\partial_t h^{(1)}|^2(0,x) + |\underline{\nabla}^I \underline{\nabla} h^{(1)}|^2(0,x) \Big\rbrace \, d^3 x \\
		& \ \ +  \int_{\mathbb{R}^3} (1 + |x|)^{1 + 2\upgamma + 2|I|} \Big\lbrace|\partial_t^k \underline{\nabla}^I 
			\Electricfield|^2(0,x) 
			+ |\partial_t^k \underline{\nabla}^I \Magneticinduction|^2(0,x) \Big\rbrace \, d^3 x. \notag
	\end{align}
	
\end{corollary}

\begin{proof}
	Corollary \ref{C:InitialEnergyInTermsofTangentialandTimeDerivatives}
	follows easily from Lemmas \ref{L:EnergyWithFarReplacedbyElectricandMagnetic}
	and \ref{L:NablaZIUL2InTermsofrWeightedPartialtkUnderlineNablaJNorms}.
\end{proof}

\begin{lemma} \label{L:SolveforTimeDerivativesinTermsofInherentDerivatives}
	Let $k \geq 1$ and $\dParameter \geq 8$ be integers, and let $J$ be a $\underline{\nabla}-$multi-index. Assume that $|J| +|K| 
	\leq \dParameter.$ Assume that $(h_{\mu \nu}^{(1)}, \Far_{\mu \nu})$ is a solution to the reduced equations 
	\eqref{E:Reducedh1Summary} - \eqref{E:ReduceddMis0Summary}, and define the arrays $V,$ $V^{(0)},$ $V^{(1)},$ $W,$ $W^{(0)},$ $W^{(1)}$ by
	
	\begin{subequations}
	\begin{align}
		V & \eqdef (h, \underline{\nabla} h, \partial_t h, \Electricfield, \Magneticinduction) = V^{(0)} + V^{(1)}, 
			\label{E:Vdef} \\
		V^{(0)} & \eqdef (h^{(0)}, \underline{\nabla} h^{(0)}, \partial_t h^{(0)}, 0, 0), \\
		V^{(1)} & \eqdef (h^{(1)}, \underline{\nabla} h^{(1)}, \partial_t h^{(1)}, \Electricfield, \Magneticinduction), \\
		W & \eqdef (0,\underline{\nabla} h, \partial_t h, \Electricfield, \Magneticinduction) = W^{(0)} + W^{(1)}, \\
		W^{(0)} & \eqdef (0,\underline{\nabla} h^{(0)}, \partial_t h^{(0)}, 0, 0), \\
		W^{(1)} & \eqdef (0,\underline{\nabla} h^{(1)}, \partial_t h^{(1)}, \Electricfield, \Magneticinduction).
		\label{E:W1def}
	\end{align}
	\end{subequations}
	In the above expressions, the tensorfields $h_{\mu \nu}^{(0)},$ $h_{\mu \nu}^{(1)}$ are defined by \eqref{E:gmhexpansion} - 	
	\eqref{E:h0defIntro}, while the electromagnetic one-forms $\Electricfield_{\mu},$ $\Magneticinduction_{\mu}$ are defined in 	
	\eqref{E:EBDHinertialcomponents}. Assume further that $|V^{(1)}| + M \leq \varepsilon.$ Then if $\varepsilon$ is sufficiently 
	small, $\partial_t^k \underline{\nabla}^J W^{(1)}$ can be written as the following finite linear combination:
	
	\begin{align} 
		\underline{\nabla}^J \partial_t^k W^{(1)} & = \sum terms,
	\end{align}
	where each $term$ can be written as
	
	\begin{align}	 \label{E:PartialtkunderlinenablaJW1inTermsofInstrinsic}
			term & = \sum_{s=0}^{|J| + k + 1} \sum_{|I_1| + \cdots + |I_s| \leq |J| + k} 
				F_{(I_1, \cdots, I_s;J;k;s)}(t,x)\mathscr{M}_{(I_1, \cdots, I_s;J;k;s)}(V)[\underline{\nabla}^{I_1} W^{(1)}, \cdots, 
				\underline{\nabla}^{I_s} W^{(1)}], 
	\end{align}
	where 
	
	\begin{enumerate}
		\item The array-valued functions $\mathscr{M}_{(I_1, \cdots, I_s;J;k;s)}(V)[\underline{\nabla}^{I_1} W^{(1)}, \cdots, 
			\underline{\nabla}^{I_s} W^{(1)}]$ are continuous in a neighborhood of $V = 0$ and
			are multi-linear in the arguments $[\underline{\nabla}^{I_1} W^{(1)}, \cdots, 
			\underline{\nabla}^{I_s} W^{(1)}].$
		\item The array-valued functions $F_{(I_1, \cdots, I_s;J;k;s)}(t,x)$
			are smooth and satisfy \\
			$|F_{(I_1, \cdots, I_s;J;k;s)}(t,x)| \lesssim  M (1 + t + |x|)^{-(3 + |J| + k )}$ if $s = 0$ 
			(i.e., if there are no multi-linear arguments $[\cdots]$), where $M$ is the ADM mass.
		\item In the case $s \geq 1,$ $|F_{(I_1, \cdots, I_s;J;k;s)}(t,x)| \lesssim (1 + t + |x|)^{-d},$ 
			where $d \geq |J| + k - (|I_1| + \cdots + |I|_s) - (s-1).$
		\end{enumerate}

\end{lemma}

\begin{proof}

	We first claim that we can write the reduced system \eqref{E:Reducedh1Summary} - \eqref{E:ReduceddMis0Summary} 		
	as a finite linear combination

	\begin{subequations}	
	\begin{align} \label{E:partialtW1isaSumofTerms}
		\partial_t W^{(1)} = \sum terms,
	\end{align}
	where each term can be written in the form
	
	\begin{align} \label{E:SolveforPartialtV}
		term & = \sum_{|I| = 1} \mathscr{M}_{(I;0;1;1)}(V)[\unabla^I W^{(1)}]
			+ \mathscr{M}_{(0;0;1;2)}(V)[W^{(1)}, W^{(1)}] \\ 
		& \ \ + f_{(0;0;1;1)}(t,x) \mathscr{M}_{(0;0;1;1)}(V)[W^{(1)}]
			+ f_{(0;0;1;0)}(t,x)\mathscr{M}_{(0;0;1;0)}(V), \notag 
	\end{align}
	\end{subequations}
	where the functions $\mathscr{M}_{(\cdots)}(V)[\cdots],$ which depend on the 
	$\dParameter +1-$times continuously differentiable Lagrangian $\Ldual$
	for the electromagnetic equations, have the properties stated in the conclusions of the theorem; and
	$f_{(0;0;1;1)}(t,x),$ $f_{(0;0;1;0)}(t,x)$ are smooth functions satisfying
	$|\nabla^{I}f_{(0;0;1;1)}(t,x)| \lesssim (1 + t + |x|)^{-2 + |I|},$ 
	$|\nabla^{I}f_{(0;0;1;0)}(t,x)| \lesssim M(1 + t + |x|)^{-3 + |I|}$ for any $\nabla-$multi-index $I.$
	Let us accept the claim \eqref{E:SolveforPartialtV} for now; we will briefly discuss the derivation of 
	\eqref{E:SolveforPartialtV} at the end of the proof. We also note that
	
	\begin{align} 
		\partial_t V & = \partial_t W^{(1)} + \Pi_1 W^{(1)} + \partial_t V^{(0)},  
			\label{E:PartialtVintermsofPartialtWandW} \\
		\unabla V & = \unabla W^{(1)} + \Pi_2 W^{(1)} + \unabla V^{(0)}, 
			\label{E:UnderlineNablaVintermsofUnderlineNablaWandW}	
	\end{align}
	where $\Pi_1 W^{(1)} \eqdef (\partial_t h^{(1)}, 0, 0, 0, 0),$
	$\Pi_2 W^{(1)} \eqdef (\unabla h^{(1)}, 0, 0, 0, 0),$ and $V^{(0)}(t,x)$ satisfies \\
	$|\nabla^{I} \partial_t V^{(0)}(t,x)|$ $+ |\nabla^{I} \unabla V^{(0)}(t,x)|$ $\lesssim  
	(1 + t + |x|)^{-2 + |I|}$ for any $\nabla-$multi-index $I$ (see Lemma \ref{L:h0decayestimates}).
	Now with the help of \eqref{E:PartialtVintermsofPartialtWandW} - \eqref{E:UnderlineNablaVintermsofUnderlineNablaWandW}, the
	chain rule, and the Leibniz rule, we repeatedly partially differentiate \eqref{E:SolveforPartialtV} with
	respect to time and spatial derivatives, using the resulting equations to replace time derivatives with spatial derivatives,
	thereby inductively arriving at an expression of the form \eqref{E:PartialtkunderlinenablaJW1inTermsofInstrinsic} featuring
	the properties (i) - (iii). The properties (ii) - (iii) capture the fact that each additional differentiation of 
	$\partial_t W^{(1)}$ either a) creates an additional decay factor of $(1 + t+ |x|)^{-1}$ \big(when the derivative falls on 
	one of the $f_{\cdots}(t,x)$\big); b) increases one of the powers $|I_j|$ (when the derivative is spatial and falls on one of 
	the multilinear factors $[\cdots, \unabla^{I_j} W^{(1)}, \cdots]);$ or c) increases $s$ by one (when the derivative falls on
	$\mathscr{M}(V),$ thereby creating an additional multi-linear factor of $\nabla W^{(1)}$ via the chain rule).

We now return to the issue of expressing $\partial_t W^{(1)}$ in the form \eqref{E:partialtW1isaSumofTerms} - \eqref{E:SolveforPartialtV}. We will make repeated use of the splitting $h = h^{(0)} + h^{(1)},$ where $h^{(0)}$ is the smooth function of $(t,x)$ with the decay properties \eqref{E:nablaIh0Linfinity}, which are proved in Section \ref{SS:PreliminaryLinfinityEstimates}. We first note that $\partial_t \Electricfield$ and $\partial_t \Magneticinduction$ can be expressed in the desired form using \eqref{E:partialtBisolated} - \eqref{E:partialtEisolated}, together with the splitting of $h$ and the properties \eqref{E:nablaIh0Linfinity}. Next, the quantities $\partial_t \unabla h_{\mu \nu}^{(1)}$ can be expressed in the desired form through the trivial identity $\partial_t \unabla h_{\mu \nu}^{(1)} = \unabla \partial_t h_{\mu \nu}^{(1)}.$ The quantities $\partial_t^2 h_{\mu \nu}^{(1)}$ can be expressed in the desired form by using equation \eqref{E:Reducedh1Summary} to isolate them. We remark that the $\mathscr{M}_{I;0;1;1}(V)[\unabla^I W^{(1)}]$ term on the right-hand side of \eqref{E:SolveforPartialtV} arises from the spatial derivatives and mixed space-time derivatives of $h^{(1)}$ contained in the term $\widetilde{\Square}_{g} h_{\mu \nu}^{(1)}$ on the left-hand side of \eqref{E:Reducedh1Summary}. Furthermore, the $\mathscr{M}_{0;0;1;2}(V)[W^{(1)}, W^{(1)}]$ term on the right-hand side of \eqref{E:SolveforPartialtV} arise from the quadratic and higher-order-in $W^{(1)}$ terms on the right-hand sides of \eqref{E:Reducedh1Summary} and \eqref{E:partialtEisolated}, while the $f_{0;0;1;1}(t,x) \mathscr{M}_{0;0;1;1}(V)[W^{(1)}]$ term on the right-hand side of \eqref{E:SolveforPartialtV} arises from the $h^{(0)}$ and $\nabla h^{(0)}-$containing factors that arise from the terms on the right-hand sides of \eqref{E:Reducedh1Summary} and \eqref{E:partialtEisolated} that contain a linear factor of $h$ or $\nabla h.$ Finally, the $f_{0;0;1;0}(t,x)\mathscr{M}_{0;0;1;0}(V)$ term on the right-hand side of \eqref{E:SolveforPartialtV} arises from the $\widetilde{\Square}_{g} h_{\mu \nu}^{(0)}$ term on the right-hand side of \eqref{E:Reducedh1Summary}, and from the $O(|\nabla h^{(0)}|^2)$ terms arising from splitting the $O(|\nabla h|^2)$ terms on the right-hand side of \eqref{E:Reducedh1Summary}. 

\end{proof}

\begin{corollary} \label{C:WeightedL2PartialtkunderlinenablaJW1inTermsofInstrinsic}
	Assume the hypotheses of Proposition \ref{P:SmallNormImpliesSmallEnergy}, 
	including the smallness condition $E_{\dParameter;\upgamma}(0) + M \leq \varepsilon.$
	Let $k\geq 0$ be an integer, let $J$ be a $\unabla$ multi-index, and 
	assume that $|J| + k \leq \dParameter .$ Let $V(t,x),$ $\cdots, W^{(1)}(t,x)$ be the array-valued functions defined in \eqref{E:Vdef} - 
	\eqref{E:W1def}, let $\mathring{V}(x) = V(0,x),$ $\cdots, \mathring{W}^{(1)}(x) = W^{(1)}(0,x),$
	and assume that $\| \mathring{V}^{(1)} \|_{L^{\infty}} + \| \mathring{W}^{(1)} \|_{H_{1/2 + \upgamma}^{\dParameter}} \leq \varepsilon.$
	Then if $\varepsilon$ is sufficiently small, the following inequality holds:

	\begin{align} \label{E:WeightedL2PartialtkunderlinenablaJW1inTermsofInstrinsic}
		\big\| (1 + |x|)^{(1/2 + \upgamma + |J| + k)} \unabla^J \partial_t^k W^{(1)}(0,x) \big\|_{L^2}
		& \lesssim \| \mathring{W}^{(1)} \|_{H_{1/2 + \upgamma}^{\dParameter}} + M.
	\end{align}
	
\end{corollary}

\begin{proof}
	
	Let us first consider the case $s=0$ in \eqref{E:PartialtkunderlinenablaJW1inTermsofInstrinsic}.
	Then using that $|F_{(0;J;k;0)}(t,x)| \lesssim M (1 + |x|)^{-(3 + |J| + k)}$ (i.e., property (ii) from Lemma 
	the conclusions of \ref{L:SolveforTimeDerivativesinTermsofInherentDerivatives}) and recalling that $0 < \upgamma < 1/2,$ it 
	follows that 
	
	\begin{align} \label{E:W0WeightedL2}
		\big\| & (1 + |x|)^{1/2 + \upgamma} F_{(0;J;k;0)}(0,x)\mathscr{M}_{(0;J;k;0)}\big(\mathring{V}(x)\big) \big\|_{L^2}^2 \\
		& = \int_{x \in \mathbb{R}^3} 
		(1 + |x|)^{1 + 2 \upgamma} |F_{(0;J;k;0)}(0,x) \mathscr{M}_{(0;J;k;0)}\big(\mathring{V}(x)\big)|^2 \, d^3 x \notag \\
		& \lesssim M^2 \int_{x \in \mathbb{R}^3} (1 + |x|)^{-4} \, d^3 x \lesssim M^2. \notag
	\end{align}

	For the case $s \geq 1,$ we first use Proposition \ref{P:SobolevEmbeddingHNdeltaCNprimedeltamprime} to deduce that
	for all $\unabla-$indices $K$ with $|K| \leq \dParameter -2,$ we have
	
	\begin{align} 
		|\unabla^K \mathring{W}^{(1)}(x)| & \lesssim (1 + |x|)^{-(|K| + 1)} 
			\| \mathring{W}^{(1)} \|_{H_{1/2 + \upgamma}^{|K| + 2}}.  
			\label{E:NablaIW1WeightedSobolevEmbeddingPointwiseBound}
	\end{align}
	Then (without loss of generality assuming $|I_1| \leq |I_2| \leq \cdots \leq |I_s$) we use \\
	$|F_{(I_1, \cdots, I_s;J;k;s)}(t,x)| \lesssim (1 + t + |x|)^{-\big(|J| + k - (|I_1| + \cdots + |I|_s) - 
	(s-1)\big)}$ (i.e., property (iii)), together with \eqref{E:NablaIW1WeightedSobolevEmbeddingPointwiseBound}, to deduce
	
	\begin{align} \label{E:CrucialWeightedL2EstimateforProductofDifferentiatedW1Terms}
		\big\| & (1 + |x|)^{1/2 + \upgamma + |J| + k} F_{(I_1, \cdots, I_s;J;k;s)}(0,x)
			\mathscr{M}_{(I_1, \cdots, I_s;J;k;s)}\big(\mathring{V}(x)\big)[\unabla^{I_1}\mathring{W}^{(1)}(x), 
			\cdots, \unabla^{I_s} \mathring{W}^{(1)}(x)] \big\|_{L^2} \\
		& \lesssim \Big\lbrace 
			\prod_{i=1}^{s-1} \big\| (1 + |x|)^{|I_1| + \cdots + |I_{s-1}| + (s-1)} \unabla^{I_i} 
			\mathring{W}^{(1)}(x) \big\|_{L^{\infty}} 
			\Big\rbrace \big\| (1 + |x|)^{1/2 + \upgamma + |I_s|} \unabla^{I_s} \mathring{W}^{(1)}(x) \big\|_{L^2} \notag \\
		& \lesssim  \big\| (1 + |x|)^{1/2 + \upgamma + |I_s|} \unabla^{I_s} \mathring{W}^{(1)}(x) \big\|_{L^2}
			\lesssim \| \mathring{W}^{(1)} \|_{H_{1/2 + \upgamma}^{\dParameter}}. \notag 
	\end{align}
	Combining \eqref{E:W0WeightedL2} and \eqref{E:CrucialWeightedL2EstimateforProductofDifferentiatedW1Terms},
	we arrive at \eqref{E:WeightedL2PartialtkunderlinenablaJW1inTermsofInstrinsic}.
	
\end{proof}

We are now ready for the proof of the proposition.
	
\textbf{Proof of Proposition \ref{P:SmallNormImpliesSmallEnergy}:}	
	We first remark that \emph{the estimates derived in this proof are valid under the assumption that $\varepsilon$ is 
	sufficiently small}. Recall that $g_{\mu \nu}(t,x) = m_{\mu \nu} + \chi\big(\frac{r}{t}\big) \chi(r) \frac{2M}{r} \delta_{\mu \nu} 
	+ h_{\mu \nu}^{(1)}(t,x).$ Also recall that according to the assumptions of the proposition, 
	
	\begin{subequations}
	\begin{align} \label{E:InitialSpacetimeh1inTermsofInstrinsich1}
		h^{(1)}(0,x) & = \begin{pmatrix}
                        0 & 0 & 0 & 0 \\
                        0 & \mathring{\underline{h}}_{11}^{(1)} & \mathring{\underline{h}}_{12}^{(1)}
                        	& \mathring{\underline{h}}_{13}^{(1)} \\
                        0 & \mathring{\underline{h}}_{21}^{(1)} & \mathring{\underline{h}}_{22}^{(1)}
                        	& \mathring{\underline{h}}_{23}^{(1)} \\
                        0 & \mathring{\underline{h}}_{31}^{(1)} & \mathring{\underline{h}}_{32}^{(1)}
                        	& \mathring{\underline{h}}_{33}^{(1)} 
                    \end{pmatrix}, \\
 		\partial_t h^{(1)}(0,x) 
 			& = \begin{pmatrix}
     		2A^3 (\mathring{\underline{g}}^{-1})^{ab} \mathring{K}_{ab} & A^2 (\mathring{\underline{g}}^{-1})^{ab} \partial_a \mathring{\underline{g}}_{b1} 
       		& A^2 (\mathring{\underline{g}}^{-1})^{ab} \partial_a \mathring{\underline{g}}_{b2} & A^2 
       		(\mathring{\underline{g}}^{-1})^{ab} \partial_a \mathring{\underline{g}}_{b3} \\
      	A^2 (\mathring{\underline{g}}^{-1})^{ab} \partial_a \mathring{\underline{g}}_{b1} & 2A \mathring{K}_{11} 
      		& 2A \mathring{K}_{12} & 2A \mathring{K}_{13} \\
      	A^2 (\mathring{\underline{g}}^{-1})^{ab} \partial_a \mathring{\underline{g}}_{b2} & 2A \mathring{K}_{21} 
      		& 2A \mathring{K}_{22} & 2A \mathring{K}_{23} \\
      	A^2 (\mathring{\underline{g}}^{-1})^{ab} \partial_a \mathring{\underline{g}}_{b3} & 2A \mathring{K}_{31} 
      		& 2A \mathring{K}_{32} & 2A \mathring{K}_{33}
   \end{pmatrix},  \label{E:InitialpartialtSpacetimeh1inTermsofInstrinsic}
	\end{align}
	\end{subequations}
	where $A(x) = \sqrt{1- 2M \chi(r)/r},$ and 
	$\mathring{\underline{g}}_{jk}(x) = \delta_{jk} + 2M \chi(r)/r \delta_{jk} + \mathring{\underline{h}}_{jk}^{(1)}(x).$
	Note that $(\mathring{\underline{g}}^{-1})^{jk} = \delta^{jk} + O^{\infty}(|M \chi(r)/r|;\mathring{\underline{h}}^{(1)}) 
	+ O^{\infty}(|\mathring{\underline{h}}^{(1)}|;M \chi(r)/r).$ Our immediate objectives are to relate
	$\| \partial_t h^{(1)}(0,\cdot) \|_{H_{1/2 + \upgamma}^{\dParameter}}$ and $\| \mathring{\Electricfield} \|_{H_{1/2 + \upgamma}^{\dParameter}}$
	to the inherent quantities $\| \mathring{\underline{h}} \|_{H_{1/2 + \upgamma}^{\dParameter}},$ 
	$\| \mathring{K} \|_{H_{1/2 + \upgamma}^{\dParameter}},$ 
	$\| \mathring{\mathfrak{\Displacement}} \|_{H_{1/2 + \upgamma}^{\dParameter}},$ 
	$\| \mathring{\mathfrak{\Magneticinduction}} \|_{H_{1/2 + \upgamma}^{\dParameter}},$ and $M.$
	To this end, we first observe that the following estimates hold for sufficiently small $M:$ 
	
	\begin{align}
		|\underline{\nabla}^I \big(M \frac{\chi(r)}{r}\big)| & \lesssim M(1 + r)^{-(1 + |I|)}, \label{E:unablaASchwarzschildTailDecayEstimates} && \\
		|A(x)| & \lesssim 1, && \\
		|\underline{\nabla}^I A(x)| & \lesssim M(1 + r)^{-(1 + |I|)}, && |I| \geq 1. \label{E:unablaADecayEstimates}
	\end{align}
	Using \eqref{E:InitialSpacetimeh1inTermsofInstrinsich1} - \eqref{E:InitialpartialtSpacetimeh1inTermsofInstrinsic}, the decay 
	estimates \eqref{E:unablaASchwarzschildTailDecayEstimates} - \eqref{E:unablaADecayEstimates},
	the Leibniz rule, Corollary \ref{C:CompositionProductHNdelta},
	the definition of $\| \cdot \|_{H_{1/2 + \upgamma}^{\dParameter}},$ the fact that 
	$0 < \upgamma < 1/2,$ and elementary calculations,
	it is easy to check that
	
	\begin{align} \label{E:PartialtSpacetimeh1inTermsofInstrinsicK}
		\| \partial_t h(0,\cdot) \|_{H_{1/2 + \upgamma}^{\dParameter}} \lesssim \| \mathring{\underline{h}}^{(1)} \|_{H_{1/2 + \upgamma}^{\dParameter}}
		+ \| \mathring{K} \|_{H_{1/2 + \upgamma}^{\dParameter}} + M.
	\end{align}
	Furthermore, by  \eqref{E:DintermsofEBh}, \eqref{E:EintermsofDBh}, and Corollary \ref{C:CompositionProductHNdelta}, 
	we have that
	
	\begin{align} \label{E:InitialElectricMagneticWeightedSobolevinTermsofInitialDisplacementMagnetic}
		\| \mathring{\Electricfield} \|_{H_{1/2 + \upgamma}^{\dParameter}} + \| \mathring{\Magneticinduction} \|_{H_{1/2 + \upgamma}^{\dParameter}}
		& \approx \| \mathring{\Displacement} \|_{H_{1/2 + \upgamma}^{\dParameter}} 
		+ \| \mathring{\Magneticinduction} \|_{H_{1/2 + \upgamma}^{\dParameter}}.
	\end{align}
	Similarly, by we have that
	\begin{align} \label{E:InitialIntrinsicDisplacementMagneticintermsofInitialDisplacementMagnetic}
		\| \mathring{\Displacement} \|_{H_{1/2 + \upgamma}^{\dParameter}} + \| \mathring{\Magneticinduction} \|_{H_{1/2 + \upgamma}^{\dParameter}}
		& \approx \| \mathring{\mathfrak{\Displacement}} \|_{H_{1/2 + \upgamma}^{\dParameter}} 
		+ \| \mathring{\mathfrak{\Magneticinduction}} \|_{H_{1/2 + \upgamma}^{\dParameter}}.
	\end{align}
	By \eqref{E:PartialtSpacetimeh1inTermsofInstrinsicK}, 
	\eqref{E:InitialElectricMagneticWeightedSobolevinTermsofInitialDisplacementMagnetic},
	\eqref{E:InitialIntrinsicDisplacementMagneticintermsofInitialDisplacementMagnetic},
	and Proposition \ref{P:SobolevEmbeddingHNdeltaCNprimedeltamprime},
	it follows that if $E_{\dParameter;\upgamma}(0) + M$ is sufficiently small, then the smallness conditions\footnote{As in the
	Lindblad-Rodnianski proof of Corollary \ref{C:WeakDecay} below, the smallness condition 
	$|h^{(1)}(0,x)| \lesssim \varepsilon (1 + r)^{-1 - \upgamma}$ 
	follows from integrating the smallness condition $|\partial_r h^{(1)}(0,x)| \lesssim \varepsilon (1 + r)^{-2 - \upgamma},$
	which is a consequence of Proposition \ref{P:SobolevEmbeddingHNdeltaCNprimedeltamprime},
	from spatial infinity and using the decay assumption \eqref{E:h1AbstractDataAsymptotics} for 
	$|\mathring{h}^{(1)}(x)|$ at spatial infinity.}
	for $\|\mathring{V}^{(1)} \|_{L^{\infty}}$ and $\| \mathring{W}^{(1)} \|_{H_{1/2 + \upgamma}^{\dParameter}}$ 
	in the hypotheses of Lemma \ref{L:SolveforTimeDerivativesinTermsofInherentDerivatives} 
	and Corollary \ref{C:WeightedL2PartialtkunderlinenablaJW1inTermsofInstrinsic} hold. 
	Therefore, combining Corollaries \ref{C:InitialEnergyInTermsofTangentialandTimeDerivatives} and 
	\ref{C:WeightedL2PartialtkunderlinenablaJW1inTermsofInstrinsic}, \eqref{E:PartialtSpacetimeh1inTermsofInstrinsicK},
	\eqref{E:InitialElectricMagneticWeightedSobolevinTermsofInitialDisplacementMagnetic}, and
	\eqref{E:InitialIntrinsicDisplacementMagneticintermsofInitialDisplacementMagnetic},
	we deduce that if $\varepsilon$ is sufficiently small, then
	
	\begin{align}
		\mathcal{E}_{\dParameter;\upgamma;\upmu}^2(0) 
		& \lesssim \| \underline{\nabla} \mathring{\underline{h}}^{(1)} \|_{H_{1/2 + \upgamma}^{\dParameter}}^2 
			+ \| \partial_t h^{(1)}(0,\cdot) \|_{H_{1/2 + \upgamma}^{\dParameter}}^2 
			+ \| \mathring{\Electricfield} \|_{H_{1/2 + \upgamma}^{\dParameter}}^2 
			+ \| \mathring{\Magneticinduction} 	\|_{H_{1/2 + \upgamma}^{\dParameter}}^2 + M^2 \\
		& \lesssim \| \underline{\nabla} \mathring{\underline{h}}^{(1)} \|_{H_{1/2 + \upgamma}^{\dParameter}}^2 
			+ \| \mathring{K} \|_{H_{1/2 + \upgamma}^{\dParameter}}^2 
			+ \| \mathring{\mathfrak{\Displacement}} \|_{H_{1/2 + \upgamma}^{\dParameter}}^2 
			+ \| \mathring{\mathfrak{\Magneticinduction}} 	\|_{H_{1/2 + \upgamma}^{\dParameter}}^2 + M^2 \notag \\
		& \eqdef E_{\dParameter;\upgamma}^2(0) + M^2. \notag
	\end{align}
	This concludes our proof of Proposition \ref{P:SmallNormImpliesSmallEnergy}. \hfill $\qed$

\section{Algebraic Estimates of the Nonlinearities} \label{S:AlgebraicEstimates}

In this section, we provide algebraic estimates for the inhomogeneous terms that arise from
differentiating the reduced equations \eqref{E:Reducedh1Summary} - \eqref{E:ReduceddMis0Summary}. We also 
use the equations of Proposition \ref{P:EOVNullDecomposition} to derive ordinary differential inequalities
for the null components of $\dot{\Far} = \Lie_{\mathcal{Z}}^I \Far.$ Furthermore, we provide algebraic estimates
for the inhomogeneous terms appearing on the right-hand sides of these inequalities. Many of the estimates derived in this section rely on the wave coordinate condition.

\noindent \hrulefill
\ \\

\subsection{Statement and proofs of the propositions}

The proofs of the propositions given in this section use the results of a collection of technical lemmas, which 
we relegate to the end of the section. We begin by quoting the following proposition proved in \cite{hLiR2010},
which is central to many of the estimates. The basic idea is the following: many of our estimates for coupled quantities would break down if we could not achieve good control of the components $h_{LL}$ and $h_{LT}.$ Amazingly, as shown in \cite{hLiR2005} and \cite{hLiR2010}, the wave coordinate condition allows for \emph{independent, improved} estimates of exactly these components.

\begin{proposition} \cite[Proposition 8.2]{hLiR2010} \label{P:harmonicgauge} 
	\textbf{(Algebraic consequences of the wave coordinate condition)}
	Let $g$ be a Lorentzian metric satisfying the wave coordinate condition \eqref{E:wavecoordinategauge1} relative to the 
	coordinate system $\lbrace x^{\mu} \rbrace_{\mu=0,1,2,3}.$ Let $I$ be a $\mathcal{Z}-$multi-index, 
	assume that $|\nabla_{\mathcal{Z}}^J h| \leq \varepsilon$ holds for all $\mathcal{Z}-$multi-indices $J$ satisfying $|J| \leq 
	\lfloor |I|/2 \rfloor,$ where $h_{\mu \nu} \eqdef g_{\mu \nu} - m_{\mu \nu}.$ Then if $\varepsilon$ is sufficiently small,
	the following pointwise estimates hold for the tensor $H^{\mu \nu} \eqdef (g^{-1})^{\mu \nu} - (m^{-1})^{\mu \nu}:$
	
	\begin{subequations}
	\begin{align} \label{E:nablaZIHLTpointwiseEstimate}
		|\nabla\nabla_{\mathcal{Z}}^I H|_{\mathcal{L} \mathcal{T}} & \lesssim \sum_{|J| \leq |I|} 
		 |\conenabla \nabla_{\mathcal{Z}}^J H|
		 \ + \ \underbrace{\sum_{|J| \leq |I| - 1} |\nabla\nabla_{\mathcal{Z}}^J H|}_{\mbox{Absent if $|I|=0.$}}
			\ + \ \sum_{|I_1| + |I_2| \leq |I|} |\nabla_{\mathcal{Z}}^{I_1} H||\nabla\nabla_{\mathcal{Z}}^{I_2}H|, \\
	|\nabla\nabla_{\mathcal{Z}}^I H|_{\mathcal{L} \mathcal{L}} & \lesssim \sum_{|J| \leq |I|} 
		|\conenabla \nabla_{\mathcal{Z}}^J H|
			\ + \ \underbrace{\sum_{|J| \leq |I| - 2} |\nabla\nabla_{\mathcal{Z}}^{J}H|}_{\mbox{Absent if $|I| \leq 1.$}}
			\ + \ \sum_{|I_1| + |I_2| \leq |I|} |\nabla_{\mathcal{Z}}^{I_1} H||\nabla\nabla_{\mathcal{Z}}^{I_2}H|. 
			\label{E:nablaZIHLLpointwiseEstimate}
	\end{align}
	\end{subequations}
	
	Furthermore, analogous estimates hold for the tensor $h_{\mu \nu}.$ 
	
\end{proposition}

\hfill $\qed$

The next lemma provides an analogous version of the proposition for the ``remainder'' pieces of $(g^{-1})^{\mu \nu}$
and $g_{\mu \nu}.$

\begin{lemma} \label{L:NablaZIh1LLh1LTwaveCoordinateAlgebraicEstimate} \cite[Slight extension of Lemma 15.4]{hLiR2010}
	\textbf{(Algebraic/analytic consequences of the wave coordinate condition)}
	Let $g$ be a Lorentzian metric satisfying the wave coordinate condition \eqref{E:wavecoordinategauge1} relative to the 
	coordinate system $\lbrace x^{\mu} \rbrace_{\mu = 0,1,2,3},$ and let $H^{\mu \nu} \eqdef (g^{-1})^{\mu \nu} - (m^{-1})^{\mu 
	\nu}.$ Let $k \geq 0$ be an integer, and assume that there is a constant $\varepsilon$ such that $|\nabla_{\mathcal{Z}}^J h| 
	\leq \varepsilon$ holds for all $\mathcal{Z}-$multi-indices $J$ satisfying $|J| \leq k/2,$ where 
	$h_{\mu \nu} \eqdef g_{\mu \nu} - m_{\mu \nu}.$ Let 
	
	\begin{align} \label{E:NablaZIh1LLh1LTwaveCoordinateAlgebraicEstimate}
		H_{(1)}^{\mu \nu} \eqdef H^{\mu \nu} - H_{(0)}^{\mu \nu}, \qquad H_{(0)}^{\mu \nu} \eqdef 
			- \chi\big(\frac{r}{t}\big) \chi(r)
			\frac{2M}{r} \delta^{\mu \nu},
	\end{align}
	where $H_{(1)}^{\mu \nu}$ is the tensor obtained by subtracting the Schwarzschild part 
	$H_{(0)}^{\mu \nu}$ from $H^{\mu \nu},$ and let $\chi_0(1/2 < z < 3/4)$ denote the characteristic function of the interval 
	$[1/2,3/4].$ Assume further that $M \leq \varepsilon.$ Then if $\varepsilon$ is sufficiently small, the following
	pointwise estimates hold
	
	\begin{align}
		\sum_{|I| \leq k} |\nabla\nabla_{\mathcal{Z}}^I H_{(1)}|_{\mathcal{L} \mathcal{L}}
		\ + \ \sum_{|J| \leq k - 1} |\nabla\nabla_{\mathcal{Z}}^J H_{(1)}|_{\mathcal{L} \mathcal{T}} 
		& \lesssim \sum_{|I| \leq k} |\conenabla \nabla_{\mathcal{Z}}^I H_{(1)}| 
			 \\
		& \ \ + \ \varepsilon \sum_{|I| \leq k} (1 + t + |q|)^{-1} |\nabla\nabla_{\mathcal{Z}}^I H_{(1)}| 
			\ + \ \varepsilon \sum_{|I| \leq k} (1 + t + |q|)^{-2} |\nabla_{\mathcal{Z}}^I H_{(1)}| \notag \\
		& \ \ + \ \sum_{|I_1| + |I_2| \leq k} |\nabla_{\mathcal{Z}}^{I_1} H_{(1)}| |\nabla\nabla_{\mathcal{Z}}^{I_2} H_{(1)}|  
		 	\ + \ \underbrace{\sum_{|J'| \leq k - 2} |\nabla\nabla_{\mathcal{Z}}^{J'} H_{(1)}|}_{\mbox{Absent if $k \leq 1$}}   
			\notag \\
		& \ \ + \ \varepsilon (1 + t + |q|)^{-2} \chi_0(1/2 < r/t < 3/4) 
			\ + \ \varepsilon^2 (1 + t + |q|)^{-3}. \notag
	\end{align}
	
	Additionally, let
	\begin{align}
		h_{\mu \nu}^{(1)} \eqdef h_{\mu \nu} - h_{\mu \nu}^{(0)}, \qquad h_{\mu \nu}^{(0)} 
			\eqdef \chi\big(\frac{r}{t}\big) \chi(r) \frac{2M}{r} \delta_{\mu \nu},
	\end{align}
	where $h_{\mu \nu}^{(1)}$ is the tensorfield obtained by subtracting the Schwarzschild part $h_{\mu \nu}^{(0)}$ 
	from $h_{\mu \nu}.$ Then an estimate analogous to \eqref{E:NablaZIh1LLh1LTwaveCoordinateAlgebraicEstimate} holds if we 
	replace the tensorfield $H_{(1)}$ with the tensorfield $h^{(1)}.$ 
	
\end{lemma}

\begin{proof}
	The estimates for the tensorfield $H_{(1)}^{\mu \nu}$ were proved as \cite[Lemma 15.4]{hLiR2010}. The analogous estimates for 
	the tensorfield $h_{\mu \nu}^{(1)}$ follow from those for $H_{(1)}^{\mu \nu},$ together with the fact that
	$H_{(1); \mu \nu} = - h_{\mu \nu}^{(1)} + O^{\infty}(|h^{(0)} + h^{(1)}|^2)$ and the decay estimates for $h^{(0)}$ stated in 
	Lemma \ref{L:h0decayestimates}.
\end{proof}

We now turn to the following proposition, which captures the algebraic structure of the inhomogeneous term $\mathfrak{H}_{\mu \nu}$ appearing on the right-hand side of the reduced equation \eqref{E:Reducedh1Summary}.

\begin{proposition} \label{P:AlgebraicInhomogeneous} 
\cite[Extension of Proposition 9.8]{hLiR2010}
	\textbf{(Algebraic estimates of $\mathfrak{H}_{\mu \nu}$ and $\nabla_{\mathcal{Z}}^I \mathfrak{H}_{\mu \nu}$)}
	Let $\mathfrak{H}_{\mu \nu}$ be the inhomogeneous term on the right-hand side of the reduced equation 
	\eqref{E:Reducedh1Summary}, and assume that the wave coordinate condition \eqref{E:wavecoordinategauge1} holds. Then
	
	\begin{subequations}
	\begin{align}
		|\mathfrak{H}|_{\mathcal{T} \mathcal{N}} & \lesssim |\conenabla h||\nabla h| 
			\ + \ \big(|\Far|_{\mathcal{L}\mathcal{N}} + |\Far|_{\mathcal{T}\mathcal{T}} \big)|\Far|
			 \ + \ O^{\infty}(|h||\nabla h|^2) \ + \ O^{\dParameter+1}(|h||\Far|^2) \ + \ O^{\dParameter+1}|\Far|^3;h), \label{E:InhomogeneousHTUAlgebraic} \\
		|\mathfrak{H}| & \lesssim |\nabla h|_{\mathcal{T} \mathcal{N}}^2 \ + \ |\conenabla h||\nabla h|
			\ + \ |\Far|^2 \ + \ O^{\infty}(|h||\nabla h|^2) \ + \ O^{\dParameter+1}(|h||\Far|^2) \ + \ O^{\dParameter+1}(|\Far|^3;h). \label{E:InhomogeneousHAlgebraic} 
	\end{align}

	In addition, assume that there exists an $\varepsilon > 0$ 
	such that $|\nabla_{\mathcal{Z}}^J h| + |\Lie_{\mathcal{Z}}^J \Far| \leq \varepsilon$ 
	holds for all $\mathcal{Z}-$multi-indices $|J| \leq \lfloor |I|/2 \rfloor.$ 
	Then if $\varepsilon$ is sufficiently small, the following pointwise estimates hold:

	\begin{align} \label{E:ZIinhomogeneoushpointwise}
		|\nabla_{\mathcal{Z}}^I \mathfrak{H}| 
		& \lesssim \sum_{|I_1| + |I_2| \leq |I|}
			\Big\lbrace |\nabla\nabla_{\mathcal{Z}}^{I_1} h|_{\mathcal{T} \mathcal{N}} 
			|\nabla\nabla_{\mathcal{Z}}^{I_2} h |_{\mathcal{T} \mathcal{N}} 
			\ + \ |\conenabla \nabla_{\mathcal{Z}}^{I_1}h | |\nabla\nabla_{\mathcal{Z}}^{I_2} h| \Big\rbrace \\
		& \ \ + \ \sum_{|I_1| + |I_2| \leq |I|} |\Lie_{\mathcal{Z}}^{I_1} \Far| |\Lie_{\mathcal{Z}}^{I_2} \Far|
			\ + \ \underbrace{\sum_{|I_1| + |I_2| \leq |I| - 2} |\nabla\nabla_{\mathcal{Z}}^{I_1} h|
			|\nabla\nabla_{\mathcal{Z}}^{I_2} h|}_{\mbox{Absent if $|I| \leq 1$}} 	
			\notag \\
		& \ \ + \ \sum_{|I_1| + |I_2| + |I_3| \leq |I|} 
			\Big\lbrace|\nabla_{\mathcal{Z}}^{I_1} h| |\nabla\nabla_{\mathcal{Z}}^{I_2} h| |\nabla\nabla_{\mathcal{Z}}^{I_3}h| 
			\ + \ |\nabla_{\mathcal{Z}}^{I_1} h||\Lie_{\mathcal{Z}}^{I_2}\Far| |\Lie_{\mathcal{Z}}^{I_3} \Far| 
			\ + \ |\Lie_{\mathcal{Z}}^{I_1} \Far||\Lie_{\mathcal{Z}}^{I_2}\Far| |\Lie_{\mathcal{Z}}^{I_3} \Far| \Big\rbrace. \notag 
\end{align}
\end{subequations}
\end{proposition}

\begin{proof}
	Using \eqref{E:ReducedhInhomogeneous}, we can decompose $\mathfrak{H}_{\mu \nu}$ 
	into
	
	\begin{align}
		\mathfrak{H}_{\mu \nu} & = (i)_{\mu \nu} + (ii)_{\mu \nu} + (iii)_{\mu \nu} + (iv)_{\mu \nu}, 
	\end{align}
	where
	
	\begin{align}
		(i)_{\mu \nu} & \eqdef \mathscr{P}(\nabla_{\mu} h, \nabla_{\nu} h),   
			\label{E:PieceiReducedhInhomogeneous} \\
		(ii)_{\mu \nu} & \eqdef \mathscr{Q}_{\mu \nu}^{(1;h)}(\nabla h, \nabla h), \\
		(iii)_{\mu \nu} & \eqdef \mathscr{Q}_{\mu \nu}^{(2;h)}(\Far, \Far), 
			\label{E:PieceiiiReducedhInhomogeneous} \\
		(iv)_{\mu \nu} & \eqdef O^{\infty}(|h||\nabla h|^2) \ + \ O^{\dParameter+1}(|h||\Far|^2) \ + \ O^{\dParameter+1}(|\Far|^3;h).
			\label{E:PieceivReducedhInhomogeneous}
	\end{align}
	We will analyze each of the four pieces separately.
	
	The facts that $|(i)|_{\mathcal{T} \mathcal{N}} \lesssim$ the 
	right-hand side of \eqref{E:InhomogeneousHTUAlgebraic} and that $|(i)| \lesssim$ the right-hand side of 
	\eqref{E:InhomogeneousHAlgebraic} follow from 
	Proposition \ref{P:harmonicgauge}, \eqref{E:PSpecialNullStructure}, and
	\eqref{E:PTUSpecialNullStructure}. The fact that $|\nabla_{\mathcal{Z}}^I(i)| \lesssim$ 
	the right-hand side of \eqref{E:ZIinhomogeneoushpointwise} follows from Proposition \ref{P:harmonicgauge},
	\eqref{E:SpecialPLeibnizRule}, and \eqref{E:PSpecialNullStructure}.
	
	The facts that $|(ii)|_{\mathcal{T} \mathcal{N}} \lesssim$ the right-hand side of 
	\eqref{E:InhomogeneousHTUAlgebraic}, and that $|(ii)| \lesssim$ the right-hand side of 
	\eqref{E:InhomogeneousHAlgebraic} both follow from \eqref{E:Q1hNullFormEstimate}.
	That $|\nabla_{\mathcal{Z}}^I(ii)| \lesssim$
	the right-hand side of \eqref{E:ZIinhomogeneoushpointwise} follows from 
	\eqref{E:Q1hLeibnizRule} and \eqref{E:Q1hNullFormEstimate}.
	
	The fact that $|(iii)|_{\mathcal{T} \mathcal{N}} \lesssim$ the right-hand side of 
	\eqref{E:InhomogeneousHTUAlgebraic} follows from \eqref{E:Q2TUhNullFormEstimate}, while the fact that 
	$|(iii)| \lesssim$ the right-hand side of \eqref{E:InhomogeneousHAlgebraic} follows from 
	\eqref{E:Q2hNullFormEstimate}. The fact that $|\nabla_{\mathcal{Z}}^I(iii)| \lesssim$ the right-hand side of 
	\eqref{E:ZIinhomogeneoushpointwise} follows from \eqref{E:LieZIinTermsofNablaZI},
	\eqref{E:Q2hLeibnizRule}, and \eqref{E:Q2hNullFormEstimate}.

	The desired estimates for term $(iv)$ follow easily with the help of the Leibniz rule and \eqref{E:LieZIinTermsofNablaZI}.

\end{proof}

The next proposition captures the special algebraic structure of the reduced inhomogeneous term $\mathfrak{F}_{(I)}^{\nu}$ defined in \eqref{E:LiemodZIdifferentiatedEOVInhomogeneousterms}.

\begin{proposition} \label{P:EnergyInhomogeneousTermAlgebraicEstimate}
	\textbf{(Algebraic estimates of $\mathfrak{F}_{(I)}^{\nu}$)}
	Let $\mathfrak{F}^{\nu}$ be the inhomogeneous term \eqref{E:EMBIFarInhomogeneous} in the reduced electromagnetic
	equations, let $I$ be a $\mathcal{Z}-$multi-index with $|I|=k,$
	and let $X_{\nu}$ be any covector. In addition, assume that there exists an $\varepsilon > 0$ 
	such that $|\nabla_{\mathcal{Z}}^J h| + |\Lie_{\mathcal{Z}}^J \Far| \leq \varepsilon$ 
	holds for all $\mathcal{Z}-$multi-indices $|J| \leq \lfloor k/2 \rfloor.$ 
	Then if $\varepsilon$ is sufficiently small, the following pointwise estimates hold:
	
	\begin{subequations}
	\begin{align} \label{E:LieZIFarNullFormInhomogeneousTermAlgebraicEstimate}
		|X_{\nu} \Liemod_{\mathcal{Z}}^I \mathfrak{F}^{\nu}|
		& \lesssim \sum_{|I_1| + |I_2| \leq k} |X| 
			|\conenabla \nabla_{\mathcal{Z}}^{I_1} h| |\Lie_{\mathcal{Z}}^{I_2} \Far|
			\ + \ \sum_{|I_1| + |I_2| \leq k} |X||\nabla\nabla_{\mathcal{Z}}^{I_1} h| 
			\big(|\Lie_{\mathcal{Z}}^{I_2} \Far|_{\mathcal{L} \mathcal{N}} + |\Lie_{\mathcal{Z}}^{I_2} \Far|_{\mathcal{T} 
			\mathcal{T}}\big) \\
		& \ \ + \ \sum_{|I_1| + |I_2| + |I_3| \leq k} 
			|X||\nabla_{\mathcal{Z}}^{I_1} h| |\nabla\nabla_{\mathcal{Z}}^{I_1} h| 
			|\Lie_{\mathcal{Z}}^{I_3} \Far|
			\ + \ \sum_{|I_1| + |I_2| + |I_3| \leq k} 
			|X||\nabla\nabla_{\mathcal{Z}}^{I_1} h| |\Lie_{\mathcal{Z}}^{I_2} \Far| 
			|\Lie_{\mathcal{Z}}^{I_3} \Far| \notag \\
		& \lesssim 
			(1 + t + |q|)^{-1} \mathop{\sum_{|I_1| + |I_2| \leq k + 1}}_{|I_2| \leq k} |X| 
			|\nabla_{\mathcal{Z}}^{I_1} h| |\Lie_{\mathcal{Z}}^{I_2} \Far|
			\ + \ (1 + |q|)^{-1} \mathop{\sum_{|I_1| + |I_2| \leq k + 1}}_{|I_2| \leq k} |X| |\nabla_{\mathcal{Z}}^{I_1} h| 
			\big(|\Lie_{\mathcal{Z}}^{I_2} \Far|_{\mathcal{L} \mathcal{N}} + |\Lie_{\mathcal{Z}}^{I_2} \Far|_{\mathcal{T} 
			\mathcal{T}}\big)  \notag \\
		& \ \ + \ (1 + |q|)^{-1} \mathop{\sum_{|I_1| + |I_2| + |I_3| \leq k + 1}}_{|I_2|, |I_3| \leq k} 
			|X| \Big\lbrace |\nabla_{\mathcal{Z}}^{I_1} h| |\Lie_{\mathcal{Z}}^{I_2} h| |\Lie_{\mathcal{Z}}^{I_3} \Far|
			\ + \ |\nabla_{\mathcal{Z}}^{I_1} h| |\Lie_{\mathcal{Z}}^{I_2} \Far| 
			|\Lie_{\mathcal{Z}}^{I_3} \Far| \Big\rbrace. \notag
	\end{align}
	In addition, the same estimates hold for $|X_{\nu} \Lie_{\mathcal{Z}}^I \mathfrak{F}^{\nu}|.$
	
	Furthermore, let $N^{\# \mu \nu \kappa \lambda}$ be the tensorfield from 
	the reduced electromagnetic equation \eqref{E:ReduceddMis0Summary}. 
	Then if $\varepsilon$ is sufficiently small and $k \geq 1,$ the following pointwise commutator estimate holds:
	
	\begin{align} \label{E:EnergyInhomogeneousTermAlgebraicEstimate}
		\Big|X_{\nu} \Big\lbrace & N^{\# \mu \nu \kappa \lambda}\nabla_{\mu} \Lie_{\mathcal{Z}}^I\Far_{\kappa \lambda}
		- \Liemod_{\mathcal{Z}}^I \big(N^{\# \mu \nu \kappa \lambda}\nabla_{\mu}\Far_{\kappa \lambda}\big) \Big\rbrace \Big|	\\
		& \lesssim (1 + |q|)^{-1} \mathop{\sum_{|I'| = k}}_{|J| \leq 1} 
				|X| |\nabla_{\mathcal{Z}}^{I'} h|_{\mathcal{L}\mathcal{L}} |\Lie_{\mathcal{Z}}^J \Far| 
			\ \ + \ (1 + |q|)^{-1} \mathop{\sum_{|J| \leq 1}}_{|I'| = k} 
				|X| |\nabla_{\mathcal{Z}}^J h|_{\mathcal{L}\mathcal{L}} |\Lie_{\mathcal{Z}}^{I'} \Far| \notag \\
		& \ \ + \ (1 + |q|)^{-1} \sum_{|I'| = k} |X||h|_{\mathcal{L}\mathcal{T}} |\Lie_{\mathcal{Z}}^{I'} \Far| 
			\ \ + \ (1 + |q|)^{-1} \mathop{\sum_{|I_1| + |I_2| \leq k + 1}}_{|I_1|, |I_2| \leq k} |X| |\nabla_{\mathcal{Z}}^{I_1} h|
		 		\big(|\Lie_{\mathcal{Z}}^{I_2} \Far|_{\mathcal{L} \mathcal{N}} +  
		 		|\Lie_{\mathcal{Z}}^{I_2} \Far|_{\mathcal{T} \mathcal{T}}\big) \notag \\
		& \ \ + \ (1 + t + |q|)^{-1} \mathop{\sum_{|I_1| + |I_2| \leq k + 1}}_{|I_1|, |I_2| \leq k}
			|X||\nabla_{\mathcal{Z}}^{I_1} h| |\Lie_{\mathcal{Z}}^{I_2} \Far|
			\ + \ (1 + |q|)^{-1} \mathop{\sum_{|I_1| + |I_2| \leq k + 1}}_{|I_1|, |I_2| \leq k}
				|X|_{\mathcal{L}}|\nabla_{\mathcal{Z}}^{I_1} h| |\Lie_{\mathcal{Z}}^{I_2} \Far| \notag \\
		& \ \ + \ (1 + |q|)^{-1} \mathop{\sum_{|I_1| + |I_2| \leq k + 1}}_{|I_1| \leq k - 1, |I_2| \leq k - 1} 
				|X||\nabla_{\mathcal{Z}}^{I_1} h|_{\mathcal{L}\mathcal{L}} |\Lie_{\mathcal{Z}}^{I_2} \Far| 
			\ \ + \ (1 + |q|)^{-1} \mathop{\sum_{|I_1| + |I_2| \leq k}}_{|I_1| \leq k - 1, |I_2| \leq k - 1} 
				|X||\nabla_{\mathcal{Z}}^{I_1} h|_{\mathcal{L}\mathcal{T}} |\Lie_{\mathcal{Z}}^{I_2} \Far| \notag \\
		& \ \ + \ (1 + |q|)^{-1} \underbrace{\mathop{\sum_{|I_1| + |I_2| \leq k - 1}}_{|I_1| \leq k - 2, |I_2| \leq k - 1} 
			|X||\nabla_{\mathcal{Z}}^{I_1} h| |\Lie_{\mathcal{Z}}^{I_2} \Far|}_{\mbox{absent if $k = 1$}} \notag \\
		& \ \ + \ (1 + |q|)^{-1} \mathop{\sum_{|I_1| + |I_2| + |I_3| \leq k + 1}}_{|I_1|, |I_2|, |I_3| \leq k} |X| 
			\Big\lbrace |\nabla_{\mathcal{Z}}^{I_1} h| |\nabla_{\mathcal{Z}}^{I_2} h| |\Lie_{\mathcal{Z}}^{I_3}\Far|
			\ + \ |\nabla_{\mathcal{Z}}^{I_1} h| |\Lie_{\mathcal{Z}}^{I_2} \Far| |\Lie_{\mathcal{Z}}^{I_3}\Far|
			\ + \ |\Lie_{\mathcal{Z}}^{I_1} \Far| |\Lie_{\mathcal{Z}}^{I_2} \Far| |\Lie_{\mathcal{Z}}^{I_3}\Far|
			\Big\rbrace. \notag
	\end{align}
	\end{subequations}
	
\end{proposition}

\begin{proof}
	Inequality \eqref{E:LieZIFarNullFormInhomogeneousTermAlgebraicEstimate} follows from
	\eqref{E:LieZIinTermsofNablaZI},
	\eqref{E:NablaLieZIinTermsofNablaNablaZI}, \eqref{E:TangentialDerivativesLieZIvsTangentialDerivativesNablaZI}
	(which allow us to estimate Lie derivatives of $h$ in terms of covariant derivatives of $h$), \eqref{E:LiemodZIFExpanded}, and 
	\eqref{E:Q2FarNullFormEstimate}. 
	
	Inequality \eqref{E:EnergyInhomogeneousTermAlgebraicEstimate} 
	follows from \eqref{E:NtriangleSmallAlgebraic}, 
	\eqref{E:LieZIinTermsofNablaZI} and \eqref{E:LieZILLinTermsofNablaZILLLieZJLTPlusJunk}
	(which allow us to estimate Lie derivatives of $h$ in terms of covariant derivatives of $h$), 
	\eqref{E:LiemodZINnablaFarCommutatorTerms},
	\eqref{E:XPFarhNablaFarNullFormEstimate}, and \eqref{E:XQ1FarhNablaFarNullFormEstimate}.

\end{proof}

As discussed at the beginning of Section \ref{SS:NullDecompElectromagnetic}, the null components of the lower-order
Lie derivatives of $\Far$ satisfy ordinary differential equations with controllable inhomogeneous terms. 
The next proposition provides convenient algebraic expressions for the inhomogeneities. In Section
\ref{S:DecayFortheReducedEquations}, these algebraic expressions will be combined with decay estimates
to deduce upgraded decay estimates for the null components of $\Far$ and its lower-order Lie derivatives.

\begin{proposition} \label{P:ODEsNullComponentsLieZIFar}
\textbf{(Ordinary differential inequalities for $\ualpha[\Lie_{\mathcal{Z}}^I \Far],$ $\alpha[\Lie_{\mathcal{Z}}^I \Far],$ 
 	$\rho[\Lie_{\mathcal{Z}}^I \Far],$ and $\sigma[\Lie_{\mathcal{Z}}^I \Far]$)}
	Let $\Far$ be a solution to the reduced electromagnetic equations \eqref{E:ReduceddFis0Summary} - 
	\eqref{E:ReduceddMis0Summary}, and let $\ualpha,$ $\alpha,$ $\rho,$ $\sigma$ denote its null components. Let $\Lambda \eqdef 
	L + \frac{1}{4} h_{LL} \uL,$ and assume that $|h| + |\Far| \leq \varepsilon$ holds. Then if $\varepsilon$ is sufficiently 
	small, the following pointwise estimate holds:
	
	\begin{align} \label{E:ODErualpha}
		r^{-1} \big|\nabla_{\Lambda} (r \ualpha) \big| 
		& \lesssim r^{-1} |h|_{\mathcal{L}\mathcal{L}} |\ualpha|
			\ + \ \sum_{|I| \leq 1} r^{-1} \big(|\Lie_{\mathcal{Z}}^I \Far|_{\mathcal{L} \mathcal{N}} 
				+ |\Lie_{\mathcal{Z}}^I \Far|_{\mathcal{T} \mathcal{T}} \big)
			\ + \ \sum_{|I_1| + |I_2| \leq 1}r^{-1} |\nabla_{\mathcal{Z}}^{I_1}h| |\Lie_{\mathcal{Z}}^{I_2} \Far| \\
		& \ \ + \ \sum_{|I| \leq 1}(1 + |q|)^{-1} |h| \big(|\Lie_{\mathcal{Z}}^I \Far|_{\mathcal{L} \mathcal{N}} 
		  	+ |\Lie_{\mathcal{Z}}^I \Far|_{\mathcal{T} \mathcal{T}} \big)  \notag \\
		& \ \ + \ \sum_{|I_1| + |I_2| + |I_3| \leq 1} (1 + |q|)^{-1} 
			\Big\lbrace |\nabla_{\mathcal{Z}}^{I_1} h| |\nabla_{\mathcal{Z}}^{I_2} h| |\Lie_{\mathcal{Z}}^{I_3}\Far|
			\ + \ |\nabla_{\mathcal{Z}}^{I_1} h| |\Lie_{\mathcal{Z}}^{I_2} \Far| |\Lie_{\mathcal{Z}}^{I_3}\Far|  
			\ + \ |\Lie_{\mathcal{Z}}^{I_1} \Far| |\Lie_{\mathcal{Z}}^{I_2} \Far| |\Lie_{\mathcal{Z}}^{I_3}\Far| \Big\rbrace. \notag
	\end{align}
	
	Similarly, for each $\mathcal{Z}-$multi-index $I,$ let
 	$\ualpha[\Lie_{\mathcal{Z}}^I \Far],$ $\alpha[\Lie_{\mathcal{Z}}^I \Far],$ 
 	$\rho[\Lie_{\mathcal{Z}}^I \Far],$ and $\sigma[\Lie_{\mathcal{Z}}^I \Far]$ denote the null components of 
	$\Lie_{\mathcal{Z}}^I \Far.$ Furthermore, let $\varpi(q)$ be any 
	differentiable function of $q.$ Assume that $|\nabla_{\mathcal{Z}}^I h| + |\Lie_{\mathcal{Z}}^I \Far| 
	\leq \varepsilon$ holds for $|I| \leq \lfloor k/2 \rfloor.$ Then if $\varepsilon$ is sufficiently small, 
	the following pointwise estimates also hold:
	
	\begin{subequations}
	\begin{align} \label{E:LambdaLieZIualphaEquationGoodqWeights}
		\sum_{|I| \leq k} r^{-1} \big|\nabla_{\Lambda} & \big(r \varpi(q) \ualpha[\Lie_{\mathcal{Z}}^I \Far] \big)\big| \\
		& \lesssim \sum_{|I| \leq k} r^{-1} \varpi(q) |h|_{\mathcal{L}\mathcal{L}} \big|\ualpha[\Lie_{\mathcal{Z}}^I \Far] \big|
			\ + \ \sum_{|I| \leq k} \varpi'(q) |h|_{\mathcal{L}\mathcal{L}} \big|\ualpha[\Lie_{\mathcal{Z}}^I \Far]\big| \notag \\
		& \ \ + \ \underbrace{\mathop{\sum_{|I| \leq k}}_{|J| \leq 1} \varpi(q) (1 + |q|)^{-1} 
				|\nabla_{\mathcal{Z}}^I h|_{\mathcal{L}\mathcal{L}} 
				\big|\ualpha[\Lie_{\mathcal{Z}}^J \Far] \big|}_{\mbox{absent if $k \leq 1$}} 
			\ + \ \underbrace{\mathop{\sum_{|J| \leq 1}}_{|I| \leq k} \varpi(q) (1 + |q|)^{-1} 
				|\nabla_{\mathcal{Z}}^J h|_{\mathcal{L}\mathcal{L}} 
				\big|\ualpha[\Lie_{\mathcal{Z}}^I \Far]\big|}_{\mbox{absent if $k = 0$}} \notag \\
		& \ \ + \ \underbrace{\sum_{|I| \leq k} \varpi(q) (1 + |q|)^{-1} 
				|h|_{\mathcal{L}\mathcal{T}} \big|\ualpha[\Lie_{\mathcal{Z}}^I \Far]\big|}_{\mbox{absent if $k = 0$}} 
			\ + \underbrace{\mathop{\sum_{|I_1| + |I_2| \leq k + 1}}_{|I_1| \leq k - 1, |I_2| \leq k - 1} \varpi(q) (1 + |q|)^{-1} 
				|\nabla_{\mathcal{Z}}^{I_1} h| \big|\ualpha[\Lie_{\mathcal{Z}}^{I_2} \Far]\big|}_{\mbox{absent if $k = 0$}}
				\notag \\
		& \ \ + \ \sum_{|I| \leq |k| + 1} \varpi(q) r^{-1} \big(|\Lie_{\mathcal{Z}}^{I} \Far|_{\mathcal{L} \mathcal{N}} 
				+ |\Lie_{\mathcal{Z}}^{I} \Far|_{\mathcal{T} \mathcal{T}} \big) 
				\ + \sum_{|I_1| + |I_2| \leq k + 1} \varpi(q) (1 + |q|)^{-1} |\nabla_{\mathcal{Z}}^{I_1} h|
		 		\big(|\Lie_{\mathcal{Z}}^{I_2} \Far|_{\mathcal{L} \mathcal{N}} +  
		 		|\Lie_{\mathcal{Z}}^{I_2} \Far|_{\mathcal{T} \mathcal{T}}\big) \notag \\
		& \ \ + \ \sum_{|I_1| + |I_2| \leq k + 1} \varpi(q) (1 + t + |q|)^{-1}
				|\nabla_{\mathcal{Z}}^{I_1} h| |\Lie_{\mathcal{Z}}^{I_2} \Far| \notag \\
		& \ \ + \ \sum_{|I_1| + |I_2| + |I_3| \leq k + 1} \varpi(q) (1 + |q|)^{-1} 
			\Big\lbrace |\nabla_{\mathcal{Z}}^{I_1} h| |\nabla_{\mathcal{Z}}^{I_2} h| |\Lie_{\mathcal{Z}}^{I_3}\Far|
				\ + \ |\nabla_{\mathcal{Z}}^{I_1} h| |\Lie_{\mathcal{Z}}^{I_2} \Far| |\Lie_{\mathcal{Z}}^{I_3}\Far|  
				\ + \ |\Lie_{\mathcal{Z}}^{I_1} \Far| |\Lie_{\mathcal{Z}}^{I_2} \Far| |\Lie_{\mathcal{Z}}^{I_3}\Far| \Big\rbrace, \notag
		\end{align}
		
	\begin{align} \label{E:alphaODE}
		\sum_{|I| \leq k} r \big| \nabla_{\uL} \big(r^{-1} \alpha[\Lie_{\mathcal{Z}}^I \Far] \big) \big| 
		& \lesssim
			\sum_{|I| \leq k + 1} r^{-1} |\Lie_{\mathcal{Z}}^I \Far|
			\ + \ \mathop{\sum_{|I_1| + |I_2| \leq k + 1}}_{|I_1| \leq k} 
			(1 + |q|)^{-1} |\nabla_{\mathcal{Z}}^{I_1} h| |\Lie_{\mathcal{Z}}^{I_2} \Far| \\
		& \ \ + \ \sum_{|I_1| + |I_2| + |I_3| \leq k + 1} 
				(1 + |q|)^{-1} \Big\lbrace |\nabla_{\mathcal{Z}}^{I_1} h| |\nabla_{\mathcal{Z}}^{I_2} h| |\Lie_{\mathcal{Z}}^{I_3}\Far|
				\ + \ |\nabla_{\mathcal{Z}}^{I_1} h| |\Lie_{\mathcal{Z}}^{I_2} \Far| |\Lie_{\mathcal{Z}}^{I_3}\Far|  
				\ + \ |\Lie_{\mathcal{Z}}^{I_1} \Far| |\Lie_{\mathcal{Z}}^{I_2} \Far| |\Lie_{\mathcal{Z}}^{I_3}\Far| \Big\rbrace, \notag
		\end{align}	
	
	\begin{align} \label{E:rhoODE}	
		\sum_{|I| \leq k} r^2 \big| \nabla_{\uL} \big(r^{-2} \rho[\Lie_{\mathcal{Z}}^I \Far] \big) \big| 
		& \lesssim
			\sum_{|I| \leq k + 1} r^{-1} |\Lie_{\mathcal{Z}}^I \Far|
			\ + \ \mathop{\sum_{|I_1| + |I_2| \leq k + 1}}_{|I_1| \leq k} 
			(1 + |q|)^{-1} |\nabla_{\mathcal{Z}}^{I_1} h| |\Lie_{\mathcal{Z}}^{I_2} \Far| \\
		& \ \ + \ \sum_{|I_1| + |I_2| + |I_3| \leq k + 1} 
				(1 + |q|)^{-1} \Big\lbrace |\nabla_{\mathcal{Z}}^{I_1} h| |\nabla_{\mathcal{Z}}^{I_2} h| |\Lie_{\mathcal{Z}}^{I_3}\Far|
				\ + \ |\nabla_{\mathcal{Z}}^{I_1} h| |\Lie_{\mathcal{Z}}^{I_2} \Far| |\Lie_{\mathcal{Z}}^{I_3}\Far|  
				\ + \ |\Lie_{\mathcal{Z}}^{I_1} \Far| |\Lie_{\mathcal{Z}}^{I_2} \Far| |\Lie_{\mathcal{Z}}^{I_3}\Far| \Big\rbrace, \notag
		\end{align}	
		
	\begin{align} \label{E:sigmaODE}
		\sum_{|I| \leq k} r^2 \big| \nabla_{\uL} \big(r^{-2} \sigma[\Lie_{\mathcal{Z}}^I \Far] \big) \big| 
		& \lesssim \sum_{|I| \leq k + 1} r^{-1} |\Lie_{\mathcal{Z}}^I \Far|.
	\end{align}	
	\end{subequations}

\end{proposition}

\begin{proof}
	Our proof of \eqref{E:ODErualpha} is based on decomposing the terms in equation \eqref{E:dotualphaEOVnulldecomp},
	where $\dot{\ualpha}_{\nu} = \ualpha_{\nu}[\Far],$ 
	$\dot{\mathfrak{F}}^{\nu'} = \mathfrak{F}^{\nu'},$ etc. in the equation. We remind the reader
	that this equation is a consequence of performing a Minkowskian null decomposition on the electromagnetic equations 
	\eqref{E:ReduceddFis0Summary} - 
	\eqref{E:ReduceddMis0Summary}. Here, $\mathfrak{F}^{\nu'}$ is defined in \eqref{E:EMBIFarInhomogeneous}. We begin by noting 
	that the first two terms in equation \eqref{E:dotualphaEOVnulldecomp} can be written as $r^{-1}\nabla_L (r \ualpha).$ We then 
	remove the dangerous $\frac{1}{4}h_{LL} \nabla_{\uL} \ualpha_{\nu}$ component from the quadratic term $\angm_{\nu \lambda} 
	\mathscr{P}_{(\Far)}^{\lambda}(h, \nabla \Far) \eqdef \angm_{\nu}^{\ \lambda} h^{\mu \kappa} \nabla_{\mu} \Far_{\kappa 
	\lambda}$ on the left-hand side of \eqref{E:dotualphaEOVnulldecomp}, and add it to the $r^{-1}\nabla_L (r \ualpha_{\nu})$ 
	term. Using the fact that $\nabla_{\Lambda} r = 1 - \frac{1}{4}h_{LL},$ it follows that
	the resulting sum can be written as $r^{-1}\nabla_{\Lambda} (r \ualpha_{\nu}) + \frac{1}{4}r^{-1} h_{LL} \ualpha_{\nu}.$ We 
	then put the $\frac{1}{4} r^{-1} h_{LL} \ualpha_{\nu}$ term on the right-hand side of \eqref{E:ODErualpha} as the first 
	inhomogeneous term; all the remaining terms in \eqref{E:dotualphaEOVnulldecomp}
	will also be placed on the right-hand side of \eqref{E:ODErualpha}. The left-over terms in 
	$\mathscr{P}_{(\Far)}^{\nu}(h,\nabla \Far)$ (after the dangerous component $\frac{1}{4}h_{LL} \nabla_{\uL} \ualpha^{\nu}$ 
	has been removed) are denoted by $\widetilde{\mathscr{P}}_{(\Far)}^{\nu}(h,\nabla \Far)$ in Lemma 
	\ref{L:AlgebraicTensorialEstimates} below. Now by \eqref{E:XPNoBadComponentFarhNablaFarNullFormEstimate}, with 
	$X_{\nu'} \eqdef \angm_{\nu \nu'}$ (so that $|X|_{\mathcal{L}} = 0),$ it follows that the left-over terms 
	$X_{\nu'} \widetilde{\mathscr{P}}_{(\Far)}^{\nu'}(h,\nabla \Far)$
	are bounded by the right-hand side of \eqref{E:ODErualpha}. The terms $\angn \rho$ and $\angn \sigma$ appearing on the 
	left-hand side of \eqref{E:dotualphaEOVnulldecomp} (see Remark \ref{R:FavorableAngular})
	can be bounded by the second term on the right-hand side of \eqref{E:ODErualpha} via Corollary 
	\ref{C:rWeightedAngularDerivativesinTermsofLieDerivatives}. 
	The remaining terms in equation \eqref{E:ODErualpha} that need to be bounded can be expressed as $X_{\nu'} 
	\widetilde{\mathscr{Q}}_{(1;\Far)}^{\nu'}(h,\nabla \Far),$ $X_{\nu'} N_{\triangle}^{\#\beta \nu' \kappa \lambda} 
	\nabla_{\beta} \Far_{\kappa \lambda},$ and $X_{\nu'} \mathfrak{F}^{\nu'}.$ The first of these can be bounded using 
	\eqref{E:XQ1FarhNablaFarNullFormEstimate}, the third with \eqref{E:LieZIFarNullFormInhomogeneousTermAlgebraicEstimate} (in 
	the case $|I| = 0$), while the second (with the help of Lemma \ref{L:PointwisetandqWeightedNablainTermsofZestiamtes}) 
	contributes to the cubic terms on the right-hand side of \eqref{E:ODErualpha}.
	
	Our proof of \eqref{E:LambdaLieZIualphaEquationGoodqWeights} is similar, but more elaborate. To begin, we
	differentiate the electromagnetic equations with the iterated modified Lie derivative $\Liemod_{\mathcal{Z}}^I$ to obtain the 
	equations of variation \eqref{E:EOVdFis0} - \eqref{E:EOVdMis0} for $\dot{\Far}_{\mu \nu} \eqdef \Lie_{\mathcal{Z}}^I 
	\Far_{\mu \nu}$ with inhomogeneous terms $\dot{\mathfrak{F}}^{\nu} = \mathfrak{F}_{(I)}^{\nu},$ where 
	$\mathfrak{F}_{(I)}^{\nu}$ is defined in \eqref{E:LiemodZIdifferentiatedEOVInhomogeneousterms}. We then perform a null 
	decomposition of the equations of variation, obtaining equation \eqref{E:dotualphaEOVnulldecomp} with
	$\dot{\ualpha}_{\nu} \eqdef \ualpha_{\nu}[\Lie_{\mathcal{Z}}^I \Far],$ $\dot{\mathfrak{F}}^{\nu'} \eqdef \mathfrak{F}_{(I)}^{\nu'},$
	etc. Next, we multiply equation \eqref{E:dotualphaEOVnulldecomp}
	by $\varpi(q),$ use the identities $\nabla_{\Lambda} r = 1 - \frac{1}{4}h_{LL}$ and $\nabla_{\Lambda} q = - 
	\frac{1}{2}h_{LL},$
	and argue as above, removing the dangerous $\frac{1}{4} h_{LL} \nabla_{\uL} \ualpha_{\nu}[\Lie_{\mathcal{Z}}^I \Far]$ 
	component from the quadratic term $\angm_{\nu \lambda} \mathscr{P}_{(\Far)}^{\lambda}(h, \nabla\Lie_{\mathcal{Z}}^I \Far) 
	\eqdef \angm_{\nu}^{\ \lambda} h^{\mu \kappa} \nabla_{\mu} \Lie_{\mathcal{Z}}^I \Far_{\kappa \lambda}$
	and denoting the remaining terms by $\angm_{\nu \lambda} \widetilde{\mathscr{P}}_{(\Far)}^{\lambda}(h,\nabla 
	\Lie_{\mathcal{Z}}^I \Far),$ to deduce that $\varpi(q) \big(\nabla_L \ualpha_{\nu}[\Lie_{\mathcal{Z}}^I \Far] 
	+ \frac{1}{4}h_{LL} \nabla_{\uL} \ualpha_{\nu}[\Lie_{\mathcal{Z}}^I \Far] 
	+ r^{-1}\ualpha_{\nu}[\Lie_{\mathcal{Z}}^I \Far] \big)$ 
	$= r^{-1} \nabla_{\Lambda} \big(r \varpi(q) \ualpha_{\nu}[\Lie_{\mathcal{Z}}^I \Far]\big) + \frac{1}{4} r^{-1} 
	\varpi(q) h_{LL} \ualpha_{\nu}[\Lie_{\mathcal{Z}}^I \Far] 
	- \frac{1}{2}\varpi'(q)h_{LL}\ualpha_{\nu}[\Lie_{\mathcal{Z}}^I \Far].$ The 
	first of these three terms is the only term on the left-hand side of \eqref{E:LambdaLieZIualphaEquationGoodqWeights}, while 
	the last two are brought over to the right-hand side of \eqref{E:LambdaLieZIualphaEquationGoodqWeights}.
	To bound $\angm_{\nu \nu'}\mathfrak{F}_{(I)}^{\nu'}$ by the right-hand side of 
	\eqref{E:LambdaLieZIualphaEquationGoodqWeights}, we again set $X_{\nu'} \eqdef \angm_{\nu \nu'}$ (so that 
	$|X|_{\mathcal{L}} = 0);$ the desired bound then follows from \eqref{E:LieZIFarNullFormInhomogeneousTermAlgebraicEstimate}
	and \eqref{E:EnergyInhomogeneousTermAlgebraicEstimate}, together with repeated use of the inequality
	$|\Lie_{\mathcal{Z}}^I \Far| \lesssim |\ualpha[\Lie_{\mathcal{Z}}^I \Far]| 
	+ |\Lie_{\mathcal{Z}}^I \Far|_{\mathcal{L}\mathcal{N}} + |\Lie_{\mathcal{Z}}^I \Far|_{\mathcal{T}\mathcal{T}}.$
	The terms $\varpi(q) \angn \rho[\Lie_{\mathcal{Z}}^I \Far]$ and 
	$\varpi(q) \angn \sigma[\Lie_{\mathcal{Z}}^I \Far]$ appearing on the left-hand side
	of \eqref{E:dotualphaEOVnulldecomp} (see Remark \ref{R:FavorableAngular}) can be bounded by the 
	seventh term on the right-hand side of \eqref{E:LambdaLieZIualphaEquationGoodqWeights} with the help of 
	Corollary \ref{C:rWeightedAngularDerivativesinTermsofLieDerivatives}. The remaining three terms on the left-hand side of 
	\eqref{E:dotualphaEOVnulldecomp} to be estimated are 
	$X_{\nu'} \widetilde{\mathscr{P}}_{(\Far)}^{\nu'}(h,\nabla\Lie_{\mathcal{Z}}^I 
	\Far),$ $X_{\nu'} \mathscr{Q}_{(1;\Far)}^{\nu'}(h, \nabla\Lie_{\mathcal{Z}}^I \Far),$ 
	and $X_{\nu'} N_{\triangle}^{\#\beta \nu' \kappa \lambda} \nabla_{\beta} \Lie_{\mathcal{Z}}^I \Far_{\kappa \lambda}.$
	The first of these can be bounded using \eqref{E:XPNoBadComponentFarhNablaFarNullFormEstimate},
	the second with \eqref{E:XQ1FarhNablaFarNullFormEstimate}, while the third (with the help of Lemma 	
	\ref{L:PointwisetandqWeightedNablainTermsofZestiamtes}) contributes to the cubic terms on the right-hand side of 
	\eqref{E:ODErualpha}.
	
	The proofs of \eqref{E:alphaODE} - \eqref{E:sigmaODE}, which are based on an analysis of equations
	\eqref{E:uLdotalphaEOVnulldecomp} - \eqref{E:uLdotnablasigmaEOVnulldecomp}, are similar, but 
	much simpler. We leave the details to the reader.
\end{proof}

The next proposition provides pointwise estimates for the challenging commutator term $\widetilde{\Square}_g \nabla_{\mathcal{Z}}^I h^{(1)} - \nablamod_{\mathcal{Z}}^I \widetilde{\Square}_g h^{(1)}$
from the right-hand side of \eqref{E:InhomogeneousTermsNablaZIh1}.

\begin{proposition} \label{P:DIPointwise} \cite[Proposition 5.3]{hLiR2010} 
\textbf{(Algebraic estimates of $[\widetilde{\Square}_g, \nabla_{\mathcal{Z}}^I]$)}
	Let $g_{\mu \nu}$ be a Lorentzian metric and let $h_{\mu \nu} \eqdef g_{\mu \nu} - m_{\mu \nu}$ and
	$H^{\mu \nu} \eqdef (g^{-1})^{\mu \nu} - m^{\mu \nu}.$
	Let $\widetilde{\Square}_g \eqdef \Square_m + H^{\kappa \lambda} \nabla_{\kappa} \nabla_{\lambda},$ and let
	$I$ be a $\mathcal{Z}-$multi-index with $1 \leq |I|.$ Let $\hat{\nabla}_{\mathcal{Z}}^I$ 
	denote the modified Minkowskian covariant derivative operator
	defined in \eqref{E:Covariantmoddef}. Assume that there is a constant $\varepsilon$ such that 
	$|\nabla_{\mathcal{Z}}^J h| \leq \varepsilon$ holds for all $\mathcal{Z}-$multi-indices $J$ satisfying 
	$|J| \leq \lfloor |I|/2 \rfloor.$ Then if $\varepsilon$ is sufficiently small, the following pointwise estimate holds:
	
	\begin{align} \label{E:waveoperatorZIcommutaorMBIIEPointwise}
		|\widetilde{\Square}_g & \nabla_{\mathcal{Z}}^I h^{(1)}  
			- \nablamod_{\mathcal{Z}}^I \widetilde{\Square}_g h^{(1)}| \\
		& \lesssim (1 + t + |q|)^{-1} \mathop{\sum_{|K| \leq |I|}}_{|J| + (|K| - 1)_{+} \leq |I|} 
			|\nabla_{\mathcal{Z}}^J H| |\nabla\nabla_{\mathcal{Z}}^K h^{(1)}| \notag \\
		& \ \ + \ (1 + |q|)^{-1} \sum_{|K| \leq |I|} |\nabla\nabla_{\mathcal{Z}}^K h^{(1)}|
		 	\bigg\lbrace \mathop{\sum_{|J| + (|K| - 1)_{+}}}_{\ \leq |I|} |\nabla_{\mathcal{Z}}^{J} H|_{\mathcal{L} \mathcal{L}} 
			\ + \	\mathop{\sum_{|J'| + (|K| - 1)_{+}}}_{\ \leq |I|-1} 
				|\nabla_{\mathcal{Z}}^{J'} H|_{\mathcal{L} \mathcal{T}}
			\ + \ \underbrace{\mathop{\sum_{|J''| + (|K| - 1)_{+}}}_{\ \leq |I|-2} |\nabla_{\mathcal{Z}}^{J''} H|}_{
				\mbox{Absent if $|I| \leq 1$ or $|K| = |I|$}} \bigg\rbrace, \notag
	\end{align}
	where $(|K|-1)_+ \eqdef 0$ if $|K| = 0$ and $(|K|-1)_+ \eqdef |K| - 1$ if $|K| \geq 1.$

\end{proposition}

\hfill $\qed$

\begin{corollary} \label{C:boxZIh1ALinfinity}
	\textbf{(Algebraic estimates of $\big|\widetilde{\Square}_g \nabla_{\mathcal{Z}}^I h^{(1)} \big|$)}
	Assume that $h_{\mu \nu}^{(1)},$ $(\mu, \nu = 0,1,2,3),$ is a solution to the reduced equation \eqref{E:Reducedh1Summary}.
	Then under the assumptions of Proposition \ref{P:DIPointwise}, we have that
	
	\begin{align} \label{E:boxZIh1ALinfinity}
		|\widetilde{\Square}_g \nabla_{\mathcal{Z}}^I h^{(1)}|		 
		& \lesssim |\nablamod_{\mathcal{Z}}^I \mathfrak{H}| \ + \ |\nablamod_{\mathcal{Z}}^I \widetilde{\Square}_g h^{(0)}| 
				\ + \ (1 + t + |q|)^{-1} \mathop{\sum_{|K| \leq |I|}}_{|J| + (|K| - 1)_{+} \leq |I|} 
					|\nabla_{\mathcal{Z}}^J H| |\nabla \nabla_{\mathcal{Z}}^K h^{(1)}| \\
			& \ \ + \ (1 + |q|)^{-1} \sum_{|K| \leq |I|}  
				|\nabla\nabla_{\mathcal{Z}}^K h^{(1)}| \Bigg\lbrace \mathop{\sum_{|J| + (|K| - 1)_{+}}}_{\ \ \leq |I|} 
				|\nabla_{\mathcal{Z}}^J H|_{\mathcal{L} \mathcal{L}} 
				\ + \ \underbrace{\mathop{\sum_{|J'| + (|K| - 1)_{+}}}_{\ \leq |I|-1} |\nabla_{\mathcal{Z}}^{J'} H|_{\mathcal{L} 
				\mathcal{T}}}_{\mbox{Absent if $|I| = 0$}}
				\ + \ \underbrace{\mathop{\sum_{|J''| + (|K| - 1)_{+}}}_{\ \leq |I|-2} |\nabla_{\mathcal{Z}}^{J''} H|}_{\mbox{Absent if 
				$|I| \leq 1$ or $|K| = |I|$}} \Bigg\rbrace. \notag
	\end{align}
\end{corollary}

\begin{proof}
	Simply use Proposition \ref{P:InhomogeneousTermsNablaZIh1} to decompose 
	$\widetilde{\Square}_g \nabla_{\mathcal{Z}}^I h^{(1)} =  
		\nablamod_{\mathcal{Z}}^I {\mathfrak{H}} - \nablamod_{\mathcal{Z}}^I \widetilde{\Square}_g h^{(0)}
		+ \Big\lbrace \widetilde{\Square}_g \nabla_{\mathcal{Z}}^I h^{(1)} 
		- \nablamod_{\mathcal{Z}}^I \widetilde{\Square}_g h^{(1)} \Big\rbrace$ 
		and apply Proposition \ref{P:DIPointwise}.
\end{proof}

\subsection{Useful lemmas} \label{S:UsefulLemmas}
In this section, we provide the lemmas that are used in the proofs of the propositions. We will make repeated use of the following decompositions of the Minkowski metric and its inverse:

\begin{subequations}
\begin{align}
	m_{\mu \nu} & = - \frac{1}{2}L_{\mu} \underline{L}_{\nu} - \frac{1}{2} \underline{L}_{\mu} L_{\nu} + \angm_{\mu \nu},
		\label{E:mdecomp} \\
	(m^{-1})^{\mu \nu} & = - \frac{1}{2}L^{\mu} \underline{L}^{\nu} - \frac{1}{2} \underline{L}^{\mu} L^{\nu} + \angm^{\mu \nu},
		\label{E:minversedecomp}	
\end{align}
\end{subequations}
where $\angm_{\mu \nu}$ is the Euclidean first fundamental form of the spheres $S_{r,t}$ defined in \eqref{E:angmdef}.

We begin with a lemma that shows that the essential algebraic structure of the quadratic terms appearing on the right-hand sides of the reduced equations \eqref{E:Reducedh1Summary} - \eqref{E:ReduceddMis0Summary} is preserved under differentiation.

\begin{lemma} \label{L:nullformvectorfieldcommutation} 
\textbf{(Leibniz rules for the quadratic terms)}
Let $\mathscr{Q}_0(\nabla \psi,\nabla \chi),$ $\mathscr{Q}_{\mu \nu}(\nabla \psi, \nabla \chi)$ denote the standard null forms defined in \eqref{E:StandardNullForm0} - \eqref{E:StandardNullFormmunu}, and let 
$\mathscr{Q}_{\mu \nu}^{(1;h)}(\nabla h, \nabla h),$
$\mathscr{Q}_{\mu \nu}^{(2;h)}(\Far, \Far),$ $\mathscr{P}(\nabla_{\mu} h, \nabla_{\nu} h),$
$\mathscr{P}_{(\Far)}^{\nu}(\nabla h, \Far),$ 
$\mathscr{Q}_{(1;\Far)}^{\nu}(h, \nabla \Far),$ and $\mathscr{Q}_{(2;\Far)}^{\nu}(h, \nabla \Far)$
denote the quadratic terms defined in \eqref{E:hAddedUpNullForms}, \eqref{E:Q2h}, \eqref{E:PNullform}, \eqref{E:PFar}, \eqref{E:Q1Far}, and \eqref{E:Q2Far} respectively. Let $I$ be a $\mathcal{Z}-$multi-index. Then there exist constants 
$C_{I_1,I_2; \mu \nu}^{\kappa \lambda \gamma \gamma' \delta \delta'},$ $C_{I_1,I_2; \mu \nu}^{0;\gamma \gamma' \delta \delta'},$ $C_{\mathscr{P};I_1,I_2; \mu \nu}^{\kappa \lambda},$ $C_{\mathscr{P};I_1,I_2},$ and $C_{i;I_1,I_2}$ such that

\begin{subequations}
\begin{align} \label{E:Q1hLeibnizRule}
	\nabla_{\mathcal{Z}}^I \mathscr{Q}_{\mu \nu}^{(1;h)}(\nabla h, \nabla h) =  
		& \sum_{|I_1| + |I_2| \leq |I|} C_{I_1,I_2; \mu \nu}^{\kappa \lambda \gamma \gamma' \delta \delta'}
			\mathscr{Q}_{\kappa \lambda}(\nabla\nabla_{\mathcal{Z}}^{I_1} h_{\gamma \gamma'}, 
			\nabla\nabla_{\mathcal{Z}}^{I_2} h_{\delta \delta'}) \\
	& + \sum_{|I_1| + |I_2| < |I|} C_{I_1,I_2; \mu \nu}^{0;\gamma \gamma' \delta \delta'}
		\mathscr{Q}_0(\nabla\nabla_{\mathcal{Z}}^{I_1} h_{\gamma \gamma'}, 
		\nabla\nabla_{\mathcal{Z}}^{I_2} h_{\delta \delta'}), \notag
\end{align}

\begin{align} \label{E:Q2hLeibnizRule}
	\nabla_{\mathcal{Z}}^I \mathscr{Q}_{\mu \nu}^{(2;h)}(\Far, \Far) =
	 \sum_{|I_1| + |I_2| \leq |I|} C_{I_1,I_2} 
		\mathscr{Q}_{\mu \nu}^{(2;h)}(\nabla_{\mathcal{Z}}^{I_1} \Far, \nabla_{\mathcal{Z}}^{I_2}\Far),
\end{align}

\begin{align} \label{E:SpecialPLeibnizRule}
	\nabla_{\mathcal{Z}}^I \mathscr{P}(\nabla_{\mu} h, \nabla_{\nu} h)
	= \sum_{|I_1| + |I_2| \leq |I|} C_{\mathscr{P};I_1,I_2; \mu \nu}^{\kappa \lambda}
	\mathscr{P}(\nabla_{\kappa} \nabla_{\mathcal{Z}}^{I_1} h, \nabla_{\lambda} \nabla_{\mathcal{Z}}^{I_2} h),
\end{align}

\begin{align}
	\Lie_{\mathcal{Z}}^I \mathscr{P}_{(\Far)}^{\nu}(\nabla h, \Far)
		& = \sum_{|I_1| + |I_2| \leq |I|} C_{\mathscr{P};I_1,I_2} 
		\mathscr{P}_{(\Far)}^{\nu}(\nabla\Lie_{\mathcal{Z}}^{I_1} h, \Lie_{\mathcal{Z}}^{I_2} \Far), 
		\label{E:PFarLeibnizRule} \\
	\Lie_{\mathcal{Z}}^I \mathscr{Q}_{(i;\Far)}^{\nu}(h, \nabla \Far) & 
		= \sum_{|I_1| + |I_2| \leq |I|} C_{i;I_1,I_2} 
		\mathscr{Q}_{(i;\Far)}^{\nu}(\Lie_{\mathcal{Z}}^{I_1} h, \nabla\Lie_{\mathcal{Z}}^{I_2}\Far), && (i=1,2). 
		\label{E:Q1and2FarLeibnizRule} 
\end{align}
\end{subequations}

\end{lemma}

\begin{proof}
	By pure calculation, if $Z \in \mathcal{Z},$ then the following identity holds for the standard null form 
	$\mathscr{Q}_{\mu \nu}(\nabla \psi, \nabla \chi):$
	
	\begin{align} \label{E:ClassicNullFormZDifferentiatedExpansion}
		\nabla_Z \mathscr{Q}_{\mu \nu}(\nabla \psi, \nabla \chi) 
			& = \mathscr{Q}_{\mu \nu}(\nabla\nabla_Z \psi, \nabla \chi) 
				+ \mathscr{Q}_{\mu \nu}(\nabla \psi, \nabla\nabla_Z \chi)
				- ^{(Z)} c_{\mu}^{\ \kappa} \mathscr{Q}_{\kappa \nu}(\nabla \psi, \nabla \chi)
				- ^{(Z)} c_{\nu}^{\ \kappa} \mathscr{Q}_{\mu \kappa}(\nabla \psi, \nabla \chi),
	\end{align}
	where $^{(Z)} c_{\mu \nu}$ is the covariantly constant tensorfield defined in \eqref{E:CovariantDerivativesofZareConstant}. A 
	similar identity holds for the standard null form $\mathscr{Q}_0(\nabla \psi, \nabla \chi).$ \eqref{E:Q1hLeibnizRule} now 
	follows from inductively from these facts and the Leibniz rule, since $\mathscr{Q}_{\mu \nu}^{(1;h)}(\nabla h, \nabla h)$ is 
	a linear combination of standard null forms. \eqref{E:SpecialPLeibnizRule} follows similarly.
	\eqref{E:Q2hLeibnizRule} follows trivially from definition \eqref{E:Q2h} and the Leibniz rule. \eqref{E:PFarLeibnizRule} and 
	\eqref{E:Q1and2FarLeibnizRule} follow from \eqref{E:LieZonmupper}, Lemma \ref{L:Liecommuteswithcoordinatederivatives}, and 
	the Leibniz rule. 
\end{proof}

The next lemma concerns the null structure of the standard null forms.

\begin{lemma} \label{L:starndardnullforms} \textbf{(Null form estimates of the standard null forms)} 
	Let $\mathscr{Q}_0(\nabla \psi,\nabla \chi) \eqdef (m^{-1})^{\kappa \lambda} (\nabla_{\kappa} \psi)(\nabla_{\lambda} \chi),$ 
	$\mathscr{Q}_{\mu \nu}(\nabla \psi, \nabla \chi) \eqdef (\nabla_{\mu} \psi)(\nabla_{\nu} \chi) 
		- (\nabla_{\nu} \psi)(\nabla_{\mu} \chi)$ denote the standard null forms defined in  
		\eqref{E:StandardNullForm0} - \eqref{E:StandardNullFormmunu}. Then
	
	\begin{align} \label{E:standardnullforms}
		|\mathscr{Q}_0 (\nabla \psi, \nabla \chi)| + |\mathscr{Q}_{\mu \nu}(\nabla \psi,\nabla \chi)| 
			\lesssim |\conenabla \psi||\nabla \chi| + |\conenabla\chi||\nabla \psi|.
	\end{align}	
	
	\begin{proof}
		The estimate \eqref{E:standardnullforms} for $\mathscr{Q}_0$ easily follows from using
		\eqref{E:mdecomp} to decompose $(m^{-1})^{\kappa \lambda}.$ To obtain the estimates for 
		$\mathscr{Q}_{\mu \nu}(\nabla \psi,\nabla \chi),$
		first consider the $\mathscr{Q}_{\mu \nu}(\nabla \psi,\nabla \chi)$ to be components of a 2-covariant tensor 
		$\mathscr{Q}(\nabla \psi,\nabla \chi).$ Inequality \eqref{E:standardnullforms} is equivalent to the following inequality:
		
		\begin{align} \label{E:standardnullforminequality}
			|\mathscr{Q}(\nabla \psi, \nabla \chi)|_{\mathcal{N}\mathcal{N}}
				\lesssim |\conenabla \psi||\nabla \chi| + |\conenabla\chi||\nabla\psi|.
		\end{align}
		Contracting $\mathscr{Q}_{\mu \nu}(\nabla \psi,\nabla \chi)$ against frame vectors $N^{\mu},N^{\nu} \in 
		\mathcal{N},$ we see that the only component on the left-hand side of \eqref{E:standardnullforminequality} that could pose 
		any difficulty is $\uL^{\mu} \uL^{\nu} \mathscr{Q}_{\mu \nu}(\nabla \psi,\nabla \chi).$ But the 
		anti-symmetry the $\mathscr{Q}_{\mu \nu}(\cdot,\cdot)$ implies that this component is $0.$
	\end{proof}
\end{lemma}

The next lemma addresses the null structure of some of the terms appearing in the reduced equations \eqref{E:Reducedh1Summary} - \eqref{E:ReduceddMis0Summary}.

\begin{lemma} \label{L:AlgebraicTensorialEstimates} \textbf{(Null form estimates for the reduced equations)}
Let $\mathscr{P}(\nabla_{\mu} \Pi, \nabla_{\nu} \Theta),$ $\mathscr{Q}_{\mu \nu}^{(1;h)}(\nabla h, \nabla h),$ $\mathscr{Q}_{\mu \nu}^{(2;h)}(\Far, \Gar),$ $\mathscr{P}_{(\Far)}^{\nu}(h, \nabla \Far),$ $\mathscr{Q}_{(1;\Far)}^{\nu}(h, \nabla \Far),$ and $\mathscr{Q}_{(2;\Far)}^{\nu}(\nabla h, \Far)$ be the quadratic forms defined in
Section \ref{SS:ReducedEquations}, and define the quadratic form $\widetilde{\mathscr{P}}_{(\Far)}^{\nu}(h, \nabla \Far)$ by
removing the $\nabla_{\uL}\ualpha^{\nu}[\Far]-$containing component of $\mathscr{P}_{(\Far)}^{\nu}(h, \nabla \Far):$

\begin{align}
	\widetilde{\mathscr{P}}_{(\Far)}^{\nu}(h, \nabla \Far) 
	& \eqdef \mathscr{P}_{(\Far)}^{\nu}(h, \nabla \Far) - \frac{1}{4}h_{LL} \angm^{\nu \nu'} \nabla_{\uL} \Far_{\uL \nu'} \\
	& = \mathscr{P}_{(\Far)}^{\nu}(h, \nabla \Far) + \frac{1}{4}h_{LL} \nabla_{\uL} \ualpha^{\nu}[\Far]. \notag
\end{align}
Let $X_{\nu}$ be any covector, let $\Pi_{\mu \nu},$ $\Theta_{\mu \nu}$ be any symmetric or anti-symmetric type $\binom{0}{2}$ tensorfields, and let $\Far_{\mu \nu},$ $\Gar_{\mu \nu}$ be any two-forms. Then the following pointwise inequalities hold:

\begin{subequations}
\begin{align}
		|\mathscr{P}(\nabla_{\mu} \Pi, \nabla_{\nu} \Theta)|
			& \lesssim |\nabla \Pi|_{\mathcal{T} \mathcal{N}} |\nabla \Theta|_{\mathcal{T} \mathcal{N}}
			\ + \ |\nabla \Pi|_{\mathcal{L}\mathcal{L}} |\nabla \Theta| 
			\ + \ |\Pi| |\nabla \Theta|_{\mathcal{L}\mathcal{L}}, \qquad (\mu, \nu = 0,1,2,3),
			\label{E:PSpecialNullStructure} \\
			\sum_{T \in \mathcal{T}, N \in \mathcal{N}}|T^{\mu} N^{\nu}\mathscr{P}(\nabla_{\mu} \Pi, \nabla_{\nu} \Theta)|
			& \lesssim |\conenabla \Pi||\nabla \Theta|, \label{E:PTUSpecialNullStructure} \\
		|\mathscr{Q}_{\mu \nu}^{(1;h)}(\nabla \Pi, \nabla \Theta)| & \lesssim 
			|\conenabla \Pi||\nabla \Theta| \ + \ |\nabla \Pi||\conenabla \Theta|, \qquad (\mu, \nu = 0,1,2,3), 
			\label{E:Q1hNullFormEstimate} \\
		\sum_{T \in \mathcal{T}, N \in \mathcal{N}} 
			|T^{\mu} N^{\nu}\mathscr{Q}_{\mu \nu}^{(2;h)}(\Far, \Gar)| 
			& \lesssim \big(|\Far|_{\mathcal{L} \mathcal{N}} + |\Far|_{\mathcal{T} \mathcal{T}}\big)|\Gar|
			\ + \ |\Far| \big(|\Gar|_{\mathcal{L} \mathcal{N}} + |\Gar|_{\mathcal{T} \mathcal{T}} \big),
			\label{E:Q2TUhNullFormEstimate} \\
		|\mathscr{Q}_{\mu \nu}^{(2;h)}(\Far, \Gar)| & \lesssim |\Far||\Gar|, \qquad (\mu, \nu = 0,1,2,3), 
			\label{E:Q2hNullFormEstimate} \\
		|X_{\nu} \mathscr{P}_{(\Far)}^{\nu}(h, \nabla \Far)| 
		& \lesssim |X||h||\conenabla \Far| 
			\ + \ |X||h|\big(|\nabla \Far|_{\mathcal{L} \mathcal{N}} + |\nabla \Far|_{\mathcal{T} \mathcal{T}}\big)
			\ + \ |X||h|_{\mathcal{L}\mathcal{L}}|\nabla \Far|
			\ + \ |X|_{\mathcal{L}}|h||\nabla \Far| \label{E:XPFarhNablaFarNullFormEstimate} \\
		& \lesssim (1 + t + |q|)^{-1} \sum_{|I| \leq 1} |X||h||\Lie_{\mathcal{Z}}^I \Far| 
			\ + \ (1 + |q|)^{-1} \sum_{|I| \leq 1} |X||h|
			\big(|\Lie_{\mathcal{Z}}^I \Far|_{\mathcal{L} \mathcal{N}} + |\Lie_{\mathcal{Z}}^I \Far|_{\mathcal{T} \mathcal{T}}\big)
			\notag \\
		& \ \ + \ (1 + |q|)^{-1} \sum_{|I| \leq 1} |X||h|_{\mathcal{L}\mathcal{L}}|\Lie_{\mathcal{Z}}^I \Far|
			\ + \ (1 + |q|)^{-1} \sum_{|I| \leq 1} |X|_{\mathcal{L}}|h||\Lie_{\mathcal{Z}}^I \Far|,
			\notag \\
		|X_{\nu} \widetilde{\mathscr{P}}_{(\Far)}^{\nu}(h, \nabla \Far)| 
		& \lesssim |X||h||\conenabla \Far| 
			\ + \ |X||h|\big(|\nabla \Far|_{\mathcal{L} \mathcal{N}} + |\nabla \Far|_{\mathcal{T} \mathcal{T}}\big)
			\ + \ |X|_{\mathcal{L}}|h||\nabla \Far| \label{E:XPNoBadComponentFarhNablaFarNullFormEstimate} \\
		& \lesssim (1 + t + |q|)^{-1} \sum_{|I| \leq 1} |X||h||\Lie_{\mathcal{Z}}^I \Far| 
			\ + \ (1 + |q|)^{-1} \sum_{|I| \leq 1} |X||h|
			\big(|\Lie_{\mathcal{Z}}^I \Far|_{\mathcal{L} \mathcal{N}} + |\Lie_{\mathcal{Z}}^I \Far|_{\mathcal{T} \mathcal{T}}\big)
			\notag \\
		& \ \ + \ (1 + |q|)^{-1} \sum_{|I| \leq 1} |X|_{\mathcal{L}}|h||\Lie_{\mathcal{Z}}^I \Far|,
			\notag \\
		|X_{\nu} \mathscr{Q}_{(1;\Far)}^{\nu}(h, \nabla \Far)| 
		& \lesssim |X||h||\conenabla \Far| 
			\ + \ |X||h|\big(|\nabla \Far|_{\mathcal{L} \mathcal{N}} + |\nabla \Far|_{\mathcal{T} \mathcal{T}}\big) 
			\label{E:XQ1FarhNablaFarNullFormEstimate} \\
		& \lesssim (1 + t + |q|)^{-1} \sum_{|I| \leq 1} |X||h||\Lie_{\mathcal{Z}}^I \Far| 
			\ + \ (1 + |q|)^{-1} \sum_{|I| \leq 1} |X||h|
			\big(|\Lie_{\mathcal{Z}}^I \Far|_{\mathcal{L} \mathcal{N}} + |\Lie_{\mathcal{Z}}^I \Far|_{\mathcal{T} \mathcal{T}}\big),
			\notag \\
		|X_{\nu}\mathscr{Q}_{(2;\Far)}^{\nu}(\nabla h, \Far)| 
			& \lesssim |X||\conenabla h||\Far| \ + \ |X||\nabla h||\Far|_{\mathcal{L} \mathcal{N}} 
				\label{E:Q2FarNullFormEstimate} \\
			& \lesssim (1 + t + |q|)^{-1} \sum_{|I| \leq 1} |X||\nabla_{\mathcal{Z}}^I h||\Far| 
			\ + \ (1 + |q|)^{-1} \sum_{|I| \leq 1} |X||\nabla_{\mathcal{Z}}^I h|
			\big(|\Far|_{\mathcal{L} \mathcal{N}} + |\Far|_{\mathcal{T} \mathcal{T}}\big).
			\notag 
		\end{align}
\end{subequations}

\end{lemma}

\begin{proof}
	Inequality \eqref{E:Q1hNullFormEstimate} follows directly from
	Lemma \ref{L:starndardnullforms}, since $\mathscr{Q}_{\mu \nu}^{(1;h)}(\nabla h, \nabla h)$
	is a linear combination of standard null forms. Inequality \eqref{E:Q2hNullFormEstimate} is trivial, 
	while \eqref{E:PSpecialNullStructure}, \eqref{E:PTUSpecialNullStructure}, and the first inequalities in 
	\eqref{E:Q2TUhNullFormEstimate} - \eqref{E:Q2FarNullFormEstimate} are easy to check using \eqref{E:mdecomp} - 
	\eqref{E:minversedecomp}. The second inequalities in \eqref{E:Q2TUhNullFormEstimate} - \eqref{E:Q2FarNullFormEstimate} then 
	follow from the first ones, Lemma \ref{L:PointwisetandqWeightedNablainTermsofZestiamtes}, and Proposition 
	\ref{P:LievsCovariantLContractionRelation}.
\end{proof}

The next lemma concerns the null structure of the cubic terms on the right-hand side of \eqref{E:Firstweightedenergyscalar}.

\begin{lemma} \cite[Lemma 4.2]{hLiR2010} \label{L:InhomogeneousWaveEquationEnergyCurrentAlgebraicNullFormEstimates}
	\textbf{(Null form estimates for quasilinear wave equations)}
	Let $\Pi$ be a type $\binom{0}{2}$ tensorfield, and let $\phi$ be a scalar function. Then the following inequalities hold:
	
	\begin{subequations}
	\begin{align}
		|\Pi^{\kappa \lambda}(\nabla_{\kappa} \phi)(\nabla_{\lambda} \phi)| & \lesssim
			|\Pi|_{\mathcal{L} \mathcal{L}}|\nabla \phi|^2 \ + \ |\Pi||\conenabla \phi| |\nabla \phi|, \\
		|L_{\kappa}\Pi^{\kappa \lambda}\nabla_{\lambda} \phi| & \lesssim
			|\Pi|_{\mathcal{L} \mathcal{L}}|\nabla \phi| \ + \ |\Pi||\conenabla \phi|, \\
		|(\nabla_{\kappa} \Pi^{\kappa \lambda})\nabla_{\lambda} \phi| & \lesssim
			|\nabla \Pi|_{\mathcal{L} \mathcal{L}} |\nabla \phi| \ + \ |\conenabla \Pi||\nabla \phi|
			\ + \ |\nabla \Pi| |\conenabla \phi|, \\
		|\Pi^{\kappa \lambda} \nabla_{\kappa} \nabla_{\lambda} \phi| & \lesssim
			|\Pi|_{\mathcal{L} \mathcal{L}}|\nabla\nabla \phi| \ + \ |\conenabla \nabla \phi|.
	\end{align}
	\end{subequations}
	
\end{lemma}

\hfill $\qed$

The following lemma addresses the null structure of the cubic terms on the right-hand 
side of \eqref{E:currentdivergence}.

\begin{lemma} \label{L:DivergenceofJAlgebraicNullFormEstimates}
	\textbf{(Null form estimates for the electromagnetic equations of variation)}
	Let $h_{\mu \nu}$ be a type $\binom{0}{2}$ tensorfield, and let $\Far_{\mu \nu}$ be two-form.
	Then the following inequalities hold:
	
	\begin{subequations}
	\begin{align}
		|(\nabla_{\mu} h^{\mu \kappa}) \Far_{\kappa \zeta} \Far_{0}^{\ \zeta}|		
		& \lesssim |\nabla h|_{\mathcal{L} \mathcal{L}}|\Far|^2 \ + \ |\conenabla h||\Far|^2
			\ + \ |\nabla h||\Far| \big(|\Far|_{\mathcal{L} \mathcal{N}} + |\Far|_{\mathcal{T} \mathcal{T}}\big), 
				\label{E:DivergenceofJFirstNablahdotFarsquaredAlgebraicNullFormEstimate} \\
		|(\nabla_{\mu} h^{\kappa \lambda}) \Far_{\ \kappa}^{\mu} \Far_{0 \lambda}|
		& \lesssim |\conenabla h||\Far|^2 
			\ + \ |\nabla h||\Far|
			\big(|\Far|_{\mathcal{L} \mathcal{N}} + |\Far|_{\mathcal{T} \mathcal{T}}\big),	\\
		|(\nabla_{t} h^{\kappa \lambda}) \Far_{\kappa \eta}\Far_{\lambda}^{\ \eta}|
		& \lesssim |\nabla h|_{\mathcal{L} \mathcal{L}}|\Far|^2
			\ + \ |\nabla h||\Far|
			\big(|\Far|_{\mathcal{L} \mathcal{N}} + |\Far|_{\mathcal{T} \mathcal{T}}\big), \\
		|L_{\mu}h^{\mu \kappa} \Far_{\kappa \zeta} \Far_{0}^{\ \zeta}|
		& \lesssim |h|_{\mathcal{L} \mathcal{L}}|\Far|^2
			\ + \ |h||\Far|\big(|\Far|_{\mathcal{L} \mathcal{N}} + |\Far|_{\mathcal{T} \mathcal{T}}\big), \\
			|L_{\mu} h^{\kappa \lambda} \Far_{\ \kappa}^{\mu} \Far_{0 \lambda}|
		& \lesssim |h||\Far||\Far|_{\mathcal{L} \mathcal{N}},  \\
			|h^{\kappa \lambda} \Far_{\kappa \eta} \Far_{\lambda}^{\ \eta}| 
		& \lesssim |h|_{\mathcal{L} \mathcal{L}}|\Far|^2
			\ + \ |h||\Far|\big(|\Far|_{\mathcal{L} \mathcal{N}} + |\Far|_{\mathcal{T} \mathcal{T}}\big).
			\label{E:DivergenceofJThirdhdotFarsquaredAlgebraicNullFormEstimate}
	\end{align}
	\end{subequations}

\end{lemma}

\begin{proof}
	Inequalities \eqref{E:DivergenceofJFirstNablahdotFarsquaredAlgebraicNullFormEstimate} -
	\eqref{E:DivergenceofJThirdhdotFarsquaredAlgebraicNullFormEstimate} are easy to check
	using \eqref{E:mdecomp}.
\end{proof}

\section{Weighted Energy Estimates for the Electromagnetic Equations of Variation and for Systems of Nonlinear Wave Equations in a Curved Spacetime} \label{S:WeightedEnergy}

In this section, we prove weighted energy estimates for the electromagnetic equations of variation

\begin{subequations} 
	\begin{align}
		\nabla_{\lambda} \dot{\Far}_{\mu \nu} + \nabla_{\mu} \dot{\Far}_{\nu \lambda} + \nabla_{\nu} \dot{\Far}_{\lambda \mu}
			& = \dot{\mathfrak{\Far}}_{\lambda \mu \nu},&& (\lambda, \mu, \nu = 0,1,2,3), \label{E:EOVdFis0again} \\
			N^{\# \mu \nu \kappa \lambda} \nabla_{\mu} \dot{\Far}_{\kappa \lambda} & = \dot{\mathfrak{F}}^{\nu},
			&& (\nu = 0,1,2,3). \label{E:EOVdMis0again} 
	\end{align}
\end{subequations}
Our estimates complement the weighted energy estimates proved in \cite{hLiR2010} for the inhomogeneous wave equation 

\begin{align}
	\widetilde{\Square}_g \phi = \mathfrak{I},
\end{align}
and for tensorial systems of inhomogeneous wave equations with principal part $\widetilde{\Square}_g:$

\begin{align}
	\widetilde{\Square}_g \phi_{\mu \nu} = \mathfrak{I}_{\mu \nu}, && (\mu, \nu = 0,1,2,3).
\end{align}

\noindent \hrulefill
\ \\

\subsection{The energy estimate weight function \texorpdfstring{$w(q)$}{}}

As in \cite{hLiR2010}, our energy estimates will involve the weight function $w(q)$ defined by

\begin{align} \label{E:weight}
	w = w(q) = \left \lbrace
		\begin{array}{lr}
    	1 \ + \ (1 + |q|)^{1 + 2 \upgamma}, &  \mbox{if} \ q > 0, \\
      1 \ + \ (1 + |q|)^{-2 \upmu}, & \mbox{if} \ q < 0,
    \end{array}
  \right.
\end{align}
where the constants $\upgamma$ and $\upmu$ are subject to the restrictions stated in Section \ref{SS:FixedConstants}.

Observe that the following inequalities follow from the definition \eqref{E:weight}:  
  
\begin{align} \label{E:weightinequality}
  	w' \leq 4(1 + |q|)^{-1} w \leq 16 \upgamma^{-1} (1 + q_-)^{2 \upmu} w',
\end{align}
where $q_- = 0$ if $q \geq 0$ and $q_- = |q|$ if $q < 0.$
	
\subsection{Weighted energy estimates}	

We begin by deriving weighted energy estimates for the electromagnetic equations of variation.

\begin{lemma} \label{L:weightedenergyFar} \textbf{(Weighted energy estimates for $\dot{\Far}$)}
	Assume that $\dot{\Far}_{\mu \nu} $ is a solution to the equations of variation \eqref{E:EOVdFis0} - \eqref{E:EOVdMis0}
	corresponding to the background $(h_{\mu \nu}, \Far_{\mu \nu}),$ 
	where $h_{\mu \nu} \eqdef g_{\mu \nu} - m_{\mu \nu}.$ Let 
	$\dot{\alpha} \eqdef \alpha[\dot{\Far}],$ $\dot{\rho} \eqdef \rho[\dot{\Far}],$ and 
	$\dot{\sigma} \eqdef \sigma[\dot{\Far}]$ denote the ``favorable'' null components of $\dot{\Far}$ as
	defined in Definition \ref{D:null}. Assume that $|h| + |\Far| \leq \varepsilon.$ Then if $\varepsilon$ is sufficiently small, 
	and $t_1 \leq t_2,$ the following integral inequality holds:
	
	\begin{align}  \label{E:FirstweightedenergyFar}
		\int_{\Sigma_{t_2}} |\dot{\Far}|^2 w(q) \,d^3x 
		\ + \ \int_{t_1}^{t_2} \int_{\Sigma_{\tau}} & \big(\dot{\alpha}^2 + \dot{\rho}^2 + \dot{\sigma}^2 \big) 
			w'(q) \,d^3x \, d \tau \\
		& \lesssim \int_{\Sigma_{t_1}} |\dot{\Far}|^2 w(q) \,d^3x \notag  \\ 						
		& \ \ + \ \int_{t_1}^{t_2} \int_{\Sigma_{\tau}} \Big| \lbrace \dot{\Far}_{0 \eta} \dot{\mathfrak{F}}^{\eta} 
			- (\nabla_{\mu} h^{\mu \kappa}) \dot{\Far}_{\kappa \zeta} \dot{\Far}_{0}^{\ \zeta}
			- (\nabla_{\mu} h^{\kappa \lambda}) \dot{\Far}_{\ \kappa}^{\mu} \dot{\Far}_{0 \lambda}
			+ \frac{1}{2} (\nabla_{t} h^{\kappa \lambda}) \dot{\Far}_{\kappa \eta}\dot{\Far}_{\lambda}^{\ \eta} \Big| w(q) \,d^3x 
			\, d \tau \notag \\
	& \ \ + \ \int_{t_1}^{t_2} \int_{\Sigma_{\tau}} 
			\Big|L_{\mu}h^{\mu \kappa} \dot{\Far}_{\kappa \zeta} \dot{\Far}_{0}^{\ \zeta} 
			+ L_{\mu} h^{\kappa \lambda} \dot{\Far}_{\ \kappa}^{\mu} \dot{\Far}_{0 \lambda} 
			+ \frac{1}{2} h^{\kappa \lambda} \dot{\Far}_{\kappa \eta} \dot{\Far}_{\lambda}^{\ \eta} \Big| w'(q) \,d^3x \, d \tau \		
			\notag  \\
	& \ \ + \ \int_{t_1}^{t_2} \int_{\Sigma_{\tau}} 
		\Big| (\nabla_{\mu}N_{\triangle}^{\# \mu \zeta \kappa \lambda}) \dot{\Far}_{\kappa \lambda} 
			\dot{\Far}_{0 \zeta} - \frac{1}{4} (\nabla_{t} N_{\triangle}^{\#\zeta \eta \kappa \lambda}) \dot{\Far}_{\zeta \eta} 
			\dot{\Far}_{\kappa \lambda} \Big| 	w(q) \,d^3x \, d \tau \notag \\
	& \ \ + \ \int_{t_1}^{t_2} \int_{\Sigma_{\tau}} 
		\Big|L_{\mu} N_{\triangle}^{\# \mu \zeta \kappa \lambda} \dot{\Far}_{\kappa \lambda} \dot{\Far}_{0 \zeta}
		+ \frac{1}{4} N_{\triangle}^{\#\zeta \eta \kappa \lambda} \dot{\Far}_{\zeta \eta} 
		\dot{\Far}_{\kappa \lambda} \Big| w'(q) \,d^3x \, d \tau. \notag
	\end{align}
\end{lemma}

\begin{proof}
	
	It follows from \eqref{E:dotJ0estimate} that if $\varepsilon$ is sufficiently small, we have that
	
	\begin{align} \label{E:J0positivity}
		\frac{1}{4} |\dot{\Far}|^2 w(q) \leq \dot{J}_{(h,\Far)}^0 \leq |\dot{\Far}|^2 w(q).
	\end{align}
	Using \eqref{E:currentdivergence} and the divergence theorem, it follows that

\begin{align}
	\int_{\Sigma_{t_2}} & \dot{J}_{(h,\Far)}^0 \, d^3 x \ + \ \frac{1}{2} \int_{t_1}^{t_2} \int_{\Sigma_{\tau}} 
		w'(q) (\dot{\alpha}^2 + \dot{\rho}^2 + \dot{\sigma}^2) \, d^3 x \, d \tau \\
	& = \int_{\Sigma_{t_1}} \dot{J}_{(h,\Far)}^0 \, d^3 x \notag \\
	& \ \ - \ \int_{t_1}^{t_2} \int_{\Sigma_{\tau}} w(q) \Big \lbrace \dot{\Far}_{0 \eta} \dot{\mathfrak{F}}^{\eta} 
			- (\nabla_{\mu} h^{\mu \kappa}) \dot{\Far}_{\kappa \zeta} \dot{\Far}_{0}^{\ \zeta}
			- (\nabla_{\mu} h^{\kappa \lambda}) \dot{\Far}_{\ \kappa}^{\mu} \dot{\Far}_{0 \lambda}
			+ \frac{1}{2} (\nabla_{t} h^{\kappa \lambda}) \dot{\Far}_{\kappa \eta}\dot{\Far}_{\lambda}^{\ \eta} \Big \rbrace 
		\, d^3 x \, d \tau \notag \\
	& \ \ - \ \int_{t_1}^{t_2} \int_{\Sigma_{\tau}} 
		w'(q) \Big\lbrace - L_{\mu}h^{\mu \kappa} \dot{\Far}_{\kappa \zeta} \dot{\Far}_{0}^{\ \zeta} 
			- L_{\mu} h^{\kappa \lambda} \dot{\Far}_{\ \kappa}^{\mu} \dot{\Far}_{0 \lambda} 
			- \frac{1}{2} h^{\kappa \lambda} \dot{\Far}_{\kappa \eta} \dot{\Far}_{\lambda}^{\ \eta} \Big\rbrace \, d^3 x \, d \tau 
			\notag \\
	& \ \ - \ \int_{t_1}^{t_2} \int_{\Sigma_{\tau}} 
		w(q) \Big\lbrace (\nabla_{\mu}N_{\triangle}^{\# \mu \zeta \kappa \lambda}) \dot{\Far}_{\kappa \lambda} 
			\dot{\Far}_{0 \zeta} - \frac{1}{4} (\nabla_{t} N_{\triangle}^{\#\zeta \eta \kappa \lambda}) \dot{\Far}_{\zeta \eta} 
			\dot{\Far}_{\kappa \lambda} \Big\rbrace \, d^3 x \, d \tau \notag \\
	& \ \ - \ \int_{t_1}^{t_2} \int_{\Sigma_{\tau}} 
		w'(q) \Big\lbrace L_{\mu} N_{\triangle}^{\# \mu \zeta \kappa \lambda} \dot{\Far}_{\kappa \lambda} \dot{\Far}_{0 \zeta}
		+ \frac{1}{4} N_{\triangle}^{\#\zeta \eta \kappa \lambda} \dot{\Far}_{\zeta \eta} 
		\dot{\Far}_{\kappa \lambda} \Big\rbrace \, d^3 x \, d \tau, \notag
\end{align}	
which, with the help of \eqref{E:J0positivity}, implies \eqref{E:FirstweightedenergyFar}.

\end{proof}

We now recall the analogous lemma proved in \cite{hLiR2010} for solutions to the inhomogeneous wave equation in curved spacetime.

\begin{lemma} \cite[Lemma 6.1]{hLiR2010} \label{L:weightedenergy} \textbf{(Weighted energy estimates for a scalar wave 
	equation)}
	Assume that the scalar-valued function $\phi$ is a solution to the equation $\widetilde{\Square}_g \phi = \mathfrak{I},$ and 
	let $H^{\mu \nu} \eqdef (g^{-1})^{\mu \nu} - (m^{-1})^{\mu \nu}.$ Assume that
	the metric $g_{\mu \nu}$ is such that $|H| \leq \frac{1}{2}.$ Then
	
	\begin{align}  \label{E:Firstweightedenergyscalar} 
		\int_{\Sigma_{t_2}} & |\nabla \phi|^2 w(q) \,d^3x \ + \ 2 \int_{t_1}^{t_2} \int_{\Sigma_{\tau}} |\conenabla \phi|^2 
			w'(q) \,d^3x \, d \tau \\
		& \leq 4 \int_{\Sigma_{t_1}} |\nabla \phi|^2 w(q) \,d^3x 
			\ + \ 4 \int_{t_1}^{t_2} \int_{\Sigma_{\tau}} \Big| \mathfrak{I}_{\kappa} \nabla_t \phi^{\kappa} 
			+ (\nabla_{\nu} H^{\nu \lambda})(\nabla_{\lambda} \phi)(\nabla_t \phi) 
			- \frac{1}{2} (\nabla_t H^{\lambda \kappa})(\nabla_{\lambda} \phi)(\nabla_{\kappa} \phi)
			\Big| w(q) \,d^3x \, d \tau \notag \\
	& \ \ + \ 4 \int_{t_1}^{t_2} \int_{\Sigma_{\tau}} \Big| \underbrace{(\omega_j H^{j \lambda} 
		-  H^{0 \lambda})}_{L_{\kappa} H^{\kappa \lambda}} (\nabla_t \phi)(\nabla_{\lambda} \phi)  
		+ \frac{1}{2} H^{\lambda \kappa} (\nabla_{\lambda} \phi)(\nabla_{\kappa} \phi) \Big| w'(q) \,d^3x \, d \tau. \notag 
	\end{align}
\end{lemma}

\hfill $\qed$

We now extend the results of the previous lemmas by estimating (under assumptions that are compatible with our global stability theorem) some of the cubic terms on the right-hand sides of \eqref{E:FirstweightedenergyFar} and \eqref{E:Firstweightedenergyscalar}. 

\begin{proposition}  \cite[extension of Proposition 6.2]{hLiR2010} \label{P:weightedenergy}
	\textbf{(Weighted energy estimates for the reduced equations)}
	Let $\phi$ be a solution to $\widetilde{\Square}_g \phi = \mathfrak{I}$ for the metric $g_{\mu \nu},$ and 
	let $H^{\mu \nu} \eqdef (g^{-1})^{\mu \nu} - (m^{-1})^{\mu \nu}.$ Let $\upgamma$ and $\upmu$ be positive constants
	satisfying the restrictions described in Section \ref{SS:FixedConstants}. Assume that the following pointwise estimates
	hold for $(t,x) \in [0,T) \times \mathbb{R}^3:$
	
	\begin{subequations}
	\begin{align}
		(1 + |q|)^{-1}|H|_{LL} \ + \ |\nabla H|_{LL} \ + \ |\conenabla H| & \leq C \varepsilon(1 + t + |q|)^{-1}, \\
		(1 + |q|)^{-1}|H| \ + \ |\nabla H| & \leq C \varepsilon(1 + t + |q|)^{-1/2} (1 + |q|)^{-1/2} (1 + q_-)^{- \upmu},
	\end{align}
	\end{subequations}
	where $q_{-} = 0$ if $q \geq 0$ and $q_{-} = |q|$ if $q < 0.$ Then there exists a constant $C_1 > 0$ such that if
	$0 < \varepsilon \leq \frac{\upgamma}{C_1},$ then the following integral inequality holds for $t \in [0,T):$
	
	\begin{align} \label{E:Secondweightedenergyscalar}
		\int_{\Sigma_{t}} |\nabla \phi|^2 w(q) \,d^3x 
			& \ + \ \int_{0}^{t} \int_{\Sigma_{\tau}} |\conenabla \phi|^2 w'(q) \,d^3x \, d \tau \\
		& \leq 8 \int_{\Sigma_{0}} |\nabla \phi|^2 w(q) \,d^3x 
			\ + \ 16 \int_{0}^{t} \int_{\Sigma_{\tau}} \Big(\frac{C \varepsilon |\nabla \phi|^2}{1 + \tau} + |\mathfrak{I}|
			|\nabla \phi| \Big) w(q) \,d^3x \, d \tau. \notag
	\end{align}

	Furthermore, let $\dot{\Far}_{\mu \nu}$ be a solution to the electromagnetic equations of variation \eqref{E:EOVdFis0} - 
	\eqref{E:EOVdMis0} corresponding to the background $(h_{\mu \nu}, \Far_{\mu \nu}),$ where $h_{\mu \nu} \eqdef g_{\mu \nu} - 
	m_{\mu \nu}.$ Assume that the following pointwise estimates hold for $(t,x) \in [0,T) \times \mathbb{R}^3:$
	
	\begin{subequations}
	\begin{align}
		(1 + |q|)^{-1}|h|_{\mathcal{L} \mathcal{L}} 
		\ + \ |\nabla h|_{\mathcal{L} \mathcal{L}} \ + \ |\conenabla h| \ + \ |\Far| & \leq C \varepsilon(1 + t + |q|)^{-1}, \\
		(1 + |q|)^{-1}|h| \ + \ |\nabla h| \ + \ |\nabla \Far| 
		& \leq C \varepsilon(1 + t + |q|)^{-1/2} (1 + |q|)^{-1/2} (1 + q_-)^{- \upmu},
	\end{align}
	\end{subequations}
	where $q_{-} = 0$ if $q \geq 0$ and $q_{-} = |q|$ if $q < 0.$  Then there exists a constant $C_1 > 0$ such that if
	$0 < \varepsilon \leq \frac{\upgamma}{C_1},$ then the following integral inequality holds for $t \in [0,T):$
	
	\begin{align}  \label{E:SecondweightedenergyFar}
	\int_{\Sigma_{t}} & |\dot{\Far}|^2 w(q) \,d^3x 
		\ + \ \int_{0}^{t} \int_{\Sigma_{\tau}} 
		\big(|\dot{\Far}|_{\mathcal{L} \mathcal{N}}^2 + |\dot{\Far}|_{\mathcal{T} \mathcal{T}}^2 \big) w'(q)
		\,d^3x \, d \tau \\  
	& \lesssim  \int_{\Sigma_{0}} |\dot{\Far}|^2 w(q) \,d^3x 
		\ + \ \varepsilon \int_{0}^{t} \int_{\Sigma_{\tau}} \frac{|\dot{\Far}|^2}{1 + \tau} w(q) \,d^3x \, d \tau
		\ + \ \int_{0}^{t} \int_{\Sigma_{\tau}} |\dot{\Far}_{0 \kappa}\dot{\mathfrak{F}}^{\kappa}| w(q) \,d^3x \, d \tau. \notag 
\end{align}
\end{proposition}

\begin{remark}
	Proposition \ref{P:weightedenergy} will not be used until the proof of Theorem \ref{T:ImprovedDecay}, where it plays
	a key role; see Section \ref{SS:MainArgument}. We also remark that the hypotheses of the proposition are implied by the 
	hypotheses of the theorem; see Section \ref{SS:FixedConstants} and Remark \ref{R:ImprovedDecay}.
\end{remark}

\begin{proof}
	Inequality \eqref{E:Secondweightedenergyscalar} was proved as Proposition 6.2 of \cite{hLiR2010}.	Their proof was
	based on using Lemma \ref{L:InhomogeneousWaveEquationEnergyCurrentAlgebraicNullFormEstimates} to estimate the inhomogeneous 
	terms on the right-hand side of \eqref{E:Firstweightedenergyscalar}. Rather than reproving this inequality, we only give the 
	proof of \eqref{E:SecondweightedenergyFar}, which is based on \eqref{E:FirstweightedenergyFar} and uses related ideas.
	
	We commence with the proof of \eqref{E:SecondweightedenergyFar}, 
	our goal being to deduce suitable pointwise bounds for some of the terms appearing on the right-hand side of 
	\eqref{E:FirstweightedenergyFar}. For the cubic terms, we use Lemma \ref{L:DivergenceofJAlgebraicNullFormEstimates}, the 
	hypotheses of the proposition, and the inequality $|ab| \lesssim a^2 + b^2$ to conclude that
	
	\begin{align} \label{E:firsttensorialdotFLinfinity}
		\Big|(\nabla_{\mu} h^{\mu \kappa}) \dot{\Far}_{\kappa \zeta} \dot{\Far}_{0}^{\ \zeta}
			- (\nabla_{\mu} h^{\kappa \lambda}) \dot{\Far}_{\ \kappa}^{\mu} \dot{\Far}_{0 \lambda}
			& + \frac{1}{2} (\nabla_{t} h^{\kappa \lambda}) \dot{\Far}_{\kappa \eta}\dot{\Far}_{\lambda}^{\ \eta} \Big| \\
		& \lesssim \big(|\nabla h|_{\mathcal{L} \mathcal{L}} + |\conenabla h| \big) |\dot{\Far}|^2
			\ + \ |\nabla h| |\dot{\Far}| \big(|\dot{\Far}|_{\mathcal{L} \mathcal{N}} + |\dot{\Far}|_{\mathcal{T} \mathcal{T}} \big) 
			\notag \\
		& \lesssim \varepsilon (1 + t + |q|)^{-1} |\dot{\Far}|^2 \ + \ \varepsilon (1 + |q|)^{-1} (1 + q_-)^{-2 \upmu} 
			\big(|\dot{\Far}|_{\mathcal{L} \mathcal{N}}^2 + |\dot{\Far}|_{\mathcal{T} \mathcal{T}}^2 \big) \notag
	\end{align}
	and 
	\begin{align} \label{E:secondtensorialdotFLinfinity}
		\Big|L_{\mu}h^{\mu \kappa} \dot{\Far}_{\kappa \zeta} \dot{\Far}_{0}^{\ \zeta} 
			+ L_{\mu} h^{\kappa \lambda} \dot{\Far}_{\ \kappa}^{\mu} \dot{\Far}_{0 \lambda} 
			& + \frac{1}{2} h^{\kappa \lambda} \dot{\Far}_{\kappa \eta} \dot{\Far}_{\lambda}^{\ \eta} \Big| \\
		& \lesssim |h|_{\mathcal{L}\mathcal{L}}|\dot{\Far}|^2 
			\ + \ |h||\dot{\Far}| \big(|\dot{\Far}|_{\mathcal{L} \mathcal{N}} + |\dot{\Far}|_{\mathcal{T} \mathcal{T}} \big) \notag \\
		& \lesssim \varepsilon (1+|q|)(1 + t + |q|)^{-1} |\dot{\Far}|^2 
			\ + \ \varepsilon (1 + q_-)^{-2 \mu} 
			\big(|\dot{\Far}|_{\mathcal{L} \mathcal{N}}^2 + |\dot{\Far}|_{\mathcal{T} \mathcal{T}}^2 \big).  \notag
	\end{align}
	
	For the higher-order terms, we use \eqref{E:NtriangleSmallAlgebraic},
	the hypotheses of the proposition, and the inequality $|ab| \lesssim a^2 + b^2$
	to deduce that
	
	\begin{align}
		\Big| (\nabla_{\mu}N_{\triangle}^{\# \mu \zeta \kappa \lambda}) \dot{\Far}_{\kappa \lambda} 
			\dot{\Far}_{0 \zeta} - \frac{1}{4} (\nabla_{t} N_{\triangle}^{\#\zeta \eta \kappa \lambda}) \dot{\Far}_{\zeta \eta} 
			\dot{\Far}_{\kappa \lambda} \Big| & \lesssim \big(|(h,\Far)||(\nabla h, \nabla \Far)| \big)  |\dot{\Far}|^2 \\
		& \lesssim \varepsilon (1 + t + |q|)^{-1} |\dot{\Far}|^2 \notag	
	\end{align}	
	and
		
		\begin{align} \label{E:fourthtensorialdotFLinfinity}
		\Big| L_{\mu} N_{\triangle}^{\# \mu \zeta \kappa \lambda} \dot{\Far}_{\kappa \lambda} \dot{\Far}_{0 \zeta}
		+ \frac{1}{4} N_{\triangle}^{\#\zeta \eta \kappa \lambda} \dot{\Far}_{\zeta \eta} 
		\dot{\Far}_{\kappa \lambda} \Big| & \lesssim |(h,\Far)|^2 |\dot{\Far}|^2 \\
		& \lesssim \varepsilon (1+|q|)(1 + t + |q|)^{-1} |\dot{\Far}|^2. \notag	
	\end{align}	

Inserting \eqref{E:firsttensorialdotFLinfinity} - \eqref{E:fourthtensorialdotFLinfinity} into
the right-hand side of \eqref{E:FirstweightedenergyFar}, and using \eqref{E:weightinequality}, 
we have that

\begin{align}  \label{E:preliminarweightedenergyFar}
	\int_{\Sigma_{t}} & |\dot{\Far}|^2 w(q) \,d^3x 
		\ + \ \int_{0}^{t} \int_{\Sigma_{\tau}} 
		\big(|\dot{\Far}|_{\mathcal{L} \mathcal{N}}^2 + |\dot{\Far}|_{\mathcal{T} \mathcal{T}}^2 \big) w'(q)
		\,d^3x \, d \tau \\  
	& \leq C \int_{\Sigma_{0}} |\dot{\Far}|^2 w(q) \,d^3x 
		\ + \ C_1 \varepsilon \int_{0}^{t} \int_{\Sigma_{\tau}} \bigg\lbrace \frac{|\dot{\Far}|^2}{1 + \tau} w(q)  
		\ + \ \big(|\dot{\Far}|_{\mathcal{L} \mathcal{N}}^2 + |\dot{\Far}|_{\mathcal{T} \mathcal{T}}^2 \big) \frac{w'(q)}{\upgamma} 
		\bigg \rbrace \,d^3x \, d \tau		
		\ + \ C \int_{0}^{t} \int_{\Sigma_{\tau}} |\dot{\Far}_{0 \kappa} \dot{\mathfrak{F}}^{\kappa}| w(q) \,d^3x \, d \tau. \notag 
\end{align}

Now if $C_1 \varepsilon/\upgamma$ is sufficiently small, we can absorb the  
$C_1 \varepsilon \int_{0}^{t} \int_{\Sigma_{\tau}} \Big\lbrace  \big(|\dot{\Far}|_{\mathcal{L} \mathcal{N}}^2 + |\dot{\Far}|_{\mathcal{T} \mathcal{T}}^2 \Big) \frac{w'(q)}{\upgamma} \Big \rbrace \,d^3x \, d \tau$ term on the right-hand side of \eqref{E:preliminarweightedenergyFar} into the second term on the left-hand side at the expense of increasing the constants $C.$ Inequality \eqref{E:SecondweightedenergyFar} thus follows.

\end{proof}

\section{Pointwise Decay Estimates for Wave Equations in a Curved Spacetime} \label{S:WaveEquationDecay}

In this section, we state a lemma and a corollary proved in \cite{hLiR2010}. 
They allow one to deduce pointwise decay estimates for solutions to inhomogeneous wave equations
(e.g., for the $h_{\mu \nu}$). The main advantage of these estimates is that if one has good control over the inhomogeneous terms, then the pointwise decay estimates provided by the lemma and its corollary are \emph{improvements over what can be deduced from the weighted Klainerman-Sobolev inequalities} of Proposition \ref{P:WeightedKlainermanSobolev}. In particular, the lemma and its corollary play a fundamental role in the proofs of Propositions \ref{P:UpgradedDecayhA} and \ref{P:UpgradedDecayh1A}.

\begin{remark}
	The Faraday tensor analogs of Lemma \ref{L:scalardecay} and Corollary \ref{C:systemdecay}
	are contained in the estimates of Proposition \ref{P:ODEsNullComponentsLieZIFar}. More specifically,
	the analogous inequalities would arise from integrating (in the direction of the first-order vectorfield differential 
	operators on the left-hand sides of the inequalities) the inequalities in the proposition. We will carry out these
	integrations in Section \ref{S:DecayFortheReducedEquations}, which will allow us to derive improved pointwise decay estimates 
	for the lower-order Lie derivatives of the Faraday tensor (improved over what can be deduced from the 
	weighted Klainerman-Sobolev inequality \eqref{E:KSIntro}).
\end{remark}

\noindent \hrulefill
\ \\

\subsection{The decay estimate weight function \texorpdfstring{$\varpi(q)$}{}}

As in \cite{hLiR2010}, our decay estimates will involve the following weight function $\varpi(q),$
which is chosen to complement the energy estimate weight function $w(q)$ defined in \eqref{E:weight}:

\begin{align} \label{E:decayweight}
	\varpi = \varpi(q) = \left \lbrace
		\begin{array}{lr}
    	(1 + |q|)^{1 + \upgamma'}, &  \mbox{if} \ q > 0, \\
      (1 + |q|)^{1/2 - \upmu'}, & \mbox{if} \ q < 0,
    \end{array}
  \right.
\end{align}
where $0 < \updelta < \upmu' < 1/2 - \upmu$ and $0 < \upgamma' < \upgamma - \updelta$ are fixed constants. Its
complementary role will become apparent in Section \ref{S:DecayFortheReducedEquations}.

\subsection{Pointwise decay estimates}

We now state the lemma concerning pointwise decay estimates for solutions to inhomogeneous quasilinear wave equations.

\begin{lemma} \cite[Lemma 7.1]{hLiR2010} \label{L:scalardecay} \textbf{(Pointwise decay estimates for solutions to a scalar 
	wave equation)}
	Let $\phi$ be a solution of the scalar wave equation \eqref{E:scalar} 
	
	\begin{align} \label{E:scalar}
		\widetilde{\Square}_g \phi = \mathfrak{I}
	\end{align}
	on a curved background with metric $g_{\mu \nu}.$ Assume that the tensor $H^{\mu \nu} \eqdef (g^{-1})^{\mu \nu} - 
	(m^{-1})^{\mu \nu}$ obeys the following pointwise estimates 
	
	\begin{align}
		|H| \leq \varepsilon',&& 			
		\int_{0}^{\infty} (1 + t)^{-1} \big\| H(t,\cdot) \big\|_{L^{\infty}(D_{t})} \, d t \leq \frac{1}{4},
		&& |H|_{\mathcal{L} \mathcal{T}} \leq \varepsilon' (1 + t + |x|)^{-1}(1 + |q|)
	\end{align}
	in the region
	\begin{align}
		D_t \eqdef \lbrace x: t/2 < r < 2 t \rbrace
	\end{align}
	for $t \in [0,T).$ Then with $\upalpha \eqdef \max(1 + \upgamma', 1/2 - \upmu'),$ the following pointwise estimate holds
	for $(t,x) \in [0,T) \times \mathbb{R}^3:$
	
	\begin{align} \label{E:scalardecay}
		(1 + t + |q|) \varpi(q) |\nabla \phi| & \lesssim \sup_{0 \leq \tau \leq t} \sum_{|I| \leq 1} 
			\big\| \varpi(q)\nabla_{\mathcal{Z}}^I \phi(\tau,\cdot) \big\|_{L^{\infty}} \\
		& \ \ + \ \int_{\tau = 0}^{t} \varepsilon' \upalpha \big\| \varpi(q) \nabla \phi(\tau, \cdot) 
			\big\|_{L^{\infty}} \, d \tau 
			\ + \ \int_{\tau = 0}^{t} (1 + \tau) \big\| \varpi(q) \mathfrak{I}(\tau, \cdot) \big\|_{L^{\infty}(D_{\tau})} 
			\, d \tau \notag \\
		& \ \ + \ \int_{\tau = 0}^{t} \sum_{|I| \leq 2} (1 + \tau)^{-1} \big\| \varpi(q) \nabla_{\mathcal{Z}}^I \phi(\tau, \cdot) 
			\big\|_{L^{\infty}(D_{\tau})} \, d \tau. \notag
	\end{align}
	
\end{lemma}

\hfill $\qed$

We now state the following corollary, which provides similar decay estimates for
the null components of tensorial systems of wave equations.

\begin{corollary} \cite[Corollary 7.2]{hLiR2010}   \label{C:systemdecay}
	\textbf{(Pointwise decay estimates for solutions to a system of tensorial wave 
	equations)}
	Let $\phi_{\mu \nu}$ be a solution of the system 
	
	\begin{align} \label{E:system}
		\widetilde{\Square}_g \phi_{\mu \nu} = \mathfrak{I}_{\mu \nu}
	\end{align}
	on a curved background with a metric $g_{\mu \nu}.$ Assume that the tensor $H^{\mu \nu} \eqdef (g^{-1})^{\mu \nu} - 
	(m^{-1})^{\mu \nu}$ obeys the following pointwise estimates 
	
	\begin{align}
		|H| \leq \frac{\varepsilon'}{4}, &&
		\int_{0}^{\infty} (1 + t)^{-1} \big \| H(t,\cdot) \big \|_{L_{\infty}(D_t)} \,dt \leq \varepsilon', &&
		|H|_{\mathcal{L} \mathcal{T}} \leq \frac{\varepsilon'}{4} (1 + t + |q|)^{-1} (1 + |q|)
	\end{align}
	in the region
	\begin{align}
		D_{t} \eqdef \lbrace x: t/2 < |x| < 2 t \rbrace
	\end{align}
	for $t \in [0,T).$ Then for any $\mathcal{U},\mathcal{V} \in \lbrace \mathcal{L},\mathcal{T},\mathcal{N} \rbrace$ and with 
	$\upalpha \eqdef \max(1 + \upgamma', 1/2 - \upmu'),$ the following pointwise estimate holds for $(t,x) \in [0,T) \times 
	\mathbb{R}^3:$
	
	\begin{align} \label{E:systemdecay}
		(1 + t + |q|) \varpi(q) |\nabla \phi|_{\mathcal{U}\mathcal{V}} & \lesssim \sup_{0 \leq \tau \leq t} \sum_{|I| \leq 1} 
			\big\| \varpi(q)\nabla_{\mathcal{Z}}^I \phi(t,x) \big\|_{L^{\infty}} \\
			& \ \ +	\ \int_{\tau = 0}^{t} \varepsilon' \upalpha \big\| \varpi(q) |\nabla \phi(\tau, \cdot)|_{\mathcal{U}\mathcal{V}} 
				\big\|_{L^{\infty}} \, d \tau 
			\ +	\ \int_{\tau = 0}^{t} (1 + \tau) \big\| \varpi(q) |\mathfrak{I}(\tau, \cdot)|_{\mathcal{U}\mathcal{V}} 
				\big\|_{L^{\infty}(D_{\tau})} \, d \tau \notag \\
			& \ \ + \ \sum_{|I| \leq 2} \int_{\tau = 0}^{t} (1 + \tau)^{-1} 
				\big\| \varpi(q) \nabla_{\mathcal{Z}}^I \phi(\tau, \cdot) \big\|_{L^{\infty}(D_{\tau})} \, d \tau. \notag
	\end{align}
	
\end{corollary}

\hfill $\qed$

\section{Local Existence and the Continuation Principle for the Reduced Equations} \label{S:LocalExistence}

In this short section, we state for convenience a standard proposition concerning local existence and a continuation principle for the reduced equations \eqref{E:Reducedh1Summary} - \eqref{E:ReduceddMis0Summary}. The continuation principle shows that an a-priori smallness condition on the energy of the solution is sufficient to deduce global existence. It therefore plays a
fundamental role in our global stability argument of Section \ref{S:GlobalExistence}.

\noindent \hrulefill
\ \\

\begin{proposition} \label{P:LocalExistence}
	\textbf{(Local existence and the continuation principle)}
	Let $(h_{\mu \nu}^{(1)}|_{t=0},\partial_t h_{\mu \nu}^{(1)}|_{t=0}, \Far_{\mu \nu}|_{t=0}),$ $(\mu, \nu = 0,1,2,3),$
	be initial data for the reduced equations \eqref{E:Reducedh1Summary} - \eqref{E:ReduceddMis0Summary} constructed from 
	abstract initial data 
	$(\mathring{\underline{h}}_{jk}^{(1)}, \mathring{K}_{jk}, \mathring{\mathfrak{\Displacement}}_j, 	
	\mathring{\mathfrak{\Magneticinduction}}_j),$ $(j,k = 1,2,3),$
	on the manifold $\mathbb{R}^3$ satisfying the constraints \eqref{E:Gauss} - \eqref{E:DivergenceB0} as described in Section 
	\ref{SS:ReducedData}. Assume that the data are asymptotically flat in the sense of \eqref{E:metricdataexpansion} -
	\eqref{E:BdecayAssumption}. Let $\dParameter \geq 3$ be an integer, and let $\upgamma > 0, \upmu > 0$ be constants. 
	Assume that $E_{\dParameter;\upgamma}(0) < \varepsilon,$ where $E_{\dParameter;\upgamma}(0)$ is the norm of the abstract data 
	defined in \eqref{E:DataNorm}. Then if $\varepsilon$ is sufficiently small\footnote{This smallness assumption 
	ensures that the reduced data lie within the regime of hyperbolicity of the reduced equations.}, these data launch a unique 
	classical solution to the reduced equations existing on a nontrivial maximal spacetime slab $[0,T_{max}) \times 
	\mathbb{R}^3.$ The energy $\mathcal{E}_{\dParameter;\upgamma;\upmu}(t)$ of the solution, which is defined in 
	\eqref{E:EnergyIntro}, satisfies $\mathcal{E}_{\dParameter;\upgamma;\upmu}(0) \lesssim \varepsilon$ and is continuous on 
	$[0,T_{max}).$ Furthermore, either $T_{max} = \infty,$ or one of the following two ``breakdown'' scenarios must occur:
	
	\begin{enumerate}
		\item $\lim_{t \uparrow T_{max}} \mathcal{E}_{\dParameter;\upgamma;\upmu}(t) = \infty.$ 
		\item The solution escapes the regime of hyperbolicity of the reduced equations.
	\end{enumerate}
	
\end{proposition}

\begin{remark}
	The classification of the two breakdown scenarios is known as a \emph{continuation principle}.
\end{remark}

The main ingredients in the proof of Proposition \ref{P:LocalExistence} are Lemma \ref{L:weightedenergyFar} and Lemma \ref{L:weightedenergy}, which provide weighted energy estimates for linearized versions of the reduced equations. Based on the availability of these estimates, the proof is rather standard, and we omit the details. Readers may consult 
e.g. \cite[Ch. VI]{lH1997}, \cite{aM1984}, \cite{jSmS1998}, \cite{cS2008}, \cite{jS2008a}, and \cite[Ch. 16]{mT1997III} for details concerning local existence, and e.g. \cite{jS2008b} for the ideas behind the continuation principle.

\section{The Fundamental Energy Bootstrap Assumption and Pointwise Decay Estimates for the Reduced Equations} \label{S:DecayFortheReducedEquations}  
\setcounter{equation}{0}

In this section, we introduce our fundamental bootstrap assumption \eqref{E:Bootstrap} for the energy of a solution to the reduced equations. Under this assumption, we derive a collection of pointwise decay estimates that will play a crucial role in the proof of Theorem \ref{T:ImprovedDecay}. In particular, these decay estimates are used to deduce the factors $(1 + \tau)^{-1}$ and $(1 + \tau)^{- 1 + C \varepsilon}$ in \eqref{E:Gronwallreadyinequalityk}, which are essential for deriving the bound \eqref{E:ImprovedEnergyInequality}. Many of the estimates we derive in this section rely upon the wave coordinate condition.

\noindent \hrulefill
\ \\
We recall that the spacetime metric $g_{\mu \nu}$ is split into the pieces 
$g_{\mu \nu} = m_{\mu \nu} + h_{\mu \nu}^{(0)} + h_{\mu \nu}^{(1)},$ and that the energy 
$\mathcal{E}_{\dParameter;\upgamma;\upmu}(t)$ (see \eqref{E:EnergyIntro}) is a functional of $(h^{(1)},\Far).$ 
Our main bootstrap assumption for the energy is 

\begin{align} \label{E:Bootstrap}
	\mathcal{E}_{\dParameter;\upgamma;\upmu}(t) \leq \varepsilon(1 + t)^{\updelta},
\end{align}
where $0 < \upgamma < 1/2$ is a fixed constant, $\updelta$ is a fixed constant satisfying both $0 < \updelta < 1/4$ and $0< \updelta < \upgamma,$ $0 < \upmu < 1/2$ is a fixed constant, (all of which will be chosen during the proof of Theorem \ref{T:MainTheorem}), and $\varepsilon$ is a small positive number whose required smallness is adjusted (as many times as necessary) during the derivation of our inequalities. With the help of \eqref{E:LieZIinTermsofNablaZI}, inequality \eqref{E:Bootstrap} implies the following more explicit consequence of the 
energy bootstrap assumption:

\begin{align} \label{E:Boostrapexplicit}
	\sum_{|I| \leq \dParameter }  \big\| w^{1/2}(q) \nabla\nabla_{\mathcal{Z}}^I h^{(1)} \big\|_{L^2}
		\ + \ \big\| w^{1/2}(q) \Lie_{\mathcal{Z}}^I \Far \big\|_{L^2} \leq C \varepsilon (1 + t)^{\updelta}.
\end{align}

In the remaining estimates in this article, \textbf{we will also often make the following smallness assumption on the ADM mass}:

\begin{align} \label{E:Missmall}
	M \leq \varepsilon.
\end{align}

\subsection{Preliminary (weak) pointwise decay estimates} \label{SS:PreliminaryLinfinityEstimates}

In this section, we provide some preliminary pointwise decay estimates that are essentially a consequence of the weighted Klainerman-Sobolev inequalities of Appendix \ref{A:WeightedKS}. Unlike the upgraded pointwise decay estimates of the next section, these estimates do not take into account the special structure of the reduced equations under the wave coordinate condition.

We begin with a lemma concerning pointwise decay estimates for the Schwarzschild tail of the metric and its derivatives.

\begin{lemma} \label{L:h0decayestimates}
		\textbf{(Decay estimates for $h^{(0)}$)}
		Let $h^{(0)}$ be as in \eqref{E:h0defIntro}, and let $I$ be any $\nabla-$multi-index. Then the following
		pointwise estimate holds for $(t,x) \in [0,\infty) \times \mathbb{R}^3:$
	
	\begin{subequations}
	\begin{align} \label{E:nablaIh0Linfinity}
		|\nabla^I h^{(0)}| \leq C M (1 + t + |q|)^{-1+ |I|},
	\end{align}
	where $M$ is the ADM mass.
	
	Furthermore, if $I$ is any $\nabla-$multi-index and $J$ is any $\mathcal{Z}-$multi-index,
	then the following pointwise estimate holds for $(t,x) \in [0,\infty) \times \mathbb{R}^3:$
	
		\begin{align} \label{E:decaynablaJnablaZIh0Linfinity}
			|\nabla^I \nabla_{\mathcal{Z}}^J h^{(0)}| \ + \ |\nabla_{\mathcal{Z}}^J \nabla^I h^{(0)}|
			\leq C M (1 + t + |q|)^{- 1 + |I|}. 
		\end{align}
  	\end{subequations}
\end{lemma}

\begin{remark}
	Since $H_{(0)\mu \nu} = - h_{\mu \nu}^{(0)}$ (where $H_{(0)}^{\mu \nu}$ is defined in 
	\eqref{E:NablaZIh1LLh1LTwaveCoordinateAlgebraicEstimate}), the above estimates also hold if we replace
	$h^{(0)}$ with $H_{(0)}.$
\end{remark}

\begin{proof}
	The lemma follows from simple computations using the definition \eqref{E:chidef} of the cut-off function $\chi,$
	the definition of $h^{(0)},$ and the definitions of the vectorfields $Z \in \mathcal{Z}.$
\end{proof}

\begin{corollary} \label{C:WeakDecay} \cite[Slight extension of Corollary 9.4]{hLiR2010}
	\textbf{(Weak pointwise decay estimates)}
	Let $\dParameter \geq 8$ be an integer. Assume that the abstract initial data are asymptotically flat in the sense 
	of \eqref{E:metricdataexpansion} - \eqref{E:BdecayAssumption}, that the ADM mass smallness condition
	\eqref{E:Missmall} holds, and that the initial data for the reduced
	system are constructed from the abstract initial data as described in Section \ref{SS:ReducedData}.
	Let $(g_{\mu \nu} \eqdef m_{\mu \nu} + h_{\mu \nu}^{(0)} + h_{\mu 
	\nu}^{(1)}, \Far_{\mu \nu})$ be the corresponding solution to the reduced system \eqref{E:Reducedh1Summary} - 
	\eqref{E:ReduceddMis0Summary} existing on a slab $(t,x) \in [0,T) \times 
	\mathbb{R}^3,$ where $h^{(1)}$ is defined in \eqref{E:hdefIntro}. Assume in addition that the pair 
	$(h^{(1)}, \Far)$ 
	satisfies the energy bootstrap assumption \eqref{E:Bootstrap} on the interval $[0,T).$
	Then if $\varepsilon$ is sufficiently small, the following pointwise estimates hold for $(t,x) \in [0,T) \times 
	\mathbb{R}^3:$
	
	\begin{subequations}
	\begin{align} \label{E:weakdecaypartialLinfinity}
		|\nabla\nabla_{\mathcal{Z}}^I h^{(1)}| 
		\ + \ |\Lie_{\mathcal{Z}}^I \Far| \leq
		\left \lbrace \begin{array}{lr}
    	C \varepsilon (1 + t + |q|)^{-1} (1 + t)^{\updelta} (1 + |q|)^{-1 - \upgamma}, &  \mbox{if} \ q > 0, \\
      C \varepsilon (1 + t + |q|)^{-1} (1 + t)^{\updelta} (1 + |q|)^{-1/2}, & \mbox{if} \ q < 0,
    \end{array}
  	\right., &&
  	|I| \leq \dParameter - 2, 
	\end{align}

	\begin{align} \label{E:weakdecayLinfinity}
			|\nabla_{\mathcal{Z}}^I h^{(1)}|  \leq
			\left \lbrace \begin{array}{lr}
    		C \varepsilon (1 + t + |q|)^{-1 + \updelta} (1 + |q|)^{- \upgamma}, &  \mbox{if} \ q > 0, \\
      	C \varepsilon (1 + t + |q|)^{-1 + \updelta} (1 + |q|)^{1/2}, & \mbox{if} \ q < 0,
    	\end{array}
  		\right., &&
  		|I| \leq \dParameter -2,  
  	\end{align}

  	\begin{align} \label{E:weakdecaybarpartialLinfinity}
		|\conenabla \nabla_{\mathcal{Z}}^I h^{(1)}| 
		\ + \ (1 + |q|)|\conenabla \Lie_{\mathcal{Z}}^I \Far| \leq
		\left \lbrace \begin{array}{lr}
    	C \varepsilon (1 + t + |q|)^{-2 + \updelta} (1 + |q|)^{- \upgamma}, &  \mbox{if} \ q > 0, \\
      C \varepsilon (1 + t + |q|)^{-2 + \updelta} (1 + |q|)^{1/2}, & \mbox{if} \ q < 0,
    \end{array}
  	\right., &&
  	|I| \leq \dParameter - 3.  
  	\end{align}
  	\end{subequations}

  	In addition, the tensorfield $H_{(1)}^{\mu \nu}$ 
  	defined in \eqref{E:NablaZIh1LLh1LTwaveCoordinateAlgebraicEstimate} satisfies the same estimates as $h_{\mu \nu}^{(1)}.$
  	Furthermore, if we make the substitution $\upgamma \rightarrow \updelta$ in the above inequalities, then the same estimates 
  	hold for the tensorfields $h_{\mu \nu}^{(0)},$ $h_{\mu \nu} \eqdef h_{\mu \nu}^{(0)} + h_{\mu \nu}^{(1)},$ 
  	$H_{(0)\mu \nu} \eqdef - h_{\mu \nu}^{(0)},$ $H^{\mu \nu} \eqdef (g^{-1})^{\mu \nu} - (m^{-1})^{\mu \nu},$
  	and $H_{(1)}^{\mu \nu} \eqdef H^{\mu \nu} - H_{(0)}^{\mu \nu}.$ 
\end{corollary}

\begin{proof}
	This Corollary is a slight extension of Corollary 9.4 of \cite{hLiR2010}, in which estimates for $h^{(0)} = 
	- H_{(0)},$ $h^{(1)},$ and $h$ were proved. The main idea in the proof is to use the weighted Klainerman-Sobolev 
	estimates of Proposition \ref{P:WeightedKlainermanSobolev} under the assumption \eqref{E:Boostrapexplicit}, together with the
	decay \eqref{E:h1AbstractDataAsymptotics} - \eqref{E:BdecayAssumption} of the initial data at $\infty.$ The estimates for 
	$\Far$ follow from the arguments of \cite[Corollary 9.4]{hLiR2010}, while the estimates for $H_{(1)}$ and $H$ follow from 
	those for $h^{(1)}$ and $h$ together with \eqref{E:Hintermsofh}.
\end{proof}

The next lemma uses the weak decay estimates to provide algebraic estimates for the Schwarzschild tail term
$\nabla_{\mathcal{Z}}^I \widetilde{\Square}_g h^{(0)}$ appearing on the right-hand side of \eqref{E:Reducedh1Summary}.

\begin{lemma}\cite[Lemma 9.9]{hLiR2010} \label{L:weakdecayLinfinitynablaZISquaregh0} 
\textbf{(Pointwise decay estimates for $\nabla_{\mathcal{Z}}^I \widetilde{\Square}_g h^{(0)}$)}
	Let $h^{(0)}$ be the Schwarzschild part of $h$ as defined in \eqref{E:h0defIntro}, and
	assume the hypotheses/conclusions of Corollary \ref{C:WeakDecay}.
	Let $I$ be a $\mathcal{Z}-$multi-index subject to the restrictions stated below.
	Then if $\varepsilon$ is sufficiently small, the following pointwise estimates hold for $(t,x) \in [0,T) \times 
	\mathbb{R}^3,$ where $M$ is the ADM mass: 
	
	\begin{subequations}
	\begin{align} \label{E:weakdecayLinfinitynablaZISquaregh0}
			|\nabla_{\mathcal{Z}}^I \widetilde{\Square}_g h^{(0)}|  \leq
			\left \lbrace \begin{array}{lr}
    		C M \varepsilon (1 + t + |q|)^{-4 + \updelta} (1 + |q|)^{- \updelta}, &  \mbox{if} \ q > 0,\\
      	C M (1 + t + |q|)^{-3}, & \mbox{if} \ q < 0,
    	\end{array}
  		\right., &&
  		|I| \leq \dParameter -2.  
  	\end{align}

  	Furthermore, the following pointwise estimates also hold for $(t,x) \in [0,T) \times \mathbb{R}^3:$
	  
	  \begin{align} \label{E:weakdecayLinfinitynablaZISquaregh0MoreGeneral}
			|\nabla_{\mathcal{Z}}^I \widetilde{\Square}_g h^{(0)}|  \leq
			\left \lbrace \begin{array}{lr}
    		C M \varepsilon (1 + t + |q|)^{-4}, &  \mbox{if} \ q > 0,\\
      	C M (1 + t + |q|)^{-3}, & \mbox{if} \ q < 0
    	\end{array}
  		\right. \ + \ C M \sum_{|J| \leq |I|} (1 + t + |q|)^{-3} |\nabla_{\mathcal{Z}}^J h^{(1)}|, &&
  		|I| \leq \dParameter .  
  	\end{align}
  	\end{subequations}
  	
\end{lemma}

\begin{proof}
	We first observe that $\widetilde{\Square}_g h^{(0)} = \Square_m h^{(0)} + H^{\kappa \lambda} \nabla_{\kappa} 
	\nabla_{\lambda}h^{(0)},$ where $\Square_m \eqdef (m^{-1})^{\kappa \lambda} \nabla_{\kappa} \nabla_{\lambda}$
	is the Minkowski wave operator. Using \eqref{E:decaynablaJnablaZIh0Linfinity}, the definition of $h^{(0)},$ 
	the Leibniz rule, and the fact that $\Square_m (1/r) = 0$ for $r > 0,$
	it follows that
	
	\begin{align}
		|\nabla_{\mathcal{Z}}^I \Square_m h^{(0)}| \  
			& \lesssim M (1 + t + |q|)^{-3} \chi_0(1/2 \leq r/t \leq 3/4),  \\
  	|\nabla_{\mathcal{Z}}^I \big(H^{\kappa \lambda} \nabla_{\kappa} \nabla_{\lambda}h^{(0)} \big)|	   
  		&  \lesssim M (1 + t + |q|)^{-3} \sum_{|J| \leq |I|} |\nabla_{\mathcal{Z}}^J H|, \label{E:BoxHappliedtoh0}
	\end{align}
	where $\chi_0(1/2 \leq z \leq 3/4)$ is the characteristic function of the interval $[1/2,3/4].$
	Furthermore, using that $H = - h^{(0)} - h^{(1)} + O^{\infty}(\big|h^{(0)} + h^{(1)}|^2 \big),$ it follows that
	
	\begin{align} \label{E:NablaZJHexpanded}
		\sum_{|J| \leq |I|}|\nabla_{\mathcal{Z}}^J H| \lesssim \varepsilon (1 + t + |q|)^{-1} 
			\ + \ \sum_{|J| \leq |I|} |\nabla_{\mathcal{Z}}^J h^{(1)}|.
	\end{align}
	Using \eqref{E:BoxHappliedtoh0}, \eqref{E:NablaZJHexpanded}, and the estimate \eqref{E:weakdecayLinfinity}, 
	we have that
	
	\begin{align} 
			|\nabla_{\mathcal{Z}}^I \big(H^{\kappa \lambda} \nabla_{\kappa} \nabla_{\lambda}h^{(0)} \big)|  
			\lesssim \left \lbrace \begin{array}{lr}
    		M \varepsilon (1 + t + |q|)^{-4 + \updelta} (1 + |q|)^{- \updelta}, &  \mbox{if} \ q > 0,\\
      	M \varepsilon (1 + t + |q|)^{-4 + \updelta} (1 + |q|)^{1/2}, & \mbox{if} \ q < 0,
    	\end{array}
  		\right., && |I| \leq \dParameter - 2, 
  	\end{align}
  	and
  	
  	\begin{align} 
			|\nabla_{\mathcal{Z}}^I \big(H^{\kappa \lambda} \nabla_{\kappa} \nabla_{\lambda}h^{(0)} \big)|  
				\lesssim M \varepsilon (1 + t + |q|)^{-4}  
			\ + \ M \varepsilon (1 + t + |q|)^{-3} \sum_{|J| \leq |I|} |\nabla_{\mathcal{Z}}^J h^{(1)}|, &&
  		|I| \leq \dParameter. 
  	\end{align}
  	Inequalities \eqref{E:weakdecayLinfinitynablaZISquaregh0} and \eqref{E:weakdecayLinfinitynablaZISquaregh0MoreGeneral}
  	now easily follow from the above estimates.
	
\end{proof}

\subsection{Initial upgraded pointwise decay estimates for $|\Lie_{\mathcal{Z}}^I \Far|_{\mathcal{L}\mathcal{N}}$
and $|\Lie_{\mathcal{Z}}^I \Far|_{\mathcal{T}\mathcal{T}}$} \label{SS:InitialFLUTTLinfinityImprovements}

In this section, we prove some upgraded pointwise decay estimates for the ``favorable'' components of the lower-order Lie derivatives of $\Far.$ Our estimates take into account the special structure revealed by our null decomposition of the electromagnetic equations of variations, a structure that was captured by Proposition \ref{P:ODEsNullComponentsLieZIFar}
and that depends in part upon the wave coordinate condition. We remark that in Section \ref{SS:FullUpgradedLinfinityEstimates}, some of these decay estimates will be further improved (hence our use of the term ``initial upgraded'' here).

\begin{proposition} \label{P:FLUTTimproveddecay}
	\textbf{(Initial upgraded pointwise decay estimates for $|\Lie_{\mathcal{Z}}^I \Far|_{\mathcal{L}\mathcal{N}}$ 
	and $|\Lie_{\mathcal{Z}}^I \Far|_{\mathcal{T}\mathcal{T}}$)}
	Assume the hypotheses/conclusions of Corollary \ref{C:WeakDecay}. Then if $\varepsilon$ is sufficiently small, the following 
	pointwise estimates hold for $(t,x) \in [0,T) \times \mathbb{R}^3:$
	
	\begin{align} \label{E:FLUTTimproveddecay}
		|\Lie_{\mathcal{Z}}^I \Far|_{\mathcal{L}\mathcal{N}} 
		\ + \ |\Lie_{\mathcal{Z}}^I \Far|_{\mathcal{T}\mathcal{T}} 
			\leq
		\left \lbrace \begin{array}{lr}
    	C \varepsilon (1 + t + |q|)^{-2 + 2 \updelta} (1 + |q|)^{- \upgamma - \updelta}, &  \mbox{if} \ q > 0, \\
      C \varepsilon (1 + t + |q|)^{-2 + 2 \updelta} (1 + |q|)^{1/2 - \updelta},	& \mbox{if} \ q < 0,
    \end{array}
  	\right., &&
  	|I| \leq \dParameter - 3.  
  	\end{align}
  	
\end{proposition}

\begin{proof}
	Since $|\Lie_{\mathcal{Z}}^I \Far|_{\mathcal{L}\mathcal{N}} +  
	|\Lie_{\mathcal{Z}}^I \Far|_{\mathcal{T}\mathcal{T}}| \approx |\alpha[\Lie_{\mathcal{Z}}^I \Far]| 
	+ |\rho[\Lie_{\mathcal{Z}}^I \Far]| + |\sigma[\Lie_{\mathcal{Z}}^I \Far]|,$ it suffices to prove the desired decay estimates 
	for $|\alpha[\Lie_{\mathcal{Z}}^I \Far]|,$ $|\rho[\Lie_{\mathcal{Z}}^I \Far]|,$ and $|\sigma[\Lie_{\mathcal{Z}}^I \Far]|$ 
	separately. We provide proof for the null component $\alpha[\Lie_{\mathcal{Z}}^I \Far].$ The proofs for the components 
	$\rho[\Lie_{\mathcal{Z}}^I \Far]$ and $\sigma[\Lie_{\mathcal{Z}}^I \Far]$ are similar, and we leave these details to the 
	reader. Let $\mathcal{W} \eqdef \big\lbrace (t,x): |x| \geq 1 + t/2 \big\rbrace \cap \big\lbrace (t,x): |x| \leq 2t - 1 
	\big\rbrace$ denote the ``wave zone'' region. Then for $(t,x) \nin \mathcal{W},$ we have that 
	$1 + |q| \approx (1 + t + |q|).$ Using this fact, for $(t,x) \nin \mathcal{W},$ we can bound $|\alpha[\Lie_{\mathcal{Z}}^I 
	\Far]|$ by the right-hand side of \eqref{E:FLUTTimproveddecay} by using the weak decay estimate	
	\eqref{E:weakdecaypartialLinfinity}.
	
	We now consider the case $(t,x) \in \mathcal{W}.$ Let $f \eqdef r^{-1} \alpha[\Lie_{\mathcal{Z}}^I \Far].$ Then using
	\eqref{E:alphaODE}, the fact that $r \approx (1 + t + |q|) \approx (1 + s + |q|)$ on $\mathcal{W},$
	and the weak decay estimates of Corollary \ref{C:WeakDecay}, it follows that (with $\partial_q$ defined in Section \ref{SS:Derivatives})
	
	\begin{align} \label{E:partialqfbound}
		|\partial_q f(t,x)| & \lesssim \left \lbrace \begin{array}{lr}
    	\varepsilon (1 + s + |q|)^{-3 + \updelta} (1 + |q|)^{-1 - \upgamma} 
    		\ + \ \varepsilon(1 + s + |q|)^{-3 + 2\updelta} (1 + |q|)^{-2 - 2 \upgamma}, &  \mbox{if} \ q > 0, \\
      \varepsilon (1 + s + |q|)^{-3 + \updelta} (1 + |q|)^{- 1/2} 
      	\ + \ \varepsilon(1 + s + |q|)^{-3 + 2\updelta} (1 + |q|)^{-1}, & \mbox{if} \ q < 0.
    \end{array}
  	\right.&&
	\end{align}
	
	Let $(\tau(q'), y(q'))$ be the $q'-$parameterized line segment of constant $s$ and angular 
	values that initiates at $(t,x)$ and terminates at the point $(t_0,x_0)$ which lies to the \emph{past} of $(t,x)$ and on
	the boundary of $\mathcal{W}.$ Let $q,s$ be the null coordinates corresponding to $(t,x).$
	Then the null coordinates corresponding to $(t_0,x_0)$ are $q_0 = s/3 - 2/3, s_0 = s.$ Integrating
	the inequality \eqref{E:partialqfbound} along this line segment (i.e., integrating $dq'$), we have that
	
	\begin{align} \label{E:qIntegratedfBound}
		|f(t,x)| & \lesssim |f(t_0,x_0)|  \\
			& \ \ + \ \int_{q'= q}^{q'= s/3 - 2/3} 
				 \left \lbrace \begin{array}{lr}
    		\varepsilon (1 + s + |q'|)^{-3 + \updelta} (1 + |q'|)^{-1 - \upgamma} 
    		\ + \ \varepsilon(1 + s + |q'|)^{-3 + 2 \updelta} (1 + |q'|)^{-2 - 2 \upgamma}, &  \mbox{if} \ q' > 0, \\
       	\varepsilon(1 + s + |q'|)^{-3 + \updelta} (1 + |q'|)^{- 1/2} 
      	\ + \ \varepsilon(1 + s + |q'|)^{-3 + 2\updelta} (1 + |q'|)^{-1}, & \mbox{if} \ q' < 0,
    \end{array}
  	\right\rbrace \, dq' \notag \\
  	& \lesssim |f(t_0,x_0)| 
  		\ + \
  			\left \lbrace \begin{array}{lr}
    		\varepsilon (1 + s)^{-3 + \updelta} (1 + |q|)^{- \upgamma} 
    		\ + \ \varepsilon (1 + s )^{-3 + 2 \updelta} (1 + |q|)^{-1 - 2 \upgamma}, &  \mbox{if} \ q > 0, \\
       	\varepsilon(1 + s)^{-3 + \updelta} (1 + |q|)^{1/2} 
      	\ + \ \varepsilon(1 + s)^{-3 + 2\updelta} \ln(1 + |q|), & \mbox{if} \ q < 0.
    \end{array}
  	\right. \notag
	\end{align}	
	
	Using the facts that $r_0 \approx 1 + |q_0| \approx 1 + t_0 + |q_0| \approx 
	 1 + s_0 + |q_0| \approx 1 + s,$ together
	with the weak decay estimate \eqref{E:weakdecaypartialLinfinity}, it follows that
	
	\begin{align} \label{E:fTerminalPointBound}
		|f(t_0,x_0)| & \lesssim \varepsilon (1 + s)^{-3 - \upgamma + \updelta}.
  \end{align}
  Combining \eqref{E:qIntegratedfBound} and \eqref{E:fTerminalPointBound}, and using the fact that
  $1 + s \approx 1 + t + |q|,$ it follows that
  $|\alpha[\Lie_{\mathcal{Z}}^I \Far(t,x)]|$ is bounded from above by the right-hand side of
  \eqref{E:FLUTTimproveddecay}. This completes our proof of \eqref{E:FLUTTimproveddecay} for the 
  $\alpha[\Lie_{\mathcal{Z}}^I \Far]$ component. 
  
\end{proof}

\subsection{Upgraded pointwise decay estimates for $|\nabla_{\mathcal{Z}}^I h|,$ 
$|\Lie_{\mathcal{Z}}^I \Far|,$ and fully upgraded pointwise decay estimates for $|\Lie_{\mathcal{Z}}^I \Far|_{\mathcal{L}\mathcal{N}},$  $|\Lie_{\mathcal{Z}}^I \Far|_{\mathcal{T}\mathcal{T}} $} \label{SS:FullUpgradedLinfinityEstimates}

In this section, we state two propositions that strengthen some of the pointwise decay estimates proved in sections 
\ref{SS:PreliminaryLinfinityEstimates} and \ref{SS:InitialFLUTTLinfinityImprovements}. Their proofs, which are provided in sections \ref{SS:ProofofUpgradedDecayhA} and \ref{SS:Proof of PropositionUpgradedDecayh1A}, are based on a careful analysis of the special structure of the reduced equations and in particular rely upon the decompositions performed in Section \ref{S:AlgebraicEstimates}, which rely in part upon the wave coordinate condition. These estimates play a central role in our derivation of the ``strong'' energy inequality \eqref{E:ImprovedEnergyInequality}, which is the main step in the proof of our stability theorem.

\begin{proposition} \cite[Extension of Proposition 10.1]{hLiR2010} \label{P:UpgradedDecayhA}
	\textbf{(Upgraded pointwise decay estimates for $\Far$ and certain components of $h,$ $\nabla h,$ and $\nabla_Z h$)}
	Assume that the abstract initial data satisfy the constraints \eqref{E:Gauss} - \eqref{E:DivergenceB0}, and
	assume the hypotheses/conclusions of Corollary \ref{C:WeakDecay}. In particular, by Proposition 
	\ref{P:PreservationofWaveCoordianteGauge}, the wave coordinate condition \eqref{E:wavecoordinategauge1} holds for $(t,x) \in 
	[0,T) \times \mathbb{R}^3.$ Then if $\varepsilon$ is sufficiently small, for every vectorfield $Z \in \mathcal{Z},$ the 
	following pointwise estimates hold for $(t,x) \in [0,T) \times \mathbb{R}^3:$ 
	
		\begin{subequations}
		\begin{align} \label{E:partialhLTpartialZhLLpointwise}
			|\nabla h|_{\mathcal{L} \mathcal{T}} 
			\ + \ |\nabla\nabla_Z h|_{\mathcal{L} \mathcal{L}} 
				\leq \left \lbrace \begin{array}{lr}
	    	C \varepsilon (1 + t + |q|)^{-2 + \updelta} (1 + |q|)^{- \updelta}, &  \mbox{if} \ q > 0, \\
	      C \varepsilon (1 + t + |q|)^{-2 + \updelta} (1 + |q|)^{1/2}, & \mbox{if} \ q < 0,
	    \end{array}
	  	\right., 
  	\end{align}
  	
  	\begin{align} \label{E:hLTZhLLpointwise}
			|h|_{\mathcal{L} \mathcal{T}} 
			\ + \ |\nabla_Z h|_{\mathcal{L} \mathcal{L}} \leq
			\left \lbrace \begin{array}{lr}
	    	C \varepsilon (1 + t + |q|)^{-1} , &  \mbox{if} \ q > 0, \\
	      C \varepsilon (1 + t + |q|)^{-1} (1 + |q|)^{1/2 + \updelta}, & \mbox{if} \ q < 0,
	    \end{array}
	  	\right., 
  	\end{align}
  	\end{subequations}
  	
  	\begin{subequations}
  	\begin{align}
  		|\nabla h|_{\mathcal{T} \mathcal{N}} 
  			& \leq C \varepsilon (1 + t + |q|)^{-1}, \label{E:partialhTUpointwise} \\
  		|\nabla h| & \leq C \varepsilon (1 + t + |q|)^{-1} \big\lbrace 1 + \ln (1 + t) \big\rbrace, \label{E:partialhpointwise}
  	\end{align}
  	\end{subequations}
  	
  \begin{align} \label{E:Farupgradedecay}
		|\Far| & \leq C \varepsilon (1 + t + |q|)^{-1}.  
	\end{align}
		
		Furthermore, the same estimates hold for the tensorfields $h_{\mu \nu}^{(0)},$ $h_{\mu \nu}^{(1)},$
		$H^{\mu \nu} \eqdef (g^{-1})^{\mu \nu} - (m^{-1})^{\mu \nu},$ $H_{(0)}^{\mu \nu},$ and $H_{(1)}^{\mu \nu}.$
\end{proposition}

\begin{proposition} \cite[Extension of Proposition 10.2]{hLiR2010} \label{P:UpgradedDecayh1A}
	\textbf{(Upgraded pointwise decay estimates for the lower-order derivatives of $h$
	and $\Far$)}
	Under the assumptions of Proposition \ref{P:UpgradedDecayhA}, let
	$0 < \upgamma' < \upgamma - \updelta$ and $0 < \updelta < \upmu' < 1/2$ be fixed constants. Let $I$ be any 
	$\mathcal{Z}-$multi-index subject to the restrictions stated below. Then there exist constants $M_k, C_k,$ and 
	$\varepsilon_k$ depending on $\upgamma', \upmu', \updelta$ such that if $\varepsilon$ is sufficiently small, then the 
	following pointwise estimates hold for $(t,x) \in [0,T) \times \mathbb{R}^3:$
	
	\begin{subequations}
	\begin{align} \label{E:partialZIh1Aupgraded}
			|\nabla\nabla_{\mathcal{Z}}^I h^{(1)}|
			\ + \ |\Lie_{\mathcal{Z}}^I \Far|  \leq
			\left \lbrace \begin{array}{lr}
	    	C_k \varepsilon (1 + t + |q|)^{-1 + M_k \varepsilon} (1 + |q|)^{-1 - \upgamma'}, &  \mbox{if} \ q > 0, \\
	      C_k \varepsilon (1 + t + |q|)^{-1 + M_k \varepsilon} (1 + |q|)^{-1/2 + \upmu'}, & \mbox{if} \ q < 0,
	    \end{array}
	  	\right., && 
	  	|I|= k \leq \dParameter - 4,
  \end{align}

  	\begin{align} \label{E:ZIh1Aupgraded}
			|\nabla_{\mathcal{Z}}^I h^{(1)}| \leq
			\left \lbrace \begin{array}{lr}
	    	C_k \varepsilon (1 + t + |q|)^{-1 + M_k \varepsilon} (1 + |q|)^{- \upgamma'}, &  \mbox{if} \ q > 0, \\
	      C_k \varepsilon (1 + t + |q|)^{-1 + M_k \varepsilon} (1 + |q|)^{1/2 + \upmu'}, & \mbox{if} \ q < 0,
	    \end{array}
	  	\right., &&  
	  	|I|= k \leq \dParameter - 4,
  	\end{align}
  	
  	\begin{align} \label{E:barpartialZIh1Aupgraded}
			|\conenabla \nabla_{\mathcal{Z}}^I h^{(1)}| \ + \ (1 + |q|)|\conenabla \Lie_{\mathcal{Z}}^I \Far| 
			\ + \ |\Lie_{\mathcal{Z}}^I \Far|_{\mathcal{L}\mathcal{N}}  
			\ + \ |\Lie_{\mathcal{Z}}^I \Far|_{\mathcal{T}\mathcal{T}} \leq
			\left \lbrace \begin{array}{lr}
	    	C_k \varepsilon (1 + t + |q|)^{-2 + M_k \varepsilon} (1 + |q|)^{- \upgamma'}, &  \mbox{if} \ q > 0, \\
	      C_k \varepsilon (1 + t + |q|)^{-2 + M_k \varepsilon} (1 + |q|)^{1/2 + \upmu'}, & \mbox{if} \ q < 0,
	    \end{array}
	  	\right., && 
	  	|I|= k \leq \dParameter - 5. 
  	\end{align}
  	\end{subequations}
  	
  	Furthermore, the same estimates hold for $h_{\mu \nu} \eqdef g_{\mu \nu} - m_{\mu \nu}$ and 
  	$H^{\mu \nu} \eqdef (g^{-1})^{\mu \nu} - (m^{-1})^{\mu \nu}$ if we replace $\upgamma'$ with $M_k \varepsilon.$

\end{proposition}

 \subsection{Proof of Proposition \ref{P:UpgradedDecayhA}} \label{SS:ProofofUpgradedDecayhA}
		We only prove the estimates for $h_{\mu \nu}$ and $\Far_{\mu \nu}.$ The estimates for $h_{\mu \nu}^{(0)},$ $h_{\mu 
		\nu}^{(1)},$ $H^{\mu \nu},$ $H_{(0)}^{\mu \nu},$ and $H_{(1)}^{\mu \nu}$ follow easily from those for 
		$h_{\mu \nu},$ \eqref{E:Hintermsofh}, and Lemma \ref{L:h0decayestimates}.

\subsubsection{Proofs of \eqref{E:partialhLTpartialZhLLpointwise} and \eqref{E:hLTZhLLpointwise}}

To prove \eqref{E:partialhLTpartialZhLLpointwise} and \eqref{E:hLTZhLLpointwise}, we will argue as in Lemma 10.4 of
\cite{hLiR2010}; we first provide a lemma that establishes a more general version of the desired estimates.

\begin{lemma} \label{L:UpgradedDecayhA} \cite[Lemma 10.4]{hLiR2010} 
	\textbf{(Pointwise estimates for $|\nabla\nabla_{\mathcal{Z}}^I h|_{\mathcal{L} \mathcal{L}},$
	$|\nabla_{\mathcal{Z}}^I h|_{\mathcal{L} \mathcal{L}},$ $|\nabla\nabla_{\mathcal{Z}}^I h|_{\mathcal{L} \mathcal{T}},$ 
	and $|\nabla_{\mathcal{Z}}^I h|_{\mathcal{L} \mathcal{T}}$)}
	Under the hypotheses of Proposition \ref{P:UpgradedDecayhA}, if $k \leq \dParameter - 3$ 
	and $\varepsilon$ is sufficiently small, then the following pointwise estimates hold for $(t,x) \in [0,T) \times 
	\mathbb{R}^3:$
	
	\begin{align} \label{E:partialZIhLLpluspartialZJhLTLinfinity}
		\sum_{|I| \leq k} |\nabla\nabla_{\mathcal{Z}}^I h|_{\mathcal{L} \mathcal{L}}
			\ + \ \underbrace{\sum_{|J| \leq k - 1} 
				|\nabla\nabla_{\mathcal{Z}}^J h|_{\mathcal{L} \mathcal{T}}}_{\mbox{absent if $k = 0$}}
			\lesssim \underbrace{\sum_{|K| \leq k - 2} |\nabla\nabla_{\mathcal{Z}}^K h|}_{\mbox{absent if $k \leq 1$}} 
			\ + \ \left\lbrace \begin{array}{lr}
	    	\varepsilon (1 + t + |q|)^{-2 + 2 \updelta} (1 + |q|)^{- 2 \updelta}, &  \mbox{if} \ q > 0, \\
	      \varepsilon (1 + t + |q|)^{-2 + 2 \updelta} (1 + |q|)^{1/2 - \updelta}, & \mbox{if} \ q < 0,
	    \end{array}
	  	\right., \\
	 	\sum_{|I| \leq k} |\nabla_{\mathcal{Z}}^I h|_{\mathcal{L} \mathcal{L}}
			\ + \ \underbrace{\sum_{|J| \leq k - 1} |\nabla_{\mathcal{Z}}^J h|_{\mathcal{L} \mathcal{T}}}_{\mbox{absent if $k = 0$}}
			\lesssim \underbrace{\sum_{|K| \leq k - 2} \int_{\varrho = |x|}^{\varrho = |x| + t}
				|\nabla\nabla_{\mathcal{Z}}^K h|(t + |q| - \varrho, \varrho x/|x|) \, d \varrho }_{\mbox{absent if $k \leq 1$}}
			\ + \ \left\lbrace \begin{array}{lr}
	    	\varepsilon (1 + t + |q|)^{-1}, &  \mbox{if} \ q > 0, \\
	      \varepsilon (1 + t + |q|)^{-1} (1 + |q|)^{1/2 + \updelta}, & \mbox{if} \ q < 0.
	    \end{array}
	  \right. \label{E:partialZIhLLpluspartialZJhLTLinfinityintegrated}
	\end{align}	
	Furthermore, the same estimates hold for the tensor $H^{\mu \nu} \eqdef (g^{-1})^{\mu \nu} - (m^{-1})^{\mu \nu}.$
\end{lemma}	

\begin{proof}
	By Proposition \ref{P:harmonicgauge}, we have that 
	
	\begin{align} \label{E:harmonicgaugeagain}
		\sum_{|I| \leq k} |\nabla\nabla_{\mathcal{Z}}^I h|_{\mathcal{L} \mathcal{L}}
			\ + \ \underbrace{\sum_{|J| \leq k - 1} |\nabla\nabla_{\mathcal{Z}}^J h|_{\mathcal{L} \mathcal{T}}}_{\mbox{absent if 
			$k = 0$}} 
			\lesssim \underbrace{\sum_{|K| \leq k - 2} |\nabla\nabla_{\mathcal{Z}}^{J} h|}_{\mbox{absent if $k \leq 1$}}
			\ + \ \sum_{|J| \leq k} |\conenabla \nabla_{\mathcal{Z}}^J h|
			\ + \ \sum_{|I_1| + |I_2| \leq k} |\nabla_{\mathcal{Z}}^{I_1} h||\nabla\nabla_{\mathcal{Z}}^{I_2} h|.
	\end{align}
	By Corollary \ref{C:WeakDecay}, we have that
	
	\begin{align} \label{E:Tangentialhderivatives}
		\sum_{|J| \leq k} |\conenabla \nabla_{\mathcal{Z}}^J h| 
		\ + \ \sum_{|I_1| + |I_2| \leq k} |\nabla_{\mathcal{Z}}^{I_1} h||\nabla\nabla_{\mathcal{Z}}^{I_2} h|  
		\lesssim \left \lbrace \begin{array}{lr}
    	\varepsilon (1 + t + |q|)^{-2 + 2\updelta} (1 + |q|)^{- 2\updelta}, &  \mbox{if} \ q > 0, \\
      \varepsilon (1 + t + |q|)^{-2 + 2\updelta} (1 + |q|)^{1/2 - \updelta}, & \mbox{if} \ q < 0,
    \end{array}
  	\right., &&
  	k \leq \dParameter - 3.
	\end{align}
	Combining \eqref{E:harmonicgaugeagain} and \eqref{E:Tangentialhderivatives},
	we deduce \eqref{E:partialZIhLLpluspartialZJhLTLinfinity}. Inequality 
	\eqref{E:partialZIhLLpluspartialZJhLTLinfinityintegrated}
	follows from integrating inequality  \eqref{E:partialZIhLLpluspartialZJhLTLinfinity}
	for $|\partial_q \nabla_{\mathcal{Z}}^I h| \lesssim |\nabla\nabla_{\mathcal{Z}}^I h|,$ $q \eqdef |x| - t,$ along the 
	lines along which the angle $\omega \eqdef x/|x|$ and the null coordinate $s = |x| + t$ are constant (i.e. integrating $dq$), 
	and using \eqref{E:weakdecayLinfinity} at $t=0.$
	
	The proofs of the estimates for $H^{\mu \nu}$ follow from the estimates for $h_{\mu \nu},$ \eqref{E:Hintermsofh},
	and Corollary \ref{C:WeakDecay}.

\end{proof}

Inequalities \eqref{E:partialhLTpartialZhLLpointwise} and \eqref{E:hLTZhLLpointwise} now follow from
inequalities \eqref{E:partialZIhLLpluspartialZJhLTLinfinity}, \eqref{E:partialZIhLLpluspartialZJhLTLinfinityintegrated},
and the weak decay estimates of Corollary \ref{C:WeakDecay}.

\hfill $\qed$

\subsubsection{Proof of \eqref{E:Farupgradedecay}} \label{SSS:ProofofFarUpgradedDecay}
	Let $\mathcal{W} \eqdef \big\lbrace (t,x): |x| \geq 1 + t/2 \big\rbrace \cap \big\lbrace 
	(t,x): |x| \leq 2t - 1 \big\rbrace$ denote the ``wave zone'' region. Note that $r \approx 1 + t + |q| \approx 1 + t + s$ for 
	$(t,x) \in \mathcal{W}.$ Now as in the proof of Proposition \ref{P:FLUTTimproveddecay}, inequality 
	\eqref{E:Farupgradedecay}
	follows from the weak decay estimates of Corollary \ref{C:WeakDecay} if $(t,x) \nin \mathcal{W}.$
	Furthermore, we have that $|\Far| \approx |\ualpha[\Far]| + |\alpha[\Far]| + |\rho[\Far]| + |\sigma[\Far]|,$ and
	by Proposition \ref{P:FLUTTimproveddecay}, inequality \eqref{E:Farupgradedecay}
	has already been shown to hold for 
	$|\alpha[\Far]| + |\rho[\Far]| + |\sigma[\Far]| \approx |\Far|_{\mathcal{L}\mathcal{N}} + |\Far|_{\mathcal{T}\mathcal{T}}.$

	It remains to prove the desired estimate for $|\ualpha[\Far(t,x)]|$ under the assumption that
	$(t,x) \in \mathcal{W}.$ To this end, we use \eqref{E:ODErualpha}, the weak decay 
	estimates of Corollary \ref{C:WeakDecay}, and Proposition \ref{P:FLUTTimproveddecay} to deduce that 
	if $(t,x) \in \mathcal{W},$ then
	
	\begin{align} \label{E:LambdarualphaBound}
		\big|\nabla_{\Lambda} \big(r\ualpha[\Far] \big)\big| 
			& \lesssim  \varepsilon (1 + t + |q|)^{-3/2 + \updelta} \big|r\ualpha[\Far] \big| 
			\ + \ \varepsilon (1 + t + |q|)^{-2 + 3 \updelta},
  \end{align}
  where $\Lambda \eqdef L + \frac{1}{4} h_{LL} \uL.$ Let $\big(\tau(\lambda),y(\lambda) \big) $ be the integral 
  curve\footnote{By integral curve,
  we mean the solution to the ODE system $\frac{d \tau}{d \lambda} = \Lambda^0(\tau,y),$ 
  $\frac{d y^j}{d \lambda^j} = \Lambda^j(\tau,y),$ $(j=1,2,3),$ passing through the point $(t,x).$} of the 
  vectorfield $\Lambda$ passing through the point $(t,x) = \big(\tau(\lambda_1),y(\lambda_1) \big) \in \mathcal{W}.$ 
	By the already-proven smallness estimate \eqref{E:hLTZhLLpointwise} for $h_{LL},$ every such integral curve must intersect 
	the boundary of $\mathcal{W}$ at a point $(t_0,x_0) = \big(\tau(\lambda_0),y(\lambda_0) \big)$ to the \emph{past} of $(t,x).$
	Furthermore, by \eqref{E:hLTZhLLpointwise} again, we have that
	$\frac{d \tau}{d\lambda} \approx 1$ along the integral curves, and for all $(\tau, y) \in \mathcal{W},$ we have
	that $|y| \approx \tau \approx 1 + |\tau| \approx 1 + |\tau| + \big||y| - \tau\big|.$
	We now set $f(\lambda) \eqdef \Big||y(\lambda)|\ualpha\big[\Far\big(\tau(\lambda),y(\lambda)\big)\big]\Big|,$
	integrate inequality \eqref{E:LambdarualphaBound} along the integral curve
	(which is contained in $\mathcal{W}$), use the assumption $0 < \updelta < 1/4,$
	and change variables so that $\tau$ is the integration variable to obtain 
		
		\begin{align} \label{E:rualphaGronwallready}
			\overbrace{\big|r \ualpha[\Far](t,x) \big|}^{f(\lambda(t))}
			& \leq \overbrace{\big|r_0 \ualpha[\Far(t_0,x_0)]\big|}^{f(\lambda_0)}
				\ + \ C \varepsilon \int_{\lambda = \lambda_0}^{\lambda = \lambda_1} [1 + \tau(\lambda)]^{-2 + 3 \updelta} d \lambda 
			 	\ + \ C \varepsilon \int_{\lambda = \lambda_0}^{\lambda = \lambda_1} [1 + \tau(\lambda)]^{-3/2 + \updelta}  
				f(\lambda) d \lambda \\
			& \leq C \varepsilon 
				\ + \ C \varepsilon \int_{\tau = t_0}^{\tau = t} (1 + \tau)^{-2 + 3 \updelta} d \tau 
				\ + \ C \varepsilon \int_{\tau = t_0}^{\tau = t} (1 + \tau)^{-3/2 + \updelta}
				f(\lambda \circ \tau) d \tau \notag \\
			& \leq C \varepsilon \ + \ C \varepsilon \int_{\tau = t_0}^{\tau = t} (1 + \tau)^{-3/2 + 
				\updelta} f(\lambda \circ \tau) d \tau, \notag
		\end{align}
		where we have used \eqref{E:weakdecaypartialLinfinity} to obtain the bound 
		$\big|r_0 \ualpha[\Far(t_0,x_0)] \big| \leq C \varepsilon$
		for points $(t_0,x_0)$ lying on the boundary of $\mathcal{W}.$ Applying Gronwall's lemma to 
		\eqref{E:rualphaGronwallready}, we have that
		
		\begin{align} \label{E:rualphaGronwall}
			\big|r \ualpha[\Far(t,x)] \big| & \leq C \varepsilon 
			\exp\bigg(C \varepsilon \int_{\tau = t_0}^{\tau = t} 
			(1 + \tau)^{-3/2 + \updelta} d \tau \bigg) \\
			& \leq C \varepsilon \exp\bigg(C \varepsilon \int_{\tau = t_0}^{\tau = t} (1 + \tau)^{-3/2 + \updelta} d \tau \bigg)
			\leq C \varepsilon, \notag
		\end{align}
		from which it trivially follows that
		
		\begin{align}
			\big| \ualpha[\Far(t,x)] \big| & \leq C \varepsilon r^{-1} \leq C \varepsilon (1 + t + |q|)^{-1}
		\end{align}
		as desired.
		\hfill $\qed$

\subsubsection{Proofs of \eqref{E:partialhTUpointwise} - \eqref{E:partialhpointwise}}

In the next two lemmas, we will use the fact that the tensorfield $h_{\mu \nu} \eqdef g_{\mu \nu} - m_{\mu \nu}$ is a solution to the
system

\begin{align} \label{E:hwavesystem}
	\widetilde{\Square}_g h_{\mu \nu} = \mathfrak{H}_{\mu \nu},
\end{align}
where the inhomogeneous term $\mathfrak{H}_{\mu \nu}$ is defined in \eqref{E:HtriangleSmallAlgebraic}. 

\begin{lemma} \cite[Extension of Lemma 10.5]{hLiR2010}\label{L:Inhomogeneousdecayestimates}
	\textbf{(Pointwise estimates for the $\mathfrak{H}_{\mu \nu}$ inhomogeneities)}
	Suppose the assumptions of Proposition \ref{P:UpgradedDecayhA} hold. Then if $\varepsilon$ is sufficiently small, the 
	following pointwise estimates hold for $(t,x) \in [0,T) \times \mathbb{R}^3:$
	
	\begin{subequations}
	\begin{align}
		|\mathfrak{H}|_{\mathcal{T} \mathcal{N}} & \leq C \varepsilon (1 + t + |q|)^{-3/2 + \updelta} |\nabla h|
			\ + \ C \varepsilon (1 + t + |q|)^{-5/2 + \updelta}, \label{E:mathfrakHTULinfinity} \\
		|\mathfrak{H}| & \leq C \varepsilon (1 + t + |q|)^{-3/2 + \updelta} |\nabla h|
			\ + \ C |\nabla h|_{\mathcal{T} \mathcal{N}}^2 \ + \ C \varepsilon^2 (1 + t + |q|)^{-2}.   \label{E:mathfrakHLinfinity} 
	\end{align}
	\end{subequations}
	
\end{lemma}

\begin{proof}
	Lemma \ref{L:Inhomogeneousdecayestimates} follows from Proposition \ref{P:AlgebraicInhomogeneous}, Corollary 
	\ref{C:WeakDecay}, Proposition \ref{P:FLUTTimproveddecay}, the already-proven estimate \eqref{E:Farupgradedecay},
	and the assumption $0 < \updelta < 1/4.$
\end{proof}

\begin{lemma} \cite[Extension of Lemma 10.6]{hLiR2010} \label{L:partialhTUpartialhdecay} 
	\textbf{(Integral inequalities for $|\nabla h|_{\mathcal{T} \mathcal{N}}$ and $|\nabla h|$)}
	Suppose the assumptions of Proposition \ref{P:UpgradedDecayhA} hold.
	Then if $\varepsilon$ is sufficiently small, the following integral inequalities hold for $t \in [0,T):$
	
	\begin{subequations}
	\begin{align}
		(1 + t) \big\| |\nabla h|_{\mathcal{T} \mathcal{N}}(t,\cdot) \big\|_{L^{\infty}}
			& \leq C \varepsilon \ + \ C \varepsilon \int_{0}^t (1 + \tau)^{- 1/2 + \updelta} 
		 	\big\| \nabla h(\tau, \cdot) \big\|_{L^{\infty}} \, d\tau,  \label{E:partialhTUdecay} \\
 		(1 + t) \big\| \nabla h(t,\cdot)  \big\|_{L^{\infty}} \label{E:partialhdecay}
		& \leq C \varepsilon \ + \ C \varepsilon^2 \ln(1 + t) \ + \ C \varepsilon \int_{0}^t (1 + \tau)^{- 1/2 + \updelta} 
		 	\big\| \nabla h(\tau,\cdot) \big\|_{L^{\infty}} \, d\tau \\
		 & \ \ + \ C \varepsilon \int_{0}^t (1 + \tau) \big\| |\nabla h|_{\mathcal{T} \mathcal{N}}^2(\tau, \cdot) 
		 \big\|_{L^{\infty}} \, d\tau. \notag
	\end{align}
	\end{subequations}

\end{lemma}

\begin{proof}
	First observe that \eqref{E:weakdecayLinfinity} and \eqref{E:hLTZhLLpointwise} (the version for the tensor $H$) imply that 
	the hypotheses of Lemma \ref{L:scalardecay} and Corollary \ref{C:systemdecay} hold. Therefore, using the lemma and the 
	corollary, with $\varpi(q) \eqdef 1$ and $\upalpha \eqdef 0,$ and noting that $h_{\mu \nu}$ verifies the system 
	\eqref{E:hwavesystem}, we have that 
	
	\begin{align} \label{E:partialhTUfirstinequality}
		(1 + t) |\nabla h|_{\mathcal{T} \mathcal{N}}
		& \lesssim \sup_{0 \leq \tau \leq t} 
			\sum_{|I| \leq 1} 
			\big\| \nabla_{\mathcal{Z}}^I h(t,\cdot) \big\|_{L^{\infty}}  
			\ + \ \int_{\tau = 0}^{t} 
			(1 + \tau) \big\| |\mathfrak{H}|_{\mathcal{T} \mathcal{N}} \big\|_{L^{\infty}(D_{\tau})} \, d\tau 
			\ + \ \sum_{|I| \leq 2} \int_{\tau = 0}^{t} 
			(1 + \tau)^{-1} \big\| \nabla_{\mathcal{Z}}^I h \big\|_{L^{\infty}(D_{\tau})} \, d \tau. 
	\end{align}
	Using \eqref{E:weakdecayLinfinity} (the version for the tensor $h$), we estimate the 
	the first and third terms on the right-hand side of \eqref{E:partialhTUfirstinequality}
	as follows:
	
	\begin{align}
		\sup_{0 \leq \tau \leq t} \sum_{|I| \leq 1} \big\| \nabla_{\mathcal{Z}}^I h \big\|_{L^{\infty}} 
		& \leq C \varepsilon (1 + t)^{-1/2 + \updelta} \leq C \varepsilon, \label{E:partialhTUdecayfirstterm} \\
		\sum_{|I| \leq 2} \int_{\tau = 0}^{t} 
		(1 + \tau)^{-1} \big\| \nabla_{\mathcal{Z}}^I h \big\|_{L^{\infty}(D_{\tau})} \, d \tau 
		& \leq C \varepsilon \int_{\tau = 0}^{\infty} (1 + \tau)^{-3/2 + \updelta} \, d \tau
			\leq C \varepsilon. \label{E:partialhTUdecaythirdterm}
	\end{align}
	To estimate the second term, we use \eqref{E:mathfrakHTULinfinity} to conclude that for $x \in D_t,$ we have that
	
	\begin{align}
		(1 + t) |\mathfrak{H}|_{\mathcal{T} \mathcal{N}}
			& \leq C \varepsilon (1 + t)^{- 1/2 + \updelta} |\nabla h| 
			\ + \ C \varepsilon (1 + t)^{-3/2 + \updelta}. \label{E:partialhTUdecaysecondterm}
	\end{align}
	Inequality \eqref{E:partialhTUdecay} now follows from \eqref{E:partialhTUfirstinequality} - \eqref{E:partialhTUdecaysecondterm}, and the fact 
	that $C \varepsilon \int_{0}^t (1 + \tau)^{-3/2 + \updelta} \, d \tau \leq C \varepsilon.$ Inequality \eqref{E:partialhdecay} 
	can be obtained in a similar fashion using \eqref{E:mathfrakHLinfinity}.

\end{proof}

To finish the proof of Proposition \ref{P:UpgradedDecayhA}, we will use the following Gronwall-type lemma.

\begin{lemma} \cite[Slight modification of Lemma 10.7]{hLiR2010} \label{L:Gronwall} \textbf{(Gronwall lemma)}
	Assume that the continuous functions $b(t) \geq 0$ and $c(t) \geq 0$ satisfy
	
	\begin{subequations}
	\begin{align}
		b(t) & \leq C \varepsilon \ + \ C \varepsilon \int_{0}^{t} (1 + \tau)^{-1 - a} c(\tau) \, d \tau, 
			\label{E:boftGronwall} \\
		c(t) & \leq C \varepsilon \ + \ C \varepsilon^2 \ln(1 + t) 
			\ + \ C \varepsilon \int_{0}^{t} (1 + \tau)^{-1 - a} c(\tau) \, d \tau 
			\ + \ C \int_{0}^{t} (1 + \tau)^{-1} b^2(\tau) \, d \tau \label{E:coftGronwall}
	\end{align}
	\end{subequations}
	for some positive constants $a, C$ such that $\varepsilon < \frac{a}{4C}$ 
	and $\varepsilon < \frac{2a}{(1 + 4C^2)}.$ Then 
	
	\begin{subequations}
	\begin{align}
		b(t) & \leq 2 C \varepsilon, \label{E:bBound} \\
		c(t) & \leq 2 C \varepsilon \big\lbrace 1 + a \ln(1 + t) \big\rbrace. \label{E:cBound}
	\end{align}
	\end{subequations}
	
\end{lemma}

\begin{proof}
	We slightly modify the proof of \cite[Lemma 10.7]{hLiR2010}. Let $T$ be the largest time such that the bounds 
	\eqref{E:bBound} - \eqref{E:cBound} hold. Then inserting these bounds into
	the inequalities \eqref{E:boftGronwall} - \eqref{E:coftGronwall}, and using the bound
	(and the change of variables $z \eqdef a \ln(1 + \tau)$)
	
	\begin{align}
		\int_{\tau=0}^{\infty} (1 + \tau)^{-1 - a} \big\lbrace1 + a \ln(1 + \tau) \big\rbrace \, d\tau 
			\leq \int_{\tau = 0}^{\infty} (1 + a^{-1} z) e^{- z} \, dz = 2 a^{-1},
	\end{align}
	we deduce that the following inequalities hold for $t \in [0,T]:$
	
	\begin{align}
		b(t) & \leq C \varepsilon \big\lbrace 1 + 4 C \varepsilon a^{-1}   \big\rbrace < 2 C \varepsilon, \\
		c(t) & \leq C \varepsilon \big\lbrace 1 + 4 C \varepsilon a^{-1} + (1 + 4C^2) \varepsilon \ln(1 + t)  \big\rbrace
			< 2 C \varepsilon \big\lbrace 1 + a \ln(1 + t) \big\rbrace.
	\end{align}
	Since the above inequalities are a strict improvement of the assumed bounds \eqref{E:bBound} - \eqref{E:cBound},
	we thus conclude that $T = \infty.$
\end{proof}

To complete the proof of \eqref{E:partialhTUpointwise} - \eqref{E:partialhpointwise}, we apply Lemmas \ref{L:partialhTUpartialhdecay} and \ref{L:Gronwall} with 
$b(t) \eqdef (1 + t) \big\| |\nabla h|_{\mathcal{T} \mathcal{N}}(t,\cdot) \big\|_{L^{\infty}}$
and $c(t) \eqdef (1 + t) \big\| \nabla h(t,\cdot) \big\|_{L^{\infty}}.$ This implies
\eqref{E:partialhTUpointwise} - \eqref{E:partialhpointwise} with $(1 + t)$ in place of $(1 + t + |q|).$
The additional decay in $|q|$ in \eqref{E:partialhTUpointwise} and \eqref{E:partialhpointwise} follows directly from \eqref{E:weakdecaypartialLinfinity} (the version for the tensor $h$). \hfill $\qed$

\subsection{Proof of Proposition \ref{P:UpgradedDecayh1A}} \label{SS:Proof of PropositionUpgradedDecayh1A}

We will prove the proposition using a series of inductive steps. We only prove the estimates for 
$h_{\mu \nu}^{(1)}$ and $\Far_{\mu \nu}.$ The estimates for $h_{\mu \nu}$ and $H^{\mu \nu}$ follow easily from those for 
$h_{\mu \nu}^{(1)},$ \eqref{E:Hintermsofh}, and Lemma \ref{L:h0decayestimates}. We first prove a technical lemma that will be used during the proof of the proposition.

\begin{lemma} \label{L:UpgradedInhomogeneousPointwise} \textbf{(Pointwise estimates for 
the $|\nabla_{\mathcal{Z}}^I \mathfrak{H}|$ inhomogeneities)}
	Suppose the hypotheses of Proposition \ref{P:UpgradedDecayhA} hold, and let $\mathfrak{H}_{\mu \nu}$ 
	be the inhomogeneous term on the right-hand side of 
	the reduced equation \eqref{E:Reducedh1Summary}. Then if $I$ is any $\mathcal{Z}-$multi-index with $|I| \leq \dParameter ,$ the following
	pointwise estimates hold for $(t,x) \in [0,T) \times \mathbb{R}^3:$
	
	\begin{align}  \label{E:UpgradedInhomogeneousPointwise}
		|\nabla_{\mathcal{Z}}^I \mathfrak{H}| 
		& \leq C \varepsilon \sum_{|J| \leq |I|} (1 + t + |q|)^{-1} 
			\Big(|\nabla\nabla_{\mathcal{Z}}^J h^{(1)}| + |\nabla_{\mathcal{Z}}^J \Far| \Big) \\
		& \ \ + \ C \mathop{\sum_{|I_1| + |I_2| \leq |I|}}_{|I_1|, |I_2| \leq |I|- 1} 
			\Big(|\nabla\nabla_{\mathcal{Z}}^{I_1} h^{(1)}| + |\Lie_{\mathcal{Z}}^{I_1} \Far| \Big)
			\Big(|\nabla\nabla_{\mathcal{Z}}^{I_2} h^{(1)}| + |\Lie_{\mathcal{Z}}^{I_2} \Far| \Big) \notag \\
		& \ \ + \ C \varepsilon^2 (1 + t + |q|)^{-4}. \notag
	\end{align}
\end{lemma}

\begin{proof}
	Lemma \ref{L:UpgradedInhomogeneousPointwise} follows from \eqref{E:ZIinhomogeneoushpointwise},
	\eqref{E:partialhTUpointwise}, \eqref{E:Farupgradedecay}, Lemma \ref{L:h0decayestimates}, the weak decay estimates of 
	Corollary \ref{C:WeakDecay}, \eqref{E:partialhTUpointwise}, and the fact that $0 < \updelta < 1/4.$ We remark that the 
	$C \varepsilon^2 (1 + t + |q|)^{-4}$ term arises from the estimate 
	${|\nabla\nabla_{\mathcal{Z}}^{I_1} h^{(0)}| |\nabla\nabla_{\mathcal{Z}}^{I_2} h^{(0)}| \leq C \varepsilon^2 (1 + t + 
	|q|)^{-4}}.$
\end{proof}

We are now ready for the proof of the proposition. To prove \eqref{E:partialZIh1Aupgraded} - \eqref{E:barpartialZIh1Aupgraded}, we will argue inductively, using the inequalities in the case $|I| \leq k$ to deduce that they hold in the case $|I| = k+1.$ We also remark that the base case $k=0$ is covered by our argument. 
\\

\noindent \textbf{Induction step 1: upgraded pointwise decay estimates for $|\nabla_{\mathcal{Z}}^J h|_{\mathcal{L}\mathcal{L}}|$
for $|I| = k + 1$ and $|\nabla_{\mathcal{Z}}^{J} h|_{\mathcal{L}\mathcal{T}}|$ for $|J| = k.$}

As a first step, we will use the wave coordinate condition to upgrade the estimates for $|\nabla_{\mathcal{Z}}^J h|_{\mathcal{L}\mathcal{L}}$ for $|J| = k + 1$ and $|\nabla_{\mathcal{Z}}^J h|_{\mathcal{L}\mathcal{T}}$ for $|J| = k.$ To this end, we appeal to inequality \eqref{E:partialZIhLLpluspartialZJhLTLinfinityintegrated}, 
using inequality \eqref{E:partialZIh1Aupgraded} for $h$ under the induction hypothesis to bound the integrand, to deduce that

\begin{align} \label{E:LieZJLiehLLZJprimeIndependentInductionEstimate}
	\sum_{|I| = k + 1} |\nabla_{\mathcal{Z}}^I h|_{\mathcal{L}\mathcal{L}}
	\ + \ \sum_{|J| = k} |\nabla_{\mathcal{Z}}^J h|_{\mathcal{L}\mathcal{T}} 
 	& \lesssim \left\lbrace \begin{array}{lr}
	    	\varepsilon (1 + t + |q|)^{-1 + M_{k-1} \varepsilon} (1 + |q|)^{- M_{k-1} \varepsilon}, &  \mbox{if} \ q > 0, \\
	      \varepsilon (1 + t + |q|)^{-1 + M_{k-1} \varepsilon} (1 + |q|)^{1/2 + \upmu'}, & \mbox{if} \ q < 0.
	    \end{array}
	  	\right.
\end{align}
In the above estimates, the constant $\upmu'$ is subject to the restrictions stated in the hypotheses 
of Proposition \ref{P:UpgradedDecayh1A}. Furthermore, since $H^{\mu \nu} = -h^{\mu \nu} + O^{\infty}(|h|^2),$ \eqref{E:weakdecayLinfinity} implies that the same estimates hold for the tensor $H.$
\\

\noindent \textbf{Induction step 2: upgraded pointwise decay estimates for $\big|\Lie_{\mathcal{Z}}^I \Far \big|,$ $|I| = k+1.$}

Let $\mathcal{W} \eqdef \big\lbrace (t,x): |x| \geq 1 + t/2 \big\rbrace \cap \big\lbrace (t,x): |x| \leq 2t - 1 \big\rbrace$
denote the ``wave zone'' region. Then for $(t,x) \nin \mathcal{W},$ we have that $1 + |q| \approx 1 + t + |q|.$ Using this fact, for $(t,x) \nin \mathcal{W},$ the weak decay estimate \eqref{E:weakdecaypartialLinfinity} implies that inequality \eqref{E:partialZIh1Aupgraded} holds for $|\Lie_{\mathcal{Z}}^I \Far|$ in the case $|I| = k + 1.$ Furthermore, by Proposition \ref{P:FLUTTimproveddecay}, the inequality \eqref{E:partialZIh1Aupgraded} holds for the null components $\big|\alpha[\Lie_{\mathcal{Z}}^I \Far]\big|,$ $\big|\rho[\Lie_{\mathcal{Z}}^I \Far]\big|,$ and 
$\big|\sigma[\Lie_{\mathcal{Z}}^I \Far]\big|$ when $|I| = k + 1.$

It remains to consider $\big|\ualpha[\Lie_{\mathcal{Z}}^I \Far(t,x)]\big|$ in the case $(t,x) \in \mathcal{W}.$ Note that $r \approx 1 + t + |q| \approx 1 + t + s$ for $(t,x) \in \mathcal{W}.$ We will make use of the weight $\varpi(q)$ defined in \eqref{E:decayweight}. Using \eqref{E:LambdaLieZIualphaEquationGoodqWeights}, Corollary \ref{C:WeakDecay} (the version for the tensorfield $h$), Proposition \ref{P:FLUTTimproveddecay}, \eqref{E:hLTZhLLpointwise}, \eqref{E:Farupgradedecay}, the induction
hypothesis, and \eqref{E:LieZJLiehLLZJprimeIndependentInductionEstimate}, it follows that

\begin{align} \label{E:ualphaGoodQweightsReadytobeIntegrated}
	\sum_{|I| \leq k + 1} \big|\nabla_{\Lambda} \big(r \varpi(q) \ualpha[\Lie_{\mathcal{Z}}^I \Far] \big)\big| 
	& \leq C \varepsilon (1 + t + |q|)^{-1} \sum_{|I| \leq k + 1} \big|r \varpi(q) \ualpha[\Lie_{\mathcal{Z}}^I \Far] \big|
	\ + \ C \varepsilon (1 + t + |q|)^{-(1 + a)}
	\ + \ C \varepsilon^2 (1 + t + |q|)^{-1 + C \varepsilon},
\end{align}
where $0 < a < \min \lbrace \upmu' - \updelta, \upgamma - \updelta - \upgamma' \rbrace$  is a fixed constant, and
$\Lambda \eqdef L + \frac{1}{4} h_{LL} \uL.$ Note the importance of the independent estimate 
\eqref{E:hLTZhLLpointwise} for bounding the second, fourth, and fifth sums on the right-hand side of
\eqref{E:LambdaLieZIualphaEquationGoodqWeights}, and of the independent estimate 
\eqref{E:LieZJLiehLLZJprimeIndependentInductionEstimate} (in the case $|I| = k + 1$) 
for bounding the third sum on the right-hand side of \eqref{E:LambdaLieZIualphaEquationGoodqWeights}.

Let $\big(\tau(\lambda),y(\lambda) \big) $ be the integral curve (as defined in Section \ref{SSS:ProofofFarUpgradedDecay}) 
of the vectorfield $\Lambda$ passing through the point $(t,x) = \big(\tau(\lambda_1),y(\lambda_1) \big) \in \mathcal{W}.$ 
By the inequality \eqref{E:hLTZhLLpointwise} for $h_{LL},$ every such integral curve must intersect 
the boundary of $\mathcal{W}$ at a point $(t_0,x_0) = \big(\tau(\lambda_0),y(\lambda_0) \big)$ lying to the \emph{past} of $(t,x).$ Using \eqref{E:hLTZhLLpointwise} again, we have that $\frac{dt}{d\lambda} \approx 1$ along the integral curves, and in the entire region $\mathcal{W},$ we have that $|y| \approx \tau \approx 1 + |\tau| \approx 1 + |\tau| + \big||y| - \tau\big|.$
Define $f(\lambda) \eqdef \sum_{|I| \leq k + 1} \Big||y(\lambda)| \varpi\big(q(\lambda)\big) 
\ualpha\big[\Lie_{\mathcal{Z}}^I \Far\big(\tau(\lambda),y(\lambda) \big)\big]\Big|,$ where $q(\lambda) 
\eqdef |y(\lambda)| - \tau(\lambda).$ Note that $f(\lambda_1) = \sum_{|I| \leq k + 1}\big|r \varpi(q) \ualpha[\Lie_{\mathcal{Z}}^I \Far]\big|,$ where $q \eqdef q(\lambda_1) = |x| - t,$ while the weak decay estimate \eqref{E:weakdecaypartialLinfinity} implies that $f(\lambda_0) \leq C \varepsilon.$ Integrating inequality \eqref{E:ualphaGoodQweightsReadytobeIntegrated} and changing variables so that $\tau$ is the integration variable, we have that

\begin{align} \label{E:ualphaGoodQweightsGronwallready}
	\overbrace{f(\lambda_1)}^{f(\lambda \circ t)}
	& \leq f(\lambda_0) \ + \ C \varepsilon \int_{\lambda = \lambda_0}^{\lambda = \lambda_1} 
		[1 + \tau(\lambda)]^{-1} f(\lambda) \, d \lambda
		\ + \ C \varepsilon \int_{\lambda = \lambda_0}^{\lambda = \lambda_1} [1 + \tau(\lambda)]^{-(1 + a)} \, d \lambda
		\ + \ C \varepsilon^2 \int_{\lambda = \lambda_0}^{\lambda = \lambda_1} [1 + \tau(\lambda)]^{-1 + C \varepsilon} 
			\, d \lambda \\
	& \leq C \varepsilon \ + \ C \varepsilon \int_{\lambda = \lambda_0}^{\lambda = \lambda_1} 
			[1 + \tau(\lambda)]^{-1} f(\lambda) \, d \lambda
		\ + \ C \varepsilon \int_{\tau = t_0}^{\tau = t} (1 + \tau)^{-(1 + a)} \, d \tau
		\ + \ C \varepsilon^2 \int_{\tau = t_0}^{\tau = t} (1 + \tau)^{-1 + C \varepsilon} \, d \tau	\notag \\
	& \leq C \varepsilon (1 + t)^{C \varepsilon}
		\ + \ C \varepsilon \int_{\tau = t_0}^{\tau = t} 
			(1 + \tau)^{-1} f(\lambda \circ \tau) \, d \tau. \notag
\end{align}
Applying Gronwall's lemma to \eqref{E:ualphaGoodQweightsGronwallready}, we have that

\begin{align} \label{E:ualphaGoodqweightsGronwall}
	f(\overbrace{\lambda \circ t}^{\lambda_1}) 
	& \leq C \varepsilon (1 + t)^{C \varepsilon}
		\exp \Big( C \varepsilon \int_{\tau= t_0}^{\tau = t} (1 + \tau)^{-1} \, d \tau \Big) \, d\tau \\
	& \leq C \varepsilon (1 + t)^{2 C \varepsilon}, \notag
\end{align}
from which it easily follows that for $(t,x) \in \mathcal{W},$ we have that

\begin{align} \label{E:ualphaGoodqweightsinequality}
	\sum_{|I| \leq k + 1} \big| \ualpha[\Lie_{\mathcal{Z}}^I \Far] \big| & \leq C \varepsilon (1 + t)^{-1 + 2 C \varepsilon} 
	\varpi^{-1}(q).
\end{align}
Combining \eqref{E:ualphaGoodqweightsinequality} and the previous arguments covering $(t,x) \nin \mathcal{W}$ 
and the other null components of $\Lie_{\mathcal{Z}}^I \Far,$ we have shown that the estimate \eqref{E:partialZIh1Aupgraded} holds for $|\Lie_{\mathcal{Z}}^I \Far|$ in the case $|I| = k + 1.$
\\

\noindent \textbf{Final induction step: upgraded pointwise decay estimates for $|\nabla\nabla_{\mathcal{Z}}^I h|$ and 
$|\nabla_{\mathcal{Z}}^I h|,$ $|I| = k + 1.$}

Our first goal is to prove the following estimate in the case $|I| = k + 1:$ 

\begin{align}    \label{E:boxZIh1ALinfinityUpgraded}
	|\widetilde{\Square}_g \nabla_{\mathcal{Z}}^I h^{(1)}| & \lesssim
		\varepsilon \sum_{|K| \leq |I|} (1 + t + |q|)^{-1} |\nabla\nabla_{\mathcal{Z}}^K h^{(1)}|
		\ + \ \left\lbrace \begin{array}{lr}
	   	\varepsilon^2 (1 + t + |q|)^{-4 + \updelta} (1 + |q|)^{- \updelta}, &  \mbox{if} \ q > 0, \\
	   	\varepsilon (1 + t + |q|)^{-3}, & \mbox{if} \ q < 0,
	    \end{array}
	  \right. \\
	 & \ + \ \left\lbrace \begin{array}{lr}
	  	\varepsilon^2 (1 + t + |q|)^{-2 + 2 M_k \varepsilon} (1 + |q|)^{-2 - 2 M_k \varepsilon}, &  \mbox{if} \ q > 0, \\
	   	\varepsilon^2 (1 + t + |q|)^{-2 + 2 M_k \varepsilon} (1 + |q|)^{-1 + 2 \upmu'}, & \mbox{if} \ q < 0.
	    \end{array}
	  \right. \notag
\end{align}
To prove \eqref{E:boxZIh1ALinfinityUpgraded}, we first recall Corollary \ref{C:boxZIh1ALinfinity}, which
states that

\begin{align} \label{E:boxZIh1ALinfinityagain}
		|\widetilde{\Square}_g \nabla_{\mathcal{Z}}^I h^{(1)}|		 
			& \lesssim |\nablamod_{\mathcal{Z}}^I \mathfrak{H}| \ + \ |\nablamod_{\mathcal{Z}}^I \widetilde{\Square}_g h^{(0)}| 
				\ + \ (1 + t + |q|)^{-1} \mathop{\sum_{|K| \leq |I|}}_{|J| + (|K| - 1)_{+} \leq |I|} |\nabla_{\mathcal{Z}}^J H||\nabla 
				\nabla_{\mathcal{Z}}^K h^{(1)}| 
				\\
			& \ \ + \ (1 + |q|)^{-1} \sum_{|K| \leq |I|}  
				|\nabla\nabla_{\mathcal{Z}}^K h^{(1)}| \Bigg\lbrace \mathop{\sum_{|J| + (|K| - 1)_{+}}}_{\ \ \leq |I|} 
				|\nabla_{\mathcal{Z}}^J H|_{\mathcal{L} \mathcal{L}} 
				\ + \ \underbrace{\mathop{\sum_{|J'| + (|K| - 1)_{+}}}_{\ \leq |I|-1} |\nabla_{\mathcal{Z}}^{J'} H|_{\mathcal{L} 
				\mathcal{T}}}_{\mbox{Absent if $|I| = 0$}}
				\ + \ \underbrace{\mathop{\sum_{|J''| + (|K| - 1)_{+}}}_{\ \leq |I|-2} |\nabla_{\mathcal{Z}}^{J''} H|}_{\mbox{Absent if 
				$|I| \leq 1$ or $|K| = |I|$}} \Bigg\rbrace, \label{E:boxZIh1ALinfinityagainsecondterm}
\end{align}
where $(|K|-1)_+ \eqdef 0$ if $|K| = 0$ and $(|K|-1)_+ \eqdef |K| - 1$ if $|K| \geq 1.$
We first bound the terms from line \eqref{E:boxZIh1ALinfinityagainsecondterm},
considering separately the cases $|K| < |I|$ and $|K| = |I| = k + 1.$ For $|K| < |I| = k+1,$ we use \eqref{E:LieZJLiehLLZJprimeIndependentInductionEstimate} (for the tensorfield $H$) and \eqref{E:ZIh1Aupgraded}
(for the tensorfield $H$) under the induction hypotheses to conclude that

\begin{align}
	(1 + |q|)^{-1} \sum_{|J|\leq k + 1,|J'|\leq k,|J''|\leq k-1} & \Big(|\nabla_{\mathcal{Z}}^J H|_{\mathcal{L} \mathcal{L}} 
		\ + \ |\nabla_{\mathcal{Z}}^{J'} H|_{\mathcal{L} \mathcal{T}} 
		\ + \ |\nabla_{\mathcal{Z}}^{J''} H| \Big) \\
	& \lesssim \left\lbrace
		\begin{array}{lr}
	   	\varepsilon (1 + t + |q|)^{-1 + M_k \varepsilon} (1 + |q|)^{-1 - M_k \varepsilon}, &  \mbox{if} \ q > 0, \\
	   	\varepsilon (1 + t + |q|)^{-1 + M_k \varepsilon} (1 + |q|)^{-1/2 + \upmu'}, & \mbox{if} \ q < 0. 		
	      \notag
	   \end{array}
	 \right. \notag
\end{align}
Also using \eqref{E:partialZIh1Aupgraded} under the induction hypotheses to bound $|\nabla\nabla_{\mathcal{Z}}^K h^{(1)}|,$ it follows that
all of the terms from \eqref{E:boxZIh1ALinfinityagainsecondterm} in the case $|K| < |I|$ can be bounded by
the last term on the right-hand side of \eqref{E:boxZIh1ALinfinityUpgraded}.

We now consider the case $|K| = |I| = k+1.$ Since $|J| \leq 1$ and $|J'| = 0$ in this case, we can use
\eqref{E:hLTZhLLpointwise} (for the tensorfield $H$) to deduce the bound

\begin{align}
	(1 + |q|)^{-1} \sum_{|K| = |I|} & \Bigg\lbrace |\nabla\nabla_{\mathcal{Z}}^K h^{(1)}|
		\bigg( \mathop{\sum_{|J| + (|K| - 1)_{+}}}_{\ \ \leq |I|} |\nabla_{\mathcal{Z}}^J H|_{\mathcal{L} \mathcal{L}} 
		\ + \ \mathop{\sum_{|J'| + (|K| - 1)_{+}}}_{\ \leq |I|-1} |\nabla_{\mathcal{Z}}^{J'} H|_{\mathcal{L} \mathcal{T}} \bigg)   
		\Bigg\rbrace \\
	& \lesssim \varepsilon \sum_{|K| = |I|} (1 + t + |q|)^{-1} |\nabla\nabla_{\mathcal{Z}}^K h^{(1)}|. \notag
\end{align}
Thus, all of the terms from \eqref{E:boxZIh1ALinfinityagainsecondterm} in the case $|K| = |I| = k + 1$ can be bounded by
the first term on the right-hand side of \eqref{E:boxZIh1ALinfinityUpgraded}.

For the $|\nablamod_{\mathcal{Z}}^I \widetilde{\Square}_g h^{(0)}|$ term from the right-hand side of \eqref{E:boxZIh1ALinfinityagain}, we simply use Lemma \ref{L:weakdecayLinfinitynablaZISquaregh0}, which shows that $|\nablamod_{\mathcal{Z}}^I \widetilde{\Square}_g h^{(0)}|$ is bounded by the 
next to last term on the right-hand side of \eqref{E:boxZIh1ALinfinityUpgraded}. 

To bound the $|\nabla_{\mathcal{Z}}^I \mathfrak{H}|$ term from the right-hand side of \eqref{E:boxZIh1ALinfinityagain},
we apply Lemma \ref{L:UpgradedInhomogeneousPointwise}; the first and third terms from the right-hand side of \eqref{E:UpgradedInhomogeneousPointwise} are manifestly bounded by the right-hand side of
\eqref{E:boxZIh1ALinfinityUpgraded}, while the term 

\begin{align*}
	\mathop{\sum_{|J| + |K| \leq |I|}}_{|J| \leq |K| < |I|} \Big(|\nabla\nabla_{\mathcal{Z}}^J h^{(1)}| 
	+ |\Lie_{\mathcal{Z}}^J \Far| \Big)
	\Big(|\nabla\nabla_{\mathcal{Z}}^K h^{(1)}| + |\Lie_{\mathcal{Z}}^K \Far| \Big) 
\end{align*}	
from the right-hand side of \eqref{E:UpgradedInhomogeneousPointwise} can be bounded by the last term on the right-hand side of
\eqref{E:boxZIh1ALinfinityUpgraded} using the induction hypotheses, since $|J| \leq |K| \leq k.$ This completes the proof of \eqref{E:boxZIh1ALinfinityUpgraded} in the case of $|I| = k + 1.$

To obtain the desired upgraded pointwise estimate for $|\nabla\nabla_{\mathcal{Z}}^I h^{(1)}|,$ we will estimate the quantity

\begin{align} \label{E:nkplusonequantity}
	n_{k+1}(t) \eqdef (1 + t) \sum_{|I| \leq k + 1} \big\| \varpi(q) \nabla\nabla_{\mathcal{Z}}^I h^{(1)}(t,\cdot) \big\|_{L^{\infty}},
\end{align}
where $\varpi(q)$ is the weight defined in \eqref{E:decayweight}. Our goal is to use 
Lemma \ref{E:scalar} with $\phi \eqdef \nabla_{\mathcal{Z}}^I h_{\mu \nu}^{(1)}$ to obtain an integral inequality for $n_{k+1}(t)$ that is amenable to Gronwall's lemma. We begin by estimating the terms on the right-hand side of \eqref{E:systemdecay}. First, with $a \eqdef \min(\upmu' - \updelta, \upgamma - \updelta - \upgamma') > 0,$ by the weak decay estimate \eqref{E:weakdecayLinfinity}, we have that

\begin{align} \label{E:ZIh1ZIAweighteddecayestimate}
	\varpi(q)|\nabla_{\mathcal{Z}}^I h^{(1)}| & \lesssim 
	\left \lbrace \begin{array}{lr}
   		\varepsilon (1 + t + |q|)^{-1 + \updelta} (1 + |q|)^{- \upgamma} (1 + |q|)^{1 + \upgamma'}, &  \mbox{if} \ q > 0, \\
    	\varepsilon (1 + t + |q|)^{-1 + \updelta} (1 + |q|)^{1/2} (1 + |q|)^{1/2 - \upmu'}, & \mbox{if} \ q < 0,
    \end{array}
  	\right., && |I| \leq \dParameter - 2 \\
  	& \lesssim \varepsilon (1 + t)^{-a},&& |I| \leq \dParameter - 2. \notag
\end{align}
This will serve as a suitable bound for estimating the first and fourth terms on the right-hand side of \eqref{E:systemdecay}.

Next, using \eqref{E:boxZIh1ALinfinityUpgraded} and the definition \eqref{E:nkplusonequantity}, 
we deduce the following pointwise estimate:

\begin{align} \label{E:boxZIh1ALinfinityWeightedUpgraded}
	\varpi(q) |\widetilde{\Square}_g \nabla_{\mathcal{Z}}^I h^{(1)}|
		\lesssim (1 + t)^{-2} \big\lbrace\varepsilon n_{k+1} \ + \ \varepsilon^2 (1 + t)^{2M_k \varepsilon}
		\ + \ \varepsilon (1 + t)^{-1/2 - \upmu'} \big\rbrace.
\end{align}
This will serve as a suitable bound for estimating the third integral on the right-hand side of \eqref{E:systemdecay}.

We now apply Corollary \ref{C:systemdecay}, using \eqref{E:ZIh1ZIAweighteddecayestimate},
\eqref{E:boxZIh1ALinfinityWeightedUpgraded}, and the assumption $k + 1 \leq \dParameter - 4$ to deduce that

\begin{align} \label{E:nkplus1Gronwallready}
	n_{k+1}(t) & \leq C \sup_{0 \leq \tau \leq t} \sum_{|I| \leq k + 2} 
		\big\| \varpi(q) \nabla_{\mathcal{Z}}^I h^{(1)}(t,\cdot) \big\|_{L^{\infty}} \\
	& \ \ + \ C \sum_{|I| \leq k + 1} \int_0^t 
		\varepsilon \big\| \varpi(q) \nabla\nabla_{\mathcal{Z}}^I h^{(1)}(\tau,\cdot) \big\|_{L^{\infty}} 
		\ + \ (1 + \tau) \big\| \varpi(q) |\widetilde{\Square}_g \nabla_{\mathcal{Z}}^I h^{(1)}|(\tau,\cdot) 
		\big\|_{L^{\infty}(D_{\tau})} \, d \tau \notag \\
	& \ \ + \ C \sum_{|I| \leq k + 3} \int_0^t (1 + \tau)^{-1} 
		\big\| \varpi(q) \nabla_{\mathcal{Z}}^I h^{(1)}(\tau,\cdot) \big\|_{L^{\infty}(D_{\tau})} \, d \tau \notag \\
	& \leq C \varepsilon (1 + t)^{-a} \ + \ C \int_{0}^{t} (1 + \tau)^{-1}  
		\Big\lbrace \varepsilon n_{k+1}(\tau) \ + \ \varepsilon^2 (1 + \tau)^{C \varepsilon} 
		\ + \ \varepsilon (1 + \tau)^{-1/2 - \upmu'} 
			\ + \ \varepsilon (1 + \tau)^{-a} \Big\rbrace \,d \tau \notag \\
	& \leq C \varepsilon \ + \ C \varepsilon(1 + t)^{C\varepsilon} 
		\ + \ C \varepsilon \int_0^t (1 + \tau)^{-1} n_{k+1}(\tau) \, d \tau. \notag
\end{align}

From \eqref{E:nkplus1Gronwallready} and Gronwall's lemma, we conclude that
$n_{k+1}(t) \leq 2 C \varepsilon (1 + t)^{2 C \varepsilon},$ which proves \eqref{E:partialZIh1Aupgraded}
in the case $|I| = k+1.$ As in our proof of Lemma \ref{L:UpgradedDecayhA}, the estimate \eqref{E:ZIh1Aupgraded} follows from integrating the bound for $|\partial_q \nabla_{\mathcal{Z}}^I h^{(1)}|$ implied by \eqref{E:partialZIh1Aupgraded} along the line $\omega \eqdef x/|x| = constant,$ $t + |x| = constant,$ from the hyperplane $t=0$ and using \eqref{E:weakdecayLinfinity} at $t=0.$ This closes the induction argument. We have completed the proof of Proposition \ref{P:UpgradedDecayh1A} with the exception of showing that inequality \eqref{E:barpartialZIh1Aupgraded} holds for 
$|\conenabla \nabla_{\mathcal{Z}}^I h^{(1)}|,$ $|\conenabla \Lie_{\mathcal{Z}}^I \Far|,$ $|\Lie_{\mathcal{Z}}^I \Far|_{\mathcal{L}\mathcal{N}},$ and $|\Lie_{\mathcal{Z}}^I \Far|_{\mathcal{T}\mathcal{T}},$ where $|I| \leq \dParameter - 5.$ In the next paragraph, we address these inequalities using an argument which is not part of the induction process.
\\

\noindent \textbf{Upgraded pointwise decay estimates for $|\conenabla \nabla_{\mathcal{Z}}^I h^{(1)}|,$ 
$|\conenabla \Lie_{\mathcal{Z}}^I \Far|,$ $|\Lie_{\mathcal{Z}}^I \Far|_{\mathcal{L}\mathcal{N}},$ and 
$|\Lie_{\mathcal{Z}}^I \Far|_{\mathcal{T}\mathcal{T}},$ $|I| \leq \dParameter - 5.$}

We first note that inequality \eqref{E:barpartialZIh1Aupgraded} for $|\conenabla \nabla_{\mathcal{Z}}^I h^{(1)}|$ and 
$|\conenabla \Lie_{\mathcal{Z}}^I \Far|$ follows from Lemma \ref{L:PointwisetandqWeightedNablainTermsofZestiamtes}, \eqref{E:LieZIinTermsofNablaZI}, \eqref{E:partialZIh1Aupgraded}, and \eqref{E:ZIh1Aupgraded}.

We now focus on proving the estimate \eqref{E:barpartialZIh1Aupgraded} for $|\Lie_{\mathcal{Z}}^I \Far|_{\mathcal{L}\mathcal{N}}$ and $|\Lie_{\mathcal{Z}}^I \Far|_{\mathcal{T}\mathcal{T}}$ in \eqref{E:barpartialZIh1Aupgraded}; all of the other estimates of Proposition \ref{P:UpgradedDecayh1A} have been proved. Recall that $|\Lie_{\mathcal{Z}}^I \Far|_{\mathcal{L}\mathcal{N}} + |\Lie_{\mathcal{Z}}^I \Far|_{\mathcal{T}\mathcal{T}} \approx |\alpha[\Lie_{\mathcal{Z}}^I \Far]| + |\rho[\Lie_{\mathcal{Z}}^I \Far]| + |\sigma[\Lie_{\mathcal{Z}}^I \Far]|.$ We will prove the desired estimate for $|\alpha[\Lie_{\mathcal{Z}}^I \Far]|$ in detail; the proofs for $|\rho[\Lie_{\mathcal{Z}}^I \Far]|$ and $|\sigma[\Lie_{\mathcal{Z}}^I \Far]|$ are similar. 

Our proof mirrors the proof of Proposition \ref{P:FLUTTimproveddecay}, except that we now are able to use the already-proven upgraded estimates of Proposition \ref{P:UpgradedDecayh1A} in place of the weak decay estimates of Corollary \ref{C:WeakDecay}. We will use the notation defined in the proof of Proposition \ref{P:FLUTTimproveddecay}. Using the upgraded 
pointwise decay estimates \eqref{E:partialZIh1Aupgraded} and \eqref{E:ZIh1Aupgraded}
(including the versions for the tensorfield $h = h^{(0)} + h^{(1)}$), inequality \eqref{E:partialqfbound} for $f(t,x) \eqdef |x|^{-1} \alpha[\Lie_{\mathcal{Z}}^I \Far(t,x)]$ can be upgraded to  

\begin{align} \label{E:nablauLralphaLieZIFarinequality}
	|\partial_q f(t,x)| 
	& \leq
		\left \lbrace \begin{array}{lr}
	    	C_k \varepsilon (1 + s)^{-3 + C \varepsilon} (1 + |q|)^{-1 - \upgamma'}, &  \mbox{if} \ q > 0, \\
	      C_k \varepsilon (1 + s)^{-3 + C \varepsilon} (1 + |q|)^{-1/2 + \upmu'}, & \mbox{if} \ q < 0,
	    \end{array}
	  	\right., && |I| \leq \dParameter - 5.
\end{align}
Arguing as in the proof of Proposition \ref{P:FLUTTimproveddecay}, and using in particular \eqref{E:fTerminalPointBound},
we deduce that

\begin{align} \label{E:ralphaLieZIFarinequality}
	\big||x|^{-1} \alpha[\Lie_{\mathcal{Z}}^I \Far(t,x)]\big| & \leq 
		C \varepsilon (1 + s)^{-3 - (\overbrace{\upgamma - \updelta)}^{> 0}}
		\ + \ \left \lbrace \begin{array}{lr}
	    	C_k \varepsilon (1 + s)^{-3 + C \varepsilon} (1 + |q'|)^{- \upgamma'}, &  \mbox{if} \ q' > 0, \\
	      C_k \varepsilon (1 + s)^{-3 + C \varepsilon} (1 + |q'|)^{1/2 + \upmu'}, & \mbox{if} \ q' < 0,
	    \end{array}
	  	\right., && |I| \leq \dParameter - 5,
\end{align}
from which it easily follows that

\begin{align} \label{E:alphaPointwiseUpgraded}
	\big|\alpha[\Lie_{\mathcal{Z}}^I \Far(t,x)]\big| 
	& \leq 
	\left \lbrace \begin{array}{lr}
	    	C_k \varepsilon (1 + t + |q|)^{-2 + C \varepsilon} (1 + |q|)^{ - \upgamma'}, &  \mbox{if} \ q > 0, \\
	      C_k \varepsilon (1 + t + |q|)^{-2 + C \varepsilon} (1 + |q|)^{1/2 + \upmu'}, & \mbox{if} \ q < 0,
	    \end{array}
	  	\right., && |I| \leq \dParameter - 5.
\end{align}
We have thus obtained the desired bound \eqref{E:barpartialZIh1Aupgraded} for 
$\big| \alpha[\Lie_{\mathcal{Z}}^I \Far] \big|.$ \hfill $\qed$

\section{Global Existence and Stability} \label{S:GlobalExistence}
In this section, we prove our main stability results. We separate our results into two theorems. The main conclusions are proved in Theorem \ref{T:MainTheorem}, which is an easy consequence of Theorem \ref{T:ImprovedDecay}. Theorem \ref{T:ImprovedDecay}, which concerns the reduced equations \eqref{E:Reducedh1Summary} - \eqref{E:ReduceddMis0Summary}, contains the crux of our bootstrap argument. In this theorem, we make certain assumptions concerning the smallness of the abstract initial data and various pointwise decay estimates for the solution on a local interval of existence $[0,T).$ We then use these assumptions to derive a ``strong'' smallness conclusion for the energy $\mathcal{E}_{\dParameter;\upgamma;\upmu}(t)$ of the reduced solution on the same interval $[0,T).$ Furthermore, in Section \ref{S:DecayFortheReducedEquations}, the pointwise decay assumptions of Theorem \ref{T:ImprovedDecay} were shown to be \emph{automatic consequences} of the smallness assumptions on the data and the ``weak'' bootstrap assumption \eqref{E:Bootstrap} for the energy $\mathcal{E}_{\dParameter;\upgamma;\upmu}(t)$ of the solution, as long as $\dParameter \geq 8.$ Consequently, in our proof of Theorem \ref{T:MainTheorem}, we will be able to appeal to the continuation principle of Proposition \ref{P:LocalExistence} to conclude that the solution to the reduced equation exists globally in time. Furthermore, this line of reasoning leads to an estimate on the size of $\mathcal{E}_{\dParameter;\upgamma;\upmu}(t),$ which can be used to deduce various decay estimates for the global solution. The wave coordinate condition plays a central role in many of the estimates in this section.

\noindent \hrulefill
\ \\

\subsection{Statement of the strong-energy-inequality theorem and proof of the global stability theorem}

We begin by recalling that the norm $E_{\dParameter;\upgamma}(0) \geq 0$ for the abstract initial data is defined by 

\begin{align}
	E_{\dParameter;\upgamma}^2(0) 
	& \eqdef \| \underline{\nabla} \mathring{\underline{h}}^{(1)} \|_{H_{1/2 + \upgamma}^{\dParameter}}^2 
		\ + \ \| \mathring{K} \|_{H_{1/2 + \upgamma}^{\dParameter}}^2 
		\ + \ \| \mathring{\mathfrak{\Displacement}} \|_{H_{1/2 + \upgamma}^{\dParameter}}^2 
		\ + \ \| \mathring{\mathfrak{\Magneticinduction}} \|_{H_{1/2 + \upgamma}^{\dParameter}}^2.
\end{align}
We furthermore recall that the energy $\mathcal{E}_{\dParameter;\upgamma;\upmu}(t) \geq 0$ for the reduced solution is defined to be

\begin{align}
	\mathcal{E}_{\dParameter;\upgamma;\upmu}^2(t) & \eqdef \underset{0 \leq \tau \leq t}{\mbox{sup}} 
		\sum_{|I| \leq \dParameter } \int_{\Sigma_{\tau}} 
		\Big\lbrace |\nabla\nabla_{\mathcal{Z}}^I h^{(1)}|^2 + |\Lie_{\mathcal{Z}}^I \Far|^2 \Big\rbrace w(q) \, d^3 x.
\end{align}

In the above expressions, the weight function $w(q)$ and its derivative $w'(q)$ are defined by
\begin{align}
	w = w(q) = \left \lbrace
		\begin{array}{lr}
    	1 \ + \ (1 + |q|)^{1 + 2 \upgamma}, &  \mbox{if} \ q > 0, \\
      1 \ + \ (1 + |q|)^{-2 \upmu}, & \mbox{if} \ q < 0,
    \end{array}
  \right., & \qquad
  w'(q) = \left \lbrace
		\begin{array}{lr}
    	(1 + 2 \upgamma)(1 + |q|)^{2 \upgamma}, &  \mbox{if} \ q > 0, \\
      2 \upmu (1 + |q|)^{-2 \upmu -1}, & \mbox{if} \ q < 0.
    \end{array}
  \right.
\end{align}
The constants $\upmu$ and $\upgamma$ are subject to the restrictions summarized in Section \ref{SS:FixedConstants}. The spacetime metric is split into the three pieces

\begin{subequations}
\begin{align}
	g_{\mu \nu} & = m_{\mu \nu} + h_{\mu \nu}^{(0)} + h_{\mu \nu}^{(1)}, \\
	h_{\mu \nu}^{(0)} & = \chi\big(\frac{r}{t}\big) \chi(r) \frac{2M}{r} \updelta_{\mu \nu},
\end{align}
\end{subequations}
where the cut-off function $\chi$ is defined in \eqref{E:chidef}. Furthermore, by Proposition \ref{P:SmallNormImpliesSmallEnergy}, if $\varepsilon$ is sufficiently small and $E_{\dParameter;\upgamma}(0) + M \leq \varepsilon,$
then the initial energy for the reduced solution satisfies
\begin{align}
	\mathcal{E}_{\dParameter;\upgamma;\upmu}(0) & \lesssim E_{\dParameter;\upgamma}(0) + M \lesssim \varepsilon.
\end{align}

We now state our technical theorem concerning the derivation of a ``strong'' energy inequality.

\begin{theorem} 		\label{T:ImprovedDecay}
	\textbf{(Derivation of a strong energy inequality)}
	Let $\dParameter \geq 0$ be an integer. Let $(g_{\mu \nu} \eqdef m_{\mu \nu} + \overbrace{h^{(0)}_{\mu \nu} + h^{(1)}_{\mu 
	\nu}}^{h_{\mu \nu}}, \Far_{\mu \nu})$ 
	be a local-in-time solution of the reduced equations \eqref{E:Reducedh1Summary} - \eqref{E:ReduceddMis0Summary} 
	satisfying the wave coordinate condition \eqref{E:wavecoordinategauge1} for $(t,x) \in [0,T) \times \mathbb{R}^3.$  
	Suppose also that for some constants $\upmu', \upgamma$ satisfying $0 < \upmu' < 1/2$ and $0 < \upgamma < 1/2,$ for all 
	vectorfields $Z \in \mathcal{Z},$ for all $\mathcal{Z}-$multi-indices $I$ subject to the restrictions stated below, and for 
	the sets $\mathcal{L} = \lbrace L \rbrace,$ $\mathcal{T} = \lbrace L, e_1, e_2 \rbrace,$
	and $\mathcal{N} = \lbrace {\uL, L, e_1, e_2} \rbrace,$ the following pointwise 
	decay estimates hold for $(t,x) \in [0,T) \times \mathbb{R}^3:$
	
	\begin{subequations}
	\begin{align}
		|\nabla h|_{\mathcal{T} \mathcal{N}} \ + \ (1 + |q|)^{-1}|h|_{\mathcal{L} \mathcal{T}}
			\ + \ (1 + |q|)^{-1}|\nabla_Z h|_{\mathcal{L} \mathcal{L}} \ + \ |\Far| 
		& \leq C \varepsilon (1 + t + |q|)^{-1}, \label{E:MainTheoremAssumptionStrongLinfinityDecayPrincipalTermCoefficients} 				
	\end{align}
	
	\begin{align}
		|\nabla\nabla_{\mathcal{Z}}^I h| \ + \ (1 + |q|)^{-1} |\nabla_{\mathcal{Z}}^I h| \ + \ |{\Lie_{\mathcal{Z}}^I \Far}|
		& \leq 
		\left \lbrace \begin{array}{lr}
    	C \varepsilon (1 + t + |q|)^{-1 + C \varepsilon} (1 + |q|)^{-1 - C\varepsilon}, &  \mbox{if} \ q > 0, \\
      C \varepsilon (1 + t + |q|)^{-1 + C \varepsilon} (1 + |q|)^{-1/2 + \upmu'}, & \mbox{if} \ q < 0,
    \end{array}
  \right., \qquad |I| \leq \lceil \dParameter/2 \rceil, 
  	\label{E:MainTheoremAssumptionStrongLinfinityDecay} \\
  |\conenabla \nabla_{\mathcal{Z}}^I h|
  \ + \ (1 + |q|)|\conenabla \Lie_{\mathcal{Z}}^I \Far|
  \ + \ |\Lie_{\mathcal{Z}}^I \Far|_{\mathcal{L}\mathcal{N}} 
  \ + \ |\Lie_{\mathcal{Z}}^I \Far|_{\mathcal{T}\mathcal{T}}
  	& \leq 
  	\left \lbrace \begin{array}{lr}
    	C \varepsilon (1 + t + |q|)^{-2 + C \varepsilon} (1 + |q|)^{- C\varepsilon}, &  \mbox{if} \ q > 0, \\
      C \varepsilon (1 + t + |q|)^{-2 + C \varepsilon} (1 + |q|)^{1/2 + \upmu'}, & \mbox{if} \ q < 0,
    \end{array}
  	\right., \qquad |I| \leq \lceil \dParameter/2 \rceil - 1. \label{E:MainTheoremAssumptionStrongerLinfinityDecayGoodComponents} 
 	\end{align}
	\end{subequations}
	In addition, assume that the following smallness conditions on the abstract initial data and ADM mass hold:
	
	\begin{align} \label{E:BootstrapTheoremDataareSmall}
		E_{\dParameter;\upgamma}(0) + M & \leq \mathring{\varepsilon}.
	\end{align}
	
	Then for any constant $\upmu$ satisfying $0 < \upmu < 1/2 - \upmu',$ there exist positive constants $\varepsilon_{\dParameter},$
	$c_{\dParameter},$ and $\widetilde{c}_{\dParameter}$ depending on $\dParameter ,$ $\upmu,$ $\upmu',$ and $\upgamma$ such that if
	$\mathring{\varepsilon} \leq \varepsilon \leq \varepsilon_{\dParameter},$ then the following energy inequality holds for $t \in [0,T):$
	
	\begin{align} \label{E:ImprovedEnergyInequality}
		\mathcal{E}_{\dParameter;\upgamma;\upmu}(t) \leq c_{\dParameter} (\mathring{\varepsilon} + \varepsilon^{3/2}) 
		(1 + t)^{\widetilde{c}_{\dParameter} \varepsilon}.
	\end{align}
	
\end{theorem}

\begin{remark} \label{R:ImprovedDecay}
	By Lemma \ref{L:h0decayestimates}, the decompositions $h = h^{(0)} + h^{(1)}$ and $H = H_{(0)} + H_{(1)}$ 
	(where $H^{\mu \nu} \eqdef (g^{-1})^{\mu \nu} - (m^{-1})^{\mu \nu}$),
	and the fact that $H_{(1)}^{\mu \nu} = - h^{(1)\mu \nu} \ + \ O^{\infty}(\big |h^{(0)} + h^{(1)}|^2 \big),$ it follows that
	the estimates stated in the assumptions of the theorem also hold if we replace $h$ with $h^{(1)},$ $h^{(1)},$ or 
	$H_{(1)}.$
\end{remark}

We now  state and (using the results of Theorem \ref{T:ImprovedDecay}) prove our main global stability theorem.

\begin{theorem} \label{T:MainTheorem}
	\textbf{(Global stability of the Minkowski spacetime solution)}
	Let $(\mathring{\underline{g}}_{jk}  \delta_{jk} + \mathring{\underline{h}}_{jk}^{(0)} + \mathring{\underline{h}}_{jk}^{(1)}, 
	\mathring{K}_{jk}, \mathring{\mathfrak{\Displacement}_j}, \mathring{\mathfrak{\Magneticinduction}}_j),$ 
	$(j,k=1,2,3),$ be abstract initial data on the manifold $\mathbb{R}^3$
	for the Einstein-nonlinear electromagnetic system \eqref{E:IntroEinstein} - \eqref{E:IntrodMis0}
	that satisfy the constraints \eqref{E:Gauss} - \eqref{E:DivergenceB0}, and let
	$(g_{\mu \nu}|_{t=0} = m_{\mu \nu} + h_{\mu \nu}^{(0)}|_{t=0} + h_{\mu \nu}^{(1)}|_{t=0}, 
	\partial_t g_{\mu \nu}|_{t=0} = \partial_t h_{\mu \nu}^{(0)}|_{t=0} + \partial_th_{\mu \nu}^{(1)}|_{t=0}, 
	\Far_{\mu \nu}|_{t=0}),$ $(\mu, \nu = 0,1,2,3),$ be the corresponding initial data 
	for the reduced system \eqref{E:Reducedh1Summary} - \eqref{E:ReduceddMis0Summary} as defined in Section \ref{SS:ReducedData}.
	Assume that the abstract initial data are asymptotically flat in the sense that 
	\eqref{E:metricdataexpansion} - \eqref{E:BdecayAssumption} hold. Let $\dParameter \geq 8$ be an integer, and let $0 < \upgamma < 1/2$ 
	be a fixed constant. Then there exist a global system of wave coordinates $(t,x)$ and
	a constant $\varepsilon_{\dParameter} > 0$ depending on $\upgamma$ and $\dParameter $ such that if 
	$\varepsilon \leq \varepsilon_{\dParameter},$ and if 
	
	\begin{subequations}
	\begin{align}
		E_{\dParameter;\upgamma}(0) + M & \leq \varepsilon,
	\end{align}
	\end{subequations}
	then the reduced data launch a unique global, classical, geodesically complete solution \\
	$(g_{\mu \nu} \eqdef m_{\mu \nu} + h_{\mu \nu}^{(0)} + h_{\mu \nu}^{(1)}, \Far_{\mu \nu})$ to 
	\textbf{both}\footnote{Of course, we technically mean here that the pair $(h_{\mu \nu}^{(1)},\Far_{\mu \nu})$
	is a solution to the version \eqref{E:Reducedh1Summary} - \eqref{E:ReduceddMis0Summary} of the reduced equations, while the 
	pair $(g_{\mu \nu},\Far_{\mu \nu})$ is a solution to equations \eqref{E:IntroEinstein} - \eqref{E:IntrodMis0}.} the reduced
	system \eqref{E:Reducedh1Summary} - \eqref{E:ReduceddMis0Summary} and the Einstein-nonlinear electromagnetic system 
	\eqref{E:IntroEinstein} - \eqref{E:IntrodMis0}. Furthermore, there exists
	a constant $0 < \upmu < 1/2$ (see Remark \ref{R:Roleofmu}), 
	and constants $c_{\dParameter} > 0,$ $\widetilde{c}_{\dParameter} > 0$ depending on $\upgamma$ and $\dParameter ,$ 
	such that the solution's energy satisfies the following bound for all $t \in (-\infty,\infty):$
	
	\begin{align} \label{E:GlobalEnergyInequality}
		\mathcal{E}_{\dParameter;\upgamma;\upmu}(t) \leq c_{\dParameter} \varepsilon (1 + |t|)^{\widetilde{c}_{\dParameter} \varepsilon}.
	\end{align}
	
	In addition, there exists a constant $C_{\dParameter} > 0$ depending on $\upgamma$ and $\dParameter ,$ such that the 
	following pointwise decay estimates hold for all $(t,x) \in (-\infty,\infty) \times \mathbb{R}^3:$
	
	\begin{subequations} 
	\begin{align} \label{E:GlobalStabilityTheoremStrongLinfinityDecayPrincipalTermCoefficients}
		& (1 + |t| + |q|)^{1 - \widetilde{c}_{\dParameter} \varepsilon}(1 + |q|)^{-1/2}
			|\nabla h|_{\mathcal{L} \mathcal{T}} 
		\ + \ (1 + |t| + |q|)^{1 - \widetilde{c}_{\dParameter} \varepsilon}(1 + |q|)^{-1/2} |\nabla \nabla_Z h|_{\mathcal{L} \mathcal{L}}
		\ + \ |\nabla h|_{\mathcal{T} \mathcal{N}} 	\ + \ \big\lbrace 1 + \ln(1 + |t|) \big\rbrace^{-1}|\nabla h| \\
		& \ \ + \ (1 + |t| + |q|)^{1 - \widetilde{c}_{\dParameter} \varepsilon}(1 + |q|)^{-3/2}|h|_{\mathcal{L} \mathcal{T}}
			\ + \ (1 + |t| + |q|)^{1 - \widetilde{c}_{\dParameter} \varepsilon}(1 + |q|)^{-3/2}|\nabla_Z h|_{\mathcal{L} \mathcal{L}} 
			\ + \ |\Far| \notag \\
		& \leq C_{\dParameter} \varepsilon (1 + |t| + |q|)^{-1}, \notag
	\end{align}
	
	\begin{align}
		|\nabla \nabla_{\mathcal{Z}}^I h^{(1)}| \ + \ (1 + |q|)^{-1} |\nabla_{\mathcal{Z}}^I h^{(1)}| 
		\ + \ |\Lie_{\mathcal{Z}}^I \Far|
		& \leq 
		\left \lbrace \begin{array}{lr}
    	C_{\dParameter} \varepsilon (1 + |t| + |q|)^{-1 + \widetilde{c}_{\dParameter} \varepsilon} 
    		(1 + |q|)^{-1 - \upgamma}, &  \mbox{if} \ q > 0, \\
      C_{\dParameter} \varepsilon (1 + |t| + |q|)^{-1 + \widetilde{c}_{\dParameter} \varepsilon} (1 + |q|)^{-1/2}, & \mbox{if} \ q < 0,
    \end{array}
  \right., \qquad |I| \leq \dParameter - 2, 
  	\label{E:GlobalStabilityTheoremStrongLinfinityDecay} \\
  |\conenabla \nabla_{\mathcal{Z}}^I h^{(1)}| \ + \ (1 + |q|)|\conenabla \Lie_{\mathcal{Z}}^I \Far|
  \ + \ |{\Lie_{\mathcal{Z}}^I \Far}|_{\mathcal{L}\mathcal{N}} 
  \ + \ |{\Lie_{\mathcal{Z}}^I \Far}|_{\mathcal{T}\mathcal{T}}
  	& \leq 
  	\left \lbrace \begin{array}{lr}
    	C_{\dParameter} \varepsilon (1 + |t| + |q|)^{-2 + \widetilde{c}_{\dParameter} \varepsilon} (1 + |q|)^{- \upgamma}, &  \mbox{if} \ q > 0, \\
      C_{\dParameter} \varepsilon (1 + |t| + |q|)^{-2 + \widetilde{c}_{\dParameter} \varepsilon} (1 + |q|)^{1/2}, & \mbox{if} \ q < 0,
    \end{array}
  	\right., \qquad |I| \leq \dParameter - 3. \label{E:GlobalStabilityTheoremStrongerLinfinityDecayGoodComponents} 
 	\end{align}
	\end{subequations}

\end{theorem}

\begin{remark} \label{R:qdecay}
	Some of the $1 + |q|-$decay estimates in inequalities 
	\eqref{E:GlobalStabilityTheoremStrongLinfinityDecayPrincipalTermCoefficients} -
	\eqref{E:GlobalStabilityTheoremStrongerLinfinityDecayGoodComponents} are not optimal, and can be improved with additional 
	work. For example, in \cite[Section 16]{hLiR2010}, with the help of the fundamental solution of the Minkowski
	wave operator $\Square_m,$ the $1 + |q|-$decay estimates 
	\eqref{E:GlobalStabilityTheoremStrongLinfinityDecay} - \eqref{E:GlobalStabilityTheoremStrongerLinfinityDecayGoodComponents}
	for the tensorfield $h^{(1)}$ are strengthened by a power of $1/2$ in the 
	interior region $q < 0.$
\end{remark}

\begin{remark}
	Proposition \ref{P:PreservationofWaveCoordianteGauge} shows that the wave coordinate condition
	\eqref{E:wavecoordinategauge1} holds in the domain of classical existence of the solution to the reduced equations; 
	this is why the reduced solution also verifies the Einstein-nonlinear electromagnetic equations \eqref{E:IntroEinstein} - \eqref{E:IntrodMis0}.
\end{remark}

\begin{remark}
	A global stability result for the reduced equations under the wave coordinate assumption, 
	without regard for the abstract initial data, can be deduced
	from the smallness of $\mathcal{E}_{\dParameter;\upgamma;\upmu}(0) + |M|$ (we could even allow for negative $M$!)
	together with the assumption $\liminf_{|x| \to \infty} |h^{(1)}(0,x)| = 0;$ this latter assumption, which is needed to deduce 
	the inequalities \eqref{E:weakdecayLinfinity} at $t=0,$ is automatically implied by the assumptions of Theorem 
	\ref{T:MainTheorem}. 
\end{remark}

\begin{proof}
	We only discuss the region of spacetime in which $t \geq 0;$ the argument for $t \leq 0$ is similar. 
	Let us set $E_{\dParameter;\upgamma}(0) + M \eqdef \mathring{\varepsilon}.$
	By Proposition \ref{P:LocalExistence}, we can 
	choose constants $\upgamma', \upmu, \upmu',$ and $\updelta$ subject to the restrictions described in Section 
	\ref{SS:FixedConstants} (and in particular depending on $\upgamma$), and a constant $A_{\dParameter} > 0$
	such that if $\varepsilon \eqdef A_{\dParameter} \mathring{\varepsilon},$ $A_{\dParameter}$ is sufficiently large, 
	and $\mathring{\varepsilon}$ is sufficiently small, then there exists a nontrivial spacetime slab 
	$[0,T) \times \mathbb{R}^3$ upon which the solution to the reduced equations exists and satisfies the
	energy bound $\mathcal{E}_{\dParameter;\upgamma;\upmu}(t) \leq \varepsilon (1 + t)^{\updelta}$ 
	for $t \in [0,T).$ We then define
	
	\begin{align*}
		T_* \eqdef \sup \big\lbrace T & \mid  
		\mbox{the solution exists classically and remains in the regime of hyperbolicity of the reduced equations}, \\ 
			& \ \mbox{and} \ \mathcal{E}_{\dParameter;\upgamma;\upmu}(t) \leq \varepsilon (1 + t)^{\updelta} \ \mbox{for} \ t \in [0,T) 
			\big\rbrace.
	\end{align*}
	Note that under the above assumptions, we have that $T_* > 0.$ 
	
	We now observe that the main energy bootstrap assumption \eqref{E:Bootstrap} is satisfied on $[0,T_*).$
	Thus, if $\varepsilon$ is sufficiently small, then by Propositions \ref{P:UpgradedDecayhA} and 
	\ref{P:UpgradedDecayh1A}, all of the hypotheses of Theorem \ref{T:ImprovedDecay} are necessarily satisfied
	on $[0,T_*).$ Here, we are using the fact that $\lceil \dParameter/2 \rceil \leq \dParameter - 4,$ which holds if
	$\dParameter \geq 8.$ Consequently, the conclusion of that theorem (i.e., estimate 
	\eqref{E:ImprovedEnergyInequality}) allows us to deduce that the following energy estimate holds for $t \in [0,T_*):$
	
	\begin{align} \label{E:ImprovedDecayConclusion}
		\mathcal{E}_{\dParameter;\upgamma;\upmu}(t) \leq c_{\dParameter} \big\lbrace \mathring{\varepsilon} + \varepsilon^{3/2} \big\rbrace
		(1 + t)^{\widetilde{c}_{\dParameter} \varepsilon} = c_{\dParameter} \bigg\lbrace \frac{\varepsilon}{A_{\dParameter}} 
			+ \varepsilon^{3/2} \bigg\rbrace (1 + t)^{\widetilde{c}_{\dParameter} \varepsilon}.
	\end{align} 
	Now if $A_{\dParameter} > 3 c_{\dParameter}$ and $\mathring{\varepsilon}$ is sufficiently small, then 
	\eqref{E:ImprovedDecayConclusion} implies that
	
	\begin{align} \label{E:ImprovedDecayConclusionWithRoom}
			\mathcal{E}_{\dParameter;\upgamma;\upmu}(t) < \frac{1}{2} A_{\dParameter} \mathring{\varepsilon} (1 + t)^{A_{\dParameter} 
			\widetilde{c}_{\dParameter} 
			\mathring{\varepsilon}} = \frac{1}{2} \varepsilon (1 + t)^{\widetilde{c}_{\dParameter} \varepsilon},
	\end{align}
	which is a strict improvement over the bootstrap assumption assumption \eqref{E:Bootstrap}.
	Thus, by \eqref{E:ImprovedDecayConclusionWithRoom}, the weighted Klainerman-Sobolev inequality \eqref{E:PhiKlainermanSobolev} 
	(which, together with \eqref{E:LieZIinTermsofNablaZI} and the smallness of $\mathcal{E}_{\dParameter;\upgamma;\upmu}(t),$ 
	implies that the solution remains within the regime of hyperbolicity of the reduced equations),
	the continuation principle of Proposition \ref{P:LocalExistence},
	and the continuity of $\mathcal{E}_{\dParameter;\upgamma;\upmu}(t),$ it follows that
	if $A_{\dParameter}$ is sufficiently large and $\mathring{\varepsilon}$ is sufficiently small, 
	then $T_* = \infty.$ Furthermore, under these assumptions, it is an obvious consequence of this reasoning
	that \eqref{E:ImprovedDecayConclusionWithRoom} holds for $t \in [0,\infty).$ After renaming the constants
	in \eqref{E:ImprovedDecayConclusionWithRoom}, we arrive at \eqref{E:GlobalEnergyInequality}.
	
	The inequalities \eqref{E:GlobalStabilityTheoremStrongLinfinityDecay} 
	follow as in the proof of Corollary \ref{C:WeakDecay}, but using the strong energy estimate 
	\eqref{E:GlobalEnergyInequality} instead of the energy bootstrap assumption \eqref{E:Bootstrap}. Similarly,
	the inequalities \eqref{E:GlobalStabilityTheoremStrongLinfinityDecayPrincipalTermCoefficients} follow	
	as in our proof of Proposition \ref{P:UpgradedDecayhA}, but using the strong energy estimate 
	\eqref{E:GlobalEnergyInequality} instead of the energy bootstrap assumption \eqref{E:Bootstrap}. The inequalities 
	\eqref{E:GlobalStabilityTheoremStrongerLinfinityDecayGoodComponents}
	for $|\conenabla \nabla_{\mathcal{Z}}^I h^{(1)}|$ and $|\conenabla \Lie_{\mathcal{Z}}^I \Far|$ follow from
	Lemma \ref{L:PointwisetandqWeightedNablainTermsofZestiamtes}, \eqref{E:LieZIinTermsofNablaZI}, and
	\eqref{E:GlobalStabilityTheoremStrongLinfinityDecay}. The inequalities 
	\eqref{E:GlobalStabilityTheoremStrongerLinfinityDecayGoodComponents} for 
	$|{\Lie_{\mathcal{Z}}^I \Far}|_{\mathcal{L}\mathcal{N}}$ and 
	$|{\Lie_{\mathcal{Z}}^I \Far}|_{\mathcal{T}\mathcal{T}}$ follow as in our proof of \eqref{E:FLUTTimproveddecay}, but 
	using the strong energy estimate \eqref{E:GlobalEnergyInequality} instead of the energy bootstrap assumption 
	\eqref{E:Bootstrap}.
	
	Based on these pointwise decay estimates, the geodesic completeness of the spacetime 
	$(\mathbb{R}^{1+3},g_{\mu \nu} \eqdef m_{\mu \nu} + h_{\mu \nu}^{(0)} 
	+ h_{\mu \nu}^{(1)})$ follows as in \cite[Section 16]{hLiR2005} and \cite[Section 9]{jL2008}.

\end{proof}

It remains to prove Theorem \ref{T:ImprovedDecay}.

\begin{center}\textbf{\Large Proof of Theorem \ref{T:ImprovedDecay}} \end{center}

\subsection{The main argument in the proof of Theorem \ref{T:ImprovedDecay}} \label{SS:MainArgument}

Our goal is to use \emph{only} the assumptions of Theorem \ref{T:ImprovedDecay} to
deduce (for all sufficiently small non-negative $\varepsilon,$ and for sufficiently large fixed constants $c_{\dParameter},\widetilde{c}_{\dParameter}$) the ``strong'' energy estimate \eqref{E:ImprovedEnergyInequality}, which reads

\begin{align} \label{E:FundamentalBootstrapMainTheorem}
	\mathcal{E}_{\dParameter;\upgamma;\upmu}(t) \leq c_{\dParameter} (\mathring{\varepsilon} + \varepsilon^{3/2}) 
		(1 + t)^{\widetilde{c}_{\dParameter} \varepsilon}.
\end{align}
The proof of \eqref{E:FundamentalBootstrapMainTheorem} is based on a hierarchy of Gronwall-amenable inequalities that arise from careful analysis of the integrals of Proposition \ref{P:weightedenergy} involving the inhomogeneous terms 
$\mathfrak{H}_{\mu \nu}^{(1;I)}$ and $\mathfrak{F}_{(I)}^{\nu}.$ We recall that these inhomogeneous terms are captured by Propositions \ref{P:InhomogeneousTermsNablaZIh1} and \ref{P:InhomogeneoustermsLieZIFar}, which state that $\nabla_{\mathcal{Z}}^I h_{\mu \nu}^{(1)}$ and $\Lie_{\mathcal{Z}}^I \Far_{\mu \nu}$ are solutions to the following system of equations:

\begin{subequations}
\begin{align}
	\widetilde{\Square}_{g} \nabla_{\mathcal{Z}}^I h_{\mu \nu}^{(1)} & = \mathfrak{H}_{\mu \nu}^{(1;I)}, &&
		(\mu, \nu = 0,1,2,3),
		\label{E:InhomogeneousTermsNablaZIh1proof} \\ 
	\nabla_{\lambda} \Lie_{\mathcal{Z}}^I \Far_{\mu \nu} + \nabla_{\mu} \Lie_{\mathcal{Z}}^I \Far_{\nu \lambda} 
		+ \nabla_{\nu} \Lie_{\mathcal{Z}}^I \Far_{\lambda \mu} & = 0, && (\lambda, \mu, \nu = 0,1,2,3),
	\label{E:InhomogeneoustermsdFis0proof} \\
	N^{\# \mu \nu \kappa \lambda} \nabla_{\mu} \Lie_{\mathcal{Z}}^I \Far_{\kappa \lambda} & = \mathfrak{F}_{(I)}^{\nu}, &&
		(\nu = 0,1,2,3). \label{E:InhomogeneoustermsdMis0proof}
\end{align}
\end{subequations}
Most of the work goes into obtaining suitable estimates for the integrals involving $\mathfrak{H}_{\mu \nu}^{(1;I)}$
and $\mathfrak{F}_{(I)}^{\nu}.$ In order to avoid impeding the flow of the proof, we prove 
most of the desired inequalities later in this section, after the main argument. For the main part of the argument, we simply quote Corollary \ref{C:NablaZIh1FundamentalEnergyEstimate} and Corollary \ref{C:LieZIFarFundamentalEnergyEstimate}, which 
are the key estimates that allow us to apply a suitable version of Gronwall's lemma. We will then return to the proofs of the corollaries, which follow from a large collection of lemmas, each of which involves the analysis of one of the constituent pieces of the integrals involving $\mathfrak{H}_{\mu \nu}^{(1;I)}$ and $\mathfrak{F}_{(I)}^{\nu}.$ 

We now proceed to the main argument. Using Proposition \ref{P:weightedenergy}, Corollary \ref{C:NablaZIh1FundamentalEnergyEstimate}, and Corollary \ref{C:LieZIFarFundamentalEnergyEstimate}, we
have that

\begin{align} \label{E:Mainenergyinequality}
	& \sum_{|I| \leq k} 
		\int_{\Sigma_{t}} \Big|\myarray[{\nabla\nabla_{\mathcal{Z}}^I h^{(1)}}]{\Lie_{\mathcal{Z}}^I \Far} \Big|^2 w(q) \,d^3x 
		\ + \ \sum_{|I| \leq k} \int_{0}^{t} \int_{\Sigma_{\tau}} 
		\bigg\lbrace |\conenabla \nabla_{\mathcal{Z}}^I h^{(1)}|^2 + |\Lie_{\mathcal{Z}}^I \Far|_{\mathcal{L}\mathcal{N}}^2 
		+ |\Lie_{\mathcal{Z}}^I \Far|_{\mathcal{T} \mathcal{T}}^2 \bigg\rbrace w'(q) \,d^3x \, d \tau \\
	& \leq C \sum_{|I| \leq k} \int_{\Sigma_{0}} \Big|\myarray[{\nabla\nabla_{\mathcal{Z}}^I h^{(1)}}]{\Lie_{\mathcal{Z}}^I 
		\Far} \Big|^2 w(q) \,d^3x 
		\ + \ C \varepsilon \sum_{|I| \leq k} \int_{0}^{t} \int_{\Sigma_{\tau}} 
			(1 + \tau)^{-1} \Big|\myarray[{\nabla\nabla_{\mathcal{Z}}^I h^{(1)}}]{\Lie_{\mathcal{Z}}^I \Far} \Big|^2 
			w(q) \, d^3x \, d \tau \notag \\
	& \ \ + \ C \sum_{|I| \leq k} \int_{0}^{t} \int_{\Sigma_{\tau}} 
		\bigg\lbrace |\mathfrak{H}^{(1;I)}| |\nabla\nabla_{\mathcal{Z}}^I h^{(1)}|  
		\ + \ |(\Lie_{\mathcal{Z}}^I \Far_{0 \nu}) \mathfrak{F}_{(I)}^{\nu} | \bigg\rbrace w(q) \,d^3x \, d \tau \notag \\
	& \leq C \sum_{|I| \leq k} \int_{\Sigma_{0}} \Big|\myarray[{\nabla\nabla_{\mathcal{Z}}^I h^{(1)}}]{\Lie_{\mathcal{Z}}^I 
		\Far} \Big|^2 w(q) \,d^3x
		\ + \ C M \sum_{|I| \leq k} \int_{0}^{t} (1 + \tau)^{-3/2} 
			\Big(\sqrt{\int_{\Sigma_{\tau}} |\nabla \nabla_{\mathcal{Z}}^I h^{(1)}|^2 w(q) \, d^3 x} \Big) \, d \tau 
		\notag \\
	& \ \ + \ C \varepsilon \sum_{|I| \leq k} \int_{0}^{t} \int_{\Sigma_{\tau}} (1 + \tau)^{-1} 
		\Big|\myarray[{\nabla\nabla_{\mathcal{Z}}^I h^{(1)}}]{\Lie_{\mathcal{Z}}^I \Far} \Big|^2 w(q) \,d^3x \, d \tau \notag \\
	& \ \ + \ C \varepsilon \sum_{|I| \leq k} \int_{0}^{t} \int_{\Sigma_{\tau}} 
			\bigg\lbrace |\conenabla \nabla_{\mathcal{Z}}^I h^{(1)}|^2 + |\Lie_{\mathcal{Z}}^I \Far|_{\mathcal{L}\mathcal{N}}^2 
		\ + \ |\Lie_{\mathcal{Z}}^I \Far|_{\mathcal{T} \mathcal{T}}^2 \bigg\rbrace w'(q)  \,d^3x \, d \tau \notag \\
	& \ \ + \ C \varepsilon  \underbrace{\sum_{|J| \leq k - 1} \int_{0}^{t} \int_{\Sigma_{\tau}} (1 + \tau)^{- 1 + C \varepsilon} 
			\Big|\myarray[{\nabla\nabla_{\mathcal{Z}}^J h^{(1)}}]{\Lie_{\mathcal{Z}}^J \Far} \Big|^2 w(q) \, d^3x \, d 
			\tau}_{\mbox{Absent if $k=0$}} \ + \ C \varepsilon^3. \notag
\end{align}

Recalling the definition (where the dependence on $\upmu,$ $\upgamma$ is through $w(q)$)

\begin{align*}
	\mathcal{E}_{k;\upgamma;\upmu}^2(t) & \eqdef \sup_{0 \leq \tau \leq t} \sum_{|I| \leq k} 
		\int_{\Sigma_{\tau}} \Big\lbrace |\nabla \nabla_{\mathcal{Z}}^I h^{(1)}|^2 + |\Lie_{\mathcal{Z}}^I \Far|^2 
		\Big\rbrace w(q) \, d^3x, 
\end{align*}
and introducing the quantity $\mathcal{S}_{k;\upgamma;\upmu}(t) \geq 0,$ which is defined by 

\begin{align}
	\mathcal{S}_{k;\upgamma;\upmu}^2(t) & \eqdef \sum_{|I| \leq k} 
		\int_0^t \int_{\Sigma_{\tau}} \Big\lbrace |\conenabla \nabla_{\mathcal{Z}}^I h^{(1)}|^2 + |\Lie_{\mathcal{Z}}^I 
		\Far|_{\mathcal{L}\mathcal{N}}^2 + |\Lie_{\mathcal{Z}}^I \Far|_{\mathcal{T} \mathcal{T}}^2 \Big\rbrace w'(q) \, d^3x \, d 
		\tau,
\end{align}
it therefore follows from the final inequality of \eqref{E:Mainenergyinequality} that

\begin{align} \label{E:Mainenergyinequalityreexpressed}
	\mathcal{E}_{k;\upgamma;\upmu}^2(t) \ + \ \mathcal{S}_{k;\upgamma;\upmu}^2(t) 
		& \leq C \mathcal{E}_{k;\upgamma;\upmu}^2(0) 
			\ + \ C M \int_{0}^{t} (1 + \tau)^{-3/2} \mathcal{E}_{k;\upgamma;\upmu}(\tau) \, d \tau
			\ + \ C \varepsilon \int_{0}^{t} (1 + \tau)^{-1} \mathcal{E}_{k;\upgamma;\upmu}^2(\tau) \, d \tau \\
		& \ \ + \ \underbrace{C \varepsilon \mathcal{S}_{k;\upgamma;\upmu}^2(t)}_{\mbox{absorb into l.h.s.}}
			\ + \ C \varepsilon \int_{0}^{t} (1 + \tau)^{- 1 + C \varepsilon} \mathcal{E}_{k-1;\upgamma;\upmu}^2(\tau) \, 
			d \tau \notag \ + \ C \varepsilon^3.
\end{align}

For $\varepsilon$ sufficiently small, we may absorb the $C \varepsilon \mathcal{S}_{k;\upgamma;\upmu}^2(t)$ term from \eqref{E:Mainenergyinequalityreexpressed} into the left-hand side at the expense of increasing all of the constants. We can similarly absorb the term $C M \int_{0}^{t} (1 + \tau)^{-3/2} \mathcal{E}_{k;\upgamma;\upmu}(\tau) \, d \tau$ by using the inequality \\
${C M \int_{0}^{t} (1 + \tau)^{-3/2} \mathcal{E}_{k;\upgamma;\upmu}(\tau) \, d \tau \leq 1/2 \mathcal{E}_{k;\upgamma;\upmu}^2(t) + C^2 M^2};$ this inequality follows from the algebraic estimate \\
$C M \mathcal{E}_{k;\upgamma;\upmu}(\tau)\leq 1/4 \mathcal{E}_{k;\upgamma;\upmu}^2(\tau) + C^2 M^2,$ 
the integral inequality $\int_{0}^{t} (1 + \tau)^{-3/2} \, d \tau \leq 2,$
and the fact that $\mathcal{E}_{k;\upgamma;\upmu}^2(\tau)$ is increasing. If we also use the fact that 
$\mathcal{E}_{k;\upgamma;\upmu}^2(0) \leq C \big\lbrace E_{\dParameter;\upgamma}^2(0) + M^2 \big\rbrace \leq 
C \mathring{\varepsilon}^2$ (i.e, Proposition \ref{P:SmallNormImpliesSmallEnergy}), and the inequality 
$M \leq \mathring{\varepsilon},$ then we arrive at the following inequality, valid for all small $\varepsilon:$

\begin{align} \label{E:EkpluSkinequality}
	\mathcal{E}_{k;\upgamma;\upmu}^2(t) \ + \ \mathcal{S}_{k;\upgamma;\upmu}^2(t) 
	\leq C \big\lbrace \mathring{\varepsilon}^2 + \varepsilon^3 \big\rbrace
	\ + \ C \varepsilon  \int_{0}^t (1 + \tau)^{-1} \mathcal{E}_{k;\upgamma;\upmu}^2(\tau) \, d \tau 
	\ + \ C \varepsilon  \underbrace{\int_{0}^t (1 + \tau)^{- 1 + C \varepsilon} \mathcal{E}_{k-1;\upgamma;\upmu}^2 (\tau) \, d 
	\tau}_{\mbox{Absent if $k=0$}}.
\end{align}

For $k=0,$ this implies that
\begin{align} \label{E:Gronwallreadyinequality0}
	\mathcal{E}_{0;\upgamma;\upmu}^2(t) \leq C \big\lbrace \mathring{\varepsilon}^2 + \varepsilon^3 \big\rbrace   
	\ + \ c_0 \varepsilon \int_{0}^t (1 + \tau)^{-1} \mathcal{E}_{0;\upgamma;\upmu}^2(\tau) \, d \tau.
\end{align}
From \eqref{E:Gronwallreadyinequality0} and Gronwall's lemma, we conclude that

\begin{align}
	\mathcal{E}_{0;\upgamma;\upmu}^2(t) \leq C \big\lbrace \mathring{\varepsilon}^2 + \varepsilon^3 \big\rbrace 
	(1 + t)^{c_0 \varepsilon}.
\end{align}

Inductively using \eqref{E:EkpluSkinequality}, we therefore derive the following estimate for $k \geq 1:$ 
\begin{align} \label{E:Gronwallreadyinequalityk}
	\mathcal{E}_{k;\upgamma;\upmu}^2(t) \ + \ \mathcal{S}_{k;\upgamma;\upmu}^2(t) 
	& \leq 
		C \big\lbrace \mathring{\varepsilon}^2 + \varepsilon^3 \big\rbrace 
		\ + \ C \varepsilon  \int_{0}^t (1 + \tau)^{-1} \mathcal{E}_{k;\upgamma;\upmu}^2(\tau) \, d \tau 
		\ + \ C \varepsilon \big\lbrace \mathring{\varepsilon}^2 + \varepsilon^3 \big\rbrace \int_{0}^t (1 + \tau)^{-1 + C 
		\varepsilon} \, d \tau \\
	& \leq 
		C \big\lbrace \mathring{\varepsilon}^2 + \varepsilon^3 \big\rbrace 
		\ + \ C \varepsilon  \int_{0}^t (1 + \tau)^{-1} \mathcal{E}_{k;\upgamma;\upmu}^2(\tau) \, d \tau. \notag
\end{align}
Finally, from \eqref{E:Gronwallreadyinequalityk} and Gronwall's lemma, we deduce that
if $\varepsilon$ is sufficiently small, then

\begin{align}
	\mathcal{E}_{k;\upgamma;\upmu}^2(t) \leq C \big\lbrace \mathring{\varepsilon}^2 + \varepsilon^3 \big\rbrace 
	(1 + t)^{c_k \varepsilon},
\end{align}
which closes the induction. Thus, we have shown \eqref{E:ImprovedEnergyInequality}, which concludes the proof of Theorem \ref{T:ImprovedDecay}.

\subsection{Integral inequalities for the $\nabla_{\mathcal{Z}}^I h_{\mu \nu}^{(1)}$ inhomogeneities}

In this section, we analyze the integrals in Proposition \ref{P:weightedenergy} 
corresponding to the inhomogeneous terms $\mathfrak{H}_{\mu \nu}^{(1;I)}$ in equation \eqref{E:InhomogeneousTermsNablaZIh1proof}. The main goal
is to arrive at Corollary \ref{C:NablaZIh1FundamentalEnergyEstimate}. As opposed to the estimates 
proved in Section \ref{SS:MainTheoremFarInhomogeneities}, most of the estimates proved in this section are a rather straightforward generalization of the ones proved in \cite{hLiR2010}; i.e., the estimates involve a similar analysis, but with
additional terms arising from the presence of the $\Far$ terms appearing on the right-hand side of 
the reduced equation \eqref{E:Reducedh1Summary}.

We begin with the following lemma, which follows easily from algebraic estimates of the form $|ab| \lesssim a^2 + b^2.$

\begin{lemma} \label{L:MainTheoremH1IInhomogeneousTermNablaZIh1PeterPaul} \textbf{(Arithmetic-geometric mean inequality)}
Let $\mathfrak{H}_{\mu \nu}^{(1;I)} = \nablamod_{\mathcal{Z}}^I \mathfrak{H}_{\mu \nu} 
- \nablamod_{\mathcal{Z}}^I \widetilde{\Square} h_{\mu \nu}^{(0)}- \big\lbrace \nablamod_{\mathcal{Z}}^I \widetilde{\Square}_{g} h_{\mu \nu}^{(1)} - \widetilde{\Square}_{g} 
\nabla_{\mathcal{Z}}^I h_{\mu \nu}^{(1)} \big\rbrace$ be the inhomogeneous term on the right-hand side of \eqref{E:InhomogeneousTermsNablaZIh1}.
Then the following algebraic inequality holds:

\begin{align}  \label{E:MainTheoremH1IInhomogeneousTermNablaZIh1PeterPaul}
	|\mathfrak{H}^{(1;I)}| |\nabla\nabla_{\mathcal{Z}}^I h^{(1)}| 
	& \leq \varepsilon^{-1}(1 + t)|\nablamod_{\mathcal{Z}}^I  \mathfrak{H}|^2 
		\ + \ \varepsilon^{-1}(1 + t)|\nablamod_{\mathcal{Z}}^I \widetilde{\Square}_{g} h_{\mu \nu}^{(1)} 
		- \widetilde{\Square}_{g} \nabla_{\mathcal{Z}}^I h_{\mu \nu}^{(1)}|^2 \\
	& \ \ + \ \varepsilon(1 + t)^{-1} |\nabla\nabla_{\mathcal{Z}}^I h^{(1)}|^2 
		\ + \ |\nablamod_{\mathcal{Z}}^I \widetilde{\Square}_g h^{(0)}| |\nabla\nabla_{\mathcal{Z}}^I h^{(1)}|. \notag
\end{align}	

\end{lemma}

The next lemma provides a preliminary pointwise estimate for the
$|\nablamod_{\mathcal{Z}}^I  \mathfrak{H}|$ term on the right-hand side of \eqref{E:MainTheoremH1IInhomogeneousTermNablaZIh1PeterPaul}. 

\begin{lemma} \label{L:MainTheoremLieZIh1InhomogeneousAlgebraicEstimate} \cite[Extension of Lemma 11.2]{hLiR2010} 
	\textbf{(Pointwise estimates for the $|\nabla_{\mathcal{Z}}^I \mathfrak{H}|$ inhomogeneities)}
	Under the assumptions of Theorem \ref{T:ImprovedDecay}, if $I$ is any $\mathcal{Z}-$multi-index
	with $|I| \leq \dParameter ,$ and if $\varepsilon$ is sufficiently small,
	then the following pointwise estimates hold for $(t,x) \in [0,T) \times \mathbb{R}^3:$
	
	\begin{align} \label{E:MainTheoremLieZIh1InhomogeneousAlgebraicEstimate}
		|\nabla_{\mathcal{Z}}^I \mathfrak{H}| 
		& \lesssim \varepsilon \sum_{|J| \leq |I|}  
			(1 + t)^{-1} \Big| \myarray[{\nabla\nabla_{\mathcal{Z}}^J h^{(1)}}]{\Lie_{\mathcal{Z}}^J \Far} \Big| 
			\ + \ \varepsilon \sum_{|J| \leq |I|} (1 + t + |q|)^{-1 + C \varepsilon} (1 + |q|)^{-1/2 + \upmu'} 
			|\conenabla \nabla_{\mathcal{Z}}^J h^{(1)}| \\
		& \ \ + \ \varepsilon^2 \sum_{|J| \leq |I|} (1 + t + |q|)^{-1} (1 + |q|)^{-1} |\nabla_{\mathcal{Z}}^J h^{(1)}| 
			\ + \ \varepsilon \underbrace{\sum_{|J'| \leq |I| - 1} (1 + t)^{- 1 + C \varepsilon} 
			\Big| \myarray[{\nabla\nabla_{\mathcal{Z}}^{J'} h^{(1)}}]{\Lie_{\mathcal{Z}}^{J'} \Far} \Big|}_{\mbox{Absent if 
			$|I|=0$}} 
			\ + \ \varepsilon^2 (1 + t + |q|)^{-4}. \notag
	\end{align}	
\end{lemma}

\begin{proof}
	By Proposition \ref{P:AlgebraicInhomogeneous}, we have that	
	\begin{align}
		|\nabla_{\mathcal{Z}}^I \mathfrak{H}| & \lesssim |(i)| \ + \ |(ii)| \ + \ |(iii)|,
	\end{align}
	where
	
	\begin{align}
		|(i)| & = \sum_{|J| + |K| \leq |I|} |\nabla\nabla_{\mathcal{Z}}^J h|_{\mathcal{T} \mathcal{N}} 
			|\nabla\nabla_{\mathcal{Z}}^K h|_{\mathcal{T} \mathcal{N}} 
			\ + \ |\conenabla \nabla_{\mathcal{Z}}^J h| |\nabla\nabla_{\mathcal{Z}}^K h| \ + \ \underbrace{\sum_{|J''| + |K''| \leq 
			|I| - 2} |\nabla\nabla_{\mathcal{Z}}^{J''} h| 
			|\nabla \nabla_{\mathcal{Z}}^{K''} h|}_{\mbox{Absent if $|I| \leq 1$}}, \\
		|(ii)| & = \sum_{|J| + |K| \leq |I|} |\Lie_{\mathcal{Z}}^{J} \Far| |\Lie_{\mathcal{Z}}^{K} \Far|, \\
		|(iii)| & = \sum_{|I_1| + |I_2| + |I_3| \leq |I|} 
			\Big\lbrace|\nabla_{\mathcal{Z}}^{I_1} h| |\nabla\nabla_{\mathcal{Z}}^{I_2} h| |\nabla\nabla_{\mathcal{Z}}^{I_3}h| 
			\ + \ |\nabla_{\mathcal{Z}}^{I_1} h||\Lie_{\mathcal{Z}}^{I_2}\Far| |\Lie_{\mathcal{Z}}^{I_3} \Far| 
			\ + \ |\Lie_{\mathcal{Z}}^{I_1} \Far||\Lie_{\mathcal{Z}}^{I_2}\Far| |\Lie_{\mathcal{Z}}^{I_3} \Far| \Big\rbrace. 
	\end{align}
	
	The desired bound for $|(i)|$ was proved in Lemma 11.2 of \cite{hLiR2010} 
	using the decomposition $h = h^{(1)} + h^{(0)},$ and by combining Lemma \ref{L:h0decayestimates} and the estimates 
	\eqref{E:MainTheoremAssumptionStrongLinfinityDecayPrincipalTermCoefficients} - 	
	\eqref{E:MainTheoremAssumptionStrongerLinfinityDecayGoodComponents}. The term (ii) is the main contribution to 
	$|\nabla_{\mathcal{Z}}^I \mathfrak{H}|$ arising from the presence of non-zero electromagnetic fields.
	To bound $|(ii)|$ by the right-hand side of \eqref{E:MainTheoremLieZIh1InhomogeneousAlgebraicEstimate}, we consider the cases
	$(|J| =\dParameter , |K|= 0),$ $(|J| = 0, |K| =\dParameter ),$ $(|J| \leq \dParameter -1, |K| \leq  \lceil \dParameter/2 \rceil),$ and
	$(|J| \leq \lceil \dParameter/2 \rceil, |K| \leq \dParameter - 1);$ clearly this exhausts all possible cases. 
	In the first two cases, we use \eqref{E:MainTheoremAssumptionStrongLinfinityDecayPrincipalTermCoefficients} to achieve the 
	desired bound, while in the last two, we use \eqref{E:MainTheoremAssumptionStrongLinfinityDecay}. The cubic terms
	from case $(iii)$ can be similarly bounded using \eqref{E:MainTheoremAssumptionStrongLinfinityDecay}.

\end{proof}

Using the previous lemma, we now derive the desired integral inequalities corresponding to the
$\varepsilon^{-1}(1 + t)|\nablamod_{\mathcal{Z}}^I  \mathfrak{H}|^2$ term on the right-hand side of \eqref{E:MainTheoremH1IInhomogeneousTermNablaZIh1PeterPaul}.

\begin{lemma} \label{L:MainTheoremLieZIh1InhomogeneousIntegralEstimate} \cite[Extension of Lemma 11.3]{hLiR2010}
	\textbf{(Integral estimates for $\varepsilon^{-1} (1 + \tau) |\nablamod_{\mathcal{Z}}^I \mathfrak{H}|^2 w(q)$)}
	Under the assumptions of Theorem \ref{T:ImprovedDecay}, if $I$ is any 
	$\mathcal{Z}-$multi-index with $|I| \leq \dParameter ,$ and if $\varepsilon$ is sufficiently small,
	then the following pointwise estimates hold for $(t,x) \in [0,T) \times \mathbb{R}^3:$
	
	\begin{align} \label{E:MainTheoremLieZIh1InhomogeneousIntegralEstimate}
		\varepsilon^{-1} \int_{0}^{t} \int_{\Sigma_{\tau}} & (1 + \tau) 
			|\nablamod_{\mathcal{Z}}^I \mathfrak{H}|^2 w(q) \,d^3x \, d \tau	\\
		& \lesssim  \varepsilon \sum_{|J| \leq |I|}	\int_{0}^{t} \int_{\Sigma_{\tau}} \bigg\lbrace (1 + \tau)^{-1} 
			\Big| \myarray[{\nabla\nabla_{\mathcal{Z}}^J h^{(1)}}]{\Lie_{\mathcal{Z}}^J \Far} \Big|^2 w(q) 
			\ + \ |\conenabla \nabla_{\mathcal{Z}}^J h^{(1)}|^2 w'(q) \bigg\rbrace \,d^3x \, d \tau \notag \\
		& \ \ + \ \underbrace{\varepsilon \sum_{|J'| \leq |I| - 1}	\int_{0}^{t} \int_{\Sigma_{\tau}} (1 + \tau)^{-1 + 
			C\varepsilon} \Big| \myarray[{\nabla\nabla_{\mathcal{Z}}^{J'} h^{(1)}}]{\Lie_{\mathcal{Z}}^{J'} \Far} \Big|^2 w(q) 
			\,d^3x \, d \tau}_{\mbox{Absent if $|I|=0$}} \ + \ \varepsilon^3. \notag
	\end{align}
\end{lemma}

\begin{proof}
	After squaring both sides of \eqref{E:MainTheoremLieZIh1InhomogeneousAlgebraicEstimate}, 
	multiplying by $\varepsilon^{-1} (1 + t) w(q),$ using the inequality \\
	${(1 + |q|)^{-1} (1 + q_-)^{2\upmu} w(q) \lesssim w'(q)}$ (i.e., inequality \eqref{E:weightinequality})
	and the fact that $\upmu + \upmu' < 1/2,$  and integrating, the only terms that are not manifestly bounded by the right-hand side
	of \eqref{E:MainTheoremLieZIh1InhomogeneousIntegralEstimate} are
	
	\begin{align}
		\varepsilon^3 \sum_{|J| \leq |I|} \int_{0}^{t} \int_{\Sigma_t} 
		(1 + \tau)^{-1} (1 + |q|)^{-2} |\nabla_{\mathcal{Z}}^J h^{(1)}|^2 w(q) \, d^3x \, d \tau.
	\end{align}
	The desired bound for these terms can be achieved with the help of the Hardy inequalities of Proposition \ref{P:Hardy},
	which imply that

	\begin{align}
		\int_{\Sigma_t} (1 + \tau)^{-1}(1 + |q|)^{-2} 
			|\nabla_{\mathcal{Z}}^J h^{(1)}|^2 w(q) \, d^3x
			\lesssim \int_{\Sigma_t} (1 + \tau)^{-1} |\nabla\nabla_{\mathcal{Z}}^J h^{(1)}|^2 w(q) \, d^3 x.
	\end{align}

\end{proof}

We now derive the desired the desired integral inequalities corresponding to the 
$|\nablamod_{\mathcal{Z}}^I \widetilde{\Square}_g h^{(0)}| |\nabla\nabla_{\mathcal{Z}}^I h^{(1)}|$
term on the right-hand side of \eqref{E:MainTheoremH1IInhomogeneousTermNablaZIh1PeterPaul}.

\begin{lemma} \label{L:mathfrakH0partialZIhAproductestimate} \cite[Lemma 11.4]{hLiR2010}
\textbf{(Integral estimates for $|\nablamod_{\mathcal{Z}}^I \widetilde{\Square}_g h^{(0)}| |\nabla\nabla_{\mathcal{Z}}^I h^{(1)}| w(q)$)} Let $M$ be the ADM mass. Under the assumptions of Theorem \ref{T:ImprovedDecay}, if $I$ is a $\mathcal{Z}-$multi-index satisfying $|I| \leq \dParameter ,$ and if $\varepsilon$ is sufficiently small,
then the following integral inequality holds for $t \in [0,T):$

\begin{align} \label{E:mathfrakH0partialZIhAproductestimate}
	\int_{0}^{t} \int_{\Sigma_{\tau}} | \nablamod_{\mathcal{Z}}^I \widetilde{\Square}_g h^{(0)}| 
	|\nabla\nabla_{\mathcal{Z}}^I h^{(1)}| w(q) \,d^3x \, d \tau 
	& \lesssim M  \sum_{|J| \leq |I|} \int_{0}^{t} \int_{\Sigma_{\tau}} (1 + \tau)^{-2} 
		|\nabla\nabla_{\mathcal{Z}}^I h^{(1)}|^2 w(q) \, d^3x \, d \tau  \\
	& \ \ + \ M \sum_{|J| \leq |I|} \int_{0}^{t} (1 + \tau)^{-3/2} \bigg( 
	\sqrt{\int_{\Sigma_{\tau}} 
			| \nabla\nabla_{\mathcal{Z}}^I h^{(1)}|^2 w(q) \, d^3 x} \bigg) \, d \tau.  \notag
\end{align}
\end{lemma}

\begin{proof}
	We first use the Cauchy-Schwarz inequality for integrals to obtain
	\begin{align} \label{E:CauchySchwarzmathfrakH0partialZIhAproductestimate}
		\int_{0}^{t} \int_{\Sigma_{\tau}} | & \nablamod_{\mathcal{Z}}^I \widetilde{\Square}_g h^{(0)}| 
			|\nabla\nabla_{\mathcal{Z}}^I h^{(1)}| w(q) \,d^3x \, d \tau \\
		& \leq 	\int_{0}^{t} \Bigg\lbrace \bigg(\int_{\Sigma_{\tau}} |\nablamod_{\mathcal{Z}}^I \widetilde{\Square}_g 
			h^{(0)}|^2 w(q) \, d^3 x\bigg)^{1/2} 
			\times \int_{\Sigma_{\tau}} \bigg(|\nabla\nabla_{\mathcal{Z}}^I h^{(1)}|^2 w(q) \, d^3x \bigg)^{1/2} 
			\Bigg\rbrace \, d \tau. 		\notag
	\end{align}
	Furthermore, under the present assumptions, the previous proof of inequality 
	\eqref{E:weakdecayLinfinitynablaZISquaregh0MoreGeneral} remains valid. Thus,
	using \eqref{E:weakdecayLinfinitynablaZISquaregh0MoreGeneral} and the Hardy inequalities of Proposition \ref{P:Hardy}, it 
	follows that
	
	\begin{align} \label{E:mathfrakH0Squaredintegratedestimate}
		\int_{\Sigma_t} |\nablamod_{\mathcal{Z}}^I \widetilde{\Square}_g h^{(0)}|^2 w(q) \, d^3 x 
		\lesssim M^2(1 + t)^{-3} 
		\ + \ M^2 (1 + t)^{-4} \sum_{|J| \leq |I|} 
			\int_{\Sigma_t} |\nabla\nabla_{\mathcal{Z}}^J h^{(1)}|^2 w(q) \, d^3 x.
	\end{align}
	The estimate \eqref{E:mathfrakH0partialZIhAproductestimate} now follows from 
	\eqref{E:CauchySchwarzmathfrakH0partialZIhAproductestimate}, \eqref{E:mathfrakH0Squaredintegratedestimate}
	and the inequalities $\sqrt{|a| + |b|} \lesssim \sqrt{|a|} + \sqrt{|b|},$ $|ab| \lesssim a^2 + b^2.$
	
\end{proof}

The following integral estimate for the commutator term $\varepsilon^{-1}(1 + t) \big|\nablamod_{\mathcal{Z}}^I \widetilde{\Square}_{g} h_{\mu \nu}^{(1)} - \widetilde{\Square}_{g} (\nabla_{\mathcal{Z}}^I h_{\mu \nu}^{(1)}) \big|^2$ on the right-hand side of \eqref{E:MainTheoremH1IInhomogeneousTermNablaZIh1PeterPaul} was proved in \cite{hLiR2010}. Its lengthy proof is similar to our proof of Lemma \ref{L:MainTheoremFarCommutatorIntegralEstimate} below, and we don't bother to repeat it here.

\begin{lemma} \label{L:NablaZIBoxCommutatorIntegrated}  \cite[Lemma 11.5]{hLiR2010}
	\textbf{(Integral estimates for $\varepsilon^{-1} \big| \nablamod_{\mathcal{Z}}^I \widetilde{\Square}_{g} h_{\mu \nu}^{(1)} 
			- \widetilde{\Square}_{g} \nabla_{\mathcal{Z}}^I h_{\mu\nu}^{(1)}\big|^2 w(q)$)}
	Under the assumptions of Theorem \ref{T:ImprovedDecay}, if $I$ is a $\mathcal{Z}-$multi-index satisfying
	$1 \leq |I| \leq \dParameter ,$ and if $\varepsilon$ is sufficiently small, then 
	the following integral inequality holds for $t \in [0,T):$
	
	\begin{align} \label{E:NablaZIBoxCommutatorIntegrated}
		\varepsilon^{-1} \int_{0}^{t} \int_{\Sigma_{\tau}} & (1 + \tau) 
			\big| \nablamod_{\mathcal{Z}}^I \widetilde{\Square}_{g} h_{\mu \nu}^{(1)} 
			- \widetilde{\Square}_{g} \nabla_{\mathcal{Z}}^I h_{\mu\nu}^{(1)}\big|^2 w(q) \,d^3x \, d \tau \\
		& \lesssim \varepsilon \sum_{|J| \leq |I|} \int_{0}^{t} \int_{\Sigma_{\tau}} \bigg\lbrace (1 + \tau)^{-1} 
			| \nabla\nabla_{\mathcal{Z}}^J h^{(1)}|^2 w(q) + |\conenabla \nabla_{\mathcal{Z}}^J h^{(1)}|^2 w'(q) 
			\bigg\rbrace \,d^3x \, d \tau \notag \\
		& \ \ + \ \varepsilon \sum_{|J'| \leq |I| - 1} \int_{0}^{t} \int_{\Sigma_{\tau}} (1 + \tau)^{-1 + C \varepsilon}
			| \nabla\nabla_{\mathcal{Z}}^{J'} h^{(1)}|^2 w(q) \, d^3x \, d \tau \ + \ \varepsilon^3. \notag
	\end{align}
\end{lemma}

\hfill $\qed$

Combining Lemmas \ref{L:MainTheoremH1IInhomogeneousTermNablaZIh1PeterPaul}, \ref{L:MainTheoremLieZIh1InhomogeneousIntegralEstimate}, \ref{L:mathfrakH0partialZIhAproductestimate}, and \ref{L:NablaZIBoxCommutatorIntegrated}, we arrive at the following corollary.

\begin{corollary} \label{C:NablaZIh1FundamentalEnergyEstimate}
	Under the assumptions of Theorem \ref{T:ImprovedDecay}, if $0 \leq k \leq \dParameter $ and
	$\varepsilon$ is sufficiently small, then 
	the following integral inequality holds for $t \in [0,T):$
	
	\begin{align} \label{E:NablaZIh1FundamentalEnergyEstimate}
		\sum_{|I| \leq k} 
			\int_{0}^{t} \int_{\Sigma_{\tau}} |\mathfrak{H}^{(1;I)}| |\nabla\nabla_{\mathcal{Z}}^I h^{(1)}| \,d^3x \, d \tau
		& \lesssim M \sum_{|I| \leq k} \int_{0}^{t} (1 + \tau)^{-3/2} 
			\Big(\sqrt{\int_{\Sigma_{\tau}} |\nabla \nabla_{\mathcal{Z}}^I h^{(1)}|^2 w(q) \, d^3 x} \Big) \, d \tau	\\
		& \ \ + \ \varepsilon  \sum_{|I| \leq k} \int_{0}^{t} \int_{\Sigma_{\tau}}  (1 + \tau)^{-1} 
			\Big|\myarray[{\nabla\nabla_{\mathcal{Z}}^I h^{(1)}}]{\Lie_{\mathcal{Z}}^I \Far} \Big|^2 w(q) 
			\,d^3x \, d \tau \notag \\
		& \ \ + \ \varepsilon  \sum_{|I| \leq k} \int_{0}^{t} \int_{\Sigma_{\tau}} 
			\Big( |\conenabla \nabla_{\mathcal{Z}}^I h^{(1)}|^2 + |\Lie_{\mathcal{Z}}^I \Far|_{\mathcal{L}\mathcal{N}}^2 
			+ |\Lie_{\mathcal{Z}}^J \Far|_{\mathcal{T} \mathcal{T}}^2 \Big) w'(q)  \, d^3x \, d \tau \notag \\
		& \ \ + \ \varepsilon  \underbrace{\sum_{|J| \leq k - 1} \int_{0}^{t} \int_{\Sigma_{\tau}} (1 + \tau)^{-1 + C 
			\varepsilon} \Big|\myarray[{\nabla\nabla_{\mathcal{Z}}^J h^{(1)}}]{\Lie_{\mathcal{Z}}^J \Far} \Big|^2 w(q) 
			\, d^3x \, d \tau}_{\mbox{Absent if $k=0$}} \ + \ \varepsilon^3. \notag
	\end{align}

\end{corollary}

This completes our analysis of the integral inequalities for the $h_{\mu \nu}^{(1)}$ inhomogeneities. 

\hfill $\qed$

\subsection{Integral inequalities for the $\Lie_{\mathcal{Z}}^I \Far_{\mu \nu}$ inhomogeneities} \label{SS:MainTheoremFarInhomogeneities}

In this section, we analyze the integrals in Proposition \ref{P:weightedenergy} 
corresponding to the inhomogeneous terms $\mathfrak{F}_{(I)}^{\nu}$ in equation \eqref{E:InhomogeneoustermsdMis0proof}. The main goal is to arrive at Corollary \ref{C:LieZIFarFundamentalEnergyEstimate}. 

We begin with the following lemma, which provides pointwise estimates for the wave coordinate-controlled quantities
$|\nabla\nabla_{\mathcal{Z}}^I h^{(1)}|_{\mathcal{L} \mathcal{L}}$ and 
$|\nabla\nabla_{\mathcal{Z}}^J h^{(1)}|_{\mathcal{L} \mathcal{T}}$ for $|I| \leq \dParameter $ and $|J| \leq \dParameter - 1.$ 
These pointwise estimates will be used to help to derive suitable integrated estimates later in this section.

\begin{lemma} \label{L:NablaZIh1LLh1TLMainTheoremWaveCoordianteAlgebraicEstimate}
	\textbf{(Pointwise estimates for $\sum_{|I| \leq k} |\nabla\nabla_{\mathcal{Z}}^I h^{(1)}|_{\mathcal{L} \mathcal{L}}
		\ + \ \sum_{|J| \leq k - 1} |\nabla\nabla_{\mathcal{Z}}^J h^{(1)}|_{\mathcal{L} \mathcal{T}}$)}
	Under the assumptions of Theorem \ref{T:ImprovedDecay}, if $0 \leq k \leq \dParameter $
	and $\varepsilon$ is sufficiently small, then the following pointwise inequality holds for
 	for $(t,x) \in [0,T) \times \mathbb{R}^3:$

	\begin{align}  \label{E:NablaZIh1LLh1TLMainTheoremWaveCoordianteAlgebraicEstimate}
		\sum_{|I| \leq k}|\nabla\nabla_{\mathcal{Z}}^I h^{(1)}|_{\mathcal{L} \mathcal{L}}
		& \ + \ \underbrace{\sum_{|J| \leq k - 1} |\nabla\nabla_{\mathcal{Z}}^J h^{(1)}|_{\mathcal{L} \mathcal{T}}}_{\mbox{Absent 
		if $k = 0$}} 
			 \\
		& \lesssim \sum_{|I| \leq k} |\conenabla \nabla_{\mathcal{Z}}^{I} h^{(1)}| 
			\ + \ \varepsilon (1 + t + |q|)^{-2} \chi_0(1/2 \leq r/t \leq 3/4) \ + \ \varepsilon^2 (1 + t + |q|)^{-3} \notag \\
		& \ \	+ \ \varepsilon \sum_{|I| \leq k} (1 + t + |q|)^{-1 + C \varepsilon}(1 + |q|)^{1/2 + 
		\upmu'}|\nabla\nabla_{\mathcal{Z}}^{I} h^{(1)}| \notag \\ 
		& \ \ + \ \varepsilon \sum_{|I| \leq k} (1 + t + |q|)^{-1 + C \varepsilon}(1 + |q|)^{-1/2 + \upmu'} 
			|\nabla_{\mathcal{Z}}^{I} h^{(1)}| \notag \\
		& \ \ + \ \underbrace{\sum_{|J'| \leq k - 2} |\nabla\nabla_{\mathcal{Z}}^{J'} h^{(1)}|}_{\mbox{absent if $k \leq 1$}}, 	
		\notag 
		\end{align}
	where $\chi_0(1/2 \leq z \leq 3/4)$ is the characteristic function of the interval $[1/2,3/4].$
		
\end{lemma}

\begin{proof}
	Lemma \ref{L:NablaZIh1LLh1TLMainTheoremWaveCoordianteAlgebraicEstimate} follows 
	from Lemma \ref{L:NablaZIh1LLh1LTwaveCoordinateAlgebraicEstimate} (for the tensorfield $h_{\mu \nu}^{(1)}$) 
	and the pointwise decay assumptions \eqref{E:MainTheoremAssumptionStrongLinfinityDecay}
	for $h_{\mu \nu}^{(1)}.$
\end{proof}


\begin{lemma} \label{L:MainTheoremEnergyInhomogeneousTermAlgebraicEstimate}
	\textbf{(Pointwise estimates for $|(\Lie_{\mathcal{Z}}^I \Far_{0 \nu})\Liemod_{\mathcal{Z}}^I \mathfrak{F}_{(I)}^{\nu}|$)}
		Let $\mathfrak{F}_{(I)}^{\nu} = \Liemod_{\mathcal{Z}}^I 
		\mathfrak{F}^{\nu} + \Big\lbrace N^{\# \mu \nu \kappa \lambda}\nabla_{\mu} \Lie_{\mathcal{Z}}^I\Far_{\kappa \lambda} 
			- \Liemod_{\mathcal{Z}}^I \big(N^{\# \mu \nu \kappa \lambda}\nabla_{\mu}\Far_{\kappa \lambda}\big)\Big\rbrace$ be the 
		inhomogeneous term \eqref{E:LiemodZIdifferentiatedEOVInhomogeneousterms} in the equations of variation \eqref{E:EOVdMis0} 
		satisfied by $\dot{\Far} \eqdef \Lie_{\mathcal{Z}}^I \Far.$ Under the assumptions of Theorem \ref{T:ImprovedDecay}, 
		if $0 \leq k \leq \dParameter $ and $\varepsilon$ is sufficiently small, then the following pointwise inequality holds for 
		$(t,x) \in [0,T) \times \mathbb{R}^3:$
	
	\begin{align} \label{E:MainTheoremEnergyInhomogeneousTermAlgebraicEstimate}
		\sum_{|I| \leq k} |(\Lie_{\mathcal{Z}}^I \Far_{0 \nu})\Liemod_{\mathcal{Z}}^I \mathfrak{F}_{(I)}^{\nu} |
			& \lesssim \varepsilon \sum_{|I| \leq k} (1 + t + |q|)^{-1} |\Lie_{\mathcal{Z}}^I \Far|^2
				\ + \ \varepsilon \sum_{|I| \leq k} (1 + t + |q|)^{-1} |\nabla\nabla_{\mathcal{Z}}^I h^{(1)}|^2 \\
			& \ \ + \ \varepsilon \sum_{|I| \leq k} (1 + |q|)^{-1}(1 + q_-)^{-2\upmu} |\conenabla \nabla_{\mathcal{Z}}^I h^{(1)}|^2 
				\notag \\
			& \ \ + \ \varepsilon \sum_{|I| \leq |k|} (1 + |q|)^{-1}(1 + q_-)^{-2\upmu} 
				\big(|\Lie_{\mathcal{Z}}^I  \Far|_{\mathcal{L} \mathcal{N}}^2 + |\Lie_{\mathcal{Z}}^I  \Far|_{\mathcal{T} 
				\mathcal{T}}^2 \big). \notag
		\end{align}
	
\end{lemma}

\begin{proof}

We first recall the decomposition \eqref{E:LiemodZIdifferentiatedEOVInhomogeneousterms} of $\mathfrak{F}_{(I)}^{\nu}:$

\begin{align}
	 \mathfrak{F}_{(I)}^{\nu} & = \Liemod_{\mathcal{Z}}^I \mathfrak{F}^{\nu}
			\ + \ \Big\lbrace N^{\# \mu \nu \kappa \lambda}\nabla_{\mu} \Lie_{\mathcal{Z}}^I\Far_{\kappa \lambda}
				- \Liemod_{\mathcal{Z}}^I \big(N^{\# \mu \nu \kappa \lambda}\nabla_{\mu}\Far_{\kappa \lambda}\big)\Big\rbrace,
				&& (\nu = 0,1,2,3). \label{E:LiemodZIdifferentiatedEOVInhomogeneoustermsagain}
\end{align}

Now using \eqref{E:LieZIFarNullFormInhomogeneousTermAlgebraicEstimate} with $X_{\nu} \eqdef \Lie_{\mathcal{Z}}^I \Far_{0 \nu},$
together with the decomposition $h = h^{(0)} + h^{(1)}$ and the $h^{(0)}-$decay estimates of Lemma \ref{L:h0decayestimates}, it follows that

\begin{align} \label{E:MainTheoremEnergyInhomogeneousTermPreliminaryAlgebraicEstimate}	
	\sum_{|I| \leq k} |(\Lie_{\mathcal{Z}}^I \Far_{0 \nu}) \Liemod_{\mathcal{Z}}^I \mathfrak{F}^{\nu}|
	& \lesssim \mathop{\sum_{|I| \leq k}}_{|I_1| + |I_2| \leq |I|} 
		|\Lie_{\mathcal{Z}}^I \Far| |\conenabla \nabla_{\mathcal{Z}}^{I_1} h^{(1)}|
		|\Lie_{\mathcal{Z}}^{I_2}\Far| \\
	& \ \ + \ \mathop{\sum_{|I| \leq k}}_{|I_1| + |I_2| \leq |I|} 
		|\Lie_{\mathcal{Z}}^I \Far| 
		|\nabla\nabla_{\mathcal{Z}}^{I_1}h^{(1)}|
		\big(|\Lie_{\mathcal{Z}}^{I_2}\Far|_{\mathcal{L}\mathcal{N}} + |\Lie_{\mathcal{Z}}^{I_2}\Far|_{\mathcal{T}\mathcal{T}} \big)\notag \\
	& \ \ + \ \mathop{\sum_{|I| \leq k}}_{|I_1| + |I_2| + |I_3| \leq |I|}
		|\Lie_{\mathcal{Z}}^I \Far|
		|\nabla\nabla_{\mathcal{Z}}^{I_1}h^{(1)}| 
		|\Lie_{\mathcal{Z}}^{I_2}\Far| 
		|\Lie_{\mathcal{Z}}^{I_3}\Far| \notag \\
	& \ \ + \ \mathop{\sum_{|I| \leq k}}_{|I_1| + |I_2| + |I_3| \leq |I|} 
		|\Lie_{\mathcal{Z}}^I \Far|
		|\nabla_{\mathcal{Z}}^{I_1}h^{(1)}| 
		|\nabla\nabla_{\mathcal{Z}}^{I_2}h^{(1)}| 
		|\Lie_{\mathcal{Z}}^{I_3}\Far| \notag \\
	& \ \ + \ \varepsilon \sum_{|I| \leq k} (1 + t + |q|)^{-1} |\nabla\nabla_{\mathcal{Z}}^{I_1}h^{(1)}|^2 \notag \\
	& \ \ + \ \varepsilon \sum_{|I| \leq k} (1 + t + |q|)^{-1} (1 + |q|)^{-2} |\nabla_{\mathcal{Z}}^{I_1}h^{(1)}|^2 \notag \\
	& \ \ + \ \varepsilon \sum_{|I| \leq k} (1 + t + |q|)^{-1} |\Lie_{\mathcal{Z}}^I \Far|^2. \notag
\end{align}

Inequality \eqref{E:MainTheoremEnergyInhomogeneousTermAlgebraicEstimate} now follows from 
the assumptions of Theorem \ref{T:ImprovedDecay}, \eqref{E:MainTheoremEnergyInhomogeneousTermPreliminaryAlgebraicEstimate}, 
and repeated application of algebraic inequalities of the form $|ab| \lesssim \varsigma a^2 + \varsigma^{-1} b^2.$ As an example, 
we consider the term $|\Lie_{\mathcal{Z}}^I \Far| |\nabla\nabla_{\mathcal{Z}}^{I_1}h^{(1)}|
|\Lie_{\mathcal{Z}}^{I_2}\Far|_{\mathcal{L}\mathcal{N}}$ in the case that $|I_1| \leq |I| \leq \lceil \dParameter/2 \rceil$ 
(such an inequality must be satisfied by either $|I_1|$ or $|I_2|$).
Then with the help of \eqref{E:MainTheoremAssumptionStrongLinfinityDecay} 
and the fact that $\upmu + \upmu' < 1/2,$ it follows that if $\varepsilon$ is sufficiently small, then

\begin{align}
	|\Lie_{\mathcal{Z}}^I \Far| |\nabla\nabla_{\mathcal{Z}}^{I_1}h^{(1)}|
	|\Lie_{\mathcal{Z}}^{I_2} \Far|_{\mathcal{L}\mathcal{N}}
	& \lesssim \varepsilon (1 + t + |q|)^{-1} |\Lie_{\mathcal{Z}}^I \Far|^2 
		\ + \ \varepsilon^{-1} (1 + t + |q|) |\nabla\nabla_{\mathcal{Z}}^{I_1}h^{(1)}|^2
		|\Lie_{\mathcal{Z}}^{I_2}\Far|_{\mathcal{L}\mathcal{N}}^2 \\
	& \lesssim \varepsilon(1 + t + |q|)^{-1} |\Lie_{\mathcal{Z}}^I \Far|^2  
		\ + \ \varepsilon(1 + |q|)^{-1} (1 + q_-)^{-2\upmu} |\Lie_{\mathcal{Z}}^{I_2}\Far|_{\mathcal{L}\mathcal{N}}^2. \notag 
\end{align}
We now observe that the right-hand side of the above inequality is manifestly bounded by the right-hand side of \eqref{E:MainTheoremEnergyInhomogeneousTermAlgebraicEstimate}.

\end{proof}

\begin{lemma} \label{L:MainTheoremFarInhomogeneousTermIntegralEstimate}
	\textbf{(Integral estimates for $\big|(\Lie_{\mathcal{Z}}^I \Far_{0 \nu})\Liemod_{\mathcal{Z}}^I \mathfrak{F}^{\nu} \big| 
	w(q)$)} 
	Under the assumptions of Lemma \ref{L:MainTheoremEnergyInhomogeneousTermAlgebraicEstimate},
	if $0 \leq k \leq \dParameter $ and
	$\varepsilon$ is sufficiently small, then the following integral inequality holds for
	$t \in [0,T):$

	\begin{align} \label{E:MainTheoremFarInhomogeneousTermIntegralEstimate}
		\sum_{|I| \leq k} \int_0^t \int_{\Sigma_{\tau}}
		\big|(\Lie_{\mathcal{Z}}^I \Far_{0 \nu})\Liemod_{\mathcal{Z}}^I \mathfrak{F}^{\nu} \big| w(q) \, d^3x \, d \tau 
		& \lesssim \varepsilon \sum_{|I| \leq k} 
			\int_0^t \int_{\Sigma_{\tau}} (1 + \tau)^{-1} |\Lie_{\mathcal{Z}}^I \Far|^2 w(q) \, d^3x \, d \tau \\
		& \ \ + \ \varepsilon \sum_{|I| \leq k} \int_0^t \int_{\Sigma_{\tau}} (1 + \tau)^{-1} 
			|\nabla\nabla_{\mathcal{Z}}^I h^{(1)}|^2 w(q) \, d^3x \, d \tau \notag \\
		& \ \ + \ \varepsilon \sum_{|I| \leq k} \int_0^t \int_{\Sigma_{\tau}} 
			|\conenabla \nabla_{\mathcal{Z}}^I h^{(1)}|^2 w'(q) \, d^3x \, d \tau \notag \\
		& \ \ + \ \varepsilon \sum_{|I| \leq k} \int_0^t \int_{\Sigma_{\tau}} 
			\big(|\Lie_{\mathcal{Z}}^I \Far|_{\mathcal{L} \mathcal{N}}^2 
			+ |\Lie_{\mathcal{Z}}^I \Far|_{\mathcal{T} \mathcal{T}}^2  \big) w'(q) \, d^3x \, d \tau.   \notag
	\end{align}

\end{lemma}

\begin{proof}
	Inequality \eqref{E:MainTheoremFarInhomogeneousTermIntegralEstimate}
	follows from multiplying inequality \eqref{E:MainTheoremEnergyInhomogeneousTermAlgebraicEstimate} by $w(q),$
	integrating, and using the fact that \\
	$(1 + |q|)^{-1}(1 + q_-)^{-2\upmu} w(q) \lesssim w'(q).$
\end{proof}

\begin{lemma} \label{L:MainTheoremFarCommutatorAlgebraicEstimate}
	\textbf{(Pointwise estimates for $\Big|(\Lie_{\mathcal{Z}}^I \Far_{0 \nu}) 
		\Big\lbrace N^{\# \mu \nu \kappa \lambda} \nabla_{\mu} \Lie_{\mathcal{Z}}^I \Far_{\kappa \lambda}
		- \Liemod_{\mathcal{Z}}^I \big(N^{\# \mu \nu \kappa \lambda}\nabla_{\mu}\Far_{\kappa \lambda}\big) \Big\rbrace \Big|$)}
	
	Let $N^{\# \mu \nu \kappa \lambda} \nabla_{\mu} \Lie_{\mathcal{Z}}^I \Far_{\kappa \lambda}
	- \Liemod_{\mathcal{Z}}^I \big(N^{\# \mu \nu \kappa \lambda}\nabla_{\mu}\Far_{\kappa \lambda}\big)$ be the inhomogeneous
	commutator term \eqref{E:LiemodZINnablaFarCommutatorTerms} in the equations of 
	variation \eqref{E:EOVdMis0} satisfied by $\dot{\Far}_{\mu \nu} \eqdef \Lie_{\mathcal{Z}} \Far_{\mu \nu}.$ Under the 
	assumptions of Theorem \ref{T:ImprovedDecay}, if $1 \leq k \leq \dParameter $ and $\varepsilon$ is sufficiently small, then the 
	following pointwise inequality holds for $(t,x) \in [0,T) \times \mathbb{R}^3:$
	
	\begin{align} \label{E:MainTheoremFarCommutatorAlgebraicEstimate}
		\sum_{|I| \leq k} \Big|(\Lie_{\mathcal{Z}}^I \Far_{0 \nu}) 
		& \Big\lbrace N^{\# \mu \nu \kappa \lambda} \nabla_{\mu} \Lie_{\mathcal{Z}}^I \Far_{\kappa \lambda}
			- \Liemod_{\mathcal{Z}}^I \big(N^{\# \mu \nu \kappa \lambda}\nabla_{\mu}\Far_{\kappa \lambda}\big) \Big\rbrace \Big|	\\
		& \lesssim \varepsilon (1 + t + |q|)^{-1} \sum_{|I| \leq |k|}|\Lie_{\mathcal{Z}}^I \Far|^2 \notag \\
		& \ \ + \ \varepsilon (1 + t + |q|)^{-1}(1 + |q|)^{-2} \sum_{|I| \leq |k|} |\nabla_{\mathcal{Z}}^I h^{(1)}|^2 \notag \\
		& \ \ + \ \varepsilon (1 + |q|)^{-1}(1+ q_-)^{-2 \mu} 
			\sum_{|I| \leq |k|} \big( |\Lie_{\mathcal{Z}}^I \Far|_{\mathcal{L}\mathcal{N}}^2
  		+ |\Lie_{\mathcal{Z}}^I \Far|_{\mathcal{T}\mathcal{T}}^2 \big) \notag \\
  	& \ \ + \ \varepsilon (1 + t + |q|)^{-1 + C \varepsilon} (1 + |q|)^{-(2 + C \varepsilon)} (1+ q_-)^{-2 \upmu}
  		\sum_{|I| \leq k} |\nabla_{\mathcal{Z}}^I h^{(1)}|_{\mathcal{L}\mathcal{L}}^2  \notag \\
  	& \ \ + \ \varepsilon (1 + t + |q|)^{-1 + C \varepsilon} (1 + |q|)^{-(2 + C \varepsilon)} (1+ q_-)^{-2 \upmu}
  		\sum_{|J| \leq k - 1}  
  		|\nabla_{\mathcal{Z}}^J h^{(1)}|_{\mathcal{L}\mathcal{T}}^2  \notag \\
  	& \ \ + \ \varepsilon (1 + t + |q|)^{-1 + C \varepsilon} (1 + |q|)^{-2}
  		\underbrace{\sum_{|J'| \leq k - 2}	|\nabla_{\mathcal{Z}}^{J'} h^{(1)}|^2}_{\mbox{absent if $k = 1$}} \notag \\
  	& \ \	+ \ \varepsilon (1 + t + |q|)^{-1 + C \varepsilon} \sum_{|J| \leq k - 1} |\Lie_{\mathcal{Z}}^J \Far|^2. \notag
	\end{align}
	
\end{lemma}

\begin{proof}
	Using inequality \eqref{E:EnergyInhomogeneousTermAlgebraicEstimate} with $X_{\nu} \eqdef \Lie_{\mathcal{Z}}^I \Far_{0 \nu},$
	the pointwise decay assumptions of Theorem \ref{T:ImprovedDecay}, together with the decomposition $h = h^{(0)} + h^{(1)}$ and 
	the $h^{(0)}$ decay estimates of Lemma \ref{L:h0decayestimates}, it follows that
	
	\begin{align} \label{E:MainTheoremEnergyInhomogeneousTermAlgebraicEstimateMultiplied}
		\sum_{|I| \leq k} 
		\Big|\Lie_{\mathcal{Z}}^I \Far_{0 \nu} & \Big\lbrace N^{\# \mu \nu \kappa 
			\lambda}\nabla_{\mu} \Lie_{\mathcal{Z}}^I\Far_{\kappa \lambda}
			- \Liemod_{\mathcal{Z}}^I \big(N^{\# \mu \nu \kappa \lambda}\nabla_{\mu}\Far_{\kappa \lambda}\big) \Big\rbrace \Big|	\\
		& \lesssim (1 + |q|)^{-1} \mathop{\sum_{|I| \leq k, |I'| \leq k}}_{|J| \leq 1} 
				|\Lie_{\mathcal{Z}}^I \Far| |\nabla_{\mathcal{Z}}^{I'} h^{(1)}|_{\mathcal{L}\mathcal{L}} |\Lie_{\mathcal{Z}}^J \Far| 
			\ \ + \ (1 + |q|)^{-1} \mathop{\sum_{|I| \leq k, |I'| \leq k}}_{|J| \leq 1} 
			|\Lie_{\mathcal{Z}}^I \Far| |\nabla_{\mathcal{Z}}^J h^{(1)}|_{\mathcal{L}\mathcal{L}} |\Lie_{\mathcal{Z}}^{I'} \Far| 
			\notag \\
		& \ \ + \ (1 + |q|)^{-1} \sum_{|I| \leq k} |\Lie_{\mathcal{Z}}^I \Far|^2 |h|_{\mathcal{L}\mathcal{T}} 
			\ + \ (1 + |q|)^{-1} \mathop{\mathop{\sum_{|I|\leq k}}_{|I_1| + |I_2| \leq k + 1}}_{|I_1|, |I_2| \leq k}
			|\Lie_{\mathcal{Z}}^I \Far| |\nabla_{\mathcal{Z}}^{I_1} h^{(1)}| 
			\big(|\Lie_{\mathcal{Z}}^{I_2} \Far|_{\mathcal{L} \mathcal{N}} + |\Lie_{\mathcal{Z}}^{I_2} \Far|_{\mathcal{T} 	
			\mathcal{T}} \big) \notag \\
		& \ \ + \ \varepsilon(1 + t + |q|)^{-1} \sum_{|I|\leq k} |\Lie_{\mathcal{Z}}^I \Far|^2 
			\ + \ \varepsilon(1 + t + |q|)^{-1} (1 + |q|)^{-2} \sum_{|I|\leq k} |\nabla_{\mathcal{Z}}^I h^{(1)}|^2 \notag \\ 
		& \ \ + \ (1 + t + |q|)^{-1} \mathop{\mathop{\sum_{|I|\leq k}}_{|I_1| + |I_2| \leq k + 1}}_{|I_1|, |I_2| \leq k}
			|\Lie_{\mathcal{Z}}^I \Far| |\nabla_{\mathcal{Z}}^{I_1} h^{(1)}| |\Lie_{\mathcal{Z}}^{I_2} \Far| 
			\ + \ (1 + |q|)^{-1} \mathop{\mathop{\sum_{|I|\leq k}}_{|I_1| + |I_2| \leq k + 1}}_{|I_1| \leq k-1, |I_2| \leq k - 1}
		 		|\Lie_{\mathcal{Z}}^I \Far||\nabla_{\mathcal{Z}}^{I_1} h^{(1)}|_{\mathcal{L}\mathcal{L}} 
		 		|\Lie_{\mathcal{Z}}^{I_2} \Far| \notag \\
		& \ \ + \ (1 + |q|)^{-1} \mathop{\mathop{\sum_{|I|\leq k}}_{|I_1| + |I_2| \leq k}}_{|I_1| \leq k - 1, |I_2| \leq k - 1}
		 	|\Lie_{\mathcal{Z}}^I \Far||\nabla_{\mathcal{Z}}^{I_1} h^{(1)}|_{\mathcal{L}\mathcal{T}} |\Lie_{\mathcal{Z}}^{I_2} \Far| 
		 \ \ + \ (1 + |q|)^{-1} \underbrace{\mathop{\mathop{\sum_{|I|\leq k}}_{|I_1| + |I_2| \leq k - 1}}_{|I_1| \leq k - 2, |I_2| 
			\leq k - 1} |\Lie_{\mathcal{Z}}^I \Far||\nabla_{\mathcal{Z}}^{I_1} h^{(1)}| 
			|\Lie_{\mathcal{Z}}^{I_2} \Far|}_{\mbox{absent if $k = 1$}} \notag \\
		& \ \ + \ (1 + |q|)^{-1} \mathop{\mathop{\sum_{|I|\leq k}}_{|I_1| + |I_2| + |I_3| \leq k + 1}}_{|I_1|, |I_2|, |I_3| \leq k}
		 	|\Lie_{\mathcal{Z}}^I \Far| |\nabla_{\mathcal{Z}}^{I_1} h^{(1)}| |\nabla_{\mathcal{Z}}^{I_2} h^{(1)}| 
		 	|\Lie_{\mathcal{Z}}^{I_3}\Far| \notag \\
		 & \ \ + \ (1 + |q|)^{-1} \mathop{\mathop{\sum_{|I|\leq k}}_{|I_1| + |I_2| + |I_3| \leq k + 1}}_{|I_1|, |I_2|, |I_3| \leq k}
		 	|\Lie_{\mathcal{Z}}^I \Far| |\nabla_{\mathcal{Z}}^{I_1} h^{(1)}| |\Lie_{\mathcal{Z}}^{I_2} \Far| 
		 		|\Lie_{\mathcal{Z}}^{I_3}\Far| 
		 	\	+ \ (1 + |q|)^{-1} \mathop{\mathop{\sum_{|I|\leq k}}_{|I_1| + |I_2| + |I_3| \leq k + 1}}_{|I_1|, |I_2|, |I_3| \leq k}
				|\Lie_{\mathcal{Z}}^I \Far| |\Lie_{\mathcal{Z}}^{I_1} \Far| |\Lie_{\mathcal{Z}}^{I_2} \Far| 
				|\Lie_{\mathcal{Z}}^{I_3}\Far|. \notag
	\end{align}
	We remark that the $\varepsilon (1 + t + |q|)^{-1} \sum_{|I|\leq k} |\Lie_{\mathcal{Z}}^I \Far|^2$ and 
	$\varepsilon (1 + t + |q|)^{-1} (1 + |q|)^{-2} \sum_{|I|\leq k} |\nabla_{\mathcal{Z}}^I h^{(1)}|^2$ sums on the right-hand 
	side of \eqref{E:MainTheoremEnergyInhomogeneousTermAlgebraicEstimateMultiplied} account for all of the terms containing a 
	factor $\nabla_{\mathcal{Z}}^J h^{(0)}$ for some $J.$
	Inequality \eqref{E:MainTheoremFarCommutatorAlgebraicEstimate} now follows from 
	\eqref{E:MainTheoremEnergyInhomogeneousTermAlgebraicEstimateMultiplied}, the pointwise decay assumptions of Theorem 
	\ref{T:ImprovedDecay} (including the implied estimates for $h^{(1)}$), and simple algebraic estimates of the form $|ab| \lesssim \varsigma a^2 + 
	\varsigma^{-1} b^2$ (as in the proof of \eqref{E:MainTheoremEnergyInhomogeneousTermPreliminaryAlgebraicEstimate}).
	
\end{proof}

\begin{lemma} \label{L:MainTheoremFarCommutatorIntegralEstimate}
	\textbf{(Integral estimates for $\Big|(\Lie_{\mathcal{Z}}^I \Far_{0 \nu}) 
		\Big\lbrace N^{\# \mu \nu \kappa \lambda} \nabla_{\mu} \Lie_{\mathcal{Z}}^I \Far_{\kappa \lambda}
		- \Liemod_{\mathcal{Z}}^I \big(N^{\# \mu \nu \kappa \lambda}\nabla_{\mu}\Far_{\kappa \lambda}\big) \Big\rbrace \Big|$)}
	Under the assumptions of Lemma \ref{L:MainTheoremEnergyInhomogeneousTermAlgebraicEstimate}, if 
	$1 \leq k \leq \dParameter $ and $\varepsilon$ is sufficiently small,
	then the following integral inequality holds for $t \in [0,T):$ 
	
	\begin{align} \label{E:MainTheoremFarCommutatorIntegralEstimate}
		\sum_{|I| \leq k} \int_0^t \int_{\Sigma_{\tau}} \Big|(\Lie_{\mathcal{Z}}^I \Far_{0 \nu}) 
		& \Big\lbrace N^{\# \mu \nu \kappa \lambda}\nabla_{\mu} \Lie_{\mathcal{Z}}^I\Far_{\kappa \lambda}
			- \Liemod_{\mathcal{Z}}^I \big(N^{\# \mu \nu \kappa \lambda}\nabla_{\mu}\Far_{\kappa \lambda}\big) \Big\rbrace \Big|
			w(q) \, d^3x \, d \tau \\
		& \lesssim \varepsilon \sum_{|I| \leq k} \int_0^t \int_{\Sigma_{\tau}} 
				(1 + \tau)^{-1} |\Lie_{\mathcal{Z}}^I \Far|^2 w(q) \, d^3x \, d \tau 
			\ + \ \varepsilon \sum_{|I| \leq k} \int_0^t \int_{\Sigma_{\tau}} 
				(1 + \tau)^{-1} |\nabla\nabla_{\mathcal{Z}}^I h^{(1)}|^2 w(q) \, d^3x \, d \tau
				\notag \\
		& \ \ + \ \varepsilon \sum_{|I| \leq k} \int_0^t \int_{\Sigma_{\tau}} 
			\Big\lbrace |\conenabla \nabla_{\mathcal{Z}}^I h^{(1)}|^2
				+ |\Lie_{\mathcal{Z}}^I \Far|_{\mathcal{L}\mathcal{N}}^2
  			+ |\Lie_{\mathcal{Z}}^I \Far|_{\mathcal{T}\mathcal{T}}^2 \Big\rbrace w'(q) \, d^3x \, d \tau \notag \\
  	& \ \ + \ \underbrace{\varepsilon \sum_{|J'| \leq k - 2} \int_0^t \int_{\Sigma_{\tau}} 
				(1 + \tau + |q|)^{-1 + C \varepsilon} |\nabla\nabla_{\mathcal{Z}}^{J'} h^{(1)}|^2 w(q) d^3x \, d \tau}_{\mbox{absent if $k = 1$}} 
				 \notag \\
		& \ \ + \ \varepsilon \sum_{|J| \leq k - 1} \int_0^t \int_{\Sigma_{\tau}} 
			(1 + \tau + |q|)^{-1 + C \varepsilon} |\Lie_{\mathcal{Z}}^J \Far|^2 w(q) d^3x \, d \tau 
			\ + \ \varepsilon^2. \notag
	\end{align}
\end{lemma}

\begin{proof}
	We begin by multiplying by both sides of \eqref{E:MainTheoremFarCommutatorAlgebraicEstimate} by $w(q)$ and integrating
	$\int_0^t \int_{\Sigma_{\tau}} d^3x \, d \tau.$ The integrals corresponding to the first and last sums on the right-hand 
	side of \eqref{E:MainTheoremFarCommutatorAlgebraicEstimate} are manifestly bounded by the first 
	and penultimate sums on the right-hand side of \eqref{E:MainTheoremFarCommutatorIntegralEstimate}. Using also the fact that 
	$(1 + |q|)^{-1}(1+ q_-)^{-2 \upmu} w(q) \lesssim w'(q),$ the integral corresponding to the
	third sum on the on the right-hand side of \eqref{E:MainTheoremFarCommutatorAlgebraicEstimate} is bounded by the third sum 
	on the right-hand side of \eqref{E:MainTheoremFarCommutatorIntegralEstimate}.
	
	To bound the integral corresponding to the second sum on the right-hand side of 
	\eqref{E:MainTheoremFarCommutatorAlgebraicEstimate}, we simply use the Hardy inequalities of Proposition \ref{P:Hardy}
	to derive the inequality
	
	\begin{align}
		\sum_{|I| \leq k} & \int_0^t \int_{\Sigma_{\tau}} 
			(1 + \tau + |q|)^{-1}(1 + |q|)^{-2} |\nabla_{\mathcal{Z}}^I h^{(1)}|^2 w(q) d^3x \, d \tau \\
		& \lesssim \sum_{|I| \leq k} \int_0^t \int_{\Sigma_{\tau}} 
			(1 + \tau + |q|)^{-1} |\nabla\nabla_{\mathcal{Z}}^I h^{(1)}|^2 w(q) d^3x \, d \tau. \notag
	\end{align}
	After multiplying by $\varepsilon,$ the right-hand side of the above inequality 
	is manifestly bounded by the right-hand side of \eqref{E:MainTheoremFarCommutatorIntegralEstimate}.
	Using the same reasoning, we obtain the following bound for the integral 
	corresponding to the sixth sum on the right-hand side 
	of \eqref{E:MainTheoremFarCommutatorAlgebraicEstimate}:
	
	\begin{align} \label{E:Hardybound}
		\sum_{|J'| \leq k - 2} & \int_0^t \int_{\Sigma_{\tau}} 
			(1 + \tau + |q|)^{-1 + C \varepsilon}(1 + |q|)^{-2} |\nabla_{\mathcal{Z}}^{J'} h^{(1)}|^2 w(q) d^3x \, d \tau \\
		& \lesssim \sum_{|J'| \leq k - 2} \int_0^t \int_{\Sigma_{\tau}} 
			(1 + \tau + |q|)^{-1 + C \varepsilon} |\nabla\nabla_{\mathcal{Z}}^{J'} h^{(1)}|^2 w(q) d^3x \, d \tau. \notag
	\end{align}
	We then multiply \eqref{E:Hardybound} by $\varepsilon$ and observe that the right-hand side of the resulting inequality 
	is manifestly bounded by the right-hand side of \eqref{E:MainTheoremFarCommutatorIntegralEstimate}.

	To address the fourth and fifth sums on the right-hand side of \eqref{E:MainTheoremFarCommutatorAlgebraicEstimate},
	we will make use of the weight $\widetilde{w}(q),$ which is defined by
	
	\begin{align}
		\widetilde{w}(q) \eqdef \min \big\lbrace w'(q), (1 + t + |q|)^{-1 + C \varepsilon} w(q) \big\rbrace.
	\end{align}
	We note that by \eqref{E:weightinequality}, the following inequality is satisfied:	

	\begin{align}
		\widetilde{w}(q) & \lesssim (1 + |q|)^{-1} w(q). \label{E:widetildewInequality1} 
	\end{align}
	
	With the help of Lemma \ref{L:NablaZIh1LLh1TLMainTheoremWaveCoordianteAlgebraicEstimate},
	\eqref{E:widetildewInequality1}, and the Hardy inequalities of Proposition \ref{P:Hardy}, we estimate the integral 
	corresponding to the fourth sum on the right-hand side of \eqref{E:MainTheoremFarCommutatorAlgebraicEstimate} as follows:
	
	\begin{align} \label{E:MainTheoremFarCommutatorIntegratedEstimate}
		\sum_{|I| \leq k} & \int_0^t \int_{\Sigma_{\tau}} 
  		(1 + t + |q|)^{-1 + C \varepsilon} (1 + |q|)^{-(2 + C \varepsilon)} (1+ q_-)^{-2 \upmu}
  		|\nabla_{\mathcal{Z}}^I h^{(1)}|_{\mathcal{L}\mathcal{L}}^2 w(q) d^3x \, d \tau 	\\
  	& \lesssim \sum_{|I| \leq k} \int_0^t \int_{\Sigma_{\tau}} 
  		|\nabla\nabla_{\mathcal{Z}}^I h^{(1)}|_{\mathcal{L}\mathcal{L}}^2   
  		\widetilde{w}(q) d^3x \, d \tau \notag \\
  	& \lesssim \sum_{|I| \leq k} \int_0^t \int_{\Sigma_{\tau}} 
  		|\conenabla \nabla_{\mathcal{Z}}^I h^{(1)}|^2 w'(q) d^3x \, d \tau \notag \\
  	& \ \ + \ \varepsilon^2 \int_0^t \int_{\Sigma_{\tau}} (1 + \tau + |q|)^{-4} \chi_0^2(1/2 < r/t < 3/4) w'(q) d^3x \, d \tau 	
  		\notag \\
  	& \ \ + \ \varepsilon^4 \int_0^t \int_{\Sigma_{\tau}} (1 + \tau + |q|)^{-6} w'(q) d^3x \, d \tau
  		\ + \ \varepsilon^2 \sum_{|I| \leq k} \int_0^t \int_{\Sigma_{\tau}} 
  		(1 + \tau)^{-1} |\nabla\nabla_{\mathcal{Z}}^I h^{(1)}|^2 w(q) d^3x \, d \tau \notag \\
  	& \ \ + \ \varepsilon^2 \sum_{|I| \leq k} \int_0^t \int_{\Sigma_{\tau}} 
  		(1 + \tau + |q|)^{-1} (1 + |q|)^{-2}|\nabla_{\mathcal{Z}}^I h^{(1)}|^2 w(q) d^3x \, d \tau \notag \\
  	& \ \ + \ \underbrace{\sum_{|J'| \leq k - 2} \int_0^t \int_{\Sigma_{\tau}} 
  		(1 + \tau + |q|)^{-1 + C \varepsilon} |\nabla\nabla_{\mathcal{Z}}^{J'} h^{(1)}|^2 w(q) d^3x \, d \tau}_{\mbox{absent if 
  		$k = 1$}} \notag \\
  	& \lesssim \sum_{|I| \leq k} \int_0^t \int_{\Sigma_{\tau}} 
  		|\conenabla \nabla_{\mathcal{Z}}^I h^{(1)}|^2 w'(q) d^3x \, d \tau
  	 	\ + \ \sum_{|I| \leq k}  \int_0^t \int_{\Sigma_{\tau}} 
  		(1 + \tau)^{-1} |\nabla\nabla_{\mathcal{Z}}^I h^{(1)}|^2 w(q) d^3x \, d \tau \notag \\
  	& \ \ + \ \underbrace{\sum_{|J'| \leq k - 2} \int_0^t \int_{\Sigma_{\tau}} 
  		(1 + \tau)^{-1 + C \varepsilon} |\nabla\nabla_{\mathcal{Z}}^{J'} h^{(1)}|^2 w(q) d^3x \, d \tau}_{
  			\mbox{absent if $k = 1$}}
  		\ + \ \varepsilon^2, \notag
	\end{align}
	where to pass to the lass inequality, we have again used Proposition \ref{P:Hardy}
	to estimate \\
	$\sum_{|I| \leq k} \int_0^t \int_{\Sigma_{\tau}} 
  		(1 + \tau + |q|)^{-1} (1 + |q|)^{-2}|\nabla_{\mathcal{Z}}^I h^{(1)}|^2 w(q) d^3x \, d \tau$
  		$\lesssim \sum_{|I| \leq k}  \int_0^t \int_{\Sigma_{\tau}} 
  		(1 + \tau)^{-1} |\nabla\nabla_{\mathcal{Z}}^I h^{(1)}|^2 w(q) d^3x \, d \tau.$
	After multiplying both sides of \eqref{E:MainTheoremFarCommutatorIntegratedEstimate} by $\varepsilon,$ the resulting
	right-hand side is manifestly bounded by the right-hand side of \eqref{E:MainTheoremFarCommutatorIntegralEstimate}
	as desired. The integral corresponding to the fifth sum on the right-hand side of 
	\eqref{E:MainTheoremFarCommutatorAlgebraicEstimate} can be bounded through the same reasoning.

\end{proof}

Combining Lemma \ref{L:MainTheoremFarInhomogeneousTermIntegralEstimate} and Lemma \ref{L:MainTheoremFarCommutatorIntegralEstimate}, we arrive at the following corollary.

\begin{corollary} \label{C:LieZIFarFundamentalEnergyEstimate}
		Let $\mathfrak{F}_{(I)}^{\nu} = \Liemod_{\mathcal{Z}}^I 
		\mathfrak{F}^{\nu} + \Big\lbrace N^{\# \mu \nu \kappa \lambda}\nabla_{\mu} \Lie_{\mathcal{Z}}^I\Far_{\kappa 
		\lambda} - \Liemod_{\mathcal{Z}}^I \big(N^{\# \mu \nu \kappa \lambda}\nabla_{\mu}\Far_{\kappa 
		\lambda}\big)\Big\rbrace$ be the 
		inhomogeneous term \eqref{E:LiemodZINnablaFarCommutatorTerms} in the equations of 
		variation \eqref{E:EOVdMis0} satisfied by $\dot{\Far}_{\mu \nu} \eqdef \Lie_{\mathcal{Z}} \Far_{\mu \nu}.$ Under the 
		assumptions of Theorem \ref{T:ImprovedDecay}, if $0 \leq k \leq \dParameter $ and $\varepsilon$ is sufficiently small, then 
		the following integral inequality holds for $t \in [0,T):$

		\begin{align} \label{E:LieZIFarFundamentalEnergyEstimate}
		\sum_{|I| \leq k} \int_0^t \int_{\Sigma_{\tau}} |(\Lie_{\mathcal{Z}}^I \Far_{0 \nu}) \mathfrak{F}_{(I)}^{\nu} |
			w(q) \, d^3x \, d \tau 
		& \lesssim \varepsilon \sum_{|I| \leq k} \int_0^t \int_{\Sigma_{\tau}} 
			(1 + \tau)^{-1} \Big\lbrace |\nabla\nabla_{\mathcal{Z}}^I h^{(1)}|^2 
			+ |\Lie_{\mathcal{Z}}^I \Far|^2 \Big\rbrace w(q) \, d^3x \, d \tau \\
		& \ \ + \ \varepsilon \sum_{|I| \leq k} \int_0^t \int_{\Sigma_{\tau}} 
			\Big\lbrace |\conenabla \nabla_{\mathcal{Z}}^I h^{(1)}|^2 
			+ |\Lie_{\mathcal{Z}}^I \Far|_{\mathcal{L}\mathcal{N}}^2
  		+ |\Lie_{\mathcal{Z}}^I \Far|_{\mathcal{T}\mathcal{T}}^2 \Big\rbrace w'(q) \, d^3x \, d \tau \notag \\
		& \ \ + \ \underbrace{\varepsilon \sum_{|J| \leq k - 1} \int_0^t \int_{\Sigma_{\tau}} 
			(1 + \tau)^{-1 + C \varepsilon} \Big\lbrace |\nabla\nabla_{\mathcal{Z}}^J h^{(1)}|^2 
			+ |\Lie_{\mathcal{Z}}^J \Far|^2 \Big\rbrace w(q) \, d^3x \, d \tau}_{\mbox{absent if $k = 0$}} \notag \\
		& \ + \ \varepsilon^3. \notag
		\end{align}
		
\end{corollary}

\hfill $\qed$

\section*{Acknowledgments}
I would like to thank Igor Rodnianski for delivering an especially illuminating set of lectures on the work
\cite{hLiR2010} at Princeton University during Spring 2009. I offer thanks to Michael Kiessling for introducing me to his work 
\cite{mK2004a}, \cite{mK2004b} on nonlinear electromagnetism, to Sergiu Klainerman for suggesting that I write the precursor \cite{jS2010a} to the present article, and to A. Shadi Tahvildar-Zadeh for introducing me to the ideas of \cite{dC2000}. I would also like to thank Mihalis Dafermos, Shadi Tahvildar-Zadeh, and Willie Wong for the useful comments and helpful discussion they provided. I am appreciative of the support offered by the University of Cambridge and Princeton University during the writing of this article.

\appendix
\setcounter{section}{0}
   \setcounter{subsection}{0}
   \setcounter{subsubsection}{0}
   \setcounter{paragraph}{0}
   \setcounter{subparagraph}{0}
   \setcounter{figure}{0}
   \setcounter{table}{0}
   \setcounter{equation}{0}
   \setcounter{theorem}{0}
   \setcounter{definition}{0}
   \setcounter{remark}{0}
   \setcounter{proposition}{0}
   \renewcommand{\thesection}{\Alph{section}}
   \renewcommand{\theequation}{\Alph{section}.\arabic{equation}}
   \renewcommand{\theproposition}{\Alph{section}-\arabic{proposition}}
   \renewcommand{\thecorollary}{\Alph{section}.\arabic{corollary}}
   \renewcommand{\thedefinition}{\Alph{section}.\arabic{definition}}
   \renewcommand{\thetheorem}{\Alph{section}.\arabic{theorem}}
   \renewcommand{\theremark}{\Alph{section}.\arabic{remark}}
   \renewcommand{\thelemma}{\Alph{section}-\arabic{lemma}}

\section{Weighted Sobolev-Moser Inequalities} \label{A:SobolevMoser}

The propositions and corollaries stated in this section were used in Section \ref{S:SmallDataAssumptions} to relate the smallness condition on the abstract initial data to a smallness condition on the initial energy of the corresponding solution to the reduced equations. The propositions were essentially proved in \cite{yCBdC1981}, while the corollaries follow from the propositions via standard arguments. Throughout the appendix, we abbreviate $C_{\eta}^{\dParameter} \eqdef C_{\eta}^{\dParameter}(\mathbb{R}^3),$ $H_{\eta}^{\dParameter} \eqdef H_{\eta}^{\dParameter}(\mathbb{R}^3)$ etc. (see Definitions \ref{D:HNdeltanorm} and \ref{D:CNdeltanorm}).

\begin{proposition} \label{P:SobolevEmbeddingHNdeltaCNprimedeltamprime} \cite[Lemma 2.4]{yCBdC1981}
	\textbf{(Weighted Sobolev embedding)}
	Let $\dParameter, \dParameter'$ be integers, and let $\eta, \eta'$ be real numbers subject to the constraints
	$\dParameter ' < \dParameter - 3/2$ and $\eta' < \eta + 3/2.$ Assume that $v(x) \in H_{\eta}^{\dParameter}.$ Then 
	$v \in C_{\eta'}^{\dParameter'},$ and
	
	\begin{align} \label{E:SobolevEmbeddingHNdeltaCNprimedeltamprime}
		\| v \|_{C_{\eta'}^{\dParameter'}} \lesssim \| v \|_{H_{\eta}^{\dParameter}}.
	\end{align}

\end{proposition}

\hfill $\qed$

\begin{proposition} \label{P:WeightedSobolevSpaceMultiplicationProperties} \cite[Lemma 2.5]{yCBdC1981}
	\textbf{(Weighted Sobolev multiplication properties)}
	Let $\dParameter_1, \cdots, \dParameter_p \geq 0$ be integers, and let $\eta_1, \cdots, \eta_p$ be real numbers.
	Suppose that $v_j(x) \in H_{\eta_j}^{\dParameter_j},$ for $j = 1, \cdots, p.$ Assume that the integer $\dParameter$ satisfies
	$0 \leq \dParameter  \leq \min \lbrace \dParameter_1, \cdots, \dParameter_p \rbrace$ and 
	$\dParameter  \leq \sum_{j=1}^p \dParameter_j - (p-1)3/2,$ and that $\eta < \sum_{j=1}^p \eta_j + (p-1)3/2.$ Then
	
	\begin{align}
		\prod_{j=1}^p v_j \in H_{\eta}^{\dParameter},
	\end{align}
	and the multiplication map
	
	\begin{align}
		H_{\eta_1}^{\dParameter_1} \times \cdots \times H_{\eta_p}^{\dParameter_p} & \rightarrow H_{\dParameter}^{\eta}, &&
		(v_1, \cdots, v_p) \rightarrow \prod_{j=1}^p v_j
	\end{align}
	is continuous.

\end{proposition}

\hfill $\qed$

\begin{corollary} \label{C:WeightedSobolevSpaceMultiplicationProperties}
	Let $\dParameter \geq 2$ be an integer, and let $\eta \geq 0.$ Assume that 
	$v_j(x) \in H_{\eta}^{\dParameter}$ for $j = 1, \cdots, p,$ and that $I_1, \cdots, I_p$ are $\unabla-$multi-indices
	satisfying $\sum_{j=1}^p |I_j| \leq \dParameter .$ Then
	
	\begin{align}
		(1 + |x|^2)^{(\eta + \sum_{j=1}^p |I_j|)/2} \prod_{i=1}^p \unabla^{I_i} v_i \in L^2
	\end{align} 
	and
	
	\begin{align}
		\Big\| (1 + |x|^2)^{(\eta + \sum_{j=1}^p |I_j|)/2} \prod_{i=1}^p \unabla^{I_i} v_i \Big\|_{L^2} 
		\lesssim \prod_{i=1}^p \|v_i \|_{H_{\eta}^{\dParameter}}.
	\end{align}
	
\end{corollary}

\hfill $\qed$

\begin{corollary} \label{C:CompositionProductHNdelta}
	Let $\dParameter \geq 2$ be an integer, let $\mathfrak{K}$ be a compact set, and let $F(\cdot) \in C^{\dParameter}(\mathfrak{K})$ 
	be a function. Assume that $v_1(x)$ is a function on $\mathbb{R}^3$ such that $v_1(\mathbb{R}^3) \subset \mathfrak{K}.$ 
	Furthermore, assume that $\underline{\nabla} v_1(x), v_2(x) \in H_{\eta}^{\dParameter}.$ Then $(F \circ v_1(x)) v_2(x) \in 
  H_{\eta}^{\dParameter},$ and 
	
	\begin{align} 
		\| (F \circ v_1) v_2 \|_{H_{\eta}^{\dParameter}} 
			& \lesssim \Big\lbrace \| v_2 \|_{H_{\eta}^{\dParameter}} |F|_{\mathfrak{K}}  
				+ \| (1 + |x|) v_2 \|_{L^{\infty}} \| \unabla v_1 \|_{H_{\eta}^{\dParameter - 1}} \sum_{j=1}^{\dParameter} 
			|F^{(j)}|_{\mathfrak{K}} \| v_1 \|_{L^{\infty}}^{j - 1} \Big\rbrace.
	\end{align}
	In the above inequality, $F^{(j)}$ denotes the array of all $j^{th}$ order partial derivatives of 
	$F$ with respect to its arguments, and $|F^{(j)}|_{\mathfrak{K}} \eqdef \sup_{v \in \mathfrak{K}} 
	|F^{(l)}(v)|.$
\end{corollary}

\hfill $\qed$

\section{Weighted Klainerman-Sobolev Inequalities} \label{A:WeightedKS}

	\setcounter{subsection}{0}
   \setcounter{subsubsection}{0}
   \setcounter{paragraph}{0}
   \setcounter{subparagraph}{0}
   \setcounter{figure}{0}
   \setcounter{table}{0}
   \setcounter{equation}{0}
   \setcounter{theorem}{0}
   \setcounter{definition}{0}
   \setcounter{remark}{0}
   \setcounter{proposition}{0}
		In this section, we recall the weighted Klainerman-Sobolev inequalities that were proved in \cite{hLiR2010}.
		Throughout this section, the weight function $w(q)$ is defined by
		
		\begin{align} \label{E:Appendixweight}
			w \eqdef w(q) \eqdef
			\left \lbrace \begin{array}{lr}
	    	1 \ + \ (1 + |q|)^{1 + 2 \upgamma}, &  \mbox{if} \ q > 0, \\
	      1 \ + \ (1 + |q|)^{- 2 \upmu}, & \mbox{if} \ q < 0.
	    \end{array}
	  	\right.
	  \end{align}
	  In this section, we assume only that $\upgamma,$ $\upmu$ are fixed constants and that
	  $0 < \upgamma < 1.$ It easily follows from \eqref{E:Appendixweight} that
	  
	  \begin{align} \label{E:Appendixweightderivative}
			w' \eqdef w'(q) =
			\left \lbrace \begin{array}{lr}
	    	(1 + 2 \upgamma)(1 + |q|)^{2 \upgamma}, &  \mbox{if} \ q > 0, \\
	      2 \upmu (1 + |q|)^{-1 - 2 \upmu}, & \mbox{if} \ q < 0,
	    \end{array}
	  	\right.,  
	  \end{align}
	  and
	  
	  \begin{align}
	  	w' \leq 4 (1 + |q|)^{-1}w \leq 16 \upgamma^{-1} (1 + q_{-})^{2 \upmu} w'.
	  \end{align}

\begin{proposition}	\cite[Proposition 14.1]{hLiR2010} \label{P:WeightedKlainermanSobolev} 
	\textbf{(Weighted Klainerman-Sobolev inequality)}
	There exists a $C > 0$ such that for all $\phi(t,\cdot) \in C_{0}^{\infty}(\mathbb{R}^3),$ the following
	inequality holds:
	\begin{align} \label{E:PhiKlainermanSobolev}
		(1 + t + |x|)[(1 + |q|) w(q)]^{1/2} |\phi(t,x)| & \leq C \sum_{|I| \leq 2} 
			\big\| w^{1/2} \nabla_{\mathcal{Z}}^I \phi(t, \cdot) \big\|_{L^2}, && q \eqdef |x| - t.
	\end{align}
	
	Furthermore, we have that
	
	\begin{align} \label{E:NablaPhiKlainermanSobolev}
		(1 + t + |x|)[(1 + |q|) w(q)]^{1/2} |\nabla \phi(t,x)| & \leq C \sum_{|I| \leq 2} 
			\big\| w^{1/2} \nabla\nabla_{\mathcal{Z}}^I \phi(t, \cdot) \big\|_{L^2}, && q \eqdef |x| - t.
	\end{align}
	
\end{proposition}

\begin{proof}
	\eqref{E:PhiKlainermanSobolev} was proved as \cite[Proposition 14.1]{hLiR2010}. \eqref{E:NablaPhiKlainermanSobolev} follows 
	from Lemma \ref{L:NablaZICommutesWithCovariantDerivativePlusErrorTerms} and \eqref{E:PhiKlainermanSobolev}.
\end{proof}

\section{Hardy-Type Inequalities} \label{S:Hardy}
   \setcounter{subsection}{0}
   \setcounter{subsubsection}{0}
   \setcounter{paragraph}{0}
   \setcounter{subparagraph}{0}
   \setcounter{figure}{0}
   \setcounter{table}{0}
   \setcounter{equation}{0}
   \setcounter{theorem}{0}
   \setcounter{definition}{0}
   \setcounter{remark}{0}
   \setcounter{proposition}{0}
   
In this section, we recall the weighted Hardy-type inequalities proved in \cite{hLiR2010}.

\begin{proposition} \cite[Corollary 13.3]{hLiR2010} \label{P:Hardy} 
	\textbf{(Hardy inequalities)}
	Let $\upgamma > 0$ and $\upmu > 0,$ $q \eqdef |x| - t,$ and let $w(q)$ and $w'(q)$
	be as defined in \eqref{E:Appendixweight} and \eqref{E:Appendixweightderivative}
	respectively. Then for any $-1 \leq a \leq 1,$ there exists a $C > 0$ such that for all $\phi \in 
	C_0^{\infty}(\mathbb{R}^3),$ we have the following pointwise inequality:
	  
	  \begin{align}
	  	\int_{\mathbb{R}^3} (1 + t + |q|)^{- 1 + a} (1 + |q|)^{-2} |\phi|^2 w(q) \,d^3 x
	  	\leq C \int_{\mathbb{R}^3} (1 + t + |q|)^{- 1 + a} |\partial_r \phi|^2 w(q) \, d^3x,
	  \end{align}
	  where $\partial_r = \omega^b \partial_b,$ $\omega^j \eqdef x^j/r,$ denotes the radial vectorfield.
	  
	  If in addition $a < 2 \min \lbrace \upgamma, \upmu \rbrace,$ then with
	  
	  \begin{align}
	  	\widetilde{w}(q) \eqdef \min\big\lbrace w'(q), (1 + t + |q|)^{-1 + a} w(q) \big\rbrace,
	  \end{align}
	  there exists a constant $C>0$ such that the following pointwise inequality holds:
	  
	  \begin{align}
	  	\int_{\mathbb{R}^3} (1 + t + |q|)^{- 1 + a} (1 + |q|)^{-(a+2)} (1 + q_-)^{-2 \upmu} |\phi|^2 w(q) \, d^3 x
	  	\leq C \int_{\mathbb{R}^3} |\partial_r \phi|^2 \widetilde{w}(q) \, d^3x,
	  \end{align}
	  where $q_- \eqdef |q|$ if $q \leq 0$ and $q_- = 0$ if $q > 0.$
		
\end{proposition}

\hfill $\qed$

\begin{corollary}
	Assume the hypotheses of Proposition \ref{P:Hardy}, and let $P_{\mu \nu}$ be a
	type $\binom{0}{2}$ tensorfield. Let $\mathcal{V}, \mathcal{W}$ be any two of the subsets of null frame-field
	vectors defined in \eqref{E:Framefieldsubsets}. Then the same conclusions of the proposition hold
	if we replace $|\phi|$ and $|\partial_r \phi|$ with the contraction seminorms $|P|_{\mathcal{V}\mathcal{W}}$
	and $|\nabla P|_{\mathcal{V}\mathcal{W}}$ respectively, where the contraction seminorms
	are defined in Definition \ref{D:contractionnomrs}. 

\end{corollary}

\begin{proof}
	Let $\angm$ be the first fundamental form of the $S_{r,t}$ defined in \eqref{E:angmdef}, 
	and recall that the tensor $\angm_{\mu}^{\ \nu}$ projects $m-$orthogonally onto the $S_{r,t}.$
	Since $\partial_r = \frac{1}{2}(L - \uL),$ it follows from \eqref{E:LanduLaregeodesic}, \eqref{E:nablaLuLis0}, and
	\eqref{E:NablaLangmandNablauLangmVanish} that 
	
	\begin{align} 
		\partial_r (\uL^{\kappa} \uL^{\lambda} P_{\kappa \lambda}) 
			& = \frac{1}{2}\uL^{\kappa} \uL^{\lambda}(\nabla_L - \nabla_{\uL})P_{\kappa \lambda},&& 
			\label{E:ParitalrCommutesWithPuLuL} \\
		\partial_r (\uL^{\kappa} L^{\lambda} P_{\kappa \lambda}) 
			& = \frac{1}{2} \uL^{\kappa} L^{\lambda}(\nabla_L - \nabla_{\uL}) P_{\kappa \lambda},&& \\
		\partial_r (L^{\kappa} \uL^{\lambda} P_{\kappa \lambda}) 
			& = \frac{1}{2} L^{\kappa} \uL^{\lambda}(\nabla_L - \nabla_{\uL}) P_{\kappa \lambda},&& \\
		\partial_r (L^{\kappa} L^{\lambda} P_{\kappa \lambda}) 
			& = \frac{1}{2} L^{\kappa} L^{\lambda} (\nabla_L - \nabla_{\uL}) P_{\kappa \lambda},&& (\mu = 0,1,2,3), \\
		\partial_r (\angm_{\mu}^{\ \kappa} \uL^{\lambda} P_{\kappa \lambda}) 
			& = \frac{1}{2}\angm_{\mu}^{\ \kappa} \uL^{\lambda}(\nabla_L - \nabla_{\uL}) P_{\kappa \lambda},&& (\mu = 0,1,2,3), \\
		\partial_r (\uL^{\kappa} \angm_{\mu}^{\ \lambda} P_{\kappa \lambda}) 
			& = \frac{1}{2}\uL^{\kappa} \angm_{\mu}^{\ \lambda}(\nabla_L - \nabla_{\uL}) P_{\kappa \lambda},&& (\mu = 0,1,2,3), \\
		\partial_r (\angm_{\mu}^{\ \kappa} L^{\lambda} P_{\kappa \lambda}) 
			& = \frac{1}{2}\angm_{\mu}^{\ \kappa} L^{\lambda} (\nabla_L - \nabla_{\uL}) P_{\kappa \lambda},&& (\mu = 0,1,2,3), \\
		\partial_r (L^{\kappa} \angm_{\mu}^{\ \lambda} P_{\kappa \lambda}) 
			& = \frac{1}{2} L^{\kappa} \angm_{\mu}^{\ \lambda} (\nabla_L - \nabla_{\uL}) P_{\kappa \lambda},&& (\mu = 0,1,2,3), \\
		\partial_r (\angm_{\mu}^{\ \kappa} \angm_{\nu}^{\ \lambda} P_{\kappa \lambda}) 
			& = \frac{1}{2}\angm_{\mu}^{\ \kappa} \angm_{\nu}^{\ \lambda}(\nabla_L - \nabla_{\uL}) P_{\kappa \lambda},&& 
			(\mu, \nu = 0,1,2,3).
	\end{align}
	That is to say, $\partial_r$ commutes with the null decomposition of $P.$ The conclusion of the corollary now easily follows
	from applying the proposition with $\phi$ equal to the scalar-valued functions
	$\uL^{\kappa} \uL^{\lambda} P_{\kappa \lambda},$ 
	$\uL^{\kappa} L^{\lambda} P_{\kappa \lambda},$ $\cdots,$ $\angm_{\mu}^{\ \kappa} \angm_{\nu}^{\ \lambda} P_{\kappa \lambda}$
	respectively.

\end{proof}

\bibliographystyle{amsalpha}
\bibliography{JBib}

\def\cprime{$'$}
\providecommand{\bysame}{\leavevmode\hbox to3em{\hrulefill}\thinspace}
\providecommand{\MR}{\relax\ifhmode\unskip\space\fi MR }
\providecommand{\MRhref}[2]{%
  \href{http://www.ams.org/mathscinet-getitem?mr=#1}{#2}
}
\providecommand{\href}[2]{#2}
\begin{thebibliography}{CBCL06}

\bibitem[BB83]{iBB1983}
Iwo Bialynicki-Birula, \emph{Nonlinear electrodynamics: Variations on a theme
  by {Born} and {Infeld}}, Quantum Theory of Particles and Fields (1983),
  31--48.

\bibitem[BI34]{mBlI1934}
Max Born and Leopold Infeld, \emph{Foundation of the new field theory}, Proc.
  Roy. Soc. London A (1934), no.~144, 425--451.

\bibitem[Boi69]{gB1969}
Guy Boillat, \emph{Nonlinear electrodynamics: {Lagrangians} and equations of
  motion}, {J. Math. Phys.} \textbf{11} (1969), no.~3, 941--951.

\bibitem[Bor33]{mB1933}
Max Born, \emph{Modified field equations with a finite radius of the electron},
  Nature \textbf{132} (1933), 282.

\bibitem[BZ09]{lBnZ2009}
Lydia Bieri and Nina Zipser (eds.), \emph{Extensions of the stability theorem
  of the {Minkowski} space in general relativity}, American Mathematical
  Society, Providence, RI, 2009.

\bibitem[CB52]{CB1952}
Yvonne~Foures (Choquet)-Bruhat, \emph{{Th\'{e}or\`{e}me d'existence pour
  certains syst\`{e}mes d'\'{e}quations aux d\'{e}riv\'{e}es partielles non
  lin\'{e}aires}}, Acta Mathematica \textbf{88} (1952), 141--225.

\bibitem[CBC81]{yCBdC1981}
Y.~Choquet-Bruhat and D.~Christodoulou, \emph{Elliptic systems in {$H_{s,\delta
  }$} spaces on manifolds which are {E}uclidean at infinity}, Acta Math.
  \textbf{146} (1981), no.~1-2, 129--150. \MR{MR594629 (82c:58060)}

\bibitem[CBCL06]{yCBpCjL2006}
Yvonne Choquet-Bruhat, Piotr~T. Chru{\'s}ciel, and Julien Loizelet,
  \emph{Global solutions of the {E}instein-{M}axwell equations in higher
  dimensions}, Classical Quantum Gravity \textbf{23} (2006), no.~24,
  7383--7394. \MR{2279722 (2008i:83022)}

\bibitem[CD02a]{pCeD2002erratum}
P.~T. Chru{\'s}ciel and E.~Delay, \emph{Erratum: ``{E}xistence of non-trivial,
  vacuum, asymptotically simple spacetimes''}, Classical Quantum Gravity
  \textbf{19} (2002), no.~12, 3389. \MR{1920322 (2003e:83024b)}

\bibitem[CD02b]{pCeD2002}
Piotr~T. Chru{\'s}ciel and Erwann Delay, \emph{Existence of non-trivial,
  vacuum, asymptotically simple spacetimes}, Classical Quantum Gravity
  \textbf{19} (2002), no.~9, L71--L79. \MR{1902228 (2003e:83024a)}

\bibitem[Chr86]{dC1986}
Demetrios Christodoulou, \emph{Global solutions of nonlinear hyperbolic
  equations for small initial data}, Comm. Pure Appl. Math. \textbf{39} (1986),
  no.~2, 267--282. \MR{820070 (87c:35111)}

\bibitem[Chr00]{dC2000}
\bysame, \emph{The action principle and partial differential equations}, Annals
  of Mathematics Studies, vol. 146, Princeton University Press, Princeton, NJ,
  2000. \MR{1739321 (2003a:58001)}

\bibitem[Chr08]{dC2008}
\bysame, \emph{Mathematical problems of general relativity. {I}}, Zurich
  Lectures in Advanced Mathematics, European Mathematical Society (EMS),
  Z\"urich, 2008. \MR{MR2391586 (2008m:83008)}

\bibitem[CK90]{dCsK1990}
Demetrios Christodoulou and Sergiu Klainerman, \emph{Asymptotic properties of
  linear field equations in {M}inkowski space}, Comm. Pure Appl. Math.
  \textbf{43} (1990), no.~2, 137--199. \MR{MR1038141 (91a:58202)}

\bibitem[CK93]{dCsK1993}
\bysame, \emph{The global nonlinear stability of the {M}inkowski space},
  Princeton Mathematical Series, vol.~41, Princeton University Press,
  Princeton, NJ, 1993. \MR{MR1316662 (95k:83006)}

\bibitem[Cor00]{jC2000}
Justin Corvino, \emph{Scalar curvature deformation and a gluing construction
  for the {E}instein constraint equations}, Comm. Math. Phys. \textbf{214}
  (2000), no.~1, 137--189. \MR{1794269 (2002b:53050)}

\bibitem[dD21]{tD1921}
{Th\'{e}ophile} de~Donder, \emph{La gravifique {Einsteinienne}},
  Gauthier-Villars, Paris, 1921.

\bibitem[Fri86]{hF1986a}
Helmut Friedrich, \emph{On the existence of {$n$}-geodesically complete or
  future complete solutions of {E}instein's field equations with smooth
  asymptotic structure}, Comm. Math. Phys. \textbf{107} (1986), no.~4,
  587--609. \MR{MR868737 (88b:83006)}

\bibitem[GH01]{gGcH2001}
G.~W. Gibbons and C.~A.~R. Herdeiro, \emph{Born-{I}nfeld theory and stringy
  causality}, Phys. Rev. D (3) \textbf{63} (2001), no.~6, 064006, 18.
  \MR{1831560 (2002h:81200)}

\bibitem[H{\"o}r97]{lH1997}
Lars H{\"o}rmander, \emph{Lectures on nonlinear hyperbolic differential
  equations}, Math\'ematiques \& Applications (Berlin) [Mathematics \&
  Applications], vol.~26, Springer-Verlag, Berlin, 1997. \MR{MR1466700
  (98e:35103)}

\bibitem[Joh81]{fJ1981}
Fritz John, \emph{Blow-up for quasilinear wave equations in three space
  dimensions}, Comm. Pure Appl. Math. \textbf{34} (1981), no.~1, 29--51.
  \MR{600571 (83d:35096)}

\bibitem[Kat05]{sK2005}
Soichiro Katayama, \emph{Global existence for systems of wave equations with
  nonresonant nonlinearities and null forms}, J. Differential Equations
  \textbf{209} (2005), no.~1, 140--171. \MR{MR2107471 (2006b:35208)}

\bibitem[Kie04a]{mK2004a}
Michael K.-H. Kiessling, \emph{Electromagnetic field theory without divergence
  problems. {I}. {T}he {B}orn legacy}, J. Statist. Phys. \textbf{116} (2004),
  no.~1-4, 1057--1122. \MR{2082203 (2005h:81003a)}

\bibitem[Kie04b]{mK2004b}
\bysame, \emph{Electromagnetic field theory without divergence problems. {II}.
  {A} least invasively quantized theory}, J. Statist. Phys. \textbf{116}
  (2004), no.~1-4, 1123--1159. \MR{2082204 (2005h:81003b)}

\bibitem[Kla86]{sK1986}
Sergiu Klainerman, \emph{The null condition and global existence to nonlinear
  wave equations}, Nonlinear systems of partial differential equations in
  applied mathematics, {P}art 1 ({S}anta {F}e, {N}.{M}., 1984), Lectures in
  Appl. Math., vol.~23, Amer. Math. Soc., Providence, RI, 1986, pp.~293--326.
  \MR{837683 (87h:35217)}

\bibitem[KN03]{sKfN2003}
Sergiu Klainerman and Francesco Nicol{\`o}, \emph{The evolution problem in
  general relativity}, Progress in Mathematical Physics, vol.~25, Birkh\"auser
  Boston Inc., Boston, MA, 2003. \MR{MR1946854 (2004f:58036)}

\bibitem[KS96]{sKtS1996}
Sergiu Klainerman and Thomas~C. Sideris, \emph{On almost global existence for
  nonrelativistic wave equations in {$3$}{D}}, Comm. Pure Appl. Math.
  \textbf{49} (1996), no.~3, 307--321. \MR{MR1374174 (96m:35231)}

\bibitem[Lin04]{hL2004}
Hans Lindblad, \emph{A remark on global existence for small initial data of the
  minimal surface equation in {M}inkowskian space time}, Proc. Amer. Math. Soc.
  \textbf{132} (2004), no.~4, 1095--1102 (electronic). \MR{MR2045426
  (2005a:35203)}

\bibitem[Lin08]{hL2008}
\bysame, \emph{Global solutions of quasilinear wave equations}, Amer. J. Math.
  \textbf{130} (2008), no.~1, 115--157. \MR{2382144 (2009b:58062)}

\bibitem[Loi06]{jL2006}
Julien Loizelet, \emph{Solutions globales des \'equations
  d'{E}instein-{M}axwell avec jauge harmonique et jauge de {L}orentz}, C. R.
  Math. Acad. Sci. Paris \textbf{342} (2006), no.~7, 479--482. \MR{2214599
  (2007f:83026)}

\bibitem[Loi08]{jL2008}
\bysame, \emph{Probl\`{e}ms globaux en relativit\'{e} general\'{e}}, {PhD}
  dissertation, Universit\`{e} Francois Rabelais, Tours, France, 2008,
  pp.~1--83.

\bibitem[Loi09]{jL2009}
\bysame, \emph{Solutions globales des \'equations d'{E}instein-{M}axwell}, Ann.
  Fac. Sci. Toulouse Math. (6) \textbf{18} (2009), no.~3, 565--610.
  \MR{2582443}

\bibitem[LR03]{hLiR2003}
Hans Lindblad and Igor Rodnianski, \emph{The weak null condition for
  {E}instein's equations}, C. R. Math. Acad. Sci. Paris \textbf{336} (2003),
  no.~11, 901--906. \MR{1994592 (2004h:83008)}

\bibitem[LR05]{hLiR2005}
\bysame, \emph{Global existence for the {E}instein vacuum equations in wave
  coordinates}, Comm. Math. Phys. \textbf{256} (2005), no.~1, 43--110.
  \MR{MR2134337 (2006b:83020)}

\bibitem[LR10]{hLiR2010}
\bysame, \emph{The global stability of {Minkowski} space-time in harmonic
  gauge}, Annals of Mathematics \textbf{171} (2010), no.~3, 1401--1477.

\bibitem[Maj84]{aM1984}
A.~Majda, \emph{Compressible fluid flow and systems of conservation laws in
  several space variables}, Applied Mathematical Sciences, vol.~53,
  Springer-Verlag, New York, 1984. \MR{748308 (85e:35077)}

\bibitem[MNS05]{jMmNcS2005b}
Jason Metcalfe, Makoto Nakamura, and Christopher~D. Sogge, \emph{Global
  existence of quasilinear, nonrelativistic wave equations satisfying the null
  condition}, Japan. J. Math. (N.S.) \textbf{31} (2005), no.~2, 391--472.
  \MR{MR2198183 (2007f:35200)}

\bibitem[MS07]{jMcS2007}
Jason Metcalfe and Christopher~D. Sogge, \emph{Global existence of null-form
  wave equations in exterior domains}, Math. Z. \textbf{256} (2007), no.~3,
  521--549. \MR{MR2299569 (2008j:35128)}

\bibitem[Ple70]{jP1970}
J.~Pleba{\`n}ski, \emph{Lecture notes on nonlinear electrodynamics}, NORDITA
  (1970).

\bibitem[Sid96]{tS1996}
Thomas~C. Sideris, \emph{The null condition and global existence of nonlinear
  elastic waves}, Invent. Math. \textbf{123} (1996), no.~2, 323--342.
  \MR{MR1374204 (97a:35158)}

\bibitem[Sog08]{cS2008}
Christopher~D. Sogge, \emph{Lectures on non-linear wave equations}, second ed.,
  International Press, Boston, MA, 2008. \MR{2455195 (2009i:35213)}

\bibitem[Spe09a]{jS2008b}
Jared Speck, \emph{The non-relativistic limit of the {E}uler-{N}ordstr\"om
  system with cosmological constant}, Rev. Math. Phys. \textbf{21} (2009),
  no.~7, 821--876. \MR{MR2553428}

\bibitem[Spe09b]{jS2008a}
\bysame, \emph{Well-posedness for the {E}uler-{N}ordstr\"om system with
  cosmological constant}, J. Hyperbolic Differ. Equ. \textbf{6} (2009), no.~2,
  313--358. \MR{MR2543324}

\bibitem[Spe10]{jS2010a}
\bysame, \emph{{The nonlinear stability of the trivial solution to the
  Maxwell-Born-Infeld system}}, arXiv preprint: http://arxiv.org/abs/1008.5018
  (2010), 1--73.

\bibitem[SS98]{jSmS1998}
Jalal Shatah and Michael Struwe, \emph{Geometric wave equations}, Courant
  Lecture Notes in Mathematics, vol.~2, New York University Courant Institute
  of Mathematical Sciences, New York, 1998. \MR{MR1674843 (2000i:35135)}

\bibitem[SY79]{rSstY1979}
Richard Schoen and Shing~Tung Yau, \emph{On the proof of the positive mass
  conjecture in general relativity}, Comm. Math. Phys. \textbf{65} (1979),
  no.~1, 45--76. \MR{MR526976 (80j:83024)}

\bibitem[SY81]{rSstY1981}
\bysame, \emph{Proof of the positive mass theorem. {II}}, Comm. Math. Phys.
  \textbf{79} (1981), no.~2, 231--260. \MR{MR612249 (83i:83045)}

\bibitem[Tay97]{mT1997III}
Michael~E. Taylor, \emph{Partial differential equations. {III}}, Applied
  Mathematical Sciences, vol. 117, Springer-Verlag, New York, 1997, Nonlinear
  equations, Corrected reprint of the 1996 original. \MR{MR1477408 (98k:35001)}

\bibitem[Wal84]{rW1984}
Robert~M. Wald, \emph{General relativity}, University of Chicago Press,
  Chicago, IL, 1984. \MR{MR757180 (86a:83001)}

\bibitem[Wit81]{eW1981}
Edward Witten, \emph{A new proof of the positive energy theorem}, Comm. Math.
  Phys. \textbf{80} (1981), no.~3, 381--402. \MR{MR626707 (83e:83035)}

\bibitem[Zip00]{nZ2000}
Nina Zipser, \emph{The global nonlinear stability of the trivial solution of
  the {Einstein-Maxwell} equations}, Ph.D. thesis, Harvard University,
  Cambridge, Massachusetts, 2000.

\end{thebibliography}
\end{document}